\documentclass[10pt]{amsart}
\usepackage[nohug,heads=littlevee]{diagrams_arxiv}
\usepackage{mathrsfs}
\usepackage{bm}
\usepackage{stmaryrd}
\usepackage[bbgreekl]{mathbbol}
\usepackage{amssymb}
\usepackage{relsize}
\usepackage[latin1]{inputenc}
\usepackage{xr-hyper}
\usepackage{xcolor}  
\usepackage{etoolbox}

\usepackage[hypertexnames=false]{hyperref} 
    \definecolor{darkblue}{rgb}{0,0,.85} 
    \definecolor{darkred}{rgb}{0.84,0,0}
    \hypersetup{colorlinks = true,
            linkcolor = darkblue,
            urlcolor  = darkblue,
            citecolor = darkred,
            anchorcolor = darkblue,
            bookmarksdepth=3}

\usepackage[
  left=2cm,
  top=2.5cm,
  right=2cm,
  bottom=1.4in,
  asymmetric,
  marginparwidth=3.2cm,
  marginparsep=6pt
]{geometry}

\usepackage{mathtools}
\usepackage[
  backend=biber,
  maxbibnames=99,
  style=alphabetic,
  maxalphanames=4,
  minalphanames=4,
  backref=true,
  backrefstyle=two,
  useprefix=true
]{biblatex}
\DefineBibliographyStrings{english}{
  backrefpage  = {cited on p\adddot},
  backrefpages = {cited on pp\adddot}
}
\usepackage{enumitem}
\usepackage{tensor}
\usepackage{tikz-cd}

\usepackage{nicefrac}
\newcommand{\wh}[1]{\widehat{#1}}
\newcommand{\mc}[1]{\mathcal{#1}}
\newcommand{\ms}[1]{\mathscr{#1}}
\newcommand{\mr}[1]{\mathrm{#1}}
\newcommand{\ov}[1]{\overline{#1}}

\newcommand*\isomto{%
        \,\xrightarrow{\raisebox{-0.2 em}{\smash{\ensuremath{\sim}}}}\,%
    }

    \newcommand*\isomfrom{%
        \,\xleftarrow{\raisebox{-0.2 em}{\smash{\ensuremath{\sim}}}}\,%
    }

 \newcommand{\stacks}[1]{\cite[\href{https://stacks.math.columbia.edu/tag/#1}{Tag~#1}]{stacks-project}}

 \newcommand{\smallprism}{\Prism}
\newcommand{\smallN}{{{\mathsmaller{\mathcal{N}}}}}

\DeclareMathOperator{\Spd}{Spd}

\renewcommand{\ll}{\llbracket}
\newcommand{\rr}{\rrbracket}

\newcommand{\proet}{{\mathrm{pro\acute{e}t}}}

    \makeatletter
    \def\paragraph{\@startsection{paragraph}{4}%
    \z@\z@{-\fontdimen2\font}%
    {\normalfont\bfseries}}
    \makeatother

\usepackage{ifoddpage}
\usepackage{ragged2e} 



\newcommand{\defnword}[1]{\textbf{#1}}
\newcommand{\comment}[1]{}
\newcommand{\on}[1]{\operatorname{#1}}
\newcommand{\bb}[1]{\mathbb{#1}}
\newcommand{\mb}[1]{\mathbf{#1}}
\newcommand{\mf}[1]{\mathfrak{#1}}
\newcommand{\ord}{\on{ord}}

\DeclareSymbolFontAlphabet{\mathbb}{AMSb}
\DeclareSymbolFontAlphabet{\mathbbl}{bbold}

\numberwithin{equation}{subsubsection}
\newtheorem{introthm}{Theorem}

\newtheorem{proposition}[subsubsection]{Proposition}
\newtheorem{theorem}[subsubsection]{Theorem}
\newtheorem{conjecture}[subsubsection]{Conjecture}
\newtheorem*{thm*}{Theorem}
\newtheorem{lemma}[subsubsection]{Lemma}
\newtheorem*{lem*}{Lemma}
\newtheorem{corollary}[subsubsection]{Corollary}

\theoremstyle{definition}
\newtheorem{definition}[subsubsection]{Definition}

\newtheorem{notation}[subsubsection]{Notation}
\newtheorem{convention}[subsubsection]{Convention}
\newtheorem{construction}[subsubsection]{Construction}
\newtheorem{setup}[subsubsection]{Setup}

\makeatletter
\BeforeBeginEnvironment{setup}{%
  \let\setup@oldthesubsubsection\thesubsubsection
  \renewcommand{\thesubsubsection}{%
    \thesection
    \ifnum\value{subsection}>0 .\arabic{subsection}\fi
    .\arabic{subsubsection}%
  }%
}
\AfterEndEnvironment{setup}{%
  \let\thesubsubsection\setup@oldthesubsubsection
}
\makeatother

\newtheorem{example}[subsubsection]{Example}

\newtheorem{remark}[subsubsection]{Remark}
\newtheorem{assumption}[subsubsection]{Assumption}

\newtheorem*{assump*}{Assumption}

\newarrow{Equals}{=}{=}{}{=}{}
\newarrow{Onto}{-}{-}{-}{-}{>>}
\newcommand{\Square}[8]{\begin{diagram}
  #1&\rTo^{#2}&#3\\
  \dTo^{#4}&&\dTo_{#5}\\
  #6&\rTo_{#7}&#8
\end{diagram}.
}

\newcommand{\Real}{\mathbb{R}}
\newcommand{\Int}{\mathbb{Z}}
\newcommand{\Nat}{\mathbb{N}}
\newcommand{\Comp}{\mathbb{C}}
\newcommand{\Aff}{\mathbb{A}}
\newcommand{\Adele}{\mathbb{A}}
\newcommand{\Field}{\mathbb{F}}
\newcommand{\Rat}{\mathbb{Q}}
\newcommand{\Lie}{\on{Lie}}
\newcommand{\Dieu}{\mathbb{D}}

\newcommand{\Map}{\on{Map}}

\newcommand{\Prism}{\mathbbl{\Delta}}
\newcommand{\BT}[2][\mathcal{G},\mu]{\mathrm{BT}^{#1}_{#2}}
\newcommand{\Wind}[3][\mathcal{G},\mu]{\mathrm{Wind}^{#1}_{#2,#3}}

\makeatletter
\newcommand{\colim@}[2]{%
  \vtop{\m@th\ialign{##\cr
    \hfil$#1\operator@font colim$\hfil\cr
    \noalign{\nointerlineskip\kern1.5\ex@}#2\cr
    \noalign{\nointerlineskip\kern-\ex@}\cr}}%
}
\newcommand{\colim}{%
  \mathop{\mathpalette\colim@{\rightarrowfill@\textstyle}}\nmlimits@
}
\makeatother

\newcommand{\gr}{\on{gr}}

\newcommand{\id}{\on{id}}

\newcommand{\Rees}{\mathcal{R}}
\newcommand{\defn}{\overset{\mathrm{defn}}{=}}

\newcommand{\Rg}{\mathscr{O}}
\newcommand{\Reg}[1]{\Rg_{#1}}

\newcommand{\Hom}{\on{Hom}}
\newcommand{\End}{\on{End}}

\newcommand{\Aut}{\on{Aut}}

\newcommand{\Gal}{\on{Gal}}

\newcommand{\Gm}{\mathbb{G}_m}

\newcommand{\Gmh}[1]{\mathbb{G}_{m,#1}}

\newcommand{\dR}{\on{dR}}

\newcommand{\Sig}{\mathfrak{S}}
\newcommand{\pow}[1]{\ll #1\rr}

\newcommand{\Rep}{\on{Rep}}

\newcommand{\Spec}{\on{Spec}}
\newcommand{\Spf}{\on{Spf}}
\newcommand{\Spa}{\on{Spa}}

\newcommand{\et}{\on{\acute{e}t}}
\newcommand{\ab}{\on{ab}}

\newcommand{\Fil}{\on{Fil}}

\newcommand{\mx}{\mathfrak{m}}

\newcommand{\rank}{\on{rank}}

\newcommand{\Sh}{\on{Sh}}
\newcommand{\Ss}{\mathscr{S}}

\newcommand{\Res}{\on{Res}}
\newcommand{\an}{\on{an}}

\newcommand{\ad}{\on{ad}}
\newcommand{\GSp}{\on{GSp}}

\newcommand{\GL}{\on{GL}}

\newcommand{\SO}{\on{SO}}

\bibliography{syntomic}

\begin{document}
\title[Canonicity of integral models of Shimura varieties]{On canonicity for integral models of Shimura varieties with hyperspecial level}
\author{Keerthi Madapusi}
\address{Keerthi Madapusi\\
Department of Mathematics\\
Maloney Hall\\
Boston College\\
Chestnut Hill, MA 02467\\
USA}
\email{madapusi@bc.edu}
\author{Alex Youcis}
\address{Alex Youcis, Department of Mathematics, Bahen Centre, University of Toronto, Toronto, ON, M5S 2E4, Canada}
\email{alex.youcis@gmail.com}

\begin{abstract}
We give a new definition---and in some cases, a new construction---of integral canonical models of Shimura varieties that uses the notion of an aperture appearing in work of Gardner--Madapusi on some conjectures of Drinfeld. This applies to Shimura varieties of pre-abelian type at odd primes of hyperspecial level, recovering and extending previous work of Kisin, Kim--Madapusi and Imai--Kato--Youcis, but also to exceptional Shimura varieties for large enough primes. The characterization in the exceptional case is \emph{a priori} different from the one recently shown by Bakker--Shankar--Tsimerman, and recovers many of their results, such as the existence of prime-to-$p$ Hecke operators, the non-emptiness of the $\mu$-ordinary stratum and the theory of the canonical lift. In fact, we give a uniform proof of the non-emptiness of \emph{all} possible Newton strata, and of the non-emptiness of Ekedahl--Oort strata and central leaves as well. An important ingredient in the proofs is a generalization of Tate's full faithfulness theorem for $p$-divisible groups to the context of apertures. This leads to a mapping property for the integral canonical model that characterizes maps into it from all normal, flat and excellent schemes over $\mathbb{Z}_{(p)}$.
\end{abstract}

\maketitle

\setcounter{tocdepth}{1}

\tableofcontents

\section{Introduction}

The purpose of this article is to lay the groundwork for applying the results of~\cite{gmm} to global questions. In doing so, we give a \emph{$p$-adic} characterization of integral canonical models of Shimura varieties at unramified primes, including for those of exceptional type considered by Bakker--Shankar--Tsimerman~\cite{bst}. This characterization is \emph{a priori} different from the $\ell$-adic ones given in \emph{op.\@ cit.\@} and others, and is closely related to those introduced by Pappas~\cite{pappas_integral} and Imai--Kato--Youcis~\cite{imai2023prismatic}. For Shimura varieties of abelian type, our method involves a new interpretation of Kisin's twisting construction, and we further extend it to give models for Shimura data of \emph{pre}-abelian type.\footnote{Pre-abelian Shimura data have a simple type-theoretic characterization: the adjoint datum is of type $A,B, C, D^{\Real}$, or $D^{\mathbb{H}}$. The first non-trivial example of one that is not of abelian type is a type $D^{\mathbb{H}}$ Shimura datum with simply connected derived subgroup---see Example \ref{ex:non-abelian_pre-abelian}.} 

\subsection{Definition and characterization}

The key observation animating this paper is that apertures, the mixed characteristic analogues of the local shtukas of Genestier--Lafforgue, introduced by Drinfeld~\cite{drinfeld2024shimurian} and studied in~\cite{gmm}, have very strong connections with the integral theory of Shimura varieties, and in fact characterize the `good' models completely, reifying an expectation of Drinfeld~\cite{drinfeld2024shimurian}. Significant progress along these lines was made in~\cite{imai2023prismatic}, where it is shown that the prismatic realization on an integral model of a Shimura variety of abelian type is sufficient to pin down the model at \emph{finite} level, without considering the whole prime-to-$p$ Hecke tower. 

Our primary aim here is to situate this result in what we believe to be the right context, and to substantially broaden the kinds of Shimura data it applies to by exploiting the geometry made available by the theory in~\cite{gmm}.

\begin{setup}
  Fix a prime $p$. An \defnword{unramified Shimura datum} is a triple $(G,\mathcal{G},X)$ where $(G,X)$ is a Shimura datum and $\mathcal{G}$ is a reductive model of $G$ over $\Int_p$. We set $K_p = \mathcal{G}(\Int_p)$, and (for now) only consider level subgroups $K\subset G(\Adele_f)$ of the form $K_pK^p$ for a \emph{neat} compact open subgroup $K^p\subset G(\Adele_f^p)$. We call the tuple $(G,\mathcal{G},X,K)$ an \defnword{unramified Shimura tuple}. We also fix a place $v\vert p$ of the reflex field $E$ of $(G,X)$. Over the associated Shimura variety $\Sh_K$, we have a canonical $\mathcal{G}^c(\Int_p)$-local system $\mathbf{Et}_{K,p}$, where $\mathcal{G}^c$ is a reductive model over $\bb{Z}_p$ of the cuspidal quotient $G^c$ (see \S\ref{subsec:etale_realization_global}). 
\end{setup}

\begin{remark}
  The Shimura datum gives an $E_v$-rational conjugacy class of cocharacters $\{\mu_v\}$. We in fact have a representative $\mu_v\colon\Gmh{\Reg{E_v}}\to \mathcal{G}_{\Reg{E_v}}$ for $\{\mu_v\}$. We may view each $\mu_v$ as a cocharacter of $\mathcal{G}^c_{\Reg{E_v}}$. We can then attach for each $1\leqslant n<\infty$ the smooth formal algebraic stacks $\BT[\mathcal{G}^c,\mbox{-}\mu_v]{n}$ over $\Reg{E_v}$ constructed in~\cite{gmm}. We can also consider their inverse limit $\BT[\mathcal{G}^c,\mbox{-}\mu_v]{\infty}$. 
  
  These stacks generalize the formal stacks of (truncated) polarized $p$-divisible groups of height $2g$ in the case of the Siegel Shimura data, and give a group theoretic way of defining `$p$-divisible groups with $\mathcal{G}^c$-structure' without actually requiring one to work with $p$-divisible groups. Moreover, they can be constructed from data that is not necessarily of `Hodge type', and so in cases where there might not be any direct connection with $p$-divisible groups at all.
\end{remark}

\begin{remark}
  For any $p$-adic formal algebraic space $\mathfrak{X}$, there is a canonical \'etale realization functor
  \[
    T_{\et}\colon \BT[\mathcal{G}^c,\mbox{-}\mu_v]{\infty}(\mathfrak{X})\to \mathrm{Loc}_{\mathcal{G}^c(\Int_p)}(\mathfrak{X}_\eta),
  \]
  where the right-hand side is the groupoid of $\mathcal{G}^c(\Int_p)$-local systems over the generic fiber of $\mathfrak{X}$. When $(G,X)$ is the Siegel Shimura datum, then what we have here is simply the functor taking a polarized $p$-divisible group to its Tate module viewed as a symplectic space.
\end{remark}

\begin{definition}
[Integral canonical models]
\label{introdef:crys_icms}
 An integral model\footnote{By `integral model', we will mean a separated \emph{algebraic space} of finite type over $\Reg{E,(v)}$ whose generic fiber is equipped with an isomorphism to $\Sh_K$ as an $E$-scheme. In practice, the models we encounter will be quasi-projective schemes, but, even for Siegel modular varieties, knowing this requires a non-trivial result of Mumford.} $\Ss_K$ for $\Sh_K$ over $\Reg{E,(v)}$ with $v$-adic completion $\widehat{\Ss}_{K}$ is an \defnword{integral canonical model} (\defnword{ICM} for short) if the following conditions hold:
 \begin{enumerate}
 \item (Serre--Tate property) There exists a formally \'etale map of $p$-adic formal stacks over $\Spf\Reg{E,v}$
  \[
      \varpi\colon \widehat{\Ss}_{K}\to \BT[\mathcal{G}^c,\mbox{-}\mu_v]{\infty},
  \]
     such that the induced $\mc{G}^c(\bb{Z}_p)$-local system on the generic fiber is isomorphic to $\mathbf{Et}_{K,p}$.
  \item (Pointwise criterion) For any mixed characteristic $(0,p)$ complete discrete valuation field $F$ over $\Reg{E_v}$ with perfect residue field and $x\in \Sh_K(F)$\footnote{Here and elsewhere, we will write $\Sh_K(F)$ for what should properly be written as $(\Sh_K\otimes_E E_v)(F)$.}, the following are equivalent:
 \begin{enumerate}
      \item $x\in \Ss_K(\Reg{F})$;
    \item $\mathbf{Et}_{K,p}|_x$ is a crystalline $\mathcal{G}^c(\Int_p)$-local system;
    \item $\mathbf{Et}_{K,p}|_x$ is a potentially crystalline $\mathcal{G}^c(\Int_p)$-local system.
  \end{enumerate}
 \end{enumerate}
\end{definition}

With this definition in hand, we get the following quite strong mapping property.
\begin{introthm}
[Mapping property of ICMs, Theorem~\ref{thm:mapping_property}]
\label{introthm:mapping_property}
  Let $\Ss_K$ be an ICM over $\Reg{E,(v)}$. Suppose that we have a normal flat separated excellent algebraic space $\mathcal{X}$ over $\Reg{E,(v)}$. Then giving a map $\mathcal{X}\to \Ss_K$ of $\Reg{E,(v)}$-schemes is equivalent to giving the following data:
\begin{enumerate}
    \item A map $f\colon \mathcal{X}_\eta\to \Sh_K$;
    \item A map $\eta\colon \widehat{\mathcal{X}}\to \BT[\mathcal{G}^c,\mbox{-}\mu_v]{\infty}$ with the associated $\mathcal{G}^c(\Int_p)$-local system on $\wh{\mc{X}}_\eta$ isomorphic to $\mathbf{Et}_{K,p}|_{\wh{\mc{X}}_\eta}$.
\end{enumerate}
In particular, the ICM $\Ss_K$ is unique up to unique isomorphism, and is functorial in the unramified tuple $(G,\mathcal{X},X,K)$.
\end{introthm}

\begin{remark}
    The proof of the theorem is based on an argument of Pappas~\cite[Theorem 7.1.7]{pappas_integral}. His use of Tate's full faithfulness theorem for $p$-divisible groups is replaced with that of a version for apertures; see \S\ref{introsub:tate} below. A closely related argument also appears in~\cite[Theorem 3.13]{imai2023prismatic}.
\end{remark}

\begin{remark}
   The mapping property in Theorem~\ref{introthm:mapping_property} was essentially established for the natural integral models of Siegel modular varieties by Vasiu~\cite{MR2043397}: In this case, one can interpret $\BT[\mathcal{G}^c,-\mu_v]{\infty}$ as the formal stack of polarized $p$-divisible groups of some fixed height.
\end{remark}

\begin{remark}
    [Mapping property for Pappas--Rapoport models]
\label{introrem:mapping_PR_models}
The method of proof of Theorem~\ref{introthm:mapping_property} combined with a fully faithfulness result of Pappas--Rapoport~\cite{pappas2023padic} establishes a mapping property for integral models at \emph{parahoric} levels satisfying the axioms from \emph{loc.\@ cit.\@}; see Theorem~\ref{thm:PR ICM-mapping-property}.  
\end{remark}

For $p>3$, we have a simpler mapping property for ICMs with respect to smooth $\Reg{E,(v)}$-schemes. 
\begin{introthm}
  [Pointwise mapping property, Proposition~\ref{prop:pointwise_smooth_canonical}]
\label{introthm:refined_neronian}
Suppose that $p>3$, that $\Ss_K$ is an ICM for $\Sh_K$, and that $\mathcal{X}$ is a smooth $\Reg{E,(v)}$-scheme with generic fiber $X$. Then a map $f\colon X\to \Sh_K$, extends to a map $\mathcal{X}\to \Ss_K$ if and only if, for all finite extensions $F/E_v$ and all $x\in \mathcal{X}(\Reg{F})$ the local system $\mathbf{Et}_{K,p}|_x$ is crystalline. In particular, if $\Ss_K$ is \emph{proper} over $\Reg{E,(v)}$, then it is a N\'eron model for $\Sh_K$ as in \emph{~\cite[Definition 1]{blr}}: For any smooth $\Reg{E,(v)}$-scheme $S$, every map $S_\eta\to \Sh_K$ of $E$-schemes extends to a map $S\to \Ss_K$.
\end{introthm}

In terms of existence of such models, we have:
\begin{introthm}
[Pre-abelian-type canonical models, Theorem~\ref{thm:abelian}]
\label{introthm:abelian}
  Suppose that one of the following conditions holds: 
  \begin{enumerate}
      \item $(G,X)$ is of abelian type;
      \item $p>2$ and $(G,X)$ is of pre-abelian type.
  \end{enumerate}
   Then, for $K^p$ sufficiently small, $\Sh_K$ admits an ICM over $\Reg{E,(v)}$. In fact, in case (1), the models constructed by Kisin~\emph{\cite{kisin:abelian}} (and Kim--Madapusi~\emph{\cite{Kim2016-fb}} for $p=2$) are ICMs as defined here.
\end{introthm}

We also have the following version of a recent theorem of Bakker--Shankar--Tsimerman~\cite{bst}. 
\begin{introthm}
[Exceptional integral canonical models, \S\ref{sub:bst}]
\label{introthm:exceptional}
  For an unramified Shimura tuple $(G,\mathcal{G},X,K)$, and all $p$ sufficiently large, an ICM $\Ss_K$ exists over $\Reg{E,(v)}$ for any $v\vert p$. 
\end{introthm}

\begin{remark}
  Given Theorem~\ref{introthm:abelian}, the content of Theorem~\ref{introthm:exceptional} is of course that it applies to the so-called exceptional Shimura data that are not of pre-abelian type and hence cannot be related to the moduli of abelian varieties in any reasonably direct way. As in~\cite{bst},  the models here are obtained from a spreading out procedure, which leads to the lack of control on prime $p$.
\end{remark}

See \S\ref{introsub:comparison} below for a more detailed comparison with existing results.

\begin{remark}
    [Yet another canonical model?]
\label{rem:yacm}
The experienced reader may be wondering at this point if the world really needs yet another notion of integral canonical models for Shimura varieties. The arguments we would proffer in favor of the definition here---which we believe to be the \emph{correct} one---are as follows: 
\begin{enumerate}
    \item First, of course, one has the quite pleasing mapping property from Theorem~\ref{introthm:mapping_property}.

   \item Accounting systematically for the prime-to-$p$ realizations, one can deduce various non-trivial properties for such models using \emph{only} their defining features, without invoking any moduli interpretation.
   
   \item Finally, in forthcoming work of the first author with Si Ying Lee, one will find that various results about integral models of abelian type (Rapoport--Zink uniformization, existence of CM lifts, Igusa stacks, Langlands--Rapoport type point counting formulas, etc.), for which the use of abelian varieties has hitherto seemed essential, are in fact \emph{properties} of limpid ICMs as defined in the next subsection.
\end{enumerate}
\end{remark}

\subsection{Limpid ICMs and their properties}
\label{introsub:limpid_icms}

To present our further results, we will need a slightly enhanced definition of an ICM.

\begin{remark}
[Prime-to-$p$ \'etale realizations]
    For $\ell\neq p$, we also have $K^c_{\ell}$-local systems $\mathbf{Et}_{K,\ell}$ over $\Sh_K$, with $K^c_{\ell}\subset G^c(\Rat_\ell)$ a compact open subgroup given by the image of $K$ under the projection $G(\Adele_f)\to G^c(\Rat_\ell)$.
\end{remark}

\begin{definition}
[Limpid integral canonical models]
 An ICM $\Ss_K$ for $\Sh_K$ over $\Reg{E,(v)}$ is \defnword{limpid} if, for all $\ell\neq p$, the $K^c_\ell$-local system $\mathbf{Et}_{K,\ell}$ extends over $\Ss_K$.\footnote{Note that this is a \emph{property} of the unique ICM when it exists.}
\end{definition}

\begin{remark}
    [Pre-abelian ICMs are limpid]
\label{rem:preabelian_limpid}
 In the context of Theorem~\ref{introthm:abelian}, ICMs are in fact limpid. This follows from our proof, which uses the three theorems of this subsection---theorems that do not use any sort of moduli interpretation---to reduce to the essentially standard fact that the moduli of principally polarized abelian schemes with appropriate level structure are limpid ICMs for Siegel modular varieties.
\end{remark}

\begin{remark}
    [Exceptional ICMs are limpid]
\label{introrem:exceptional_icms_limpid}
 The ICMs appearing in Theorem~\ref{introthm:exceptional} are also limpid for large enough $p$: This follows from work of Klevdal--Patrikis~\cite{Klevdal2025-ka}.
\end{remark}

\begin{introthm}
[Integral canonical models from reduction of structure group, \S\ref{sub:reduction_of_structure_group}]
\label{introthm:reduction_of_structure}
Suppose that we have a map $(G,\mathcal{G},X,K)\to (G^\sharp,\mathcal{G}^\sharp,X^\sharp,K^\sharp)$ of unramified Shimura tuples with $\mathcal{G}\to \mathcal{G}^\sharp$ a closed immersion of reductive group schemes. If $\Sh_{K^\sharp}$ admits an ICM over $\Reg{E^\sharp,(v^\sharp)}$ (for the place $v^\sharp\mid p$ of the reflex field $E^\sharp\subset E$ lying under $v$), then so does $\Sh_K$. Moreover, if the ICM for $\Sh_{K^\sharp}$ is limpid, then so is the one for $\Sh_K$.
\end{introthm}

\begin{remark}
  Theorem~\ref{introthm:reduction_of_structure} is already known by work of Pappas--Rapoport and Daniels--van Hoften--Kim--Zhang in the context of Pappas--Rapoport canonical models; see~\cite[Theorem 4.1.8]{daniels2025conjecturepappasrapoport}.
\end{remark}

Next, we have two results concerning the behavior of limpid ICMs under central modifications.

\begin{setup}
    \label{introsetup:central}
We will fix a map of unramified Shimura tuples
$(G_1,\mathcal{G}_1,X_1,K_1) \to (G_2,\mathcal{G}_2,X_2,K_2)$ with reflex fields $E_2\subset E_1$ such that $G_1\to G_2$ is surjective with central kernel. We will also fix a place $v_1\vert p$ of $E_1$ and let $v_2\vert p$ be the place of $E_2$ under it.
\end{setup}

First, we have the following \emph{ascent} statement. which implies that constructing ICMs for adjoint Shimura data is sufficient to construct them for all Shimura data:
\begin{introthm}
    [Ascent of ICMs along central isogenies, Theorem~\ref{thm:ascent}]
\label{introthm:ascent}
Suppose that $p>2$ and that $G_2 = G_2^c$. If $\Sh_{K_2}$ admits a limpid ICM $\Ss_{K_2}$ over $\Reg{E_2,(v_2)}$, then $\Sh_{K_1}$ admits a limpid ICM $\Ss_{K_1}$ over $\Reg{E_1,(v_1)}$, and the map $\Ss_{K_1}\to \Ss_{K_2}$ is finite \'etale. 
\end{introthm}

\begin{remark}
    There is an obvious way to construct an integral model $\Ss_{K_2}$ for $\Sh_{K_2}$ from an integral model $\Ss_{K_1}$ for $\Sh_{K_1}$: One can take the normalization of $\Ss_{K_1}$ in $\Sh_{K_2}$. The issue is to show that this preserves the property of being a limpid ICM. The heart of the matter is to see that $\Ss_{K_1}$ is finite \'etale over $\Ss_{K_2}$. We prove this by showing it holds for the restriction to the $\mu$-ordinary locus and using purity. Over the $\mu$-ordinary locus, we need the theory of the canonical lift; see Theorem~\ref{introthm:mu_ordinary} below. The fact that the canonical lift is CM makes use of the prime-to-$p$ realizations, and so requires the additional limpidity hypothesis.
\end{remark}

\begin{remark}
    The idea of using the $\mu$-ordinary locus and canonical lifts to prove something like an ascent theorem shows up in a different form in work of Vasiu~\cite[\S 4.9]{vasiu2001pointsintegralcanonicalmodels}.
\end{remark}

\begin{remark}
    [Sufficiency of adjoint data]
\label{introrem:adjoint_is_enough}
Theorem~\ref{introthm:ascent} tells us that, when $p>2$, in order to construct limpid ICMs for tuples with underlying data $(G,\mathcal{G},X)$, it suffices to do so for tuples with \emph{adjoint} Shimura data $(G^{\ad},\mathcal{G}^{\ad},X^{\ad})$.
\end{remark}

Next, we have an abstract form of the \emph{twisting} construction from~\cite{Kisin2017-qa}, which uses the modified Weil restriction given in Definition~\ref{defn:modified_weil_restriction} below. 
\begin{introthm}
    [Descent of ICMs along central isogenies, Theorem~\ref{thm:descent_for_icms}]
\label{introthm:descent}
Suppose that  $G = G^c$ and that the following additional condition holds for every totally real number field $F$ in which $p$ is unramified:  All unramified Shimura tuples with underlying unramified Shimura datum  $\Res'_{F/\Rat}\left(G_1,\mathcal{G}_1,X_1\right)$ admit limpid ICMs over $\Reg{E_1,(v_1)}$, and $Z(G^{\mathrm{der}})(\Adele_f^p\otimes F)/Z(G^{\mathrm{der}})(\Reg{F,(p)})$ acts freely on the prime-to-$p$ Hecke tower of such ICMs. Then, for $K_2^p$ sufficiently small, $\Sh_{K_2}$ admits a limpid ICM $\Ss_{K_2}$ over $\Reg{E_2,(v_2)}$, and the map $\Ss_{K_1}\to \Ss_{K_2}$ is finite \'etale.
\end{introthm}

\begin{remark}
[Role of the modified Weil restriction condition]
    The role of the condition is to overcome the fact that, in the generality that we work with here, we do not have a good \emph{intrinsic} notion of `$F$-linear isogenies' between points on the Shimura variety. 
    
    In the context of Shimura varieties of Hodge type, such isogenies, defined through the $F$-linearization of the isogeny category of abelian varieties, are an important actor in Kisin's twisting construction. This is replaced here by a systematic use of the ICMs for the modified Weil restrictions (in fact from the actual Weil restrictions). The condition of being of Hodge type is preserved under the modified Weil restriction construction, and so we obtain a different perspective on the construction of Shimura varieties of abelian type from those of Hodge type. More than anything else, this is a proof of concept that one can replace arguments using isogenies between abelian varieties with ones involving structural properties of ICMs.
\end{remark}

\subsection{Comparison with existing results}
\label{introsub:comparison}

\begin{remark}
  [Comparison with existing notions of canonicity]
\label{rem:different_notions_of_canonicity}
  Our notion of canonicity is \emph{different} from those appearing in~\cite{moonen:models}, \cite{kisin:abelian} or \cite{bst}, which use $\ell$-adic realizations for $\ell\neq p$ to characterize the model. Also, the characterizations appearing in the first two of these, are not of models at any finite level, but of the whole prime-to-$p$ Hecke tower. As such, they are slightly tricky to formulate and check. Also, our definition of ICMs is \emph{intrinsic} to an integral model, with the mapping property arising as a consequence, while the other three require some external input: consideration of the full prime-to-$p$ Hecke tower, or the existence of a sufficiently nice compactification.
  
  The notion of an ICM used here was essentially introduced by Imai--Kato--Youcis~\cite{imai2023prismatic}, and this in turn has its antecedents in the work of Pappas~\cite{pappas_integral} and Pappas--Rapoport~\cite{pappas2023padic}. The important observation that this should give a characterization at \emph{finite level} is found in~\cite{imai2023prismatic}.
\end{remark}

\begin{remark}
  [Comparison with the work of Kisin]
\label{introrem:kisin}
The first general construction of integral canonical models of abelian type is due to Kisin~\cite{kisin:abelian}; see also the earlier work of Vasiu~\cites{vasiu:preab,vasiu2001pointsintegralcanonicalmodels}. For Shimura varieties of Hodge type, what one will find in this paper is essentially a streamlined version of the arguments in Kisin's work, making use of the geometry of the stacks $\BT[\mathcal{G}^c,\mbox{-}\mu_v]{n}$ to compress some of the $p$-adic Hodge theory used there. 

To construct Shimura varieties of abelian type, Kisin introduces his \emph{twisting} construction, which ultimately uses the global moduli interpretation of Siegel-type Shimura varieties in terms of abelian varieties. We actually give a \emph{different} approach to twisting here. The main new observation here is that one does not need a moduli description: One can formulate the construction in a way that is valid for all maps between Shimura data where the underlying map of groups is a central cover. We also make the additional new observation that the theory of the canonical lift for $\mu$-ordinary points suffices to also \emph{lift} ICMs along central covers, at least when $p>2$. This leads to our construction of ICMs of pre-abelian type.
\end{remark}

\begin{remark}
 [Comparison with Imai--Kato--Youcis]
\label{introrem:imai_kato_youcis}
That the models constructed by Kisin are ICMs in the sense used here is originally due to Imai--Kato--Youcis~\cite{imai2023prismatic}. The proof we give here is different from that in \cite{imai2023prismatic} and works also when $p=2$ where the models were constructed by Kim--Madapusi. But we do make use of intermediate results from~\cites{imai2023prismatic,imai2024tannakianframeworkprismaticfcrystals,IKY3}.
\end{remark}

\begin{remark}
[Comparison with Bakker--Shankar--Tsimerman]
\label{introrem:bst_comparison}
The idea that one should be able to \emph{provably} characterize integral canonical models of \emph{all} Shimura varieties, including those not of abelian type, first appears in the already mentioned recent work of Bakker--Shankar--Tsimerman. Their results extend beyond the ambit of Shimura varieties, and apply to integral models of \emph{non-minuscule} variations of Hodge structure. This is in some sense the main philosophical difference between these works: Here, we make full use of the minusculeness (though in the sense of $G$-bundles, not just vector bundles) of the family of realizations over a Shimura variety, for instance, in the form of the full faithfulness result Theorem~\ref{introthm:tate}. But these methods say nothing about what happens beyond the minuscule regime.

The ICMs here are also canonical in the sense of \cite{bst} for $p$ large enough, and our proof of canonicity is independent from theirs, though it uses essentially the same inputs. Moreover, the lower bound on the prime $p$ is \emph{a priori} smaller than the one in \emph{op.\@ cit.}: for instance, it depends only on the underlying adjoint Shimura datum.
\end{remark}

\subsection{Applications}

Our first application is essentially an algebraization result for apertures. 

\begin{introthm}
[Surjectivity of the syntomic realization, Theorem~\ref{thm:surjectivity}]
\label{introthm:surjective}
  Suppose that $\Sh_K$ admits an ICM $\Ss_K$ over $\Reg{E,(v)}$ and that one of the following conditions holds: 
  \begin{enumerate}
      \item The Shimura datum $(G,X)$ is of pre-abelian type.
      \item The prime $p$ is sufficiently large (as in \emph{Theorem \ref{introthm:exceptional}}).
  \end{enumerate}
 Then, for all $1\leqslant n\leqslant \infty$, the map
  \[
    \widehat{\Ss}_{K}\to \BT[\mathcal{G}^c,\mbox{-}\mu_v]{n}
  \]
  is surjective, and it is smooth if $n<\infty$.
\end{introthm}

\begin{remark} \label{rem:de-Jong-extension}
The proof of Theorem~\ref{introthm:surjective} shows that the surjectivity holds when $(G,\mc{G},X,K)$ satisfies the \defnword{de Jong extension property} given in Definition \ref{def:de_jong_property} below. The cases listed in Theorem~\ref{introthm:surjective} are the ones where we can verify this extension property; but we expect that it is valid for all ICMs.
\end{remark}

\begin{remark}
  The surjectivity assertion above is already known for Shimura varieties of PEL type by work of Viehmann--Wedhorn~\cite{ViehmannWedhornPEL}. Their proof amounts to first establishing the non-emptiness of Newton strata, followed by an argument via isogenies. A similar argument works for Shimura varieties for Hodge and perhaps even abelian type, but one has to work much harder by appealing to the non-emptiness results of~\cite{Lee2018-tj} or~\cite{Kisin2022-eo}, as well as to some deeper results from~\cite{Kisin2017-qa} on $p$-isogenies between mod-$p$ points of the Shimura variety. 
  
  In contrast, the proof given here is geometric, works uniformly for all known ICMs, and requires no \emph{a priori} knowledge of non-emptiness statements or the existence of isogenies. In fact, it can be used to \emph{deduce} such results.
\end{remark}

\begin{introthm}
[Non-emptiness of Ekedahl--Oort strata, \S\ref{sub:global_mu-ord} and \S\ref{sub:non-emptiness}]
\label{introthm:Ekedahl--Oort}
Suppose that $\Ss_K$ is an ICM. Then there exists a smooth map of stacks
\[
\Ss_{K}\otimes k(v)\to \mathcal{G}^c\text{-}\,\mathrm{zip}_{\mbox{-}\mu_v}.
\]
Here, the right-hand side is the stack of $\mathcal{G}^c$-zips of type $\mbox{-}\mu_v$ defined in~\emph{\cite{Pink2015-ye}}, and admits a canonical open and dense stratum known as the \emph{$\mu$-ordinary stratum}. Moreover:
\begin{enumerate}
    \item The $\mu$-ordinary locus of $\Ss_{K}\otimes k(v)$ (which is the pre-image of the $\mu$-ordinary stratum) is dense.
    \item If $\Ss_K$ satisfies the hypotheses of \emph{Theorem~\ref{introthm:surjective}}, then this map is also surjective: That is, all Ekedahl--Oort strata of the special fiber are non-empty.
\end{enumerate}
\end{introthm}

\begin{remark}
    The density of the $\mu$-ordinary locus is shown in~\cite{Kisin2022-eo} for most Shimura varieties of Hodge type with parahoric level at $p$, extending results of Wortmann~\cite{wortmann:ordinary}. It is also known in the exceptional case if $p$ is sufficiently large by work of Bakker--Shankar--Tsimerman~\cite[\S8]{bst}.
\end{remark}

\begin{remark}
   The non-emptiness of Ekedahl--Oort strata for Shimura varieties of PEL type is due to Viehmann--Wedhorn~\cite{ViehmannWedhornPEL}, and is also known in the abelian type case by work of Shen--Zhang~\cite{Shen2017-sf} (though also see earlier work of Wortmann~\cite{wortmann:ordinary} and Zhang \cite{Zhang_2018}). The proof of Shen--Zhang once again uses deep results of Kisin from~\cite{Kisin2017-qa}. On the other hand, our proof is essentially an immediate consequence of Theorem~\ref{introthm:surjective}, and the fact (shown in~\cite{gmm}) that there is a natural smooth surjective map $\BT[\mathcal{G}^c,\mbox{-}\mu_v]{1}\otimes\Field_p\to \mathcal{G}^c\text{-}\,\mathrm{zip}_{\mbox{-}\mu_v}$. 
\end{remark}

There is another stratification where we can show non-emptiness for an abstract ICM (i.e., without requiring an assumption as in Remark \ref{rem:de-Jong-extension}): the \emph{Newton stratification} using the Kottwitz set $B(G^c,\{\mbox{-}\mu_v\})$ defined in \cite{kottwitz:isoc} and studied in \cite{rapoport_richartz}.
\begin{introthm}
[Non-emptiness of Newton strata, Theorem~\ref{thm:newton_non_emptiness}]
\label{introthm:newton_strata}
  Let $\Ss_K$ be an ICM for $\Sh_K$ over $\Reg{E,(v)}$. Then the natural map $\Ss_{K}\to B(G^c,\{\mbox{-}\mu_v\})$ is surjective.
\end{introthm}

Next, we have applications to the $\mu$-ordinary locus.
\begin{introthm}
    [Canonical CM lifts for $\mu$-ordinary points, Theorems~\ref{thm:properties_mu_ord_locus},~\ref{thm:canonical_lifts_are_CM}]
\label{introthm:mu_ordinary}
Suppose that $\Ss_K$ is an ICM for $\Sh_K$ and that $x\in \Ss_K(\kappa)$ is a closed point in the $\mu$-ordinary locus where $\kappa$ is a finite field. Let $\widehat{U}_x$ be the formal completion of $\Ss_K$ at $x$. Then:
\begin{enumerate}
    \item There exists a canonical tower
    \[
    \widehat{U}_x = \widehat{U}_{x,n}\to \widehat{U}_{x,n-1}\to \cdots \to \widehat{U}_{x,0} = \Spf W(\kappa)
    \]
    of formal schemes over $W(\kappa)$, where, $\widehat{U}_{x,i}$ is a formal $p$-divisible group over $\widehat{U}_{x,i-1}$ for $i>0$.
    \item Suppose that $\Ss_K$ is limpid. Then the identity section\footnote{By this, we mean the successive composition of the identity sections of each map in the tower.} $x^{\mathrm{can}}\in \widehat{U}_x(W(\kappa))$, the \emph{canonical lifting}, is a CM point of $\Ss_K$. 
\end{enumerate}
\end{introthm}

\begin{remark}
    Bakker--Shankar--Tsimerman also show the existence of the CM canonical lift for their integral models~\cite[\S 8]{bst}. Unlike their proof, ours does not use \emph{a priori} knowledge of the $\ell$-independence of the conjugacy classes of Frobenius elements acting on the \'etale realizations (see Klevdal--Patrikis~\cite{Klevdal2025-ka}): We only use the defining properties of ICMs, and deduce this $\ell$-independence from the existence of the canonical CM lift. The key point is to exhibit a lift of Frobenius as a quasi-isogeny, which relies ultimately on the stack-theoretic methods used here.
\end{remark}

\begin{remark}
    [Integral canonical models for subhyperspecial levels]
Finally, we mention that the results here can be combined with work of Takaya~\cite{Takaya} to show the existence of integral canonical models in the sense of Pappas--Rapoport when the level at $p$ is parahoric and \emph{contained} in a hyperspecial subgroup; see Proposition~\ref{prop:PR ICM-for-minorizing}.
\end{remark}

\begin{remark}
    [Expected applications]
We expect the results and methods here to be applied (among other questions) to: 
\begin{itemize} 
\item The problem of defining $p$-Hecke correspondences and Igusa stacks for limpid ICMs, as well as counting their mod-$p$ points; 
\item An extension of the Langlands--Kottwitz--Scholze method; 
\item The understanding of central leaves; 
\item The study of special cycles on Shimura varieties.
\end{itemize}
In addition, in ongoing work, joint with Teruhisa Koshikawa and Kentaro Inoue, we will also extend these methods to the logarithmic context and define integral canonical models for toroidal compactifications of Shimura varieties.
\end{remark}

\subsection{Tate full faithfulness}
\label{introsub:tate}

Finally, a key ingredient in our proofs of the canonicity results, and especially for establishing the mapping property in Theorem~\ref{introthm:mapping_property}, is a generalization of Tate's full faithfulness theorem for the Tate module of $p$-divisible groups, and so is of independent interest.

This result holds in the following setting: $\mathcal{G}$ will be a  smooth affine group scheme over $\Int_p$ and $\mu$ will be a cocharacter of $\mathcal{G}$ defined over an unramified ring of integers $\hat{\mathcal{O}}$. We will assume that $\mu$ is $1$-\emph{bounded} in the sense of Lau~\cite[Definition 6.3.1]{MR4355476} (when $\mathcal{G}$ is reductive, this is the same as being minuscule). Then we have the associated formal algebraic stacks $\BT{n}$ constructed in~\cite[\S 9]{gmm}, and their inverse limit $\BT{\infty}$. We also have the \'etale realization functor from $\mr{BT}^{\mc{G},\mu}_\infty$ to $\mc{G}(\bb{Z}_p)$-local systems on the generic fiber, generalizing the Tate module of a $p$-divisible group.

\begin{introthm}
[Tate full faithfulness, Theorem~\ref{thm:full_faithfulness}]
\label{introthm:tate}
If $\mathfrak{X}$ is a normal flat formal scheme over $\hat{\mathcal{O}}$, the \'etale realization functor from $\BT{\infty}(\mathfrak{X})$ to $\mathcal{G}(\Int_p)$-local systems over $\mathfrak{X}_\eta$ is fully faithful.
\end{introthm}

The proof proceeds---as in Tate's original proof---by reduction to the case of a discrete valuation ring. This reduction requires algebraizing $\BT{n}$ to an algebraic stack over $\mathcal{O}$ with affine diagonal; see Theorem~\ref{thm:algebraization}.

\subsection{Acknowledgements}
\label{ss:acknowledgements} 
We would like to thank Ching-Li Chai, Pol van Hoften, Jake Huryn, Christian Klevdal, Ayan Nath, George Pappas, Stefan Patrikis, Ananth Shankar, Xu Shen and Jacob Tsimerman for helpful conversations and correspondence.

\subsection{Notation and conventions}
\label{sub:notational_conventions}

\begin{itemize}
    \item We follow the notational conventions of~\cite[\S 2]{gmm} for ($\infty$-)category theory and related notions.
    \item We will also follow the definitions and conventions of~\cite{gmm} for graded and filtered objects; see in particular \S 4.1--4.3 of \emph{op.\@ cit.}
    \item As a specific instance, for a ring $A$ and ideal $I$, we write $\mr{Fil}^\bullet_I A$ for the $I$-adic filtration. The trivial filtration is $\mr{Fil}_\mr{triv} A\defn\mr{Fil}_0 A$. If $I$ is invertible, we write $\mr{Fil}^\bullet_{I,\pm} A$ for the two-sided filtration with $\mr{Fil}^k_{I,\pm} A=I^k$ for all $k\in \Int$.
    \item If $V$ is a $\Gm$-equivariant complex over some ring $R$, then we will write $V^i\subset V$ for the direct summand on which the $\Gm$-action is via $z\mapsto z^{-i}$: This is the \defnword{$i$-th weight space} for the action. This applies in particular if $\mu:\Gm\to G$ is a cocharacter of a smooth affine group scheme $G$ over $R$, where we take $V$ to be $\Lie(G)$ equipped with the adjoint action of $\Gm$ via $\mu$.
    
    \item If $R$ is a ring (or a scheme), we denote by $\widehat{R}_I$ the completion along an ideal (sheaf) $I$. If $I$ is clear from context (it will usually be the ideal generated by $p$), we will consistently suppress it from the notation.

    \item We follow standard terminology and notation concerning Tannakian theory (e.g., $\Lambda^{\otimes}$); see, for example, \cite[Appendix A]{imai2024tannakianframeworkprismaticfcrystals} for a quick summary.

    \item For a ringed topos $X$, we write $\mr{Vect}(X)$ for the category of vector bundles and $\mr{Perf}(X)$ for the $\infty$-category of perfect complexes.
    \item For a $p$-adic formal stack $\mf{X}$, we denote by $\mf{X}_\mr{crys}$ the absolute crystalline site of $\mf{X}$.

        \item For any groupoid $G$, and any algebraic space $X$, we will write $\underline{G}$ for the \'etale sheafification over $X$ of the constant groupoid associated with $G$.

    \item We will use a subscript $\eta$ to denote `generic fiber' in the following contexts:
    \begin{enumerate}
        \item If $\mathcal{X}$ is an algebraic stack over a discrete valuation ring, then $\mathcal{X}_\eta$ will be the usual generic fiber over its fraction field.
        \item If $\mathfrak{X}$ is a formal algebraic stack over a $p$-complete discrete valuation ring $\mathcal{O}$, then $\mf{X}_\eta$ will denote the corresponding adic stack over $\Spa(\mathcal{O}[\nicefrac{1}{p}],\mc{O})$; see \cite[\S3-4]{ShimizuOh}.
    \end{enumerate}

    \item If $(F,F^+)$ is an affinoid field, and $X$ is an algebraic stack over a $F$, we will write $X^{\an}$ for the associated adic stack $X\times_{\Spec(F)}\Spa(F,F^+)$ over $(F,F^+)$; see \cite[\S4]{HuberGeneralization}.
    
\end{itemize}

\section{Prismatic \texorpdfstring{$F$}{F}-gauges}
\label{sec:prismatic_f_gauges}

\subsection{The stacks of Bhatt--Lurie and Drinfeld}
\label{sub:the_stacks_of_bhatt_lurie_and_drinfeld}

Canonically associated with any derived $p$-adic formal scheme $\mf{X}$ are three derived $p$-adic formal stacks\footnote{This means that they are \'etale (and in fact fpqc) sheaves on  $p$-nilpotent derived affine schemes, i.e. derived schemes of the form $\Spec A$ where $A$ is a $p$-nilpotent animated commutative ring.} $\mf{X}^\Prism$, $\mf{X}^\smallN$, and $\mf{X}^{\mathrm{syn}}$ (shortened to $R^\Prism$, $R^\smallN$, and $R^\mr{syn}$ when $\mf{X}=\Spf(R)$), the \defnword{prismatization}, \defnword{filtered prismatization} and \defnword{syntomification}, respectively. We will review what we'll need to know of these cohomological stacks for this paper. The reader can find more details in~\cite{drinfeld2022prismatization}, \cite{bhatt_lectures},~\cite[\S 6]{gmm} and~\cite[\S 3]{Madapusi2025-ay}. The first two named sources explain how to construct these stacks as classical formal prestacks, and the third extends the constructions---following Bhatt--Lurie~\cite{bhatt2022prismatization}---to the animated framework.

 \begin{remark}
  [Prisms and the prismatization]
\label{rem:prisms_and_prismatization}
The prismatization $\mf{X}^{\Prism}$ parameterizes \emph{Cartier--Witt divisors} on $\mf{X}$. These are related to the prismatic site for $\mf{X}$. Indeed, suppose that we have an object $(A,I,\Spf(\ov{A})\to \mf{X})$ in the absolute prismatic site for $\mf{X}$. Here $A$ is endowed with the $(p,I)$-adic topology, and $\ov{A}=A/I$. Then, as in~\cite[Construction 3.10]{bhatt2022prismatization}, we find a canonical map $\iota_{(A,I)}\colon\Spf A\to \mf{X}^{\smallprism}$ taking a map $\Spec(B)\to \Spf(A)$, where $B$ is a $p$-nilpotent ring, to the Cartier--Witt divisor $I\otimes_AW(B)\to W(B)$ where $A\to W(B)$ is the unique $\delta$-map lifting $A\to B$, and where we equip this Cartier--Witt divisor with structure map
\[
\Spec(W(B)/^{\mathbb{L}}(I\otimes_AW(B)))\to \Spf(\ov{A})\to \mf{X}.
\]

\end{remark}

\begin{remark}
  [Relationship between the stacks]
The stack $\mf{X}^\smallN$ is a \emph{filtered} formal stack, meaning that it is fibered naturally over the formal Artin stack $\Aff^1/\Gm$ parameterizing line bundles $L$ over $p$-complete rings $C$ equipped with a cosection $L\to C$. The pre-image of the open point 
\begin{equation*}
\Spf \Int_p\simeq \Spf \bb{Z}_p \times \Gm/\Gm\subset \Spf \bb{Z}_p\times\Aff^1/\Gm
\end{equation*}
is canonically isomorphic to $\mf{X}^{\smallprism}$, and is called the \defnword{de Rham locus} of $\mf{X}^\smallN$ and its inclusion is denoted $j_\mr{dR}\colon 
\mf{X}^\Prism\to \mf{X}^\smallN$. There is also another open immersion $j_{\mathrm{HT}}\colon \mf{X}^\Prism\to \mf{X}^\smallN$ called the \defnword{Hodge--Tate locus} which is physically disjoint from the de Rham locus. The stack $\mf{X}^{\mathrm{syn}}$ is the coequalizer of these two open immersions and is therefore equipped with a canonical open immersion $j_{\Prism}\colon\mf{X}^\Prism\to \mf{X}^{\mathrm{syn}}$.
\end{remark}

\begin{remark}
   [Nygaard filtered absolute prismatic cohomology]
\label{rem:nygaard_filtered_cohomology}
Bhatt--Lurie show that the values of these stacks on quasiregular semiperfectoid (qrsp) inputs can be described in terms of Nygaard filtered absolute prismatic cohomology. For such rings $R$, their \defnword{absolute prismatic cohomology} is a classical $p$-complete ring $\Prism_R$, which is in fact a $\delta$-ring equipped with a canonical prism structure $(\Prism_R,I_R)$: this is the \emph{initial} prism for $R$. In~\cite[\S 12]{BhattScholzePrisms}, Bhatt--Scholze construct the \defnword{Nygaard filtration} $\Fil^\bullet_\smallN\Prism_R$ on absolute prismatic cohomology. For $R$ qrsp, we have
\[
\Fil^i_{\smallN}\Prism_R = \{x\in \Prism_R:\;\varphi(x)\in \Fil^i_{I_R}\Prism_R\},
\]
where $\Fil^\bullet_{I_R}\Prism_R$ is the $I_R$-adic filtration. The Nygaard filtered prismatization $R^{\smallN}$ is now canonically isomorphic to the \emph{formal Rees stack} $\Rees(\Fil^\bullet_{\smallN}\Prism_R)$ associated with this filtration (e.g., see \cite[\S1.1.2]{imai2023prismatic}). The open immersion $j_{\dR}$ amounts to `forgetting' the filtration (that is, restriction to the open substack $\Gm/\Gm\subset \Aff^1/\Gm)$, while $j_{\mathrm{HT}}$ arises from the filtered Frobenius lift
\[
  \Fil^\bullet_{\smallN}\Prism_R\to \Fil^\bullet_{I_R}\Prism_R.
\]
This description can be extended to all \emph{semiperfectoid} rings; see~\cite[Theorem 6.11.7]{gmm}.

\end{remark}

\begin{remark}
  [The (filtered) de Rham point]
\label{rem:de_rham_point}
There are canonical maps 
\[
  x^\smallN_{\dR}\colon\Aff^1/\Gm\times \mf{X}\to \mf{X}^\smallN\;;\; x_{\dR}\colon\mf{X}\to \mf{X}^{\smallprism}
\]
with the second being the restriction of the first over $\Gm/\Gm$. In terms of the `affine' description from Remark~\ref{rem:nygaard_filtered_cohomology} for semiperfectoid rings, the first map corresponds to the map on formal Rees stacks associated with the canonical map of filtered rings from $\Fil^\bullet_{\smallN}\Prism_R$ to $R$ equipped with its trivial descending filtration supported in graded degree $0$.
\end{remark}

\begin{remark}
 [Functoriality and \'etale descent]
Each of the assignments $\mf{X}\mapsto \mf{X}^?$, for $?$ equal to $\Prism$, $\mathcal{N}$, or $\mathrm{syn}$, is functorial in the derived formal scheme $\mf{X}$, and preserves products: That is, we have canonical isomorphisms $(\mf{X}\times \mf{Y})^? \simeq \mf{X}^?\times \mf{Y}^?$. Furthermore, if $\mf{Y}\to\mf{X}$ is an \'etale and faithfully flat map\footnote{Recall this means that the induced map of schemes $\mf{Y}\otimes^\bb{L}\bb{F}_p\to \mf{X}\otimes^\bb{L}\bb{F}_p$ is \'etale and faithfully flat.}, then $\mf{Y}^?\to \mf{X}^?$ is an \'etale cover of derived formal stacks (see \cite[Proposition 6.12.3]{gmm}).
 \end{remark} 

\begin{definition}
   A map $\mf{Y}\to\mf{X}$ of derived $p$-adic formal schemes is \defnword{quasisyntomic} if it is $p$-completely flat (that is, $\mf{Y}\otimes^\bb{L}\bb{F}_p\to\mf{X}\otimes^\bb{L}\bb{F}_p$ is flat), and if $\mathbb{L}_{\mf{Y}/\mf{X}}$ has $p$-complete Tor amplitude $[-1,0]$: that is, $\mathbb{L}_{\mf{Y}/\mf{X}}/{}^{\mathbb{L}}p$ has Tor amplitude $[-1,0]$ as a quasi-coherent sheaf on $\mf{Y}\otimes^\bb{L}\bb{F}_p$. The map $\mf{Y}\to \mf{X}$ is a \defnword{quasisyntomic cover} if it is quasisyntomic and $\mf{Y}\otimes^\bb{L}\bb{F}_p\to\mf{X}\otimes^\bb{L}\bb{F}_p$ is an fpqc cover. 
\end{definition}

\begin{proposition}
   \label{prop:nygaard_qsynt_descent}
If $\mf{Y}\to \mf{X}$ is a quasisyntomic cover, then $\mf{Y}^\smallN \to \mf{X}^\smallN$ is surjective in the $p$-completely flat topology. In fact, if $\mf{Y}$ and $\mf{X}$ are the formal spectra of semiperfectoid rings, then $\mf{Y}^\smallN\to \mf{X}^\smallN$ is faithfully flat. 
\end{proposition}
\begin{proof}
   See~\cite[Proposition 6.12.3 and Corollary 6.12.8]{gmm}.
\end{proof}

\begin{remark}
  \label{rem:quasisyntomic_qrsp}
If $\mf{X} = \Spf R$ is affine and the structure map $\mf{X}\to\mr{Spf}(\mathbb{Z}_p)$ is quasisyntomic, there exists an affine quasisyntomic cover $\Spf R_\infty = \mf{X}_\infty\to \mf{X}$ such that $R_\infty^{\otimes_{R} m}$, the $m$-fold completed self-tensor product over $R$, is qrsp and $p$-torsion free for all $m$ (see \cite[Lemmas 4.28 and 4.30]{BMSII}). Therefore, in many cases, Proposition~\ref{prop:nygaard_qsynt_descent} reduces questions about the stacks $\mf{X}^{?}$ to ones about the corresponding stacks associated with (the formal spectra of) qrsp and $p$-torsion free rings.
\end{remark}

\subsection{\texorpdfstring{$F$}{F}-gauges and prismatic \texorpdfstring{$F$}{F}-crystals}
\label{sub:fontaine_laffaille_modules_and_f_gauges}

Let $\mf{X}$ be a derived $p$-adic formal scheme.

\begin{definition}
  [(Prismatic) $F$-gauges]
A \defnword{vector bundle in prismatic $F$-gauges} or simply \defnword{vector bundle $F$-gauge} over $\mf{X}$ is a vector bundle over $\mf{X}^{\mathrm{syn}}$. We will write $\mathrm{Vect}(\mf{X}^{\mathrm{syn}})$ for the category of such objects.\footnote{In general, $\mr{Vect}(\mf{X}^\mr{syn})$ is actually an $\infty$-category, but when $\mf{X}$ is $p$-quasisyntomic and flat over $\mathbb{Z}_p$, Remark~\ref{rem:quasisyntomic_qrsp} combined with Remark~\ref{rem:nygaard_filtered_cohomology} tells us that this is in fact a classical category.}
\end{definition}

\begin{remark}
  [Prismatic $F$-crystals]
\label{rem:prismatic_f-crystals}
Recall from~\cite{MR4600546} the notion of \emph{prismatic crystals} and \emph{prismatic $F$-crystals} over $\mf{X}$: These are objects over the absolute prismatic site $
\mf{X}_{\Prism}$ (shortened to $R_\Prism$ when $\mf{X}=\Spf(R)$). To begin, we have the \emph{structure sheaf} $\Reg{\Prism}$ and a generalized Cartier divisor $\mathcal{I}_{\Prism}\to \Reg{\Prism}$ given by the assignment
\[
(\mathcal{I}_{\Prism}\to \Reg{\Prism})\colon (A,I,\Spf(\overline{A})\to\mf{X})\mapsto (I\to A).
\]
A \defnword{prismatic crystal} (in vector bundles) over $\mf{X}$ is a vector bundle $\mathcal{E}$ over $(\mf{X}_{\Prism},\Reg{\Prism})$, and a \defnword{prismatic $F$-crystal} is a pair $(\mathcal{E},\varphi_{\mathcal{E}})$ where $\mathcal{E}$ is a prismatic crystal and
\[
  \varphi_{\mathcal{E}}\colon \varphi^*\mathcal{E}[\nicefrac{1}{\mc{I}_\Prism}]\isomto \mathcal{E}[\nicefrac{1}{\mathcal{I}_{\Prism}}]
\] 
is an $\Reg{\Prism}[\nicefrac{1}{\mc{I}_\Prism}]$-linear isomorphism of sheaves on $\mf{X}_\Prism$. We write $\mathrm{Vect}^{\varphi}(\mf{X}_\Prism,\Reg{\Prism})$ for the category of prismatic $F$-crystals.
\end{remark}

\begin{remark}
  \label{rem:concrete_prismatic_f-crystals}
Concretely, a prismatic crystal is an assignment 
\begin{equation*}
(A,I,\Spf(\overline{A})\to\mf{X})\mapsto \mathcal{E}(A,I,\Spf(\overline{A})\to\mf{X})
\end{equation*}
with the value being a finite locally free $A$-module, and satisfying the usual crystal property: For all maps $(A,I,\Spf(\overline{A})\to\mf{X})\to (B,J,\Spf(\overline{B})\to\mf{X})$ in $\mf{X}_{\Prism}$, we have isomorphisms
\[
  \mathcal{E}(A,I,\Spf(\overline{A})\to\mf{X})\otimes_A B{\isomto} \mathcal{E}(B,J,\Spf(\overline{B})\to\mf{X})
\]
satisfying the expected compatibility relations. Endowing this with the structure of a prismatic $F$-crystal now amounts to giving isomorphisms
\[
  \mathcal{E}(A,I,R\to\overline{A})[\nicefrac{1}{I}]\otimes_{A,\varphi}A{\isomto}\mathcal{E}(A,I,R\to\overline{A})[\nicefrac{1}{I}]
\]
that are compatible with the isomorphisms coming from the crystal property.
\end{remark}

\begin{remark}
  [$F$-gauges and prismatic $F$-crystals]
\label{rem:f-gauges-prismatic} 
By Remark~\ref{rem:prisms_and_prismatization}, any vector bundle on $\mf{X}^\Prism$ yields a prismatic crystal over $\mf{X}$. If the vector bundle arises from a vector bundle $F$-gauge via pullback along $j_{\Prism}\colon\mf{X}^\Prism\to \mf{X}^{\mathrm{syn}}$, then $\mathcal{E}$ has a canonical structure of a prismatic $F$-crystal over $\mf{X}$; see \cite[Remark 6.3.4]{bhatt_lectures} and \cite[Construction 1.21]{imai2023prismatic}. This gives rise to a natural functor $\mr{Vect}(\mf{X}^\mr{syn})\to \mr{Vect}^\varphi(\mf{X}_\Prism,\Reg{\Prism})$.
\end{remark}

\begin{proposition}
  \label{prop:f-gauges_to_analytic_prismatic}
Suppose that $\mf{X}$ is $p$-quasisyntomic and flat over $\mathbb{Z}_p$. Then the natural functor
\[
  \mathrm{Vect}(\mf{X}^{\mathrm{syn}})\to \mathrm{Vect}^{\varphi}(\mf{X}_{\Prism},\Reg{\Prism})
\]
is fully faithful
\end{proposition}
\begin{proof}
    See~\cite[Corollary 3.53]{GuoLi}.
\end{proof}

\subsection{Analytic prismatic \texorpdfstring{$F$}{F}-crystals and crystalline local systems}
\label{sub:full_faithfulness_of_the_etale_realization}

We now recall some generalizations due to Du--Liu--Moon--Shimizu~\cite{Du2024-to} and Guo--Reinecke~\cite{Guo2024-al} of the classification of Bhatt--Scholze~\cite{MR4600546} of crystalline Galois representations. For more details, see the discussion in~\cite[\S 2.3]{imai2024tannakianframeworkprismaticfcrystals}. In this subsection, $K$ will denote a complete mixed characteristic $(0,p)$ discrete valuation field with perfect residue field.

\begin{remark}
  [Analytic prismatic $F$-crystals]
We will take $\mathrm{Vect}^{\mathrm{an},\varphi}(\mf{X}_\Prism,\Reg{\Prism})$ to be the category of \emph{analytic prismatic $F$-crystals} defined in~\cite[\S 3]{Guo2024-al}. Concretely, this amounts to giving essentially the same data as that of a prismatic $F$-crystal, explained in Remark~\ref{rem:concrete_prismatic_f-crystals}, except that $\mathcal{E}(A,I,\Spf(\overline{A})\to\mf{X})$ is now a vector bundle over $\Spec(A)\backslash V(p,I)$; \cite[Lemma 3.4]{Guo2024-al} ensures that this is a sensible notion. 
\end{remark}

\begin{remark}
  [The \'etale realization]
\label{rem:etale_realization}
Via~\cite[Constructions 6.3.1 and 6.3.2]{bhatt_lectures} (see also~\cite[Construction 3.9]{Guo2024-al}), we find a canonical functor
\[
T_{\et}\colon  \mathrm{Vect}^{\mathrm{an},\varphi}(\mf{X}_\Prism,\Reg{\Prism})\to \mathrm{Loc}_{\Int_p}(\mf{X}_\eta)
\]
where the right-hand side is the category of pro\'etale $\Int_p$-local systems on the generic fiber of $\mf{X}$. Since both sides satisfy quasisyntomic descent (see~\cite[Lemma 3.6]{Guo2024-al} for descent for analytic prismatic $F$-crystals), it suffices to specify the values for semiperfectoid $\Spf(R)\to \mf{X}$ with $R$ semiperfectoid as: 
\[
\Gamma\big((\Spf R)_{\eta},T_{\et}(\mathcal{E})\big) = \big(\mathcal{E}(\Prism_R,I_R,\Spf(\ov{\Prism}_R)\to\Spf(R)\to\mf{X})[\nicefrac{1}{I_R}]^{\wedge}\big)^{\varphi_\mc{E} = \mathrm{id}}.
\]
Here, we are viewing $\mathcal{E}\big(\Prism_R,I_R,\Spf(\ov{\Prism}_R)\to\Spf(R)\to\mf{X}\big)[\nicefrac{1}{I_R}]$ as a finite locally free module over $\Prism_R[\nicefrac{1}{I_R}]$, and $(-)^{\wedge}$ denotes the $p$-adic completion.
\end{remark}

\begin{remark}[{Relationship between various types of local systems}] In the sequel we will frequently make use of the following natural equivalences:
\begin{enumerate}
    \item Let $X$ be a locally of finite type $K$-scheme, and $X^\mathrm{an}\defn X\times_{\Spec(K)}\Spa(K)$ its analytification. There is a (bi-)exact monoidal equivalence $\mr{Loc}_{\bb{Z}_p}(X)\isomto \mr{Loc}_{\bb{Z}_p}(X^\mr{an})$ (see \cite[\S2.1.3]{imai2024tannakianframeworkprismaticfcrystals} and the references therein).
    \item If $(A,A^+)$ is a Huber pair, there is a (bi-)exact monoidal equivalence 
    \begin{equation*}
        \mathrm{Loc}_{\bb{Z}_p}(\Spec(A))\isomto \mr{Loc}_{\bb{Z}_p}(\Spa(A,A^+)).
    \end{equation*} 
    See \cite[\S2.1.4]{imai2024tannakianframeworkprismaticfcrystals} and the references therein.

\end{enumerate}
We will often glibly use the above identifications without further comment. Moreover, by (2) (and the independence of base points for connected spaces) there is no ambiguity in the notation $\mr{Loc}_{\bb{Z}_p}(A)$ for a Huber pair $(A,A^+)$, and we use this notational shortcut often.
\end{remark}

\begin{proposition}
  \label{prop:prismatic_ff_in_analytic}
Suppose that $\mf{X}$ is $p$-quasisyntomic and flat over $\bb{Z}_p$. Then the natural functor 
\begin{equation*}
   \mathrm{Vect}^{\varphi}(\mf{X}_\Prism,\Reg{\Prism})\to \mathrm{Vect}^{\mathrm{an},\varphi}(\mf{X}_\Prism,\Reg{\Prism}) 
\end{equation*} is fully faithful.
\end{proposition}
\begin{proof}
  See~\cite[Proposition 3.7]{Guo2024-al}.
\end{proof}

\begin{remark}
  [Crystalline local systems]
\label{rem:crystalline_local_system}
Suppose that $\mf{X}$ is a \emph{base formal $\Reg{K}$-scheme} as defined in~\cite[\S 1.1]{imai2024tannakianframeworkprismaticfcrystals} and the references therein. This implies in particular that $R$ is $p$-quasisyntomic and $p$-torsion free. Then within $\mathrm{Loc}_{\Int_p}(\mf{X}_\eta)$, we have the full subcategory $\mathrm{Loc}^{\mathrm{crys}}_{\Int_p}(\mf{X})$ of \emph{crystalline} $\Int_p$-local systems; see~\cite[\S 2.2]{Du2024-to},~\cite[\S 2.4]{Guo2024-al}, or ~\cite[\S 2.3.3]{imai2024tannakianframeworkprismaticfcrystals}. When $\mf{X}= \Spf(\Reg{K})$, this recovers the classical notion of a Galois stable $\Int_p$-lattice in a crystalline $\Rat_p$-representation of the absolute Galois group of $K$. Recent work of Guo--Yang~\cite[Theorem 1.1]{guo2024pointwisec} shows that, when $\mf{X}$ is smooth over $\Reg{K}$, a $\Int_p$-local system over $\mf{X}$ is crystalline if and only if its restriction over every classical point of $\mf{X}_\eta$ is crystalline.

\end{remark}

\begin{theorem}
\label{thm:guo_reinecke}
Suppose that $\mf{X}$ is $p$-quasisyntomic and flat over $\bb{Z}_p$. Then:
\begin{enumerate}
  \item The \'etale realization functor $T_{\et}\colon \mathrm{Vect}^{\mathrm{an},\varphi}(\mf{X}_\Prism,\Reg{\Prism})\to \mathrm{Loc}_{\Int_p}(\mf{X}_\eta)$ is faithful. In particular, its restriction to $\mathrm{Vect}(\mf{X}^{\mathrm{syn}})$ is also faithful.

  \item Suppose that $\mf{X}$ is smooth over $\Reg{K}$; or, more generally, that $\mf{X}$ is a base formal $\Reg{K}$-scheme. Then the functor induces a bi-exact symmetric monoidal equivalence 
    \[
      \mathrm{Vect}^{\mathrm{an},\varphi}(\mf{X}_\Prism,\Reg{\Prism}){\isomto}\mathrm{Loc}^{\mathrm{crys}}_{\Int_p}(\mf{X}_\eta).
    \]
  In particular, its restriction to $\mathrm{Vect}(\mf{X}^{\mathrm{syn}})$ is fully faithful.
\end{enumerate}
\end{theorem}
\begin{proof}
This is essentially a translation of results of Guo--Reinecke~\cite{Guo2024-al} and Du--Liu--Moon--Shimizu~\cite{Du2024-to}). 

The first result follows from Propositions~\ref{prop:f-gauges_to_analytic_prismatic} and~\ref{prop:prismatic_ff_in_analytic} combined with~\cite[Proposition 3.7]{Guo2024-al} and the proof of~\cite[Lemma 4.1]{Guo2024-al}. Note that the last result only claims to prove faithfulness for $R$ $p$-completely smooth over $\Reg{K}$. However, the proof actually shows that the functor $T_{\et}$ is faithful for any $p$-torsion free qrsp algebra\footnote{The argument in this generality is actually needed to apply this lemma to the proof of Theorem 4.5 in \emph{op.\@ cit.\@}}, and so implies what we need by quasisyntomic descent. 

The second claim follows from~\cite[Theorem 4.5]{Guo2024-al},~\cite[Theorem 3.29]{Du2024-to}, and~\cite[Proposition 2.22]{imai2024tannakianframeworkprismaticfcrystals}.
\end{proof}

\begin{definition}
  [Local systems for smooth affine group schemes]
 \label{def:local_systems_smooth_affine}
 Suppose that $\mathcal{G}$ is a smooth affine group scheme over $\Int_p$ and that $X$ is a locally Noetherian scheme or adic space. A $\mathcal{G}(\Int_p)$\defnword{-local system} over $X$ is a $\underline{\mathcal{G}(\Int_p)}$-torsor on the pro\'etale site of $X$.\footnote{Concretely, this amounts to giving a compatible inverse system $\mathcal{P} = \{\mathcal{P}_n\}_{n\geqslant 1}$ where $\mathcal{P}_n\to X$ is a finite \'etale torsor for the finite group $\mathcal{G}(\Int/p^n\Int)$ (viewed as a locally constant sheaf), and the maps $\mathcal{P}_{n+1}\to \mathcal{P}_n$ are equivariant for the natural surjections $\mathcal{G}(\Int/p^{n+1}\Int)\to \mathcal{G}(\Int/p^n\Int)$.} We will write $\mathrm{Loc}_{\mathcal{G}(\Int_p)}(X)$ for the groupoid of such local systems. If $X$ is affine/affinoid with global sections $R$, we will also write $\mathrm{Loc}_{\mathcal{G}(\Int_p)}(R)$ instead. We will also use analogous notation for the variants where $\mathcal{G}(\Int_p)$ is replaced by $\mathcal{G}(\Int/p^n\Int)$. See \cite[\S2.1]{imai2024tannakianframeworkprismaticfcrystals} for a more detailed discussion.
\end{definition}

\begin{remark}
  [Tannakian perspective on local systems]
\label{rem:local_systems_tannakian}
Giving a $\mathcal{G}(\Int_p)$-local system $\mathcal{P}$ over $X$ is equivalent to giving an exact symmetric monoidal functor $V \mapsto (V)_{\mathcal{P}}$ from the category $\Rep_{\Int_p}(\mathcal{G})$ of algebraic representations $\mathcal{G}\to \GL(V)$ defined over $\Int_p$ to the category $\mathrm{Loc}_{\Int_p}(X)$. See~\cite[Propositions 2.3 and 2.8]{imai2024tannakianframeworkprismaticfcrystals}. In fact, the cited result shows that giving a $\mathcal{G}(\Int/p^n\Int)$-local system $\mathcal{P}_n$ over $X$ is the same as giving an exact symmetric monoidal functor $\Rep_{\Int_p}(\mathcal{G})\to \mathrm{Loc}_{\Int/p^n\Int}(X)$.
\end{remark} 

\begin{definition}
  [Crystalline $\mathcal{G}(\Int_p)$-local systems]
\label{def:crystalline_GZ_p}
Let $\mf{X}$ be a base formal $\Reg{K}$-scheme. A $\mathcal{G}(\Int_p)$-local system $\mathcal{P}$ on $\mf{X}_\eta$ is \defnword{crystalline} if, for every representation $V$ in $ \Rep_{\Int_p}(\mathcal{G})$, the local system $(V)_{\mathcal{P}}$ is crystalline. Write $\mathrm{Loc}_{\mathcal{G}(\Int_p)}^{\mathrm{crys}}(\mf{X}_\eta)$ for the full subcategory of $\mathrm{Loc}_{\mathcal{G}(\Int_p)}(\mf{X}_\eta)$ of crystalline $\mc{G}(\bb{Z}_p)$-local systems.
\end{definition}

\begin{remark}
  \label{rem:crystalline_GZ_p}
By~\cite[Proposition 2.20]{imai2024tannakianframeworkprismaticfcrystals}, to check that a $\mathcal{G}(\Int_p)$-local system $\mathcal{P}$ is crystalline, it is sufficient to check that $(V)_{\mathcal{P}}$ is crystalline for a single faithful representation $V$.
\end{remark}

\begin{remark}
[$\mathcal{G}$-bundles in analytic prismatic $F$-crystals]
  \label{rem:crystalline_GZ_p_tannakian}
Suppose that $\mf{X}$ is a base formal $\Reg{K}$-scheme. Combining Remark~\ref{rem:local_systems_tannakian} with Theorem~\ref{thm:guo_reinecke}, we see that the category $\mathrm{Loc}_{\mathcal{G}(\Int_p)}^{\mathrm{crys}}(\mf{X}_\eta)$ is equivalent to the category of exact symmetric monoidal functors
\[
  \omega\colon\Rep_{\Int_p}(\mathcal{G})\to \mathrm{Vect}^{\mathrm{an},\varphi}(\mf{X}_{\Prism},\Reg{\Prism}).
\]

\end{remark}

\begin{corollary}
\label{cor:guo_reinecke}
Suppose that $\mf{X}$ is $p$-quasisyntomic and flat over $\bb{Z}_p$.
\begin{enumerate}
  \item   There is a natural faithful functor
\[
T_{\et}\colon B \mathcal{G}(\mf{X}^{\mathrm{syn}})\to \mathrm{Loc}_{\mathcal{G}(\Int_p)}(\mf{X}_\eta).
\]
\item If $\mf{X}$ is a base $\Reg{K}$-formal scheme, then the functor is also full and takes its values in $\mathrm{Loc}_{\mathcal{G}(\Int_p)}^{\mathrm{crys}}(\mf{X}_\eta)$. 

\end{enumerate}
\end{corollary}
\begin{proof}
  Under these hypotheses, $\mf{X}^{\mathrm{syn}}$ is a classical formal stack. Therefore, we find from~\cite[Corollary 9.3.7.3]{Lurie2018-kh} that the groupoid $(B \mathcal{G})(\mf{X}^{\mathrm{syn}})$ is equivalent to the groupoid of exact symmetric monoidal functors
  \[
    \Rep_{\Int_p}(\mathcal{G})\to \mathrm{Vect}(\mf{X}^{\mathrm{syn}}).
  \]
  The corollary now follows from Theorem~\ref{thm:guo_reinecke} and Remark~\ref{rem:local_systems_tannakian}. 
\end{proof}

\subsection{\texorpdfstring{$F$}{F}-gauges and \texorpdfstring{$p$}{p}-divisible groups}
\label{sub:f_gauges_and_p_divisible_groups}

\begin{definition}
  [Hodge--Tate weights]
Let $\mathfrak{X}$ be a $p$-adic formal scheme. Every vector bundle $F$-gauge $\mathcal{V}$ over $\mf{X}$ yields a graded finite locally free $R$-module via pullback along the composition
\[
  B\Gm\times \mf{X}\to \Aff^1/\Gm\times \mf{X}\xrightarrow{x^{\smallN}_{\dR}}\mf{X}^{\smallN}\to \mf{X}^{\mathrm{syn}}.
\]
The \defnword{Hodge--Tate weights} of $\mathcal{V}$ are the graded degrees in which this graded module is non-zero.
\end{definition}

The following result is essentially due to Ansch\"utz--Le Bras~\cite{MR4530092} over quasisyntomic bases; the general statement here is shown in~\cite{gmm}. See Theorem 11.1.4 and Proposition 11.8.2 of \emph{op.\@ cit.}
\begin{theorem}
[Prismatic Dieudonn\'e theory]
\label{thm:dieudonne}
Let $\mf{X}$ be a derived $p$-adic formal scheme and let $\mr{BT}_p(\mf{X})$ be the category of $p$-divisible groups over $\mf{X}$, and let $\mathrm{Vect}_{[0,1]}(\mf{X}^{\mathrm{syn}})$ be the category of vector bundle $F$-gauges over $\mf{X}$ with Hodge--Tate weights in $\{0,1\}$. Then: 
\begin{enumerate}
  \item There is a canonical equivalence of categories
\[
\mathcal{M}\colon\mr{BT}_p(\mf{X})\isomto \mathrm{Vect}_{[0,1]}(\mf{X}^{\mathrm{syn}})
\]
compatible with Cartier duality.
\item There is a natural equivalence $T_p\simeq T_{\et}\circ \mc{M}$, where $T_p$ is the functor associating to a $p$-divisible group $H$ its Tate module $T_p(H)=\varprojlim H[p^n]_\eta$, viewed as a $\bb{Z}_p$-local system on $\mf{X}_\eta$.
\end{enumerate}
\end{theorem}

\subsection{\texorpdfstring{$F$}{F}-gauges and cohomology}\label{ss:F-gauges-and-cohomology} 

In this final subsection we recall the relationship between $F$-gauges, \'etale realization, and higher pushforwards. Throughout $K$ denotes a complete mixed characteristic $(0,p)$ discrete valuation field with perfect residue field $k$, and $C$ a completed algebraic closure of $K$.

\begin{setup} 
Let $\mf{X}$ be a formal $\ms{O}_K$-scheme with generic fiber $X$. We denote by $D^b_\mr{lisse}(X,\mathbb{Z}_p)$ the category of derived $p$-complete locally bounded $\underline{\bb{Z}}_p$-modules $\bb{L}$ on $X_\proet$ whose cohomology sheaves $H^i(\bb{L})$ are locally constant on $X_\proet$ with finitely generated stalks. If $\mf{X}$ is a smooth formal $\Reg{K}$-scheme, let $D^b_{\mr{crys}}(X,\bb{Z}_p)$ be the full subcategory of those $\bb{L}$ such that $H^i(\bb{L})[\nicefrac{1}{p}]$ is crystalline for all $i$.
\end{setup}

\begin{remark}[{\'Etale realization for perfect $F$-gauges}] By \cite[Constructions 6.3.1 and 6.3.2]{bhatt_lectures} there is an \'etale realization functor
\begin{equation*}
    T_{\et}\colon \mr{Perf}(\mf{X}^\mr{syn})\to D^b_{\mr{lisse}}(X,\bb{Z}_p).
\end{equation*}
By \cite[Corollary 4.1]{Pentland}, if $\mf{X}$ is smooth, then this functor takes values in $D^b_\mr{crys}(X,\bb{Z}_p)$. Moreover, the \'etale realization functor is compatible with that from Remark \ref{rem:etale_realization} via the natural embedding of $\mr{Vect}^{\an,\varphi}(\mf{X}_\Prism)$ into $\mr{Perf}(\mf{X}^\mr{syn})$ via the functor $\Pi_{\mf{X}}$ from \cite[Theorem 3.32]{GuoLi} (see also \cite[\S3.2]{Pentland}). 
\end{remark}

\begin{remark}[{Crystalline realization functor}] 
Suppose that $\mf{X}$ is a $p$-quasisyntomic $p$-adic formal scheme. Then there is a natural functor
\begin{equation*}
    T_\mr{crys}\colon \mr{Perf}(\mf{X}^\mr{syn})\to\mr{Perf}^\varphi\big((\mf{X}_{p=0}\big)_{\mr{crys}}),
\end{equation*}
where the target is the category of perfect complexes in $F$-crystals on the absolute crystalline site of $\mf{X}_{p=0}$. More precisely, if $\mc{V}$ is a perfect $F$-gauge, then Remark \ref{rem:f-gauges-prismatic} gives us a perfect complex of prismatic $F$-crystals on $\mf{X}$ and by restriction a perfect complex of prismatic $F$-crystals on $\mf{X}_{p=0}$, and one then applies \cite[Theorem 6.4]{Guo2024-al}.
\end{remark}

\begin{remark}[{Perfect complexes in $F$-isocrystals on smooth formal $\mathscr{O}_K$-schemes}] Suppose now that $\mf{X}$ is a smooth formal $\mathscr{O}_K$-scheme. In this case, one has a natural identification
\begin{equation*}
    \mr{Perf}^\varphi\big((\mf{X}_{p=0})_{\mr{crys}}\big)[\nicefrac{1}{p}]\simeq \mr{Perf}^\varphi\big((\mf{X}_k/W(k))_\mr{crys}\big)[\nicefrac{1}{p}].
\end{equation*}
Here the left-hand side is the category of perfect complexes in $F$-isocrystals on the absolute crystalline site of $\mf{X}_{p=0}$, and the right-hand side the category of perfect complexes in $F$-isocrystals on the relative crystalline site $(\mf{X}_k/W(k))_\mr{crys}$. This identification is a version of Dwork's trick (see \cite[Remark 2.15]{Guo2024-al} and the references therein). In the sequel we will often implicitly make this identification.
\end{remark}

\begin{remark}[{The filtered crystalline realization functor on smooth formal $\mathscr{O}_K$-scheme}] 
Suppose again that $\mf{X}$ is a smooth formal $\mathscr{O}_K$-scheme. Then, the crystalline realization functor admits a filtered upgrade
\begin{equation*}
    T_\mr{crys}^+\colon \mr{Perf}(\mf{X}^\mr{syn})\to \mr{PerfF}^\varphi(\mf{X}_\mr{crys})[\nicefrac{1}{p}],
\end{equation*}
where the target is the category of perfect complexes in \emph{filtered} $F$-isocrystals on $\mf{X}$ as in \cite[Definition 4.18]{Pentland}, and where $T^+_\mr{crys}$ is denoted by $\mathsf{Beil}$ in \emph{op.\@ cit.}
\end{remark}

The following shows that all of these constructions are compatible with pushforwards. For the proof (due to combining work from \cite{Guo2024-al}, \cite{GuoLi}, and \cite{Pentland}) one can see \cite[Corollary 4.7]{Pentland}, \cite[Proposition 4.23]{Pentland}, and their proofs, and the references therein.

\begin{theorem}\label{thm:etale-realization-and-pushforward} 
Suppose that $f\colon \mf{X}\to\mf{Y}$ is a smooth proper map of smooth formal $\mathscr{O}_K$-schemes with generic fiber map $f_\eta\colon X\to Y$. Then, pushforward defines a well-defined commutative square
\begin{equation*}
    \begin{tikzcd}
  {\mr{Perf}(\mf{X}^\mr{syn})} & {D^b_\mr{crys}(X,\bb{Z}_p)} \\
  {\mr{Perf}(\mf{Y}^\mr{syn})} & {D_{\mr{crys}}(Y,\bb{Z}_p).}
  \arrow["{T_{\et}}", from=1-1, to=1-2]
  \arrow["{Rf_\ast}"', from=1-1, to=2-1]
  \arrow["{R (f_\eta)_\ast}", from=1-2, to=2-2]
  \arrow["{T_{\et}}"', from=2-1, to=2-2]
\end{tikzcd}
\end{equation*}
Moreover, for any object $\mc{E}$ of $\mr{Perf}(\mf{X}^\mr{syn})$ and $i\geqslant 0$ there is a natural identification of filtered $F$-isocrystals
\begin{equation*}
    D_\mr{crys}\big(R^i(f_\eta)_\ast T_{\et}(\mc{E})[\nicefrac{1}{p}]\big)\simeq R^i (f_k)_\ast T^+_\mr{crys}(\mc{E})[\nicefrac{1}{p}].
\end{equation*}
\end{theorem}

\begin{remark}
    [Pushforward over a general base]
\label{rem:pushforward_general}
In fact, for any smooth proper map $\mf{X}\to \mf{Y}$ of formal $\Reg{K}$-schemes, we always have a well-defined pushforward map 
\[
Rf_\ast\colon\mr{Perf}(\mf{X}^{\mr{syn}})\to \mr{Perf}(\mf{Y}^{\mr{syn}}).
\]
This has the property that it carries perfect $F$-gauges over $\mf{X}$ with Hodge--Tate weights in $[a,b]$ to those over $\mf{Y}$ with Hodge--Tate weights in $[a-d,b]$, where $d$ is the relative dimension of $\mf{X}$ over $\mf{Y}$; see~\cite[Proposition 8.2.6]{Madapusi2025-ay}.
\end{remark}

\begin{corollary} Suppose that $\mathfrak{X}$ is a smooth proper formal $\mathscr{O}_K$-scheme and $\mc{V}$ is a prismatic $F$-gauge on $\mf{X}$. Then, for all $i\geqslant 0$ the $\mr{Gal}(\overline{K}/K)$-representation $H^i_\mathrm{et}(\mf{X}_C,T_{\et}(\mc{V}))[\nicefrac{1}{p}]$ is crystalline and there is a natural identification of filtered $F$-isocrystals on $K$:
\begin{equation*}
    D_\mr{crys}\big(H^i_\mathrm{et}(\mf{X}_C,T_{\et}(\mc{V}))[\nicefrac{1}{p}]\big)\simeq H^i_\mr{crys}\big((\mf{X}_k/W(k)_\mr{crys},T_\mr{crys}^+(\mc{E})\big).
\end{equation*}
Moreover, for all $1\leqslant n\leqslant \infty$ we have a natural identification
\begin{equation*}
    T_{\et}(Rf_\ast(\mc{V}/p^n))\simeq R(f_\eta)_\ast T_{\et}(\mc{V}/p^n).
\end{equation*}
\end{corollary}

\begin{definition}[{Bloch--Kato's finite part of Galois cohomology}] Suppose now that $K$ is a finite extension of $\bb{Q}_p$. In \cite[\S3]{BlochKato} one finds the definition, for a Galois $\bb{Q}_p$-representation $V$ of $K$, of a subspace $H^1_f(K,V)$ of the Galois cohomology of $V$ called its \emph{finite part}. When $V$ is crystalline one naturally has
\begin{equation*}
    H^1_f(K,V)\simeq \mathrm{Ext}^1(\bb{Q}_p,V)
\end{equation*}
where these extensions are taken in the category $\mr{Rep}^{\mr{crys}}_{\bb{Q}_p}(\mr{Gal}(\ov{K}/K))$ (e.g., see \cite[Proposition 4.5.2]{EmertonKisin}).
\end{definition}

The following is a result of combining Theorem \ref{thm:etale-realization-and-pushforward} and \cite[Proposition 4.9]{Pentland}.

\begin{proposition} Suppose that $K$ is a finite extension of $\bb{Q}_p$, and let $\mathfrak{X}$ be a smooth proper formal $\mathscr{O}_K$-scheme and $\mc{V}$ a prismatic $F$-gauge on $\mf{X}$. Then, for all $i\geqslant 0$ we have a natural identification
\begin{equation*}
    H^1_f\big(H^i_\mathrm{et}(\mf{X}_C,T_{\et}(\mc{V}))[\nicefrac{1}{p}]\big)\simeq H^1(\mathscr{O}_K^\mr{syn},R^if_\ast\mc{V})[\nicefrac{1}{p}].
\end{equation*}
\end{proposition}

\section{\texorpdfstring{$(\mathcal{G},\mu)$}{(G,mu)}-apertures}\label{sec:gmudisp}

The purpose of this section is to recall the results of~\cite{gmm}, and to record some useful complements.

\subsection{The main representability result}
\label{subsec:the_main_results}

This subsection is a quick review of~\cite[\S 9]{gmm}.

\begin{setup}
\label{setup:g_mu}
  In the following, $\mathcal{G}$ will be a smooth connected affine group scheme over $\Int_p$, $\hat{\mathcal{O}}$ the ring of integers in a finite unramified extension of $\Rat_p$, and $\mu\colon\Gmh{\hat{\mathcal{O}}}\to \mathcal{G}_{\hat{\mathcal{O}}}$ a \defnword{$1$-bounded} cocharacter, i.e., one whose weights on $\Lie(\mc{G})_{\hat{\mathcal{O}}}$ via the adjoint action are bounded above by $1$. Let $\mathcal{P}^-_\mu\subset \mathcal{G}$ be the smooth subgroup scheme whose Lie algebra is identified with the sum of the weight-$i$ spaces in $\Lie(\mathcal{G})_{\hat{\mathcal{O}}}$ for $i\leqslant 0$. When $\mathcal{G}$ is reductive, $\mu$ is minuscule and $\mathcal{P}^-_\mu$ is a parabolic subgroup associated with $\mu$.
\end{setup}

\begin{remark}
  \label{rem:torsor_Qmu}
Associated with $\mu$ is the map of classifying stacks $B\Gmh{\hat{\mathcal{O}}}\to B \mathcal{G}_{\hat{\mathcal{O}}}$. This classifies a $\mathcal{G}$-torsor over $B\Gmh{\hat{\mathcal{O}}}$ which we denote by $\mathcal{Q}_\mu$.
\end{remark}

\begin{definition}
For $R\in \mathrm{CRing}^{p\text{-comp}}_{\hat{\mathcal{O}}/}$, an $n$-\defnword{truncated} $(\mathcal{G},\mu)$-\defnword{aperture} over $R$ is a $\mathcal{G}$-torsor $\mathfrak{Q}$ over $R^{\mathrm{syn}}\otimes\Int/p^n\Int$\footnote{By this, we mean a map of $p$-adic formal stacks $R^{\mathrm{syn}}\otimes\Int/p^n\Int\to B \mathcal{G}$. } satisfying the following condition: For every geometric point $R\to \kappa$ of $\Spf R$, the restriction of $(x^{\smallN}_{dR})^*\mathfrak{Q}$ (see Remark~\ref{rem:de_rham_point}) over $B\Gm\times \Spec \kappa$ is isomorphic to that of $\mathcal{Q}_\mu$. These organize into an $\infty$-groupoid $\BT{n}(R)$, and the assignment $R\mapsto \BT{n}(R)$ is a derived $p$-adic formal prestack. We will also set
\[
\BT{\infty} = \varprojlim_n \BT{n}.
\]
Objects in $\BT{\infty}(R)$ will be called \defnword{$(\mathcal{G},\mu)$-apertures} over $R$. They classify $\mathcal{G}$-torsors on $R^{\mathrm{syn}}$ that are bounded by $\mu$ in the sense explained above.
\end{definition}

\begin{remark}
\label{rem:mu_only_conj_class_needed}
Note that the isomorphism class of $\mathcal{Q}_\mu$ depends only on the conjugacy class of $\mu$. This implies that the stacks $\BT{n}$ also only depend on this conjugacy class and not on the particular choice of representative.
\end{remark}

The next result is~\cite[Theorem 9.3.2]{gmm} (see Remark \ref{rem:hodge-filtered_de_rham} for the definition of the maps to $B\mc{P}_\mu^{-}$).
\begin{theorem}
\label{thm:gmm}
The formal prestack $\BT{n}$ is represented by a zero-dimensional connected quasi-compact smooth $p$-adic formal Artin stack over $\hat{\mathcal{O}}$ with affine diagonal. Moreover, for any pro-nilpotent divided power thickening $R'\twoheadrightarrow R$ in $\mathrm{CRing}^{p\emph{-comp}}_{\hat{\mathcal{O}}/}$, we have a Cartesian \emph{Grothendieck--Messing} square\footnote{This diagram does depend on the choice of PD structure on $\ker(R'\to R)$, as well as the fact that it is (pro-)nilpotent. For example, one may apply it to the usual trivial PD structure on $\ker(\bb{Z}/4\bb{Z}\to \bb{Z}/2\bb{Z})$ but not to that inherited from the PD structure on $2\mathbb{Z}_2\subseteq \mathbb{Z}_2$.}
\begin{equation}\label{eqn:commuting_square}
\Square{\BT{n}(R')}{}{B \mathcal{P}^{-}_\mu(R'/{}^{\mathbb{L}}p^n)}{}{}{\BT{n}(R)}{}{B \mathcal{P}^{-}_\mu(R/{}^{\mathbb{L}}p^n)\times_{B \mathcal{G}(R/{}^{\mathbb{L}}p^n)}B \mathcal{G}(R'/{}^{\mathbb{L}}p^n).}
\end{equation}
The transition maps $\BT{n+1}\to\BT{n}$ are smooth and surjective. 
\end{theorem}

\begin{example}
    [The case of $\GL_n$ and truncated $p$-divisible groups]
\label{ex:gl_n_truncated_bts}
A basic example is $(\GL_h,\mu_d)$ where $d\leqslant h$ and $\mu_d$ is a minuscule cocharacter of $\GL_h$ defined over $\Int_p$ and splitting the direct summand $\Int^{h-d}\times \{0\}\subset \Int^h$. In this case, it is shown in~\cite[Theorem 11.2.7]{gmm} that $\BT[\GL_h,\mu_d]{n}$ is canonically isomorphic to the stack of $n$-truncated Barsotti--Tate (or $p$-divisible) groups of height $h$ and dimension $d$. This is a refinement of Theorem~\ref{thm:dieudonne}.
\end{example}

\begin{remark}
    [The case of central cocharacters]
\label{rem:central_cochar_case}
When $\mu = 0$ is the trivial cocharacter, $\BT{n}$ is canonically isomorphic to the classifying stack $\mathrm{Loc}_{\mathcal{G}(\Int/p^n\Int)}$ over $\Spf \hat{\mathcal{O}}$; see \cite[Proposition 9.4.1]{gmm}. When $\mu$ is central, $\BT{n}$ is a \emph{non-canonically} trivial gerbe over $\Spf \hat{\mathcal{O}}$ banded by $\mathcal{G}(\Int/p^n\Int)$; see \cite[Proposition 10.3.4]{gmm}. More precisely, in this case, we have $\mathcal{P}^-_\mu = \mathcal{G}$, and so $\BT{n}$ is \'etale over $\Spf\hat{\mathcal{O}}$. Moreover, $\BT{n}(\hat{\mathcal{O}})$ is non-empty, and, for any $\mf{Q}\in \BT{n}(\hat{\mathcal{O}})$, the automorphism scheme $\underline{\Aut}(\mf{Q})$ (defined via the diagonal of $\BT{n}$) is finite \'etale and \'etale locally isomorphic to $\underline{\mathcal{G}(\Int/p^n\Int)}$. Finally, the scheme $\underline{\mr{Isom}}(\mf{Q},\mf{Q}')$ of isomorphisms from $\mf{Q}$ to any other section of $\BT{n}$ is non-empty.
\end{remark}

\begin{remark}
    [Lubin--Tate apertures]
\label{rem:lt_apertures}
Still assuming $\mu$ to be central, the objects of $\BT{n}(\hat{\mc{O}})$ were termed \emph{Lubin--Tate apertures} in~\cite{gmm}. To explain this terminology, we assume for simplicity that $\mathcal{G}$ is reductive. Let $\mathcal{Z} \defn Z({\mathcal{G}})^\circ\subset \mathcal{G}$ be the connected center. Then $\mu$ factors through $\mathcal{Z}$, and we can consider its `reflex norm'
\[
r_\mu\colon \Res_{\hat{\mc{O}}/\Int_p}\bb{G}_{m,\hat{\mc{O}}}\xrightarrow{\Res_{\hat{\mc{O}}/\Int_p}\mu}\Res_{\hat{\mc{O}}/\Int_p}\mathcal{Z}\xrightarrow{\mathrm{Nm}_{\hat{\mc{O}}/\Int_p}}\mathcal{Z}.
\]
If $\mathcal{T}_0 = \Res_{\hat{\mc{O}}/\Int_p}\bb{G}_{m,\hat{\mc{O}}}$, we have a canonical cocharacter $\mu_0\colon \Gmh{\hat{\mc{O}}}\to \mc{T}_{0,\hat{\mc{O}}}$ such that $r_\mu\circ \mu_0 = \mu$. This gives us a map of pairs $(\mathcal{T}_0,\mu_0)\to (\mathcal{G},\mu)$, and hence a map $\BT[\mathcal{T}_0,\mu_0]{n}\to \BT{n}$. Now, examples of objects in $\BT[\mathcal{T}_0,\mu_0]{n}(\hat{\mc{O}})$ can be obtained from Lubin--Tate formal $\hat{\mc{O}}$-modules; see~\cite[Proposition 11.7.6]{gmm}. In particular, the twist of the standard representation of $\mathcal{T}_0$ on $\hat{\mc{O}}$ by such an aperture will yield a vector bundle over $\hat{\mc{O}}^{\mathrm{syn}}\times \Spf\hat{\mc{O}}$, whose stack of global sections will be isomorphic via Theorem~\ref{thm:dieudonne} to the associated Lubin--Tate formal $\hat{\mc{O}}$-module.
\end{remark}

\begin{remark}
  [Deformation rings]
 \label{rem:bt_deformation_rings}
For any point $x\in \BT{\infty}(\kappa)$ valued in a perfect field $\kappa$ of characteristic $p$, the deformation functor for $\BT{\infty}$ at $x$ takes discrete values and is represented by a complete local formally smooth $W(\kappa)$-algebra of relative dimension $\dim \mathcal{G} - \dim \mathcal{P}^-_\mu$; see \cite[Lemma 10.2.5]{gmm}. It can also be seen from results of Ito~\cite{ito2023deformation}, combined with \emph{a priori} knowledge of the formal smoothness of $\BT{\infty}$; see~\cite[Proposition 3.32]{imai2023prismatic}. 
\end{remark}

\begin{remark}
 [Classicality]
\label{rem:bt_classicality}
There is a somewhat interesting phenomenon here. When $R$ is a discrete $p$-complete ring, even though $R^{\mathrm{syn}}$ is in general only a derived stack that is covered flat locally by the spectra of animated (not necessarily discrete) rings, $\BT{n}(R)$ is still just a $1$-groupoid (or a $1$-truncated $\infty$-groupoid). When $R/pR$ satisfies a mild finiteness condition, this is now explained by~\cite[Remark 9.9.11]{Madapusi2025-ay}, which tells us that $\BT{n}(R)$ can be computed using only the classical truncation of $R^{\mathrm{syn}}$.
\end{remark}

\subsection{Frames and windows}
\label{sub:frames_and_windows}

In~\cite[\S 5 and \S 10]{gmm}, we find a general discussion of frames and windows, which allow us to give more concrete descriptions of the category of $(\mathcal{G},\mu)$-apertures in certain situations. Here, we will isolate a particular instance that will prove very useful to us.

\begin{definition}
[Breuil--Kisin frames]
Let $R$ be a $p$-complete ring. A \defnword{Breuil--Kisin frame} for $R$ is a prism $\underline{A} = (A,I')$ such that $A$ is $p$-complete and $p$-torsion free, $I'\subset A$ is an ideal and $R = A/I'$.
\end{definition}

\begin{example}
  [Frames for base $\Reg{K}$-algebras]
\label{ex:bk_frames_for_base_algebras}
Suppose that $K$ is a complete discrete valuation field with perfect residue field $\kappa$, and set $K_0 = W(\kappa)[\nicefrac{1}{p}]\subset K$. Fix a monic Eisenstein polynomial $E(u)\in W(\kappa)[u]$ such that $E(\pi) = 0$ for some uniformizer $\pi\in \Reg{K}$. Suppose that $R$ is a base $\Reg{K}$-algebra in the sense of~\cite[\S 1.1]{imai2024tannakianframeworkprismaticfcrystals}: Then $R = R_0\otimes_{W(\kappa)}\ms{O}_K$ where $R_0$ is a flat $W(\kappa)$-algebra that admits a Frobenius lift and hence a $\delta$-structure. The prism $\big(\Sig_{R},(E(u))\big)$ where $\Sig_{R} = R_0\ll u\rr$ is equipped with the $\delta$-$R_0$-algebra structure satisfying $\delta(u) = 0$ is now a Breuil--Kisin frame for $R$.
\end{example}

\begin{example}
  [Frames for perfectoid rings]
\label{ex:perfectoid_frame}
If $R$ is a perfectoid ring, then $\underline{\Prism}_R = (\Prism_R,\Fil^1_{\smallN}\Prism_R)$ is a Breuil--Kisin frame for $R$. Moreover, one has an identification $\Prism_R\simeq \mr{A}_\mr{inf}(R)$ under which $\Fil^1_\smallN \Prism_R$ corresponds to $\varphi^{-1}(
\ker\theta)$ where $\theta\colon \mr{A}_\mr{inf}(R)\to R$ is Fontaine's map.
\end{example}

\begin{example}
  [Frames via Frobenius liftings]
\label{ex:frobenius_lifting_frame}
Suppose that $R$ is a flat $\Int_p$-algebra equipped with a Frobenius lift $\varphi\colon R\to R$. We can then view $\big(R,(p)\big)$ as a Breuil--Kisin frame $\underline{R}$ for $R/pR$.
\end{example}

\begin{remark}
  [Liftings of \'etale covers]
\label{rem:frames_liftings}
Suppose that $\underline{A}$ is a Breuil--Kisin frame for $R$. Then any $p$-completely \'etale map $R\to \tilde{R}$ lifts uniquely to a $(p,I')$-completely \'etale map $A\to \tilde{A}$ of $\delta$-rings, and $(\tilde{A},I'\tilde{A})$ is a Breuil--Kisin frame for $\tilde{R}$; see~\cite[Proposition 5.4.23]{gmm}.
\end{remark}

\begin{assumption}
    For the remainder of this subsection, all our rings will be $p$-complete $\hat{\mc{O}}$-algebras. Given such an algebra $R$ and a Breuil--Kisin frame $\underline{A}$ for $R$, the $\hat{\mc{O}}$-algebra structure of $R$ lifts uniquely to one on $\underline{A}$.
\end{assumption}

\begin{remark}
   [The Rees stack associated with a Breuil--Kisin frame]
\label{rem:rees_bk_frame}
Suppose that $\underline{A}$ is a Breuil--Kisin frame for $R$. Set $I = \varphi^*I'\subset A$. We then have the $(p,I')$-complete formal Rees stack $\Rees(\Fil^\bullet_{I'}A)$ associated with the $I'$-adic filtration on $A$, and the Frobenius lift on $A$ extends to a map of Rees stacks $\Rees(\Fil^\bullet_IA)\to \Rees(\Fil^\bullet_{I'}A)$. We obtain two maps $\tau,\sigma\colon\Spf A\to \Rees(\Fil^\bullet_{I'}A)$, where $\tau$ is the pullback of $\Gm/\Gm\to \Aff^1/\Gm$ and $\sigma$ is obtained from the filtered Frobenius lift and the natural isomorphism
\[
   \Spf A {\isomto}\Rees(\Fil^\bullet_{I,\pm}A).
\]
We have a map $x_{\underline{A}}\colon\Aff^1/\Gm\times \Spf R\to \Rees(\Fil^\bullet_{I'} A)$ associated with the map $\Fil^\bullet_{I'}A\to \Fil^\bullet_{\mathrm{triv}}R$ of filtered rings. 
\end{remark}

\begin{definition}
  [$\mathcal{G}$-torsors over the Rees stack bounded by $\mu$]
\label{def:mu_bounded_G-torsors}
A $\mathcal{G}$-torsor $\mathcal{Q}_{\mathrm{Rees}}$ over $\Rees(\Fil^\bullet_{I'}A)$ is \defnword{bounded by $\mu$} if, for any map $R\to \kappa$ to an algebraically closed field $\kappa$, the restriction of $x_{\underline{A}}^*\mathcal{Q}_{\mathrm{Rees}}$ along the composition
\[
  B\Gm\times \Spec \kappa\to B\Gm\times \Spf R \to \Aff^1/\Gm\times\Spf R
\]
is isomorphic to $\mathcal{Q}_\mu$.
\end{definition}

\begin{remark}
[Consequences of $1$-boundedness]
  \label{rem:1-boundedness}
Since we have assumed that $\mu$ is $1$-bounded, a $\mathcal{G}$-torsor $\mathcal{Q}_{\mathrm{Rees}}$ over $\Rees(\Fil^\bullet_{I'}A)$ bounded by $\mu$ is determined up to isomorphism by the following data:
\begin{itemize}
  \item The $\mathcal{G}$-torsor $\mathcal{Q}\defn \tau^* \mathcal{Q}_{\mathrm{Rees}}$ over $\Spf A$;
  \item The $\mathcal{G}$-torsor $x_{\underline{A}}^* \mathcal{Q}_{\mathrm{Rees}}$ over $\Aff^1/\Gm\times \Spf R$, which, by the bounded-by-$\mu$ condition, is equivalent to giving a reduction of structure $Q^-\subset Q$ to a $\mathcal{P}^-_\mu$-torsor for the $\mathcal{G}$-torsor $Q=\mc{Q}|_R$.
\end{itemize}
This can be deduced for instance from~\cite[Proposition 4.12.3]{gmm}.
\end{remark}

\begin{remark}
  [Description in terms of modifications]
\label{rem:modifications}
As a particular consequence of the description in Remark~\ref{rem:1-boundedness}, we find that we can associate with the pair $(\mathcal{Q},Q^-\subset Q)$ the $\mathcal{G}$-torsor $\sigma^* \mathcal{Q}_{\mathrm{Rees}}$ over $\Spf A$. This can be obtained directly from the pair. For simplicity, we will restrict ourselves to the case where $I' = (E)$ is principal with a distinguished choice of generator $E$. Consider the sheaf of groups 
     \[
  H_\mu\colon \tilde{R}\mapsto \mathcal{G}(\tilde{A})\times_{\mathcal{G}(\tilde{R})}\mathcal{P}^-_\mu(\tilde{R})\subset \mathcal{G}(\tilde{A}) \defn (L^+ \mathcal{G})(\tilde{R})
     \]
  on the \'etale site of $\Spf R$,  and note that conjugation by $\mu(E)\in \mathcal{G}(A[\nicefrac{1}{E}])$ yields a map 
  \begin{equation*}
  \mathrm{int}(\mu(E))\colon H_\mu\to L^+ \mathcal{G}.
  \end{equation*}
  We can view $\mathcal{Q}$ as an $L^+\mathcal{G}$-torsor, and $Q^-$ as yielding a reduction of structure group to an $H_\mu$-torsor $\mathcal{Q}^-\subset \mathcal{Q}$. Pushforward of $\mathcal{Q}^-$ along $\mathrm{int}(\mu(E))$ now yields a modification $\mathcal{Q}'$ of $\mathcal{Q}$ and a further pushforward along the endomorphism of $L^+\mathcal{G}$ induced by the Frobenius lift $\varphi$ on $A$ now yields the $\mathcal{G}$-torsor $\sigma^* \mathcal{Q}_{\mathrm{Rees}}$.
\end{remark}

\begin{definition}
  [$(\mathcal{G},\mu)$-windows over Breuil--Kisin frames]
\label{def:windows_bk_frames}
Suppose that $\underline{A}$ is a Breuil--Kisin frame for $R$. A $(\mathcal{G},\mu)$\defnword{-window} over $\underline{A}$ is a $\mathcal{G}$-torsor $\mathcal{Q}$ over $\Rees(\Fil^\bullet_{I'}A)$ bounded by $\mu$ and equipped with an isomorphism $\alpha\colon\sigma^* \mathcal{Q} {\isomto }\tau^* \mathcal{Q}$. Write $\Wind{\underline{A}}{\infty}(R)$ for the groupoid of $(\mathcal{G},\mu)$-windows over $\underline{A}$.
\end{definition}

\begin{example}
  [The case of $\mathcal{G} = \GL_h$]
\label{ex:gl_n_case}
Suppose that $(\mathcal{G},\mu) = (\GL_h,\mu_d)$ from Example~\ref{ex:gl_n_truncated_bts}. Then, by~\cite[Proposition 5.6.8 and Corollary 5.7.4]{gmm} and the argument from~\cite[Proposition 9.5.9]{Madapusi2025-ay}, there is a canonical equivalence of groupoids
\[
  \bigsqcup_{d\leqslant h\in \Int_{\geqslant 0}}\Wind[\GL_h,\mu_d]{\underline{A}}{\infty}(R) \isomto \mathrm{BK}_{\underline{A},\infty}(R)
\]
where the right-hand side is the groupoid of triples $(\mathcal{M},\varphi_{\mathcal{M}},V_{\mathcal{M}})$ where:
\begin{itemize}
  \item $\mathcal{M}$ is a finite locally free $A$-module;
  \item $\varphi_{\mathcal{M}}\colon\varphi^* \mathcal{M}\to \mathcal{M}$ and $V_{\mathcal{M}}\colon I'\otimes_A\mathcal{M}\to \varphi^* \mathcal{M}$ such that the maps
  \[
    V_{\mathcal{M}}\circ (1\otimes \varphi_{\mathcal{M}})\colon I'\otimes_A \varphi^* \mathcal{M}\to \varphi^* \mathcal{M}\;;\; \varphi_{\mathcal{M}}\circ V_{\mathcal{M}}\colon I'\otimes_A \mathcal{M}\to \mathcal{M}
  \]
  are the maps induced by the inclusion $I'\subset A$;
\end{itemize}
It is easy to see that in fact the datum of $V_{\mathcal{M}}$ is redundant and can be replaced with the \emph{condition} that $\varphi_{\mathcal{M}}$ is injective with cokernel a finite locally free module over $R$. In the case of Example~\ref{ex:bk_frames_for_base_algebras}, this is the category of \emph{Breuil windows} over $\Sig_R$ introduced in~\cite{vasiu_zink}.
\end{example}

\begin{remark}
  [Alternate description via modifications]
\label{rem:windows_alternate}
The groupoid $\Wind{\underline{A}}{\infty}(R)$ can also be described as follows using Remark~\ref{rem:modifications} (after having chosen a generator $E\in I'$). An object here is a triple $(\mathcal{Q},Q^-,\alpha)$ where $\mathcal{Q}$ is a $\mathcal{G}$-torsor over $\Spf A$, $Q^-\subset Q$ is a reduction of structure group of $Q \defn \mathcal{Q}\vert_{R}$ to a $\mathcal{P}^-_\mu$-torsor, and $\alpha\colon\varphi^* \mathcal{Q}'{\isomto}\mathcal{Q}$ is an isomorphism of $\mathcal{G}$-torsors, where $\mathcal{Q}'$ is the modification of $\mathcal{Q}$ along $Q^-$. More explicitly, the sheaf of groupoids 
\[
  \tilde{R}\mapsto \Wind{\underline{\tilde{A}}}{\infty}(\tilde{R})
\]
on the \'etale site of $\Spf R$ is the quotient $\big[L^+ \mathcal{G}/H_\mu\big]$ where $H_\mu$ acts on $L^+\mathcal{G}$ via the right action 
\begin{equation*} 
g\cdot h = \tau(h)^{-1}g\varphi\big(\mathrm{int}(\mu(E))(h)\big),
\end{equation*}
where $\tau\colon H_\mu\to L^+ \mathcal{G}$ is the tautological map.
\end{remark}

\begin{remark}
  [Modifications in terms of representations]
\label{rem:modifications_in_terms_of_representations}
Suppose that we have $\Lambda\in \mathrm{Rep}_{\Int_p}(\mathcal{G})$ and let $\mathcal{V}$ be its twist by $\mathcal{Q}$. Then the associated twist $\mathcal{V}'$ by $\mathcal{Q}'$ can be described as follows: Let $V = R\otimes_A \mathcal{V}$. The reduction of structure $Q^-$ yields a filtration $\Fil^\bullet V$. Working \'etale locally on $R$ if necessary we can assume that $\Fil^\bullet V$ is split by a cocharacter of $\mathcal{G}_R$ that is geometrically conjugate to $\mu$: This gives us a splitting $ V= \oplus_i V^i$ where $\Fil^iV = \oplus_{j\geqslant i}V^j$. Lift this to a cocharacter of $\mathcal{G}_A$ and consider the associated splitting $\mathcal{V} = \bigoplus_i \mathcal{V}^i$ (cf.\@ \cite[Proposition (1.1.5)]{Kisin2017-qa}). Then we have the equality
\[
  \mathcal{V}' = \bigoplus_i E^{-i}\mathcal{V}^i = \sum_i E^{-i}\Fil^i \mathcal{V}\subset \mathcal{V}[E^{-1}],
\]
where $\Fil^i \mathcal{V} = \oplus_{j\geqslant i}\mathcal{V}^i$. In particular, the isomorphism $\alpha$ yielding a window gives an isomorphism
\[
  \varphi^*\bigg(\sum_iE^{-i}\Fil^i \mathcal{V}\bigg)\isomto\mathcal{V}.
\]
\end{remark}

\begin{setup}
[Tensor packages]
\label{setup:faithful_repn_tensors}
It will be convenient to make the following choices. Let $\mathcal{G}\to \GL(\Lambda)$ be a faithful representation.  We choose finitely many tensors $\{s_{\alpha}\}\subset \Lambda^\otimes$ such that $\mathcal{G}$ is their pointwise stabilizer: This is always possible by \cite[Theorem 1.1]{Broshi}.
\end{setup}

\begin{remark}
  [Explicit description using tensor packages]
\label{rem:expl_tensor_packages}
Using Remark~\ref{rem:modifications}, we see that giving $\mathfrak{Q}$ in $\Wind{\underline{A}}{\infty}(R)$ is equivalent to the groupoid of triples $(\mathcal{Q},Q^-,\xi)$, where:
\begin{itemize}[leftmargin=.4cm]
  \item $\mathcal{Q}$ is a $\mathcal{G}$-torsor over $A$;
  \item $Q^-\subset Q\defn \mathcal{Q}\vert_{\Spec R}$ is a reduction of structure group to a $\mathcal{P}^-_\mu$-torsor;
  \item $\xi$ is an isomorphism $\varphi^*\mathcal{Q}'\isomto\mathcal{Q}$ of $\mathcal{G}$-torsors over $A$, where $\mathcal{Q}'$ is the modification of $\mathcal{Q}$ along $Q^-$.
\end{itemize}
This can be made even more concrete using a tensor package as in Setup~\ref{setup:faithful_repn_tensors}.
\begin{enumerate}
  \item Giving $\mathcal{Q}$ is equivalent to specifying a finite locally free $A$-module $\mathcal{L}$ and tensors $\{s_{\alpha,\mathcal{L}}\}\subset \mathcal{L}^\otimes$ such that, for some $p$-completely \'etale and faithfully flat map $R\to\widetilde{R}$ with corresponding lift $A\to \widetilde{A}$, there is an isomorphism 
\[
\eta\colon \widetilde{A}\otimes_{\Int_p}\Lambda\isomto \widetilde{A} \otimes_{A}\mathcal{L}
\]
carrying $\{1\otimes s_{\alpha}\}$ to $\{1\otimes s_{\alpha,\mathcal{L}}\}$. 
\item The reduction of structure group $Q^-\subset Q$ amounts to giving a filtration $\Fil^\bullet L$ on $L \defn R\otimes_{A}\mathcal{L}$ such that for some $R\to \widetilde{R}$ as above the isomorphism $\eta$ can be chosen so that the induced filtration on $\widetilde{R}\otimes_{\Int_p}\Lambda$ is split by $\mu$.

\item If we set $\mathcal{L}^m\subset \mathcal{L}$ to be the pre-image of $\Fil^mL$, the modification $\mathcal{Q}'$ now corresponds to the $A$-module
\[
  \mathcal{L}' = \sum_m E^{-m}\mathcal{L}^m\subset \mathcal{L}[\nicefrac{1}{E}]
\]
along with the same collection of tensors $\{s_{\alpha,\mathcal{L}}\}$, now viewed inside $\mathcal{L}^{',\otimes}$.
\item The isomorphism $\xi$ corresponds to an isomorphism of $A$-modules
\[
  \varphi^* \mathcal{L}' = \sum_m\varphi(E)^{-m}\varphi^* \mathcal{L}^m{\isomto} \mathcal{L}
\]
carrying $\{\varphi^*s_{\alpha, \mathcal{L}}\}$ to $\{s_{\alpha, \mathcal{L}}\}$.
\end{enumerate}
\end{remark}

\begin{proposition}
   [Mapping the Rees stack to the filtered prismatization]
\label{prop:mapping_rees_to_filt_prismatization}
There is a canonical map of formal stacks $\iota^{\smallN}_{\underline{A}}\colon\Rees(\Fil^\bullet_{I'} A)\to R^\smallN$ such that
\[
\iota^{\smallN}_{\underline{A}}\circ \tau = j_{\dR}\circ \iota_{(A,I)}\;;\;\iota^{\smallN}_{\underline{A}}\circ \sigma = j_{\mathrm{HT}}\circ \iota_{(A,I)}\;;\; \iota^{\smallN}_{\underline{A}}\circ x_{\underline{A}} = x^{\smallN}_{\dR}.
\]
Therefore, there is a natural functor
\[
  \BT{\infty}(R) \to \Wind{\underline{A}}{\infty}(R).
\]
\end{proposition}
\begin{proof}
   See~\cite[Example 6.10.5]{gmm}.
\end{proof}

The next result is a special case of~\cite[Lemma 9.2.3]{gmm}. 
\begin{proposition}
[Description for perfectoid rings]
  \label{prop:apertures_over_perfectoid_rings}
Suppose that $R$ is perfectoid. Then the functor 
\[
\BT{\infty}(R)\to\Wind{\underline{\Prism_R}}{\infty}(R)
\]
is an equivalence.
\end{proposition}

\begin{remark}
  [Quotient description for perfectoid rings]
\label{rem:quotient_description_perfectoid}
Let $\xi\in \Prism_R$ be a generator for $\Fil^1_{{\smallN}}\Prism_R$.\footnote{Note that in other references, this is often denoted $\tilde{\xi}$.} Another way of phrasing the proposition, using Remark~\ref{rem:windows_alternate}, is that the restriction of the stack $\BT{\infty}$ to perfectoid $R$-algebras is the quotient $[L^+ \mathcal{G}/H_\mu]$ (computed in the $p$-completely \'etale topology) where $H_\mu$ acts on $L^+\mathcal{G}$ via the right action 
\begin{equation*}
g\cdot h = \tau(h)^{-1}g\varphi(\mathrm{int}(\mu(\xi))(h)).
\end{equation*}
Here, $\tau\colon H_\mu\to L^+ \mathcal{G}$ is the tautological map.
\end{remark}

\begin{remark}
[Quotient description for truncated apertures]
    \label{rem:quotient_description_perfectoid_truncated}
In fact a similar description holds for the $n$-truncated versions as well, which can be deduced by combining Remarks 5.5.5 and 5.5.6 and Lemma 9.2.3 of~\cite{gmm}. One now looks at the group
\[
H_\mu^{(n)}(R) = \mathcal{G}(\Prism_R/p^n\Prism_R)\times_{\mathcal{G}(R/{}^{\mathbb{L}}p^n)}\mathcal{P}^-_\mu(R/{}^{\mathbb{L}}p^n).
\]
This admits two maps 
\[
\tau,\sigma\colon H_\mu^{(n)}(R)\to \mathcal{G}(\Prism_R/p^n \Prism_R),
\]
where $\tau$ is simply the projection onto the first coordinate, and $\sigma$ is another map, which we will only describe here for $R$ $p$-torsion free. In this case, we have
\begin{align}\label{eqn:Hmu_n_torsion_free_desc}
 H_\mu^{(n)}(R) &= \mathcal{G}(\Prism_R/p^n\Prism_R)\times_{\mathcal{G}(R/p^nR)}\mathcal{P}^-_\mu(R/p^nR) = \mathcal{G}(\Prism/p^n\Prism_R)\cap \mu(\xi)^{-1} \mathcal{G}(\Prism/p^n\Prism_R)\mu(\xi) \subset \mathcal{G}(\Prism/p^n\Prism_R).
\end{align}
and $\sigma$ is the map $h\mapsto \varphi(\mu(\xi)h\mu(\xi)^{-1})$.

The restriction of $\BT{n}$ to perfectoid $\hat{\mathcal{O}}$-algebras is now given by the \'etale sheafification of $R\mapsto \mathcal{G}(\Prism_R/p^n\Prism_R)/H_{\mu}^{(n)}(R)$, where the action is given by
\[
H_\mu^{(n)}(R)\times \mathcal{G}(\Prism_R/p^n\Prism_R) \xrightarrow{(h,g)\mapsto \tau(h)^{-1}g\sigma(h)}\mathcal{G}(\Prism_R/p^n\Prism_R).
\]
\end{remark}

\subsection{Isocrystals with \texorpdfstring{$G$}{G}-structure}

We now recall the relationship between $(\mc{G},\mu)$-apertures and $F$-(iso)crystals with $\mc{G}$-structure.

\begin{remark}
    [$F$-isocrystals with $G$-structure]
\label{rem:F-isoc_G}
Set $G = \mathcal{G}_{\Rat_p}$ and let the assignment $LG\colon R\mapsto G\big(W(R)[\nicefrac{1}{p}]\big)$ on perfect $\Field_p$-algebras be the Witt vector loop group associated with $G$. Then we can consider the \'etale quotient 
\[
\mathrm{Isoc}_{G} \defn \big[LG/_{\mathrm{Ad}_\varphi} LG\big],
\]
associated with the $\varphi$-twisted adjoint action
\begin{align*}
    \mr{Ad}_\varphi\colon LG \times LG &\to LG\\
    (g,h)&\mapsto h^{-1}g\varphi(h).
\end{align*}
For any perfect field $\kappa$, $\mathrm{Isoc}_G(\kappa)$ is isomorphic to the groupoid of $F$-isocrystals over $\kappa$ with $G$-structure, as in~\cite{kottwitz:isoc}.
\end{remark}

\begin{construction}
[Apertures to isocrystals]
\label{const:newton_map}
By Remark~\ref{rem:quotient_description_perfectoid}, when restricted to perfect $\Field_p$-algebras, we have
\[
  \BT{\infty}(R) = \big[L^+\mathcal{G}/H_\mu\big](R),
\]
where $H_\mu$ is the \'etale sheaf on perfect $\Field_p$-algebras given by $H_\mu = L^+ \mathcal{G}\times_{\mathcal{G}}\mathcal{P}^-_\mu$. The action is via
\begin{align*}
    L^+\mathcal{G} \times H_\mu &\to L^+\mathcal{G}\\
    (g,h)&\mapsto h^{-1}g\varphi(\mu(p)h\mu(p)^{-1}).
\end{align*}

The map $L^+ \mathcal{G}\xrightarrow{g\mapsto g\varphi(\mu(p))} LG$ now descends to a functorial map
\[
  \BT{\infty}(R)\to \mathrm{Isoc}_{G}(R)
\]
for perfect $\Field_p$-algebras $R$.
\end{construction}

\begin{remark}\label{rem:field-desciption-of-BT-mu}
When $R = \kappa$ is an algebraically closed field, then the quotient description of $\BT{\infty}(\kappa)$ appearing above shows that we have 
\[
\BT{\infty}(\kappa) = \big[\mathcal{G}(W(\kappa))/H_\mu(\kappa)\big] \xrightarrow[\sim]{g\mapsto g\varphi(\mu(p))} \big[\mathcal{G}(W(\kappa))\varphi(\mu(p))\mathcal{G}(W(\kappa))/_{\mathrm{Ad}_\varphi}\mathcal{G}(W(\kappa))\big].
\]
The set of isomorphism classes on the right hand side is the set $C(\mathcal{G},\{\varphi(\mu)\})$ from \cite[\S8.2]{ViehmannWedhornPEL}, which we will denote by $C_\kappa(\mathcal{G},\{\varphi(\mu)\})$ to keep track of the dependence on $\kappa$. The set of isomorphism classes in $\mathrm{Isoc}_{G}(\kappa)$ can be identified with Kottwitz's set $B(G)$, which is independent of the choice of algebraically closed field $\kappa$. Within $B(G)$, we have the \emph{$\mu$-admissible} subset $B(G,\{\mu\})$, which depends only on the Galois orbit of the conjugacy class $\{\mu\}$. By~\cite[Theorem 4.2]{rapoport_richartz}, the natural map
\[
C_\kappa(\mathcal{G},\{\varphi(\mu)\})\to B(G)
\]
induced by inclusion $G(W(\kappa))\subset G(W(\kappa)[\nicefrac{1}{p}])$ lands inside $B(G,\{\mu\}) = B(G,\{\varphi(\mu)\})$.
\end{remark}

\begin{lemma}
  \label{lem:kottwitz_map}
For an algebraically closed field $\kappa$, the map $\BT{\infty}(\kappa)\to B(G)$ induced from \emph{Construction~\ref{const:newton_map}} has image $B(G,\{\mu\}) = B(G,\{\varphi(\mu)\})$.
\end{lemma}
\begin{proof}
 Given the previous remark, we only need to know that every $[b]\in B(G,\{\varphi(\mu)\})$ is in the image of $C_\kappa(\mathcal{G},\{\varphi(\mu)\})$. This follows from a result of Wintenberger~\cite[Corollaire 3]{Wintenberger2005-be} (see also the more general non-emptiness result of Gashi~\cite[Theorem 5.2]{MR2778454}): There exists elements $g\in G(W(\kappa)[\nicefrac{1}{p}])$ and $h_1,h_2\in G(W(\kappa))$ such that $g^{-1}b\varphi(g) = h_1\varphi(\mu(p))h_2$.
\end{proof}

\subsection{The de Rham realization}
\label{sub:the_de_rham_realization}

\begin{construction}
  [The Hodge-filtered de Rham realization]
Suppose that we have $\mathfrak{Q}\in \BT{\infty}(R)$. Pulling back along the filtered de Rham point $x^{\smallN}_{\dR}$ from Remark~\ref{rem:de_rham_point} gives us a $\mathcal{G}$-bundle $\Fil^\bullet_{\mathrm{Hdg}}T_{\dR}(\mathfrak{Q})$ over $\Aff^1/\Gm\times \Spf R$. 
\end{construction}

\begin{remark}
\label{rem:hodge-filtered_de_rham}
It follows from the definition of $\BT{\infty}(R)$ and \cite[Remark 4.9.7]{gmm} that $\Fil^\bullet_{\mathrm{Hdg}}T_{\dR}(\mathfrak{Q})$ is \'etale locally on $\Spf R$ isomorphic to the pullback to $\bb{A}^1/\bb{G}_m\times\Spf R$ of the $\mathcal{G}$-torsor $\mathcal{Q}_\mu$ over $B\Gm\times\Spf \hat{\mathcal{O}}$ from Remark~\ref{rem:torsor_Qmu}. Moreover, \emph{loc.\@ cit.\@} also shows that giving such a $\mathcal{G}$-bundle over $\Aff^1/\Gm\times\Spf R$ is equivalent to giving a $\mathcal{P}^-_\mu$-torsor over $R$. That is, we have constructed a canonical map $\BT{\infty}\to B \mathcal{P}^-_\mu$. This is precisely the one showing up in the deformation theory explained in Theorem~\ref{thm:gmm}.
\end{remark}

\begin{remark}
 [Versality and a local model diagram]
\label{rem:versality_bt_n}
Here is a reformulation of the deformation theory in Theorem~\ref{thm:gmm} that will be useful. Let $\mathrm{Gr}_{\mu}$ be the Grassmannian scheme over $\hat{\mathcal{O}}$ associated with $\mu$: It depends only on the conjugacy class of $\mu$ and can be presented as the fppf quotient $\mathcal{G}_{\hat{\mathcal{O}}}/\mathcal{P}^-_\mu$ for our choice of representative $\mu$. We can view the classifying stack $B \mathcal{P}^-_\mu$ from this perspective as the fppf quotient $\mathrm{Gr}_{\mu}/\mathcal{G}_{\hat{\mathcal{O}}}$. Now, $\Fil^\bullet_{\mathrm{Hdg}}T_{\dR}(\mathfrak{Q})$ can be viewed as arising from a canonical map $\BT{\infty}\to \mathrm{Gr}_{\mu}/\mathcal{G}_{\hat{\mathcal{O}}}$. Then Grothendieck--Messing theory tells us that the resulting map on cotangent complexes of $p$-adic formal stacks
\[
\mathbb{L}^{\wedge}_{\mathrm{Gr}_{\mu}/\mathcal{G}_{\hat{\mathcal{O}}}}\vert_{\BT{\infty}}\to \mathbb{L}^\wedge_{\BT{\infty}}
\]
factors through an isomorphism
\[
\mathbb{L}^{\wedge}_{( \mathrm{Gr}_{\mu}/\mathcal{G}_{\hat{\mathcal{O}}} ) / B \mathcal{G}_{\hat{\mathcal{O}}}}\vert_{\BT{\infty}}\isomto\mathbb{L}^\wedge_{\BT{\infty}}.
\]
In particular, this tells us that there is a canonical isomorphism
\[
\mathbb{L}^\wedge_{\BT{\infty}/(\mathrm{Gr}_\mu/\mathcal{G}_{\hat{\mathcal{O}}})}\isomto \mathbb{L}^\wedge_{B\mathcal{G}_{\widehat{\mathcal{O}}}}[1]\vert_{\BT{\infty}}.
\]
So, $\BT{\infty}$ is formally smooth over $\mathrm{Gr}_\mu/\mathcal{G}_{\hat{\mathcal{O}}}$. Setting $\widetilde{\BT{\infty}} = \BT{\infty}\times_{\mathrm{Gr}_\mu/\mathcal{G}_{\hat{\mc{O}}}}\mathrm{Gr}_\mu$, we obtain a \emph{local model diagram}
\[\begin{tikzcd}[cramped,column sep=2.25em]
	& {\widetilde{\BT{\infty}}} \\
	{\BT{\infty}} && {\mathrm{Gr}_\mu}
	\arrow[from=1-2, to=2-1]
	\arrow[from=1-2, to=2-3]
\end{tikzcd}\]
where the left arrow is a $\mathcal{G}_{\hat{\mc{O}}}$-torsor and the right arrow is formally smooth and $\mc{G}_{\hat{\mc{O
}}}$-equivariant with relative cotangent complex a vector bundle of rank $\dim \mathcal{G}$.
\end{remark}

\subsection{Realization in \texorpdfstring{$F$}{F}-zips with additional structure}
\label{sub:Fzips}

We now quickly review the relationship between apertures and $F$-zips as in \cite{Pink2015-ye}.

\begin{notation}
    Let $\mathcal{P}^+_{\varphi(\mu)}\subset \mathcal{G}$ be the smooth subgroup scheme whose Lie algebra is identified with the sum of the weight-$i$ spaces in $\Lie(\mc{G})_{\hat{\mathcal{O}}}$ for $\varphi(\mu)$ with $i\geqslant 0$. Let $\mathcal{U}^-_\mu$ (resp.\@ $\mathcal{U}^+_{\varphi(\mu)}$) be the smooth subgroup scheme of $\mathcal{P}^-_\mu$ (resp.\@ $\mathcal{P}^+_{\varphi(\mu)}$) whose Lie algebra is the sum of the weight-$i$ spaces for $\mu$ (resp.\@ for $\varphi(\mu)$) with $i<0$ (resp.\@ $i>0$). Set $\mathcal{M}_\mu = \mathcal{P}^-_\mu/\mathcal{U}^-_\mu$ and $\mathcal{M}_{\varphi(\mu)}= \mathcal{P}^+_{\varphi(\mu)}/\mathcal{U}^+_{\varphi(\mu)}$. We have a canonical isomorphism of group schemes $\varphi^* \mathcal{M}_{\mu}{\isomto}\mathcal{M}_{\varphi(\mu)}$. 
\end{notation}

\begin{definition}
    [$\mathcal{G}$-zips]
Let $R$ be a $k$-algebra where $k$ is the residue field of $\hat{\mathcal{O}}$.  An \defnword{$F$-zip with $\mathcal{G}$-structure} or \defnword{$\mathcal{G}$-zip} of type $\mu$ over $R$ (see~\cite[Definition 1.4]{Pink2015-ye}) is a tuple
\[
(\mathcal{Q},\mathcal{Q}^-,\mathcal{Q}^+,\iota)
\]
where $\mathcal{Q}$ is a $\mathcal{G}$-torsor over $R$, $\mathcal{Q}^-\subset \mathcal{Q}$ is a reduction of structure to a $\mathcal{P}^-_\mu$-torsor, $\mathcal{Q}^+\subset \mathcal{Q}$ is a reduction of structure to a $\mathcal{P}^+_{\varphi(\mu)}$-torsor and
\[
\alpha\colon \varphi^*(\mathcal{Q}^-/\mathcal{U}^-_\mu){\isomto}\mathcal{Q}^+/\mathcal{U}^+_{\varphi(\mu)}
\]
is an isomorphism of $\mathcal{M}_{\varphi(\mu)}$-torsors. We denote the groupoid of such objects by $\mathcal{G}\text{-}\mathrm{zip}_{\mu}(R)$. 
\end{definition}

\begin{remark}
    By \cite[Theorem 1.5]{Pink2015-ye}, $R\mapsto \mathcal{G}\text{-}\mathrm{zip}_{\mu}(R)$ is a smooth $0$-dimensional Artin stack over $k$.
\end{remark}

\begin{remark}
    [Stratification of $ \mathcal{G}\text{-}\mathrm{zip}_{\mu}$]
\label{rem:stratification_G-zips}
Suppose that $\mathcal{G}$ is \emph{reductive}. Then by~\cite[Theorem 3.20]{Pink2015-ye}, the stack $ \mathcal{G}\text{-}\mathrm{zip}_{\mu}$ has an underlying finite $T_0$-topological space corresponding to a certain subset ${}^J W\subset W$ of the Weyl group of $\mathcal{G}$ endowed with a refinement of the Bruhat ordering; see the explanation in~\cite[\S5.3]{wortmann:ordinary}.\footnote{Note that what Wortmann denotes by $\mu$ corresponds to our $\varphi(\mu)$ here.} This yields a stratification of $ \mathcal{G}\text{-}\mathrm{zip}_{\mu}$---explained in \S 3.6 of \emph{loc.\@ cit.\@}---where each stratum is a smooth, locally closed substack. 
\end{remark}

The next result follows from the proof of~\cite[Theorem 9.3.2]{gmm}.
\begin{proposition}
    [Apertures to $\mathcal{G}$-zips]
\label{prop:apertures_to_gzips}
There is a natural smooth surjective map 
\[
\BT{1}\otimes_{\hat{\mathcal{O}}}k\to \mathcal{G}\text{-}\mathrm{zip}_{\mu}
\]
of $0$-dimensional Artin stacks over $k$ that is a gerbe for a connected $p$-torsion finite flat commutative group scheme. Moreover, the $\mathcal{P}^-_\mu$-torsor over $\BT{1}\otimes_{\hat{\mathcal{O}}}k$ pulled back from the universal such one on $\mc{G}\text{-}\mr{zip}_\mu$ agrees with the one described in \emph{Remark~\ref{rem:hodge-filtered_de_rham}}. 
\end{proposition}

\begin{remark}
    [Set-theoretic version]
\label{rem:settheoretic_EO}
A particular consequence of the proposition is that, for any algebraically closed field $\kappa$ over $k$, the map
\[
\BT{1}(\kappa)\to  \mathcal{G}\text{-}\mathrm{zip}_{\mu}(\kappa)
\]
is an isomorphism of groupoids. 
 
If $\mathcal{G}$ is reductive, then, for any algebraically closed field $\kappa$ over $k$, looking at isomorphism classes of $\kappa$-points for the above map gives a map $\tilde{\zeta}\colon C_\kappa(\mathcal{G},\{\varphi(\mu)\})\to {}^JW$, which is described explicitly in~\cite[Proposition 6.7]{wortmann:ordinary}.
\end{remark}

\subsection{The \texorpdfstring{$\mu$}{mu}-ordinary locus}
\label{sub:the_ordinary_locus}

In this subsection, we will suppose that $\mathcal{G}$ is reductive. We'll now look at the $\mu$-ordinary locus of $\BT{n}$. Throughout, $\kappa$ will be a perfect field over $k$.

\begin{definition}
    [The $\mu$-ordinary stratum]
\label{defn:mu_ord_stratum}
Recall the map $\tilde{\zeta}$ from Remark~\ref{rem:settheoretic_EO}. The element $\tilde{\zeta}(\varphi(\mu(p)))\in {}^JW$ corresponds to the unique open stratum of $\mathcal{G}\text{-}\mathrm{zip}_{\mu}$; see for instance the proof of~\cite[Theorem 6.10]{wortmann:ordinary}. Write $\mathcal{G}\text{-}\mathrm{zip}^{\mathrm{ord}}_{\mu}$ for this open stratum. For $n\in \Nat\cup\{\infty\}$, we will write $\BT[\mathcal{G},\mu,\mathrm{ord}]{n}\subset \BT{n}$ for the unique open formal substack characterized by the property that its mod-$p$ fiber is
\[
\BT[\mathcal{G},\mu,\mathrm{ord}]{n}\otimes k = (\BT{n}\otimes k)\times_{\mathcal{G}\text{-}\mathrm{zip}_{\mu}}\mathcal{G}\text{-}\mathrm{zip}^{\mathrm{ord}}_{\mu}.
\]
\end{definition}

\begin{setup}
  Choose a maximal torus $\mathcal{T}\subset \mathcal{G}$ contained in a Borel subgroup of $\mathcal{G}$. We can assume that $\mu$ has been chosen within its conjugacy class to factor through $\mathcal{T}_{\hat{\mathcal{O}}}$.\footnote{By~\cite[Lemma 1.1.3]{kottwitz:twisted}, $\mu$ can be conjugated to a cocharacter factoring through $\mathcal{T}_{\hat{\mathcal{O}}[\nicefrac{1}{p}]}$ dominant with respect to the choice of Borel. This will necessarily extend to a cocharacter of $\mathcal{T}_{\hat{\mathcal{O}}}$.} Suppose that $[\hat{\mathcal{O}}[\nicefrac{1}{p}]:\Rat_p] = d$, so that the cocharacter $\nu = \sum_{i=0}^{d-1}\varphi^i(\mu)$ of $\mathcal{T}$ is defined over $\Int_p$. Let $\mathcal{P}^+_\nu\subset \mathcal{G}$ be the parabolic subgroup whose Lie algebra is the sum of the non-negative eigenspaces for $\nu$, and let $\mathcal{U}^+_\nu$ be its unipotent radical. 
\end{setup}

\begin{lemma}
  \label{lem:mu_Mnu_central}
Let $\mc{M}_\nu\subset \mathcal{G}$ and $\mc{M}_\mu\subset \mathcal{G}_{\hat{\mathcal{O}}}$ be the centralizers of $\nu$ and $\mu$. Then $\mathcal{M}_{\nu,\hat{\mathcal{O}}}\subset \mathcal{M}_\mu$.
\end{lemma}
\begin{proof}
  It is enough to know that $\mu$ is central in $\mc{M}_{\nu,\hat{\mathcal{O}}}$. By~\cite[Lemma 2.7]{Shankar2021-rp}, this would follow if the average $\overline{\mu}$ of the Galois conjugates of $\mu$ is central in $\mc{M}_{\nu,\hat{\mathcal{O}}}$. But this is clear, since $\nu$ is an integer multiple of $\overline{\mu}$.
\end{proof}

\begin{remark}
\label{rem:BTPnu_BTMnu}
    We can view $\mu$ as a minuscule (in fact, central) cocharacter of $\mathcal{M}_\nu$, and as a $1$-bounded cocharacter of $\mathcal{P}^-_\mu$. This gives us formal algebraic stacks $\BT[\mathcal{M}_\nu,\mu]{n}$ and $\BT[\mathcal{P}^+_\nu,\mu]{n}$ over $\hat{\mc{O}}$. Viewing $\mathcal{M}_\nu$ as the Levi quotient of $\mathcal{P}^-_\nu$ we obtain maps
    \[
    \BT[\mathcal{M}_\nu,\mu]{n}\to \BT[\mathcal{P}^+_\nu,\mu]{n}\to \BT[\mathcal{M}_\nu,\mu]{n}
    \]
    whose composition is the identity.
\end{remark}

\begin{remark}
    \label{rem:BTMnu_k-points}
Since $\mu$ is central in $\mathcal{M}_\nu$, by Remark~\ref{rem:field-desciption-of-BT-mu}, we find
\[
\BT[\mathcal{M}_\nu,\mu]{n}(\kappa)\simeq [\mathcal{M}_\nu(W_n(\kappa))\varphi(\mu(p))/_{\mathrm{Ad}_\varphi}\mathcal{M}_\nu(W_n(\kappa))].
\]
By Lang's theorem, the quotient on the right can be identified wtih $[\{\varphi(\mu(p))\}/\mathcal{M}_\nu(\Int/p^n\Int)]$. In particular, if $\kappa$ is algebraically closed, then, up to isomorphism, we have a single object in $\BT[\mathcal{M}_\nu,\mu]{n}(\kappa)$.
\end{remark}

\begin{remark}
[The canonical $F$-gauges]
    \label{rem:BTMnu_canonical_f-gauges}
Suppose that we have $\mf{Q}\in \BT[\mathcal{M}_\nu,\mu]{n}(R)$. Now, $\Lie \mathcal{U}^+_\nu$ admits a canonical ascending filtration $\Lie \mathcal{U}^+_{\nu,\bullet}$, whose associated graded module is isomorphic to the weight decomposition for $\Lie \mathcal{U}^+_\nu$ with respect to $\nu$. This in turn integrates to an ascending filtration $\mathcal{U}^+_{\nu,\bullet}$ of unipotent group schemes on $\mathcal{U}^+_\nu$ that is stable under the action of $\mathcal{M}_\mu$, and whose graded pieces $\mathcal{U}^+_{\nu,[i-1,i]} \defn \mathcal{U}^+_{\nu,i}/\mathcal{U}^+_{\nu,i-1}$ are commutative, and on which $\nu(p)$ acts via multiplication by $p^i$. Twisting each $\mathcal{U}^+_{\nu,[i-1,i]}$ by $\mf{Q}$ gives us a vector group scheme $\mathbf{V}(\mathcal{M}_i)$ associated with a vector bundle $\mathcal{M}_{i}$ over $R^{\mathrm{syn}}\otimes\Int/p^n\Int$---in other words, with a $p^n$-torsion $F$-gauge over $R$. 
\end{remark}

\begin{lemma}
[The canonical finite flat group schemes]
\label{lem:BTMnu_canonical_ffgs}
 The functors
\begin{equation*}
C \mapsto \tau^{\leqslant 0}R\Gamma\big(C^{\mathrm{syn}}\otimes\Int/p^n\Int,\mathcal{M}_{i}\big),
\end{equation*}
and
\begin{equation*}C\mapsto \tau^{\leqslant 0}R\Gamma\big(C^{\mathrm{syn}}\otimes\Int/p^n\Int,\mathcal{M}_{i}[1]\big) \simeq \Map_{/R^{\mathrm{syn}}\otimes\Int/p^n\Int}\big(C^{\mathrm{syn}}\otimes\Int/p^n\Int,B\mathbf{V}(\mathcal{M}_{i})\big),
\end{equation*}
on $p$-nilpotent $R$-algebras are represented by a connected $n$-truncated Barsotti--Tate group scheme $\mathcal{F}_{\nu,i}(n)_R$ over $R$ and its classifying stack $B\mc{F}_{\nu,i}(n)_R$, respectively. 
\end{lemma}
\begin{proof}
    We first claim that the $F$-gauges $\mathcal{M}_i$ have Hodge--Tate weights $0,1$. Given this, everything except the connectedness of $\mathcal{F}_{\nu,i}(n)_R$ follows from~\cite[Theorem 11.3.3]{gmm} and~\cite[Proposition 4.3.2]{Madapusi2025-ay}.

    To see the condition on the Hodge--Tate weights, as well as the connectedness, we can assume that $R = \kappa$ is an algebraically closed field over $k$. Here, unwinding definitions, and using Remark~\ref{rem:BTMnu_k-points}, shows that $\mathcal{M}_i$ is isomorphic to the mod-$p^n$ reduction of the $F$-gauge associated with the Dieudonn\'e module $W(\kappa)\otimes_{\Int_p}\Lie\mathcal{U}^+_{\nu,[i-1,i]}$ equipped with the Frobenius operator $\varphi\otimes \varphi(\mu(p))^{-1}$, where $\varphi(\mu(p))$ is acting via the adjoint action: Note that, since $\varphi(\mu)$ acts on $\Lie \mathcal{U}^+_\nu$ via $z\mapsto 1$ and $z\mapsto z^{-1}$, this indeed gives a Dieudonn\'e module. This shows both that the Hodge--Tate weights are as desired, and also that the group scheme $\mathcal{F}_{\nu,i}(n)_\kappa$ is connected: The latter is because the slope cocharacter of the $F$-isocrystal underlying $W(\kappa)\otimes_{\Int_p}\Lie\mathcal{U}^+_{\nu,[i-1,i]}$ is a rational multiple of $\nu$ by construction, and is therefore constant and non-zero.
\end{proof}

\begin{proposition}
  \label{prop:BTPnu_structure}
Suppose that $1\leqslant n<\infty$. The formal stack $\BT[\mathcal{M}_\nu,\mu]{n}$ is \'etale over $\hat{\mathcal{O}}$ and is non-canonically isomorphic to the formal classifying stack for the finite group $\mathcal{M}_\nu(\Int/p^n\Int)$. Moreover, there is a canonical (non-commutative) finite flat group scheme $\mathcal{F}_{\nu}(n)$ over  $\BT[\mathcal{M}_\nu,\mu]{n}$ such that:
\begin{enumerate}
    \item $\mathcal{F}_{\nu}(n)$ admits a canonical filtration by finite flat subgroup schemes over $\BT[\mathcal{M}_\nu,\mu]{n}$, whose graded pieces $\mathcal{F}_{\nu,i}(n)$ are the $n$-truncated connected Barsotti--Tate group schemes from \emph{Lemma~\ref{lem:BTMnu_canonical_ffgs}}; 
    \item For $n>1$ and all $i$, there is a canonical isomorphism
    \[
    \mathcal{F}_{\nu,i}(n-1) \simeq \mathcal{F}_{\nu,i}(n)[p^{n-1}]
    \]
    of $(n-1)$-truncated Barsotti--Tate group schemes;
    \item The map $\BT[\mathcal{P}^+_\nu,\mu]{n}\to \BT[\mathcal{M}_\nu,\mu]{n}$ is canonically isomorphic to the classifying stack of $\mathcal{F}_{\nu}(n)$.
\end{enumerate}
\end{proposition}
\begin{proof}
The first assertion follows from Remark~\ref{rem:central_cochar_case}.

For the rest, suppose that we have a finite \'etale $\hat{\mathcal{O}}$-algebra $R$, and $\mathfrak{Q}\in \BT[\mathcal{M}_\nu,\mu]{n}(R)$ corresponding to an $\mathcal{M}_\nu$-torsor over $R^{\mathrm{syn}}\otimes\Int/p^n\Int$. We can twist (via conjugation) $\mathcal{U}^+_\nu$ by this torsor to get a unipotent (relative) group scheme $(\mathcal{U}^+_\nu)_{\mathfrak{Q}}$ over $R^{\mathrm{syn}}\otimes\Int/p^n\Int$. Unwinding definitions shows that the fiber of $\BT[\mathcal{P}^+_\nu,\mu]{n}$ over $\mathfrak{Q}$ is the functor on $p$-nilpotent $R$-algebras $C$ given by
\begin{align}
\label{eqn:BPnu_over_BMnu_explicit}
C&\mapsto \Map_{/R^{\mathrm{syn}}\otimes\Int/p^n\Int}(C^{\mathrm{syn}}\otimes\Int/p^n\Int,B(\mathcal{U}^+_\nu)_{\mathfrak{Q}}).
\end{align}
Now, Remark \ref{rem:BTMnu_canonical_f-gauges} and Lemma~\ref{lem:BTMnu_canonical_ffgs} tells us
\[
C\mapsto \Map_{/R^{\mathrm{syn}}\otimes\Int/p^n\Int}(C^{\mathrm{syn}}\otimes\Int/p^n\Int,(\mathcal{U}^+_\nu)_{\mathfrak{Q}})
\]
is represented by a finite flat group scheme $\mathcal{F}_{\nu}(n)_R$ over $R$ admitting a filtration as stated, and that the functor~\eqref{eqn:BPnu_over_BMnu_explicit} is the classifying stack of the finite flat group scheme. By varying $R$ and $\mathfrak{Q}$, one now gets the desired finite flat group scheme over $\BT[\mathcal{M}_\nu,\mu]{n}$.
\end{proof}

\begin{corollary}
    \label{cor:BPnu_complete_local_ring}
Suppose that we have $\mathfrak{P}_0\in \BT[\mathcal{P}^+_\nu,\mu]{\infty}(\kappa)$. Then the universal deformation space $\widehat{U}_{\mathfrak{P}_0}$ of $\BT[\mathcal{P}^+_\nu,\mu]{\infty}$ at $\mathfrak{P}_0$ admits a canonical structure of a tower
    \[
    \widehat{U}_{\mathfrak{P}_0} = \widehat{U}_{n}\to \widehat{U}_{n-1}\to \cdots \to \widehat{U}_{0} = \Spf W(\kappa)
    \]
    of formal schemes over $W(\kappa)$, where, for each $i>0$, $\widehat{U}_{i}$ is a formal $p$-divisible group over $\widehat{U}_{i-1}$;
\end{corollary}
\begin{proof}
    Given Proposition~\ref{prop:BTPnu_structure}, this is implied by Lemma~\ref{lem:classifying_space_complete_local} below.
\end{proof}

\begin{lemma}
\label{lem:classifying_space_complete_local}
Let $R$ be a complete local ring with residue field $\kappa$, and let $\mathrm{Art}_{R}$ be the category of Artin local $R$-algebras with residue field $\kappa$. Let $\mathcal{H}$ be a $p$-divisible group over $R$, and let $0$ denote the trivial object of $(B \mathcal{H})(\kappa)$. Then the deformation functor 
\[
\mathrm{Def}_{B \mathcal{H},0}\colon C\mapsto \mathrm{fib}_{0}\bigg((B \mathcal{H})(C)\to (B \mathcal{H})(\kappa)\bigg)
\]
on $\mathrm{Art}_R$ is pro-represented by the $p$-divisible formal group $\hat{\mathcal{H}}$.
\end{lemma} 
\begin{proof}
This is classical and can be deduced from results in~\cite{katz:serre-tate}.
\end{proof}

\begin{proposition}
    [Canonical Frobenius lifting]
\label{prop:frobenius_liftings_BTP}
Set $q_0 = p^d = |k|$. There is a canonical endomorphism 
\begin{equation*} 
\tilde{\Phi}\colon \BT[\mathcal{P}_\nu^+,\mu]{n}\to \BT[\mathcal{P}_\nu^+,\mu]{n}
\end{equation*}
of formal $\hat{\mc{O}}$-stacks lifting the $q_0$-Frobenius endomorphism of $\BT[\mathcal{P}_\nu^+,\mu]{n}\otimes k(v)$.
\end{proposition}
\begin{proof}
    To begin, since $\BT[\mathcal{M}_\nu,\mu]{n}\otimes k(v)$ is isomorphic to the classifying stack of a constant group scheme, the action of $q_0$-Frobenius on it is canonically trivial. Now, conjugation by $\nu(p)^{-1}$ induces an $\mathcal{M}_\nu$-equivariant endomorphism of $\mathcal{U}^+_\nu$, and so in turn induces an endomorphism of the finite flat group scheme $\mathcal{F}_\nu(n)$, and hence of its classifying stack. By Proposition~\ref{prop:BTPnu_structure}, this gives an endomorphism $\tilde{\Phi}$ of $\BT[\mathcal{P}_\nu^+,\mu]{n}$ as a stack over $\BT[\mathcal{M}_\nu,\mu]{n}$, and one can use the discussion in the proof of Lemma~\ref{lem:BTMnu_canonical_ffgs} to deduce that this is a lift of the $q_0$-Frobenius endomorphism.
\end{proof}

The connection of the above discussion with the ordinary locus is made by:
\begin{proposition}
\label{prop:ordinary_locus_description}
    For $n\in \Nat$, the natural map $\BT[\mathcal{P}^+_\nu,\mu]{n}\to \BT{n}$ induces an isomorphism
    \[
    \BT[\mathcal{P}^+_\nu,\mu]{n}\isomto\BT[\mathcal{G},\mu,\mathrm{ord}]{n}
    \]
 \end{proposition}
 \begin{proof}
     That the natural map is \'etale is shown in~\cite[Proposition 9.5.1]{gmm}. 
     
     Now, Proposition~\ref{prop:BTPnu_structure} and Remark~\ref{rem:BTMnu_k-points} show that for any $\kappa$ we have 
     \[
     \BT[\mathcal{P}^+_\nu,\mu]{n}(\kappa)\isomto  \BT[\mathcal{M}_\nu,\mu]{n}(\kappa)\simeq [\mathcal{M}_\nu(W_n(\kappa))\varphi(\mu(p))/_{\mathrm{Ad}_\varphi}\mathcal{M}_\nu(W_n(\kappa))].
     \]
     where the target has a single isomorphism class represented by $\varphi(\mu(p))$, whose automorphism group is $\mathcal{M}_\nu(\Int/p^n\Int)$. It now follows from Lemma~\ref{lem:kottwitz_lemma} below that the image of $\BT[\mathcal{P}^+_\nu,\mu]{n}$ is precisely the $\mu$-ordinary locus in $\BT{n}$.

     Moreover, the automorphism group of $\varphi(\mu(p))$ viewed as a point of $\BT{n}(\kappa) = [\mathcal{G}(W_n(\kappa))/H^{(n)}_\mu(\kappa)]$ is also $\mathcal{M}_\nu(\Int/p^n\Int)$. This shows that the map
     \[
      \BT[\mathcal{P}^+_\nu,\mu]{n}(\kappa)\to \BT{n}(\kappa)
     \]
     is fully faithful, and thus that $\BT[\mathcal{P}^+_\nu,\mu]{n}\to \BT{n}$ is an open immersion, as desired.
 \end{proof}

\begin{notation}
   Following~\cite{wortmann:ordinary}, given $g\in \mathcal{G}(W(\kappa))\varphi(\mu(p))\mathcal{G}(W(\kappa))$, we will write $\langle g\rangle$ for the associated class in $C(\mathcal{G},\{\varphi(\mu)\})$, $[g]$ for the class in $B(G,\{\mu\})$, and $\pow{g}\subset C(\mathcal{G},\{\varphi(\mu)\})$ for the subset $\tilde{\zeta}^{-1}(\zeta(\langle g\rangle))$ where $\tilde{\zeta}$ is the map from Remark~\ref{rem:settheoretic_EO}.
\end{notation}

\begin{lemma}
    \label{lem:kottwitz_lemma}
With the above notation, we have equivalences
\[
[g] = [\varphi(\mu(p))] \Leftrightarrow \langle g\rangle = \langle \varphi(\mu(p))\rangle \Leftrightarrow \pow{g} = \pow{\varphi(\mu(p))}.
\]
\end{lemma}
\begin{proof}
    See~\cite[Proposition 7.2]{wortmann:ordinary}.
\end{proof}

\begin{corollary}
    \label{cor:complete_local_ring_ordinary_locus}
Suppose that we have $\mathfrak{Q}\in\BT[\mathcal{G},\mu,\mathrm{ord}]{\infty}(\kappa)$. Then the universal deformation space $\widehat{U}_{\mathfrak{Q}}$ of $\BT{\infty}$ at $\mathfrak{Q}$ admits a canonical structure of a tower
    \[
    \widehat{U}_{\mathfrak{Q}} = \widehat{U}_{n}\to \widehat{U}_{n-1}\to \cdots \to \widehat{U}_{0} = \Spf W(\kappa)
    \]
    of formal schemes over $W(\kappa)$, where, for each $i>0$, $\widehat{U}_{i}$ is a formal $p$-divisible group over $\widehat{U}_{i-1}$;
\end{corollary}

\begin{remark}
    [Work of Moonen and Shankar-Zhou]
In~\cite{Moonen2004-cc}, Moonen defines the notion of a \emph{cascade} and shows that the deformation spaces of $\mu$-ordinary formal $\mathcal{O}$-modules have a natural cascade structure. This was extended to Hodge type $\mu$-ordinary deformation spaces by Shankar--Zhou~\cite{Shankar2021-rp}. One can recover these structures from the optic of Corollary~\ref{cor:complete_local_ring_ordinary_locus} when $(\mathcal{G},\mu)$ is of \emph{Hodge type}---that is, when we have a closed embedding of it into $(\GL_h,\mu_d)$ for appropriately chosen $h$ and $d$ (see Example~\ref{ex:gl_n_truncated_bts}). However, an interesting point is that the cascade structure does not appear to be canonical: It depends on the choice of representation into $\GL_h$.
\end{remark}

 \begin{remark}
     \label{rem:complete_local_ring_ordinary_locus}
Corollary~\ref{cor:complete_local_ring_ordinary_locus} tells us that, for any aperture $\mathfrak{Q}\in \BT[\mathcal{G},\mu,\mathrm{ord}]{\infty}(\kappa)$, the deformation space $\mathrm{Def}_{\mathfrak{Q}}$ over $W(\kappa)$ admits a canonical section obtained as the successive composition of identity sections of a tower of relative formal $p$-divisible groups.
 \end{remark}

 \begin{definition}
    [Canonical lift for $\mu$-ordinary apertures]
\label{defn:canonical_lift_local}
The lift of $\mathfrak{Q}$ to $\BT{\infty}(W(\kappa))$ corresponding to the canonical section of the deformation space $\mathrm{Def}_{\mathfrak{Q}}$ is called the \defnword{canonical lift} of $\mathfrak{Q}$ and is denoted $\mathfrak{Q}^{\mathrm{can}}$. Unwinding definitions, we find the following: $\mathfrak{Q}$ is in the image of the map $\BT[\mathcal{M}_\nu,\mu]{\infty}(\kappa)\to \BT{\infty}(\kappa)$, where the source is the space of $\kappa$-points of a formally \'etale stack over $W(\kappa)$, giving us an isomorphism 
\[
\BT[\mathcal{M}_\nu,\mu]{\infty}(W(\kappa))\isomto\BT[\mathcal{M}_\nu,\mu]{\infty}(\kappa).
\]
Now $\mathfrak{Q}^{\mathrm{can}}$ is the lift of $\mathfrak{Q}$ along this isomorphism.
\end{definition}

\begin{remark}
[Description in terms of Lubin--Tate groups]
    \label{rem:canonical_lift_algebraically_closed}
If $\kappa$ is algebraically closed, by Remarks~\ref{rem:central_cochar_case} and~\ref{rem:lt_apertures}, $\mf{Q}^{\mr{can}}$ can be obtained as follows. Set $\mathcal{T}_0 = \Res_{\hat{\mc{O}}/\Int_p}\Gm$, and let $r_\mu\colon\mathcal{T}_0\to \mathcal{M}_\nu$ be the reflex norm. Then any Lubin--Tate formal $\hat{\mc{O}}$-module over $W(\kappa)$ will correspond to a $(\mathcal{T}_0,\mu_0)$-aperture over $W(\kappa)$ and its pushforward along $r_\mu$ will give a $(\mathcal{M}_\nu,\mu)$-aperture isomorphic to $\mf{Q}^{\mr{can}}$.
\end{remark}

\subsection{Filtered \texorpdfstring{$F$}{F}-crystal realization}
\label{sub:filtered_f_crystals_and_f_isocrystals}

We now review some material on crystalline realizations of apertures (cf.\@ \cite{IKY3}, which allows one to recover much of the material here via Tannakian considerations.).

\begin{remark}
  [Crystalline realization over semiperfectoid rings]
\label{rem:apertures_semiperfect}
For any semiperfect (animated commutative) $\Field_p$-algebra $R$, we have a canonical comparison isomorphism $\Prism_{R}{\isomto}A_{\mathrm{crys}}(R)$ (see for instance~\cite[Remark 6.9.5]{gmm}). Hence a $\mathcal{G}$-bundle over $R^{\Prism} = \Spf \Prism_R$ is the same as one over $A_{\mathrm{crys}}(R)$. Suppose now that we have $\mathfrak{Q}\in \BT{\infty}(R)$: Then just as in~\cite[Remark 6.3.4]{bhatt_lectures}, the  $\mathcal{G}$-bundle $T_{\mathrm{crys}}(\mathfrak{Q})$ in crystals associated with $j_{\Prism}^* \mathfrak{Q}$ is equipped with an isomorphism $\varphi^*T_{\mathrm{crys}}(\mathfrak{Q})[\nicefrac{1}{p}]{\isomto}T_{\mathrm{crys}}(\mathfrak{Q})[\nicefrac{1}{p}]$ of $\mathcal{G}$-bundles over $A_{\mathrm{crys}}(R)$.
\end{remark}

\begin{construction}
  [Crystalline realization]
\label{const:crystalline_realization}
By quasisyntomic descent, Remark~\ref{rem:apertures_semiperfect} tells us that for any $\Field_p$-algebra $R$, every $\mathfrak{Q}\in \BT{\infty}(R)$ yields a $\mathcal{G}$-bundle in $F$-crystals over $R$, which we denote by $T_{\mathrm{crys}}(\mathfrak{Q})$. See~\cite[\S 1.1.1]{IKY3} for an explanation in the context of vector bundle $F$-gauges.
\end{construction}

\begin{remark}
    [$F$-isocrystals over perfect rings]
Suppose that $R$ is a perfect $\Field_p$-algebra. Then, $\mathrm{Isoc}_G(R)$ is the groupoid of pairs $(\mathcal{E},\varphi_{\mathcal{E}})$, where $\mathcal{E}$ is a $G$-torsor over $W(R)[\nicefrac{1}{p}]$ and $\varphi_{\mathcal{E}}\colon\varphi^*\mathcal{E}{\isomto}\mathcal{E}$ is an isomorphism. The functor $\BT{\infty}(R)\to \mathrm{Isoc}_G(R)$ from Construction~\ref{const:newton_map} can now be understood as follows. Given $\mf{Q}\in \BT{\infty}(R)$, one considers the $G$-bundle in $F$-isocrystals underlying the $\mathcal{G}$-bundle in $F$-crystals $T_{\mathrm{crys}}(\mathfrak{Q})$: This corresponds to a $G$-bundle over $W(R)[\nicefrac{1}{p}]$ equipped with an isomorphism from its Frobenius twist, and hence an object of $\mathrm{Isoc}_G(R)$. 
\end{remark}

\begin{remark}
  [Crystals and the de Rham realization]
\label{rem:crystals_and_de_rham}
Suppose that $R$ is a $p$-complete $\hat{\mathcal{O}}$-algebra. Then every crystal in $\mathcal{G}$-bundles over $R/pR$ yields a $\mathcal{G}$-bundle over $R$; here, we are using the fact that $R\to R/pR$ is a pro-divided power thickening. In particular, if we have $\mathfrak{Q}\in \BT{\infty}(R)$ reducing to $\mathfrak{Q}_0\in \BT{\infty}(R/pR)$, then we obtain a $\mathcal{G}$-bundle over $R$ associated with $T_{\mathrm{crys}}(\mathfrak{Q}_0)$. This $\mathcal{G}$-bundle is in fact canonically isomorphic to $T_{\dR}(\mathfrak{Q})$; see~\cite[Theorem 1.19]{IKY3} and its proof.
\end{remark}

\begin{remark}
  [Frobenius lifts and filtrations]
\label{rem:frobenius_lifts}
Suppose that $R$ is a flat $\hat{\mathcal{O}}$-algebra equipped with a Frobenius lift $\varphi\colon R\to R$. As in Example~\ref{ex:frobenius_lifting_frame}, this gives us a Breuil--Kisin frame $\underline{R}$ for $R$, and we obtain a functor $\Phi_{\underline{R}}\colon\BT{\infty}(R/pR)\to \Wind{\underline{R}}{\infty}(R/pR)$. Suppose that we have $\mathfrak{Q}_0\in \BT{\infty}(R/pR)$. One finds that in the notation of Remark~\ref{rem:windows_alternate}, $\Phi_{\underline{R}}(\mathfrak{Q}_0)$ corresponds to the tuple
\[
(T_{\mathrm{crys}}(\mathfrak{Q}_0),\Fil^\bullet_{\mathrm{Hdg}}T_{\dR}(\mathfrak{Q}_0),\alpha)
\]
with $\alpha$ an isomorphism
\[
  \varphi^*T_{\mathrm{crys}}(\mathfrak{Q}_0)'{\isomto}T_{\mathrm{crys}}(\mathfrak{Q}_0),
\]
where $T_{\mathrm{crys}}(\mathfrak{Q}_0)'$ is the modification along $\Fil^\bullet_{\mathrm{Hdg}}T_{\dR}(\mathfrak{Q}_0)$ from Remark \ref{rem:modifications_in_terms_of_representations}.
\end{remark}

\begin{remark}
  [Strong divisibility]
\label{rem:strong_divisibility}
In the situation of Remark~\ref{rem:frobenius_lifts}, suppose that $\mathfrak{Q}_0$ lifts to $\mathfrak{Q}\in \BT{\infty}(R)$. Then we have a lift $\Fil^\bullet_{\mathrm{Hdg}}T_{\dR}(\mathfrak{Q})$ of the filtered bundle $\Fil^\bullet_{\mathrm{Hdg}}T_{\dR}(\mathfrak{Q})$. For $\Lambda\in \mathrm{Rep}_{\Int_p}(\mathcal{G})$, let $\Fil^\bullet_{\mathrm{Hdg}}V$ be the filtered finite locally free $R$-module obtained by twisting $\Lambda$ by $\Fil^\bullet_{\mathrm{Hdg}}T_{\dR}(\mathfrak{Q})$. Then Remark~\ref{rem:modifications_in_terms_of_representations}  tells us that we have a functorial-in-$\Lambda$ isomorphism
\[
  \varphi^*\bigg(\sum_i p^{-i}\Fil^i_{\mathrm{Hdg}}V\bigg)\isomto V
\]
of finite locally free $R$-modules.
\end{remark}

\begin{construction}
  [Filtered $F$-crystals]
\label{const:filtered_F-crystal}
Let $\kappa$ be a perfect field, and suppose that $\Spf R$ is a base formal $W(\kappa)$-scheme. Suppose that we have $\mathfrak{Q}\in \BT{\infty}(R)$ lifting $\mathfrak{Q}_0\in \BT{\infty}(R/pR)$. Then $T_{\dR}(\mathfrak{Q})$ is equipped with a topologically nilpotent integrable connection $\nabla$, which completely determines the $\mathcal{G}$-bundle in crystals $T_{\mathrm{crys}}(\mathfrak{Q}_0)$. If $R$ is equipped with a Frobenius lift $\varphi\colon R\to R$, then the $F$-crystal structure on $T_{\mathrm{crys}}(\mathfrak{Q}_0)$ gives us an isomorphism
\[
  \varphi^*T_{\dR}(\mathfrak{Q})[\nicefrac{1}{p}]{\isomto}T_{\dR}(\mathfrak{Q})[\nicefrac{1}{p}]
\]
of $\mathcal{G}$-bundles over $R[\nicefrac{1}{p}]$ that is parallel for $\nabla$. 

For any representation $\Lambda$, this gives us an isomorphism $\varphi^* V[\nicefrac{1}{p}]{\isomto}V
[\nicefrac{1}{p}]$, which restricts to the integral isomorphism explained in Remark~\ref{rem:strong_divisibility}. Moreover, the filtration $\Fil^\bullet_{\mathrm{Hdg}}T_{\dR}(\mathfrak{Q})$ satisfies Griffiths transversality (in the sense that the associated filtration on the twist of any representation of $\mathcal{G}$ by $T_{\dR}(\mathfrak{Q})$ satisfies Griffiths transversality); see~\cite[Theorem 2.10]{IKY3}. In sum, in the language of~\cite[\S 2.1.2]{IKY3}, for any base formal $W(\kappa)$-scheme $\mathfrak{X}$, and $\mathfrak{Q}\in \BT{\infty}(\mathfrak{X})$, we can associate with $\mathfrak{Q}$ an exact and monoidal functor from $\Rep_{\Int_p}(\mathcal{G})$ to the category of \emph{strongly divisible filtered $F$-crystals} over $\mathfrak{X}$.
\end{construction}

\subsection{The \'etale realization and \texorpdfstring{$p$}{p}-adic comparison}
\label{sub:the_etale_realization}

Here, we recall the construction of the \'etale realization of apertures and some $p$-adic comparison results for it that make precise its relationship with the objects defined in the previous subsection.

\begin{construction}
  [The \'etale realization for apertures]
\label{const:etale_realization_apertures}
Let $R$ be a derived $p$-complete animated commutative ring. To an object $\mathfrak{Q}$ of $B\mc{G}(R^{\mathrm{syn}}\otimes\Int/p^n\Int)$,  we will attach an object $T_{\et}(\mathfrak{Q})$ that specializes to the construction given in Corollary~\ref{cor:guo_reinecke} when $R$ is $p$-quasisyntomic and $p$-torsion free. 

For every object $(A,I,R\to \overline{A})$ in $R_\Prism$, Remark~\ref{rem:prisms_and_prismatization} tells us that $j^*_{\Prism}\mathfrak{Q}\in B\mc{G}(R^\Prism\otimes\Int/p^n\Int)$ yields a $\mathcal{G}$-torsor $\mathcal{P}_{\mathfrak{Q}}(A,I,R\to\overline{A})$ over $A/{}^{\mathbb{L}}p^n$. Furthermore, just as in~\cite[Remark 6.3.4]{bhatt_lectures}, the fact that we have the underlying syntomic torsor $\mathfrak{Q}$ implies that there is a natural isomorphism of $\mathcal{G}$-torsors 
\[
  \varphi_\mc{P}\colon \varphi^*\mathcal{P}_{\mathfrak{Q}}(A,I,R\to\overline{A})[\nicefrac{1}{I}]{\isomto}\mathcal{P}_{\mathfrak{Q}}(A,I,R\to\overline{A})[\nicefrac{1}{I}]
\]
over $(A/{}^{\mathbb{L}}p^n)[\nicefrac{1}{I}]$. Now, just as in Constructions 6.3.1 and 6.3.2 of \emph{op.\@ cit.\@}, we obtain an object $T_{\et}(\mathfrak{Q})$ in $\mathrm{Loc}_{\mathcal{G}(\Int/p^n\Int)}(R[\nicefrac{1}{p}])$ characterized by the equality
\[
  \Gamma\big(\Spec T[\nicefrac{1}{p}],T_{\et}(\mathfrak{Q})\big) = \mathcal{P}_{\mathfrak{Q}}(\Prism_T,I_T,R\to T)[\nicefrac{1}{I_T}]^{\varphi_\mc{P} = \mr{id}},
\]
for $p$-torsion free perfectoid $R$-algebras $T$. Taking the limit over $n$ gives us a realization functor
\[
  T_{\et}\colon B\mc{G}(R^{\mathrm{syn}})\to \mathrm{Loc}_{\mathcal{G}(\Int_p)}(R[\nicefrac{1}{p}]).
\]
This construction globalizes to a general derived $p$-adic formal scheme $\mf{X}$ in the obvious way.
\end{construction}

\begin{remark}
    [The \'etale realization of the canonical lift]
\label{rem:etale_realization_canonical_lift}
Suppose that $\kappa$ is a perfect field over $k$ and that we have a $\mu$-ordinary aperture $\mathfrak{Q}\in \BT[\mathcal{G},\mu,\mathrm{ord}]{\infty}(\kappa)$ with canonical lift $\mathfrak{Q}^{\mathrm{can}}\in \BT[\mathcal{G},\mu,\mathrm{ord}]{\infty}(W(\kappa))$. The \'etale realization $T_{\et}(\mf{Q}^{\mr{can}})$ is obtained as follows: Take $\mathfrak{P}\in \BT[\mathcal{M}_\nu,\mu]{\infty}(W(\kappa))$ lifting $\mathfrak{Q}^{\mathrm{can}}$, consider the $\mathcal{M}_\nu(\Int_p)$-local system $T_{\et}(\mathfrak{P})$ and look at the associated $\mathcal{G}(\Int_p)$-local system obtained via change of structure group. If $\kappa$ is algebraically closed, then Remark~\ref{rem:canonical_lift_algebraically_closed} tells us more: $T_{\et}(\mf{P})$ is isomorphic to the pushforward along the reflex norm $r_\mu\colon \mathcal{T}_0\to \mathcal{M}_\nu$ of the $\mathcal{T}_0(\Int_p)$-torsor over $W(\kappa)[\nicefrac{1}{p}]$ obtained from some (hence any) Lubin--Tate formal $\hat{\mc{O}}$-module over $W(\kappa)$.
\end{remark}

\begin{remark}
  [Filtered $F$-isocrystals and crystalline comparison]
\label{rem:filtered_f-isocrystals}
Let $\kappa$ be a perfect field over $\hat{\mathcal{O}}$ of characteristic $p$ and let $\mathfrak{X}$ be a base formal $W(\kappa)$-scheme. Suppose that we have $\mathfrak{Q}\in \BT{\infty}(\mathfrak{X})$. For any $\Lambda\in \mathrm{Rep}_{\Int_p}(\mathcal{G})$, we can now associate two filtered $F$-isocrystals\footnote{See~\cite[p. 28]{imai2024tannakianframeworkprismaticfcrystals} for this notion} over the special fiber $\mathfrak{X}_{\kappa}$: 
\begin{itemize}
  \item The filtered $F$-isocrystal associated with the filtered $F$-crystal $\Fil^\bullet_{\mathrm{Hdg}}V$ from Construction~\ref{const:filtered_F-crystal};
  \item The filtered $F$-isocrystal associated with the crystalline $\Int_p$-local system $(\Lambda)_{T_{\et}(\mathfrak{Q})}$.
\end{itemize}
By~\cite[Theorem 2.10]{IKY3}, these two filtered $F$-isocrystals are canonically isomorphic.
\end{remark}

\begin{remark}
    [Filtered bundles and de Rham comparison]
\label{rem:filtered_de_rham_comparison}
As usual, let $\Reg{K}$ be a complete discrete valuation ring over $\hat{\mathcal{O}}$ with perfect residue field, and let $\mathfrak{X}$ be a base formal $\Reg{K}$-scheme. Suppose that we are given $\mathfrak{Q}\in \BT{\infty}(\mathfrak{X})$ lifting $\mathsf{Q}\in \mathrm{Loc}^{\mathrm{crys}}_{\mathcal{G}(\Int_p)}(\mathfrak{X}_\eta)$. We can now associate two filtered $G$-bundles over $\mathfrak{X}_\eta$ equipped with integrable connections:
\begin{itemize}
    \item The generic fiber of the Hodge-filtered de Rham realization $\Fil^\bullet_{\mathrm{Hdg}}T_{\dR}(\mathfrak{Q})$;
    \item As $\mathsf{Q}$ is crystalline, and therefore in particular de Rham, we have the filtered $G$-bundle $\Fil^\bullet D_{\dR}(\mathsf{Q})$ associating with each $V\in \Rep_{\Rat_p}(G)$, the filtered vector bundle $\Fil^\bullet D_{\dR}((V)_{\mathsf{Q}})$ from~\cite[Theorem 3.7(iv)]{Liu2017-yz}.
\end{itemize}
\end{remark}

\begin{lemma}
    \label{lem:filtered_de_rham_comparison}
The two filtered bundles with integrable connections from \emph{Remark \ref{rem:filtered_de_rham_comparison}} are isomorphic.
\end{lemma}
\begin{proof}
   When $W(\kappa) = \Reg{K}$, this is immediate from Remark~\ref{rem:filtered_f-isocrystals}. The proof of Imai--Kato--Youcis cited there works also in this context, if one ignores $F$-structures. Indeed, everything after the first paragraph in \emph{loc.\@ cit.\@} works exactly the same, replacing $\mathbb{D}_\mathrm{crys}[\nicefrac{1}{p}]$ and $D_\mathrm{crys}$ with $\Fil^\bullet_\mr{Hdg}T_\mr{dR}(\mf{Q})[\nicefrac{1}{p}]$ and $\Fil^\bullet D_\mr{dR}(\mathsf{Q})$, respectively.
\end{proof}

\subsection{\'Etale realization over perfectoid rings}
\label{sub:etale_perfectoid}

In this subsection, we will give a more transparent proof that the \'etale realization on $\BT{n}$ is faithful for base formal $\Reg{K}$-schemes.

\begin{remark}
  [\'Etale $\mathcal{G}(\Int/p^n\Int)$-torsors over perfectoid rings]
\label{rem:GZpn_perfectoid}
Suppose that $R$ is perfectoid and $p$-torsion free. By Remark~\ref{rem:local_systems_tannakian}, an object of $\mr{Loc}_{\mc{G}(\bb{Z}/p^n\bb{Z})}(R[\nicefrac{1}{p}])$ is equivalent to an exact symmetric monoidal functor
\[
  \Rep_{\Int_p}(\mathcal{G})\to \mathrm{Loc}_{\Int/p^n\Int}(R[\nicefrac{1}{p}]).
\]
Now, tilting combined with Katz's Riemann-Hilbert equivalence (see~\cite[Proposition 3.6]{MR4600546}) tells us that the right-hand side is equivalent to the category of pairs $(\mathcal{V},\eta)$, where $\mathcal{V}$ is a finite locally free  module over $\Prism_R/p^n\Prism_R[\nicefrac{1}{\xi}]$ and $\eta\colon \varphi^* \mathcal{V}{\isomto}\mathcal{V}$ is an isomorphism of $\Prism_R/p^n\Prism_R[\nicefrac{1}{\xi}]$-modules. Therefore, we conclude that giving a $\mathcal{G}(\Int/p^n\Int)$-local system over $R[\nicefrac{1}{p}]$ is equivalent to giving a $\mathcal{G}$-torsor $\mathcal{Q}$ over $(\Prism_R/p^n\Prism_R)[\nicefrac{1}{\xi}]$ along with an isomorphism $\varphi_{\mathcal{Q}}\colon\varphi^* \mathcal{Q}{\isomto}\mathcal{Q}$.
\end{remark}

\begin{remark}
  [Quotient description of the image of the \'etale realization]
\label{rem:quotient_description_perfectoid_local_systems}
In terms of the description of the target from Remark~\ref{rem:GZpn_perfectoid}, the  \'etale realization functor $\BT{n}(R)\to \mathrm{Loc}_{\mathcal{G}(\Int/p^n\Int)}(R[\nicefrac{1}{p}])$ factors through the fully faithful sub-groupoid spanned by pairs $(\mathcal{Q},\varphi_{\mathcal{Q}})$ where $\mathcal{Q}$ can be trivialized over $(\Prism_{\tilde{R}}/p^n\Prism_{\tilde{R}})[\nicefrac{1}{\xi}]$ for some $p$-completely faithfully flat and \'etale map $R\to \tilde{R}$ of perfectoid rings. 

Observe that this sub-groupoid is the evaluation at $R$ of the \'etale sheafification of the functor on $p$-torsion free perfectoid rings given by
\[
  R\mapsto \big[\mathcal{G}(\Prism_R/p^n\Prism_R[\nicefrac{1}{\xi}])/_{\mathrm{Ad}_{\varphi}}(\mathcal{G}(\Prism_R/p^n\Prism_R[\nicefrac{1}{\xi}]))\big].
\]
Here, the right-hand side is the quotient by the $\varphi$-semilinear adjoint action
\begin{align*}
\mr{Ad}_\varphi\colon \mathcal{G}(\Prism_R/p^n\Prism_R[\nicefrac{1}{\xi}])\times \mathcal{G}(\Prism_R/p^n\Prism_R[\nicefrac{1}{\xi}])&\to \mathcal{G}(\Prism_R/p^n\Prism_R[\nicefrac{1}{\xi}])\\
(g,h)&\mapsto h^{-1}g\varphi(h).  
\end{align*}
Therefore, the \'etale realization map from Construction~\ref{const:etale_realization_apertures} can be viewed as the \'etale sheafification of the map
\[
\big[\mathcal{G}(\Prism_R/p^n\Prism_R)/H^{(n)}_\mu(R)\big]\to   \big[\mathcal{G}(\Prism_R/p^n\Prism_R[\nicefrac{1}{\xi}])/_{\mathrm{Ad}_{\varphi}}(\mathcal{G}(\Prism_R/p^n\Prism_R[\nicefrac{1}{\xi}]))\big]
\]
induced by 
\[
\mathcal{G}(\Prism_R/p^n\Prism_R)\xrightarrow{g\mapsto g\varphi(\mu(\xi))} \mathcal{G}(\Prism_R/p^n\Prism_R[\nicefrac{1}{\xi}])
\]
and the map of groups
\begin{equation}
\label{eqn:hmu_to_LG}
   H^{(n)}_\mu(R)=\mathcal{G}(\Prism_R/p^n\Prism_R)\times_{\mathcal{G}(R/p^nR)}\mathcal{P}^-_\mu(R/p^nR)\subset \mathcal{G}(\Prism/p^n\Prism_R)\to \mathcal{G}(\Prism_R/p^n\Prism_R[\nicefrac{1}{\xi}]).
\end{equation}
\end{remark}

\begin{remark}
\label{rem:Hnmu_description}
    In the situation above, the group $H^{(n)}_\mu(R)$ can be described more simply: We have
    \[
    H^{(n)}_\mu(R) = \mathcal{G}(\Prism_R/p^n\Prism_R)\cap  \mu(\xi)\mathcal{G}(\Prism_R/p^n\Prism_R)\mu(\xi)^{-1}\leqslant \mathcal{G}(\Prism_R/p^n\Prism_R[\nicefrac{1}{\xi}]).
    \]
    To see this, note that by~\cite[Remark 6.3.3]{MR4355476} the map
    \[
    \mathcal{P}^-_\mu(\Prism_R/p^n\Prism_R)\times (\Prism_R/p^n\Prism_R\otimes_{\hat{\mathcal{O}}}\mathfrak{g}_1) \xrightarrow{(h,X)\mapsto h\exp(\xi X)}\mathcal{G}(\Prism_R/p^n\Prism_R)
    \]
    is a bijection onto $H^{(n)}_\mu(R)$. Here $\mathfrak{g}_1\subset \mathfrak{g}_{\hat{\mathcal{O}}}$ is the weight-$1$ eigenspace for $\mu$.
\end{remark}

 \begin{proposition}
  \label{prop:perfectoid_faithful}
 Suppose that $R$ is perfectoid and $p$-torsion free. Then the \'etale realization functor
 \[
   \BT{n}(R)\to \mathrm{Loc}_{\mathcal{G}(\Int/p^n\Int)}(R[\nicefrac{1}{p}])  
 \]
 is faithful. 
 \end{proposition}
 \begin{proof}
  Given Remark~\ref{rem:quotient_description_perfectoid_local_systems} and the description of $H^{(n)}_\mu(R)$ in~\eqref{eqn:Hmu_n_torsion_free_desc}, this reduces to the easy assertion that the map~\eqref{eqn:hmu_to_LG} is injective.   
 \end{proof}

 \begin{corollary}
  \label{cor:p-completely_smooth_faithful}
 Suppose that $\mf{X}$ is a base formal $\Reg{K}$-scheme. Then the \'etale realization functor
 \[
   \BT{n}(\mf{X})\to \mathrm{Loc}_{\mathcal{G}(\Int/p^n\Int)}(\mf{X}_\eta)  
 \]
 is faithful. 
 \end{corollary}
 \begin{proof}
  This is because $\mf{X}$ admits a quasisyntomic cover $\{\Spf(R_i)\to \mf{X}\}$ with $R_i$ perfectoid and $p$-torsion free for all $i$; see~\cite[Lemma 1.15]{imai2024tannakianframeworkprismaticfcrystals}.
 \end{proof}

\section{Tate's full faithfulness for apertures and its consequences}
\label{sec:full_faithfulness_for_apertures_and_its_consequences}

In this section, we prove a generalization Tate's full faithfulness theorem for $p$-divisible groups to the context of apertures. We also derive some corollaries to this that will be of importance for the main theorems of this paper.

\subsection{Algebraization and Tate's full faithfulness}
\label{subsec:tate}

Let $\mathcal{O}$ be a mixed characteristic $(0,p)$ discrete valuation ring with completion $\hat{\mathcal{O}}$ and fraction field $E$. We will now define an algebraic variant of $\BT{n}$ over $\mathcal{O}$ by gluing it to the stack of $\mathcal{G}(\Int/p^n\Int)$-local systems over $E$, and we will use it to prove the aperture analogue of Tate's full faithfulness theorem for $p$-divisible groups by reducing it ultimately to Corollary~\ref{cor:guo_reinecke}. 

\begin{construction}
  [Algebraic BT]
For any animated commutative $\mathcal{O}$-algebra $R$, set
\[
  \BT[\mathcal{G},\mu,\mathrm{alg}]{n}(R) \defn \BT{n}(\hat{R})\times_{T_{\et},\mathrm{Loc}_{\mathcal{G}(\Int/p^n\Int)}(\hat{R}[\nicefrac{1}{p}])}\mathrm{Loc}_{\mathcal{G}(\Int/p^n\Int)}(R[\nicefrac{1}{p}]),
\]
where $T_{\et}$ is the \'etale realization functor (which we henceforth suppress from the notation). Also, set
\[
  \BT[\mathcal{G},\mu,\mathrm{alg}]{\infty}(R)  = \varprojlim_n \BT[\mathcal{G},\mu,\mathrm{alg}]{n}(R).
\]
\end{construction}

\begin{remark}
  For any $R$ and $1\leqslant n\leqslant \infty$, we obtain a canonical \'etale realization map
  \[
    T_{\et}\colon \BT[\mathcal{G},\mu,\mathrm{alg}]{n}(R)\to \mathrm{Loc}_{\mathcal{G}(\Int/p^n\Int)}(R[\nicefrac{1}{p}])
  \]
  obtained by projecting to the second component.

  This globalizes: For any Deligne--Mumford stack $\mathcal{X}$ over $\mathcal{O}$ with generic fiber $X$, we obtain a functor
  \[
  T_{\et}\colon \BT[\mathcal{G},\mu,\mathrm{alg}]{n}(\mathcal{X})\to \mathrm{Loc}_{\mathcal{G}(\Int/p^n\Int)}(X).
  \]
\end{remark}

\begin{theorem}
[Algebraic representability and affineness of diagonal]
  \label{thm:algebraization}
The functor $\BT[\mathcal{G},\mu,\mathrm{alg}]{n}$ is represented by a smooth $0$-dimensional Artin stack over $\mathcal{O}$ with \emph{affine} diagonal.
\end{theorem}

The rest of this subsection and the next will be devoted to the proof of this theorem, with the proof of the affineness of the diagonal completed in the next subsection. For now, we record an important consequence.

\begin{definition}
    [$\eta$-normality]
\label{defn:eta_normal}
  Suppose that $\mathcal{X}$ is a Deligne--Mumford stack over $\mathcal{O}$. We will say that $\mathcal{X}$ is \defnword{$\eta$-normal}\footnote{See~\cite[Appendix A]{Achinger2022-ws}} if it admits an \'etale cover by affine schemes of the form $\Spec R$ with $R$ both $p$-torsion free and integrally closed in $R[\nicefrac{1}{p}]$. Note that, if $\mathcal{X}$ is normal then it is automatically $\eta$-normal.
\end{definition}

\begin{theorem}
[Tate full faithfulness for apertures]
\label{thm:full_faithfulness}
Let $\mathcal{X}$ be a Noetherian, $\eta$-normal Deligne--Mumford stack over $\mathcal{O}$ with generic fiber $X$. Then the functor 
\begin{equation*}T_{\et}\colon \BT[\mathcal{G},\mu,\mathrm{alg}]{\infty}(\mathcal{X})\to \mathrm{Loc}_{\mathcal{G}(\Int_p)}(X)
\end{equation*}
is fully faithful.
\end{theorem}
\begin{proof}
By \'etale descent, we reduce to the case where $\mathcal{X} = \Spec R$ is affine. We essentially follow part of the proof by Tate of his theorem for $p$-divisible groups~\cite{tate_pdiv}. By Theorem~\ref{thm:algebraization}, for any pair
\[
(\mathfrak{Q}_1,\mathfrak{Q}_2)\in \BT[\mathcal{G},\mu,\mathrm{alg}]{\infty}(R)\times \BT[\mathcal{G},\mu,\mathrm{alg}]{\infty}(R),
\]
the scheme of isomorphisms from $\mathfrak{Q}_1$ to $\mathfrak{Q}_2$ is represented by an affine scheme $\mathcal{I}\to \Spec R$. The theorem comes down to knowing that every $R[\nicefrac{1}{p}]$-point of $\mathcal{I}$ extends to an $R$-valued point. 

For this, by the $\eta$-normality of $R$, it is enough to know that an $R[\nicefrac{1}{p}]$-valued point of $\mathcal{I}$ extends over the localization at every height $1$ prime of $R$ that is minimal over $pR$.\footnote{The argument for this is identical to that of algebraic Hartogs lemma for normal rings; see \stacks{031T}.} This reduces us to the case where $R$ is a mixed characteristic discrete valuation ring with maximal ideal $\mx$. Now, let $\kappa$ be the residue field of $R$. For any purely inseparable finite extension $\kappa'/\kappa$, there exists a finite flat quasisyntomic map of discrete valuation rings $R\to R'$ such that $R'/\mx R' \simeq \kappa'$\footnote{If $\kappa' = \kappa(a^{1/p})$ for some $a\in \kappa\backslash \kappa^p$, then we can take $R' = R[X]/(X^p-\pi X - \tilde{a})$, for some uniformizer $\pi\in R$ and some lift $\tilde{a}\in R$ of $a$. The general case follows by induction on $[\kappa':\kappa]$.} Using this, one sees that there exists an ind-finite flat quasisyntomic cover $R\to R_\infty$ where $R_\infty$ is a discrete valuation ring with perfect residue field. Via flat descent, it is now enough to show:
\begin{enumerate}
  \item $\mathcal{I}(R_\infty) \to \mathcal{I}(R_\infty[\nicefrac{1}{p}])$ is a bijection.
  \item $\mathcal{I}(R_\infty^{\otimes_R m})\to \mathcal{I}(R_\infty^{\otimes_Rm}[\nicefrac{1}{p}])$ is injective. 
\end{enumerate}
By the definition of $\BT[\mathcal{G},\mu,\mathrm{alg}]{n}$, to prove assertion (1) (resp.\@ (2)), we can replace $R_\infty$ (resp.\@ $R_\infty^{\otimes_R m}$) with its (derived) $p$-completion. Note that in assertion (2) the algebras involved are quasisyntomic over $R_\infty$. Therefore, both assertions follow from Theorem~\ref{thm:guo_reinecke}. 
\end{proof}

\begin{corollary}
\label{cor:full_faithfulness}
Suppose that $\mathcal{X}$ is a normal flat Noetherian Deligne-Mumford stack over $\mathcal{O}$ and $\mc{X}'\to\mc{X}$ is a smooth cover. Then, for $\mathsf{Q}\in \mathrm{Loc}_{\mathcal{G}(\Int_p)}(\mathcal{X}_\eta)$, there exists a $\mathfrak{Q}\in \BT[\mathcal{G},\mu,\mathrm{alg}]{\infty}(\mathcal{X})$ with $T_{\et}(\mathfrak{Q})$ isomorphic to $\mathsf{Q}$ if and only if there is a $\mathfrak{Q}'\in \BT[\mathcal{G},\mu,\mathrm{alg}]{\infty}(\mathcal{X}')$ with $T_{\et}(\mathfrak{Q}')$ isomorphic to $\mathsf{Q}\vert_{\mathcal{X}'[\nicefrac{1}{p}]}$.
\end{corollary}
\begin{proof}
  The only if direction is obvious. For the if direction, since $\mathcal{X}'\to \mathcal{X}$ is smooth each $k$-fold self-fiber product $(\mc{X}')^{\times_{\mc{X}} k}$ is still a normal flat Noetherian DM stack over $\mc{O}$. Therefore, Theorem~\ref{thm:full_faithfulness} tells us that the tautological descent datum for $\mathsf{Q}\vert_{\mc{X}'_\eta}$ yields a descent datum for $\mathfrak{Q}'$, which is effective.
\end{proof}

\begin{remark}
  Combining the theorem with Theorem~\ref{thm:dieudonne}, we obtain a proof of Tate's seminal full faithfulness theorem for $p$-divisible groups~\cite[Theorem 4]{tate_pdiv}. Unwinding everything, one sees that the proof obtained here is not very different from that of Tate's, except that his final reduction is to the case where there exists a homomorphism of $p$-divisible groups over a complete valuation ring inducing a given isomorphism of Tate modules. Since we do not have any direct way of defining a map between apertures that is not an isomorphism, we end up appealing to the faithfulness assertion of Theorem~\ref{thm:guo_reinecke} instead.
\end{remark}

Let us now begin preparations for the proof of Theorem~\ref{thm:algebraization}.

\begin{lemma}
[Completed finite presentation]
\label{lem:completed_finite_presentation}
Suppose that $\{R_i\}_{i\in I}$ is an inductive system of derived $p$-complete animated commutative rings with colimit $R$, and let $\hat{R}$ be the derived $p$-completion of $R$. Then the natural map
\[
   \colim_i \BT{n}(R_i)\to \BT{n}(\hat{R})
\]
is an equivalence.
\end{lemma}
\begin{proof}
   If the inductive system is one of $\Int/p^m\Int$-algebras, then the conclusion holds because $\BT{n}$ is a finitely presented formal Artin stack. 

   For the general case, set $\overline{R}_i \defn R_i/{}^{\mathbb{L}}p^2$ and $\overline{R} = R/{}^{\mathbb{L}}p^2$,\footnote{The choice of $p^2$ here is to ensure that $R\to \overline{R}$ is an inverse limit of \emph{nilpotent} divided power extensions even when $p=2$: For $p$ odd, one could also choose to work with $R/{}^{\mathbb{L}}p$ instead.} and note that by Grothendieck--Messing theory we have a Cartesian square
     \[
     \begin{diagram}
           \BT{n}(\hat{R}) &\rTo& B \mathcal{P}_\mu^-(R/{}^{\mathbb{L}}p^n)\\
           \dTo&&\dTo\\
           \BT{n}(\overline{R})&\rTo&B\mathcal{P}_\mu^-(\overline{R}/{}^{\mathbb{L}}p^n)\times_{B\mathcal{G}(\overline{R}/{}^{\mathbb{L}}p^n)}B \mathcal{G}(R/{}^{\mathbb{L}}p^n),
     \end{diagram}
     \]
     and similarly with $\hat{R}$ replaced by $R_i$. We can now conclude using the previous paragraph, and the finite presentation of the stacks $B \mathcal{P}^-_\mu$ and $B \mathcal{G}$.\footnote{We are using the fact that filtered colimits of $\infty$-groupoids commute with finite limits; see~\cite[Proposition 5.3.3.3]{Lurie2009-oc}.}
\end{proof}

\begin{lemma}
  [Algebraic finite presentation]
\label{lem:algebraic_finite_presentation}
The prestack $\BT[\mathcal{G},\mu,\mathrm{alg}]{n}$ over $\mathcal{O}$ is locally of finite presentation.
\end{lemma}
\begin{proof}
  We need to know that the natural map
     \begin{equation*}
       \begin{tikzcd}
  {\colim_i\left[\BT[\mathcal{G},\mu]{n}(\hat{R}_i)\times_{\mathrm{Loc}_{\mathcal{G}(\Int/p^n\Int)}(\hat{R}_i[\nicefrac{1}{p}])}\mathrm{Loc}_{\mathcal{G}(\Int/p^n\Int)}(R_i[\nicefrac{1}{p}])\right]} \\
  {\BT[\mathcal{G},\mu]{n}(\hat{R})\times_{\mathrm{Loc}_{\mathcal{G}(\Int/p^n\Int)}(\hat{R}[\nicefrac{1}{p}])}\mathrm{Loc}_{\mathcal{G}(\Int/p^n\Int)}(R[\nicefrac{1}{p}])}
  \arrow[from=1-1, to=2-1]
\end{tikzcd}
     \end{equation*}
      is an isomorphism for any inductive system $\{R_i\}$ of animated commutative $\mathcal{O}$-algebras with colimit $R$. 

     Now, since $\mathrm{Loc}_{\mathcal{G}(\Int/p^n\Int)}$ is a finite Deligne--Mumford stack the natural map
     \begin{align}\label{eqn:generic_finite_presentation}
      \colim_i\mathrm{Loc}_{\mathcal{G}(\Int/p^n\Int)}(R_i[\nicefrac{1}{p}]) &\to \mathrm{Loc}_{\mathcal{G}(\Int/p^n\Int)}(\colim_i R_i[\nicefrac{1}{p}]),
     \end{align}
     is an isomorphism. Therefore, it is enough to verify that the following commutative square is Cartesian
     \[
       \begin{diagram}
        \colim_i \BT{n}(\hat{R}_i)&\rTo&\BT{n}(\hat{R})\\
        \dTo&&\dTo\\
        \mathrm{Loc}_{\mathcal{G}(\Int/p^n\Int)}(\colim_i\hat{R}_i[\nicefrac{1}{p}])&\rTo&\mathrm{Loc}_{\mathcal{G}(\Int/p^n\Int)}(\hat{R}[\nicefrac{1}{p}]).
       \end{diagram}
     \]
     The ring $\colim_i\hat{R}_i$ is $p$-Henselian with $p$-adic completion $\hat{R}$. Therefore, by~\cite[Corollary 2.1.20]{MR4553656}, the bottom arrow is an isomorphism. The top arrow is an isomorphism by Lemma~\ref{lem:completed_finite_presentation}.
\end{proof}

The next result is immediate from the definitions, the (formal) \'etaleness of the Deligne--Mumford stack $\mathrm{Loc}_{\mathcal{G}(\Int/p^n\Int)}$ over $\mc{O}[\nicefrac{1}{p}]$, and Theorem~\ref{thm:gmm}.
\begin{lemma}
  [Algebraic deformation theory]
\label{lem:algebraic_deformation_theory}
For every nilpotent divided power extension $R'\twoheadrightarrow R$ of animated commutative $\mathcal{O}$-algebras, there is a canonical Cartesian diagram
\[
\begin{diagram}
  \BT[\mathcal{G},\mu,\mathrm{alg}]{n}(R')&\rTo& B \mathcal{P}^-_\mu(R'/{}^{\mathbb{L}}p^n)\\
  \dTo&&\dTo\\
  \BT[\mathcal{G},\mu,\mathrm{alg}]{n}(R)&\rTo&B \mathcal{P}^-_\mu(R/{}^{\mathbb{L}}p^n)\times_{B \mathcal{G}(R/{}^{\mathbb{L}}p^n)}B \mathcal{G}(R'/{}^{\mathbb{L}}p^n).
\end{diagram}
\]
\end{lemma}

\begin{proof} 
     [Proof of Theorem~\ref{thm:algebraization}]
  The proof will be via verifying the conditions for Artin--Lurie representability~\cite[Theorem 7.1.6]{lurie_thesis}. It is important for this that we allow animated inputs, though with this caveat the reader will note that the argument itself (except for the parts establishing finite presentation and affineness of the diagonal) is quite formal. In what follows, we will use without comment the fact that $\BT{n}\otimes\Int/p^m\Int$ (for $m\geqslant 1$) and $\mathrm{Loc}_{\mathcal{G}(\Int/p^n\Int)}$ are both locally finitely presented algebraic stacks over $\mathcal{O}$, and so satisfy the conditions of \emph{loc.\@ cit.}

  Condition (1: local finite presentation) is verified by Lemma~\ref{lem:algebraic_finite_presentation}. Conditions (2: being an \'etale sheaf), (3: integrability) , (5: infinitesimal cohesiveness), and (6: nilcompleteness) all involve behavior with respect to limits and are easily verified.

     For condition (4: existence of a cotangent complex), note that, for any $\mathcal{O}$-algebra $R$ and a map $x\colon\Spf \hat{R}\to \BT{n}$, the cotangent complex $\mathbb{L}_{\BT{n}/\Reg{E_v},x}$ is a $\Int/p^n\Int$-module object in the $\infty$-category $\mathrm{Perf}(\hat{R})$. In particular, it algebraizes canonically to a perfect complex over $R$. If $x$ lifts to a point $(x,x')\in \BT[\mathcal{G},\mu,\mathrm{alg}]{n}(R)$, then Lemma~\ref{lem:algebraic_deformation_theory} shows that the assignment $(x,x')\mapsto \mathbb{L}_{\BT{n}/\Reg{E_v},x}$ is the desired cotangent complex for $\BT[\mathcal{G},\mu,\mathrm{alg}]{n}$ over $\mathcal{O}$. Note that this is still a perfect complex over $\BT[\mathcal{G},\mu,\mathrm{alg}]{n}$ with Tor amplitude in $[0,1]$.

  Condition (7: $1$-truncatedness) is verified by seeing that $\BT[\mathcal{G},\mu,\mathrm{alg}]{n}(R)$ is $1$-truncated for any discrete $\mathcal{O}$-algebra $R$: Indeed, this is true for all the spaces involved in the fiber product defining it.

  Thus we have verified all the conditions of Lurie's theorem, and so can conclude using this, the description of the cotangent complex, and the quasi-compactness of both $\BT{n}$ and $\mathrm{Loc}_{\mathcal{G}(\Int/p^n\Int)}$ that $\BT[\mathcal{G},\mu,\mathrm{alg}]{n}$ is a smooth $0$-dimensional Artin stack over $\mathcal{O}$.

    We are now tasked with showing that the diagonal $\Delta$, which we now know is a locally of finite type algebraic space, is actually affine. For this, note that $\Delta[\nicefrac{1}{p}]$ is affine (in fact, finite \'etale) and the $p$-adic completion $\widehat{\Delta}$ is also affine by Theorem~\ref{thm:gmm}. Moreover, it follows from Proposition~\ref{prop:perfectoid_faithful} that the diagonal is a \emph{separated} algebraic space: Indeed, the cited result shows that the map $\Delta(R)\to \Delta(F)$ is injective for any perfectoid valuation ring $R$ of mixed characteristic with fraction field $F$. Since every mixed characteristic discrete valuation ring is dominated by a perfectoid valuation ring, we see that the injectivity also holds for such discrete valuation rings. For discrete valuation rings of equal characteristic, this injectivity is clear from the already stated facts about $\Delta[\nicefrac{1}{p}]$ and $\widehat{\Delta}$. 

  Choose a smooth cover 
     \[
  \Spec R\to \BT[\mathcal{G},\mu,\mathrm{alg}]{n}\times\BT[\mathcal{G},\mu,\mathrm{alg}]{n}
     \]
  by a smooth affine scheme over $\mathcal{O}$. Let $B$ be the finite \'etale $R[\nicefrac{1}{p}]$-algebra representing the pullback of $\Delta[\nicefrac{1}{p}]$, and let $S$ be the topologically of finite type $\hat{R}$-algebra representing the pullback of $\widehat{\Delta}$. Then the criterion of Achinger--Youcis from ~\cite[Proposition 3.6]{achinger_youcis} tells us that $\Delta\vert_{\Spec R}$ is affine if and only if the natural map $B\to S[\nicefrac{1}{p}]$ has dense image. This follows from Proposition~\ref{prop:diagonal_rational_open} below. 
     \end{proof}

\subsection{Analytic properties of the diagonal}
\label{sub:analytic_properties_of_the_diagonal}

The purpose of this subsection is to prove Proposition~\ref{prop:diagonal_rational_open}, and so to complete the proof of the affineness of the diagonal of $\BT[\mathcal{G},\mu,\mathrm{alg}]{n}$.

\begin{setup}
\label{setup:diagonal}
    Suppose that we have a smooth covering $\Spec R\to \BT[\mathcal{G},\mu,\mathrm{alg}]{n}\times\BT[\mathcal{G},\mu,\mathrm{alg}]{n}$ as in the proof of Theorem~\ref{thm:algebraization}. This classifies a pair of $n$-truncated apertures $(\mathfrak{Q}_1,\mathfrak{Q}_2)$ over the $p$-completion $\hat{R}$ and a pair of $\mathcal{G}(\Int/p^n\Int)$-local systems $(\mathsf{Q}_1,\mathsf{Q}_2)$ over $R[\nicefrac{1}{p}]$ whose restrictions over $\hat{R}[\nicefrac{1}{p}]$ are isomorphic to $(T_{\et}(\mathfrak{Q}_1),T_{\et}(\mathfrak{Q}_2))$. Write $\Delta_R\to \Spec R$ for the restriction of the diagonal, and let $\hat{\Delta}_R\to \Spf\hat{R}$  be the associated formal algebraic space. As we observed in the proof of Theorem~\ref{thm:algebraization}, $\Delta_R$ is a finitely presented \emph{separated} algebraic space over $R$. 
\end{setup}

\begin{remark}
    \label{rem:diagonal_setup_affine}
By definition, $\hat{\Delta}_R$ parameterizes isomorphisms between the $n$-truncated apertures $\mathfrak{Q}_1$ and $\mathfrak{Q}_2$. The diagonal of $\BT{n}$ is an affine, finitely presented map, and so $\hat{\Delta}_R$ is represented by a topologically of finite type $\hat{R}$-algebra $S$. Similarly, $\Delta_R[\nicefrac{1}{p}]$ is the finite \'etale scheme over $R[\nicefrac{1}{p}]$ parameterizing isomorphisms between the $\mathcal{G}(\Int/p^n\Int)$-local systems $\mathsf{Q}_1$ and $\mathsf{Q}_2$, and so is represented by a finite \'etale $R[\nicefrac{1}{p}]$-algebra $B$.
\end{remark}

\begin{remark}
\label{rem:diagonal_setup_adic}
We can associate with $\Delta_R$ the generic fiber $X\defn (\widehat{\Delta}_R)_\eta$ of the formal scheme $\hat{\Delta}_R$: This is a finitely presented affinoid adic space over $Y \defn \Spa(\hat{R}[\nicefrac{1}{p}],\hat{R})$. We can also consider the affinoid adic space $X'$ over $Y$ obtained by taking the analytification (relative to $R$ in the sense of \cite[Definition 2.15]{achinger_youcis}) of the finite \'etale $\hat{R}[\nicefrac{1}{p}]$-scheme $\hat{R}[\nicefrac{1}{p}]\otimes_R\Delta_R$.

Concretely, we have $X = \Spa(S[\nicefrac{1}{p}],S)$ and $X' = \Spa(\hat{B},\hat{B}^+)$, where $\hat{B} = B\otimes_{R[\nicefrac{1}{p}]}\hat{R}[\nicefrac{1}{p}]$, and $\hat{B}^+\subset \hat{B}$ is the integral closure of $\hat{R}$. We then obtain a diagram of adic spaces
  \[
    \begin{diagram}
      X&\rTo&X'\\
      &\rdTo&\dTo\\
      &&Y
    \end{diagram}
  \]
  where the vertical arrow is finite~\cite[(1.4.2)]{Huber1996-ly}, and where the top horizontal map is an open immersion~\cite[Proposition 1.9.6]{Huber1996-ly}: For the latter, see also~\cite[Proposition 2.16]{achinger_youcis}. 
\end{remark}

Note that the map $X\to X'$ from Remark \ref{rem:diagonal_setup_adic} yields maps of $R[\nicefrac{1}{p}]$-algebras $B\to \hat{B}\to S[\nicefrac{1}{p}]$. Our goal is to prove the following:
\begin{proposition}
\label{prop:diagonal_rational_open}
The map $B\to S[\nicefrac{1}{p}]$  has dense image.
\end{proposition}

\begin{remark}
  [The diagonal as an analytic diamond]
\label{rem:diagonal_diamond}
Suppose that $T$ is a perfectoid $R$-algebra and fix a generator $\xi$ of $\mr{Fil}^1_\smallN \,\Prism_T$. Unwinding Remark~\ref{rem:quotient_description_perfectoid_truncated}, and after \'etale localization on $\Spf T$ if necessary, we can assume that the restriction of $(\mathfrak{Q}_1,\mathfrak{Q}_2)$ over $T$ arises from a pair $(x_1,x_2)$ belonging to the set $\mathcal{G}(\Prism_T/p^n\Prism_T)\times \mathcal{G}(\Prism_T/p^n\Prism_T)$. Then, for any $p$-torsion free perfectoid $T$-algebra $T'$ that is integrally closed in $T'[\nicefrac{1}{p}]$, we have canonical identifications
\begin{align*}
  X(T'[\nicefrac{1}{p}],T') &= \{h\in H^{(n)}_\mu(T'):\; x_2 = h^{-1}x_1\varphi(\mu(\xi)h\mu(\xi)^{-1}) \}\\  &= \{h\in H^{(n)}_\mu(T'):\;  h = x_1\varphi(\mu(\xi)h\mu(\xi)^{-1})x_2^{-1}\},
\end{align*}
and
\begin{align*}
    X'(T'[\nicefrac{1}{p}],T') &= \{h\in \mathcal{G}(\Prism_{T'}/p^n\Prism_{T'}[\nicefrac{1}{\xi}]):\; x_2 = h^{-1}x_1\varphi(\mu(\xi)h\mu(\xi)^{-1}) \}\\ &= \{h\in \mathcal{G}(\Prism_{T'}/p^n\Prism_{T'}[\nicefrac{1}{\xi}]):\; h = x_1\varphi(\mu(\xi)h\mu(\xi)^{-1})x_2^{-1}\}
\end{align*}
In other words, we have concretely described the diamonds $X_T^\lozenge$, $X^{',\lozenge}_T$ underlying the adic spaces $\Spa(T[\nicefrac{1}{p}],T)\times_YX$ and $\Spa(T\nicefrac{1}{p}],T)\times_YX'$ as in \cite[\S10.1]{Scholze2020-bx}.
\end{remark}

\begin{remark}
  [Linearization]
\label{rem:linearization}
Keep the setup from the previous remark. Choose a faithful representation $\mathcal{G}\hookrightarrow \GL(\Lambda)$, and set $M = \End(\Lambda)\otimes_{\Int_p}\Prism_T/p^n\Prism_T$: This is equipped with a canonical $\varphi$-semilinear bijection $1\otimes\varphi$, which we will denote simply by $\varphi$. Write $\Phi\colon M[\nicefrac{1}{\xi}]\isomto M[\nicefrac{1}{\xi}]$ for the operator $m\mapsto x_1\varphi(\mu(\xi))\varphi(m)\varphi(\mu(\xi))^{-1}x_2^{-1}$, and define functors $Z^\lozenge$ and $Z^{',\lozenge}$ on affinoid perfectoid spaces over $(T[\nicefrac{1}{p}],T)$ by
\begin{align*}
  Z^\lozenge(T'[\nicefrac{1}{p}],T') &= \{m\in \Prism_{T'}\otimes_{\Prism_T}M:\; m = \Phi(m)\in (\Prism_{T'}\otimes_{\Prism_T}M)[\nicefrac{1}{\xi}]\}\\
   Z^{',\lozenge}(T'[\nicefrac{1}{p}],T') &= \{m\in (\Prism_{T'}\otimes_{\Prism_T}M)[\nicefrac{1}{\xi}]:\; m = \Phi(m)\in (\Prism_{T'}\otimes_{\Prism_T}M)[\nicefrac{1}{\xi}]\}
\end{align*}
Now, Remark~\ref{rem:Hnmu_description} tells us that, we have an isomorphism $X^\lozenge_T \isomto Z^{\lozenge}\times_{Z^{',\lozenge}}X^{',\lozenge}_T$ of diamonds over $\mathrm{Spd}(T[\nicefrac{1}{p}],T) = \Spa(T^\flat[\nicefrac{1}{\varpi}],T^\flat)$, where $\varpi\in T^\flat$ is the pseudouniformizer given by the image of $\xi$ under the natural map $\Prism_T\to T^\flat$.
\end{remark}

\begin{lemma}
  \label{lem:linearization}
The diamonds $Z^\lozenge$ and $Z^{',\lozenge}$ are represented by affinoid perfectoid spaces $Z = \Spa(D,D^+)$ and $Z' = \Spa(D',D^{',+})$ over $\mathrm{Spa}(T[\nicefrac{1}{p}],T)$. Moreover, the map $Z\to Z'$ is a rational open embedding, and the map $D\to D'$ has dense image.
\end{lemma}
\begin{proof}
  Via the usual d\'evissage, we reduce to the case where $n=1$. Choose $k\geqslant 0$ with $\Phi(M)\subset \varpi^{-k}M$, and set $\tilde{\Phi} = \varpi^k\Phi$. In the following we identify $\Prism_T/p\Prism_T[\nicefrac{1}{\xi}]$ with $T^\flat[\nicefrac{1}{\varpi}]$.

  By Artin--Schreier theory, the functor $C\mapsto (C\otimes_{T^\flat[\nicefrac{1}{\varpi}]}M[\nicefrac{1}{\varpi}])^{\Phi = 1}$ is represented by a finite \'etale group scheme over $T^\flat[\nicefrac{1}{\varpi}]$ locally isomorphic to $\Field_p^{\rank_{T^\flat}M}$. Therefore, there exists $r\geqslant 0$ such that, for all affinoid perfectoids $\Spa(T'[\nicefrac{1}{p}],T')$ over $T$, we have
  \begin{align}
  \label{eqn:z_prime_diamond}
    Z^{\lozenge}(T'[\nicefrac{1}{p}],T') &= \left\{m\in \varpi^{-r}(T^{',\flat}\otimes_{T^\flat}M):\;\Phi(m)=m\right\}\nonumber\\
    &\isomto\left\{n\in T^{',\flat}\otimes_{T^\flat}M:\;\Phi(n) = \varpi^{pr}n\right\}\nonumber\\
    &=\left\{n\in T^{',\flat}\otimes_{T^\flat}M:\;\tilde{\Phi}(n) = \varpi^{k+pr}n\right\},
  \end{align}
  where the middle identification is given by $m\mapsto \varpi^r m$. Moreover, via this identification, we have
  \begin{align}
  \label{eqn:z_diamond}
    Z^{',\lozenge}(T'[\nicefrac{1}{p}],T') \simeq Z^{',\lozenge}(T'[\nicefrac{1}{p}],T')\cap \varpi^r(T^{',\flat}\otimes_{T^\flat}M)\subset T^{',\flat}\otimes_{T^\flat}M.
  \end{align}
  Now,~\eqref{eqn:z_prime_diamond} tells us that $Z^{',\lozenge}$ is represented by an affinoid perfectoid over $\Spa(T^\flat[\nicefrac{1}{\varpi}],T^\flat)$: It is a closed subsheaf of the perfectoid closed unit disk
  \[
  \mathbf{D}_{\mathrm{perf}}(M)\colon (T^{',\flat}[\nicefrac{1}{\varpi}],T^{',\flat})\mapsto T^{',\flat}\otimes_{T^\flat}M,
  \]
  associated with $M$. Furthermore,~\eqref{eqn:z_diamond} tells us that $Z^{\lozenge}$ is a rational open in $Z^{',\lozenge}$, obtained as the pre-image of the rational open immersion $\mathbf{D}_{\mathrm{perf}}(M)\xrightarrow{\varpi^r}\mathbf{D}_{\mathrm{perf}}(M)$. 
  
  By the tilting equivalence~\cite[Theorem 6.2.7]{Scholze2020-bx}, we conclude that $Z=\Spa(D,D^+)$ and $Z'=\Spa(D',D^{',+})$ are represented by affinoid perfectoids over $\Spa(T[\nicefrac{1}{p}],T)$. It remains to see that the properties of being a rational open embedding and of $D^\flat \to D^{',\flat}$ having dense image are preserved by untilting: The first follows from~\cite[Theorem 3.6.14]{KedlayaLiu}, and the second can be deduced from the explicit description of untilts of certain rational open embeddings in Lemma 3.6.13 of \emph{op.\@ cit.}
\end{proof}

\begin{proof}
[Proof of Proposition~\ref{prop:diagonal_rational_open}]
Note that $\hat{R}$ is $p$-completely smooth over $\mathcal{O}$ and so admits a $p$-completely flat map $\hat{R}\to \hat{R}_\infty$ where $\hat{R}_\infty$ is perfectoid and is the $p$-completed filtered colimit of finite flat $\hat{R}$-algebras. Lemma~\ref{lem:linearization} and Remark~\ref{rem:linearization} together tell us that the map
\begin{equation}\label{eq:density-after-tensoring}
B\hat{\otimes}\hat{R}_\infty \to (S\hat{\otimes}\hat{R}_\infty)[\nicefrac{1}{p}]
\end{equation}
has dense image. This implies that $B\to S[\nicefrac{1}{p}]$ has dense image. Indeed, it suffices to show that, for all $n\geqslant 1$, if $Q_n$ denotes the cokernel of the map of $R$-modules $B\to S[\nicefrac{1}{p}]/p^nS$ then $Q_n=0$. Let us note that $Q_n$ is $p^\infty$-torsion, and as $R\to \hat{R}_\infty$ is $p$-completely faithfully flat, we see that $Q_n$ is zero if and only if $Q_n\otimes_R\hat{R}_\infty$ is zero. But this follows from the density of the image in \eqref{eq:density-after-tensoring}.
\end{proof}

\subsection{An extension property}
\label{sub:extension_condition}

The following proposition will be used to prove Theorem~\ref{introthm:mapping_property}.
\begin{proposition}
  \label{prop:uniqueness_maps}
Suppose that we have two $\eta$-normal separated algebraic spaces $\mathcal{X}_1$ and $\mathcal{X}_2$ over $\mathcal{O}$, and that we are given $\mathfrak{Q}_i\in \BT[\mathcal{G},\mu,\mathrm{alg}]{\infty}(\mathcal{X}_i)$ lifting $\mathsf{Q}_i\in \mathrm{Loc}_{\mathcal{G}(\Int_p)}(X_i)$ where $X_i\defn \mathcal{X}_{i,\eta}$. Suppose that:
\begin{enumerate}
    \item The stack $\mathcal{X}_2$ is of finite type over $\mathcal{O}$ and the stack $\mathcal{X}_1$ is excellent.
  \item The formal classifying map $\widehat{\mathcal{X}}_2 \to \BT{\infty}$ for $\mathfrak{Q}_2$ is formally unramified.
  \item There is a map $f\colon X_1\to X_2$ and an isomorphism $\zeta\colon\mathsf{Q}_2\vert_{X_1}{\isomto}\mathsf{Q}_1$.
  \item Given any mixed characteristic $(0,p)$ complete discrete valuation field $F$ over $\mathcal{O}$ with perfect residue field, $f$ maps $\mathcal{X}_1(\Reg{F})\subset X_1(F)$ to $\mathcal{X}_2(\Reg{F})\subset X_2(F)$.
\end{enumerate}
Then there are unique extensions of $f$ to a map $\mathcal{X}_1\to \mathcal{X}_2$ and of $\zeta$ to an isomorphism $\mathfrak{Q}_1\vert_{\mathcal{X}_1}{\isomto} \mathfrak{Q}_2$. 
\end{proposition}
\begin{proof}
     The proof we are about to see is essentially a generalization of an argument of Pappas from~\cite[Theorem 7.1.7]{pappas_integral} (see also~\cite[Theorem 3.13]{imai2023prismatic}).

  Let $\mathcal{Y}\to \mathcal{X}_1\times_{\Spec \mathcal{O}} \mathcal{X}_2$ be constructed as follows: Take the scheme-theoretic image $\mathcal{Y}'$ of the graph
     \[
       X_1\xrightarrow{\mathrm{id}\times f}X_1\times_{\Spec E}X_2 \hookrightarrow \mathcal{X}_1\times_{\Spec \mathcal{O}} \mathcal{X}_2.
     \]
  Since $\mathcal{X}_2$ is of finite type over $\mathcal{O}$, we see $\mathcal{Y}'$ is of finite type over $\mathcal{X}_1$. Now  take $\mathcal{Y}$ to be the $\eta$-normalization of $\mathcal{Y}'$\footnote{More specifically, take $\mathcal{Y}$ to be relative spectrum over $\mathcal{Y}'$ of the quasi-coherent $\mc{O}_{\mc{Y}'}$-algebra given by integral closure of $\mathcal{O}_{\mathcal{Y}'}$ in $j_\ast\mc{O}_{\mc{Y}'_\eta}$. where $j\colon \mc{Y}'_\eta\to \mc{Y}'$ is the open inclusion.}: This is of finite type over $\mathcal{X}_1$ because of the excellence; see hypothesis~\stacks{035S}. It suffices to show $\mathcal{Y}\isomto\mathcal{X}_1$, as this implies that $\mathcal{Y}$ is the graph of a unique extension $\mathcal{X}_1\to \mathcal{X}_2$ with the desired properties.
  
  The restrictions of $\mathfrak{Q}_1$ and $\mathfrak{Q}_2$ yield two objects in $\BT{\infty}(\mathcal{Y})$ equipped with an isomorphism of the underlying local systems in $\mathrm{Loc}_{\mathcal{G}(\Int_p)}(\mathcal{Y}_\eta)$. Therefore, by Theorem~\ref{thm:full_faithfulness}, we see that there is an isomorphism $\mathfrak{Q}_1\vert_{\mathcal{Y}}{\isomto} \mathfrak{Q}_2\vert_{\mathcal{Y}}$. In other words, the  classifying maps
  \[
    \mathcal{Y}\to \mathcal{X}_1\to \BT[\mathcal{G},\mu,\mathrm{alg}]{\infty}\;;\; \mathcal{Y}\to \mathcal{X}_2\to \BT[\mathcal{G},\mu,\mathrm{alg}]{\infty},
  \]
  are isomorphic. 
  
  Suppose that we have a closed geometric point $y\in \mathcal{Y}(\kappa)$ of the special fiber mapping to $x_i\in \mathcal{X}_i(\kappa)$. We claim that the map of complete local Noetherian rings $\widehat{\Rg}_{\mathcal{X}_1,x_1}\to \widehat{\Rg}_{\mathcal{Y},y}$ is finite.
  
  Write $z\in \BT{\infty}(W(\kappa))$ for the point corresponding to $\mathfrak{Q}_{1,x_1}\simeq\mathfrak{Q}_{2,x_2}$. Let $\widehat{\Rg}_y$ (resp.\@ $\widehat{\Rg}_{x_1}$, $\widehat{\Rg}_{x_2}$, and $\widehat{\Rg}_z$) be the universal deformation rings over $W(\kappa)$ for $\mathcal{Y}$ (resp.\@ $\mathcal{X}_1$, $\mathcal{X}_2$, and $\BT[\mathcal{G},\mu,\mathrm{alg}]{\infty}$) at their respective $\kappa$-points. We now have factorizations
  \[
    \widehat{\Rg}_{z}\to \widehat{\Rg}_{x_1}\to \widehat{\Rg}_{y}\;;\; \widehat{\Rg}_{z}\to \widehat{\Rg}_{x_2}\to \widehat{\Rg}_{y},
  \]
  of the map $ \widehat{\Rg}_{z}\to \widehat{\Rg}_{y}$. By our formal unramifiedness hypothesis, the first map in the second factoring is surjective. By construction, the combined map
  \[
     \widehat{\Rg}_{x_1}\hat{\otimes}_{W(\kappa)}\widehat{\Rg}_{x_2}\to \widehat{\Rg}_{y},
  \]
  whose image is now the same as that of the first factor $\widehat{\Rg}_{x_1}$, is finite (since $\mathcal{Y}\to \mathcal{X}_1\times_{\Spf\mathcal{O}}\mathcal{X}_2$ is finite). We conclude that the map $\widehat{\Rg}_{x_1}\to \widehat{\Rg}_{y}$ is already finite.

 What we have seen above shows that the map $\mathcal{Y}\to \mathcal{X}_1$ is quasi-finite; see for instance \cite[Ch. IV, Propositions 2 and 3]{Raynaud1970-nd}. Since it is also an isomorphism after inverting $p$, it follows from Zariski's Main Theorem and the $\eta$-normality of $\mathcal{X}_1$ that it is in fact an open immersion: one may reduce by \'etale descent to \cite[Corollary 12.88]{GortzWedhorn}. 

 We now claim that every algebraically closed point $x_1\in \mathcal{X}_1(\kappa)$ of the special fiber is inside the open $\mathcal{Y}$, which implies that $\mc{Y}\to\mc{X}_1$ is an isomorphism as desired. Indeed, by the flatness of $\mathcal{X}_1$, we can find a mixed characteristic complete discrete valuation ring $\Reg{F}$ with residue field $\kappa$ such that $x_1$ lifts to $\mathcal{X}_1(\Reg{F})$; see \stacks{0CM2}. But hypothesis (4) tells us that this lift lies in $\mathcal{Y}(\Reg{F})$.
\end{proof}

\section{Criteria for the existence of apertures}
\label{sec:criteria_for_the_existence_of_apertures}

The purpose of this section is to develop criteria for proving the existence of apertures when $\mathcal{G}$ is reductive. These are mainly of a `pointwise' nature, and will be employed in the next section to prove the canonicity of integral models. We will maintain the notation from Setup~\ref{setup:g_mu}. 

\subsection{Apertures over \texorpdfstring{$p$}{p}-completely smooth \texorpdfstring{$\Reg{K}$}{OK}-algebras}
\label{sub:apertures_over_base_algebras}

Let $K$ be a complete mixed characteristic $(0,p)$ discrete valuation with perfect residue field $\kappa$ and absolute ramification index $e$, and fix a uniformizer $\pi\in K$.  Our goal here is to characterize the image of the fully faithful \'etale realization functor for base formal $\Reg{K}$-schemes. We will assume that $\mathcal{G}$ is \emph{reductive}.

\begin{definition}
[The type of a crystalline local system]
\label{def:type_crystalline}
Suppose that $G$ is reductive and that we have $\mathsf{Q}\in \mathrm{Loc}^{\mathrm{crys}}_{\mathcal{G}(\Int_p)}(K)$. Then, for every representation $V$ of $\mc{G}$ we obtain a filtered $K$-vector space $\Fil^\bullet D_{\dR}((V)_{\mathsf{Q}})$. By~\cite[Lemma (1.4.5)]{kisin:abelian}, there exists a cocharacter of $G_K$ that splits these filtrations simultaneously for all $V$. In particular, this gives us a well-defined geometric conjugacy class of cocharacters of $G$. This conjugacy class is the \defnword{type} of $\mathsf{Q}$. If $\lambda$ is a representative of the conjugacy class defined over $K$, then we will abuse terminology and say that $\mathsf{Q}$ has type $\lambda$. More generally, if $\mf{X}$ is a base formal $\Reg{K}$-scheme, a crystalline $\mc{G}(\bb{Z}_p)$-local systems $\mathsf{Q}$ over $\mf{X}$ is of \defnword{type $\lambda$} if, for all $V\in \Rep_{\Rat_p}(G)$, the filtration $\Fil^\bullet_{\mathrm{Hdg}}D_{\dR}((V)_{\mathsf{Q}})$ is, \'etale locally on $\mf{X}_\eta$, split by the cocharacter $\lambda$.
\end{definition}

\begin{notation}
For a cocharacter $\lambda$ of $G_K$, let $\mathrm{Loc}^\mathrm{crys,\lambda}_{\mathcal{G}(\Int_p)}(K)$ be the groupoid of crystalline $\mc{G}(\bb{Z}_p)$-local systems of type $\lambda$ on $K$. More generally, for any base formal $\Reg{K}$-scheme $\mf{X}$, we define $\mathrm{Loc}^{\mathrm{crys},\mu}_{\mathcal{G}(\Int_p)}(\mf{X}_\eta)$ to be the groupoid of crystalline $\mc{G}(\bb{Z}_p)$-local systems of type $\lambda$ on $\mf{X}_\eta$.
\end{notation}

The two following results characterize the image of the fully faithful \'etale realization functor on $\BT{\infty}(R)$ in some cases.
\begin{proposition}
\label{prop:gmu_rings_of_integers}
The fully faithful functor 
\begin{equation*}
T_{\et}\colon \BT{\infty}(\Reg{K})\to \mathrm{Loc}_{\mathcal{G}(\Int_p)}(K)
\end{equation*}
has essential image $\mathrm{Loc}^\mathrm{crys,\mu}_{\mathcal{G}(\Int_p)}(K)$.
\end{proposition}

\begin{theorem}
  \label{thm:p_comp_smooth_characterization}
Suppose that $\mf{X}$ is a base formal $\Reg{K}$-scheme and that $2e<p-1$. Then the fully faithful functor
\[
T_{\et}\colon \BT{\infty}(\mathfrak{X})\to \mathrm{Loc}_{\mathcal{G}(\Int_p)}(\mf{X}_\eta)
\]
has essential image $\mathrm{Loc}^\mathrm{crys,\mu}_{\mathcal{G}(\Int_p)}(\mf{X}_\eta)$.
\end{theorem}

\begin{remark}
  [Crystallinity from pointwise crystallinity]
\label{rem:guo_yang}
The work of Guo--Yang~\cite{guo2024pointwisec}, as explained in Remark~\ref{rem:crystalline_local_system},  tells us that, when $\mf{X}$ is $p$-completely smooth over $\Reg{K}$, then the following conditions are equivalent for a $\mathcal{G}(\Int_p)$-local system $\mathsf{Q}$ over $\mf{X}_\eta$:
\begin{enumerate}
    \item $\mathsf{Q}$ belongs to $\mathrm{Loc}^\mathrm{crys,\mu}_{\mathcal{G}(\Int_p)}(\mf{X}_\eta)$;
    \item For every classical point $x\colon \Spa(K')\to\mf{X}_\eta$, one has that $\mathsf{Q}_x$ is in $\mathrm{Loc}^\mathrm{crys,\mu}_{\mathcal{G}(\Int_p)}(K')$.
\end{enumerate}
Therefore, in this case, the second pointwise condition can be used to describe the essential image of $T_{\et}$.
\end{remark}

\begin{remark}
    [Crystalline specialization functors]
\label{rem:specialization}
In the situation of either Theorem~\ref{thm:p_comp_smooth_characterization} or Proposition~\ref{prop:gmu_rings_of_integers} (the latter with $\mf{X} = \Spf\Reg{K}$), we obtain functors 
\[
\mathrm{Loc}^{\mathrm{crys},\mu}_{\mathcal{G}(\Int_p)}(\mf{X}_\eta)\isomfrom\BT{\infty}(\mf{X})\to \BT{\infty}(\mf{X}_\kappa),
\]
where $\mf{X}_\kappa$ is the special fiber of $\mf{X}$. Using this, we can functorially associate with every $\mathsf{Q}\in \mathrm{Loc}^{\mathrm{crys},\mu}_{\mathcal{G}(\Int_p)}(\mf{X}_\eta)$ a $(\mathcal{G},\mu)$-aperture $\Dieu_{\mathrm{crys}}(\mathsf{Q})$ over $\mf{X}_\kappa$. In particular, for every $\Lambda\in \Rep_{\Int_p}(\mathcal{G})$, we obtain the $F$-gauge
\[
\Dieu_{\mathrm{crys}}\left((\Lambda)_{\mathsf{Q}}\right) \defn (\Lambda)_{\Dieu_{\mathrm{crys}}(\mathsf{Q})}\in \mathrm{Vect}(\mf{X}_\kappa^{\mathrm{syn}}).
\]
If we take the underlying $F$-crystal over $\mf{X}_\kappa$, then we are essentially getting a special case of the similarly denoted functor from~\cite[\S 2]{IKY3}.
\end{remark}

To prove our main results, we will first need a bit of preparation. Theorem \ref{thm:p_comp_smooth_characterization} quickly reduces to the affine case $\mf{X}=\Spf(R)$, and so we will focus on this below. 
\begin{lemma}
    \label{lem:type_of_bt_mu}
The \'etale realization functor $T_{\et}\colon \BT{\infty}(\Reg{K})\to \mathrm{Loc}^{\mathrm{crys}}_{\mathcal{G}(\Int_p)}(K)$ lands in $\mathrm{Loc}^{\mathrm{crys},\mu}_{\mathcal{G}(\Int_p)}(K)$
\end{lemma}
\begin{proof}
Given $\mf{Q}\in \BT{\infty}(\Reg{K})$ lifting $\mathsf{Q}\in \mathrm{Loc}_{\mathcal{G}(\Int_p)}(K)$, Lemma~\ref{lem:filtered_de_rham_comparison} shows that $\Fil^\bullet D_{\dR}(\mathsf{Q})$ is canonically isomorphic to $\Fil^\bullet_{\mathrm{Hdg}}T_{\dR}(\mf{Q})$. This, combined with the definition of $\BT{\infty}$, implies the claim.
\end{proof}

\begin{remark}
  [Reduction to a question of algebraicity]
\label{rem:algebraicity}
Theorem~\ref{thm:guo_reinecke} tells us that the tensor functor $\Rep_{\Int_p}(\mathcal{G})\to \mathrm{Loc}_{\Int_p}(R[\nicefrac{1}{p}])$ associated with $\mathsf{Q}$ factors through an exact symmetric monoidal functor
\[
  \Rep_{\Int_p}(\mathcal{G})\to \mathrm{Vect}^{\mathrm{an},\varphi}(R_{\Prism},\Reg{\Prism}).
\]
By~\cite[Propositions 1.28 and 1.39]{imai2023prismatic}, the theorem now amounts to showing that this functor factors through the fully faithful subcategory $\mathrm{Vect}^{\varphi}(R_{\Prism},\Reg{\Prism})$. Indeed, the cited results would then imply that $\mathsf{Q}$ arises from a $\mathcal{G}$-torsor $\mathfrak{Q}$ over $R^{\mathrm{syn}}$, and our pointwise assumption on the type, combined with Lemma~\ref{lem:type_of_bt_mu}, implies that $\mathfrak{Q}$ is in $\BT{\infty}(R)$.

To establish the desired factorization, it is now sufficient to show that the analytic prismatic $F$-crystal associated with $(\Lambda)_{\mathsf{Q}}$ is in fact a prismatic $F$-crystal for \emph{any} faithful representation of $\mathcal{G}$. 
\end{remark}

\begin{remark}
  [Reduction to a question of finite freeness]
\label{rem:finite_free_criterion}
Working \'etale locally on $\Spf R$, we can assume that $R$ is a base $\Reg{K}$-algebra admitting a Breuil--Kisin frame $\underline{\Sig}_R$ as in Example~\ref{ex:bk_frames_for_base_algebras}. The analytic prismatic $F$-crystal associated with $(\Lambda)_{\mathsf{Q}}$ gives a vector bundle $\mathcal{V}^{\mathrm{an}}$ over $\Spec \Sig_R\backslash V\big(p,E(u)\big)$. Knowing that we in fact have a prismatic $F$-crystal amounts to knowing that the reflexive $\Sig_R$-module
\[
  \mathcal{L} = H^0\bigg(\Spec \Sig_R\backslash V\big(p,E(u)\big),\mathcal{V}^{\mathrm{an}}\bigg)
\]
is in fact finite locally free (cf.\@ \cite[Proposition 1.26]{imai2024tannakianframeworkprismaticfcrystals}).
\end{remark}

\begin{proof}
  [Proof of Proposition~\ref{prop:gmu_rings_of_integers}]
In this case, $\Sig_{\Reg{K}}$ is a regular local Noetherian ring of dimension $2$, and so every reflexive finitely generated module over it is finite free.
\end{proof}

\begin{lemma}
  \label{lem:case_of_torus}
Suppose that $\mathcal{G} = \mathcal{T}$ is a torus over $\Int_p$. Then, \emph{Theorem~\ref{thm:p_comp_smooth_characterization}} holds.
\end{lemma}
\begin{proof}
We can reduce to the situation in Remark~\ref{rem:finite_free_criterion}, and so need to check that the reflexive $\Sig_R$-module $\mathcal{L}$ associated with a faithful representation $\Lambda$ is finite locally free. Since this can be checked after a faithfully flat base change, we can assume without loss of generality that $\mathcal{T}$ is split over $\Sig_R$. In particular, $\mathcal{L}^{\an}$ is now a direct sum of line bundles, and this implies that $\mathcal{L}$ is finite locally free; see~\cite[Expos\'e XI, Th\'eor\`eme 3.18]{Grothendieck2005-ro}.
\end{proof}

\begin{lemma}
  [Reduction to the complete local case]
\label{lem:reduction_to_complete_local}
Suppose that $e$ is arbitrary and that the following equivalent conditions hold for every maximal ideal $\mathfrak{m}\subset R$:
\begin{enumerate}
  \item The restriction of $\mathsf{Q}$ over $\Spec\hat{R}_{\mathfrak{m}}[\nicefrac{1}{p}]$ lifts to $\BT{\infty}(\hat{R}_{\mathfrak{m}})$.
  \item The analytic prismatic $F$-crystal over $\hat{R}_{\mathfrak{m}}$ associated with $(\Lambda)_{\mathsf{Q}}$ is in fact a prismatic $F$-crystal.
\end{enumerate}
Then $\mathsf{Q}$ lifts to $\BT{\infty}(R)$.
\end{lemma}
\begin{proof}
   As in Remark~\ref{rem:finite_free_criterion}, we can assume that $R$ is a base $\Reg{K}$-algebra of the form $R_0\otimes_{W(\kappa)}\Reg{K}$, and are therefore reduced to checking that the $\Sig_R$-module $\mathcal{L}$ is finite locally free. For this, it is enough to check that, for every maximal ideal $\mathfrak{m}\subset R$, lifting to a maximal ideal $\mathfrak{n}\subset \Sig_R$, the completion of $\mathcal{L}$ at $\mathfrak{n}$ is finite locally free. However, the completion $(\Sig_R)^{\wedge}_{\mathfrak{n}}$ underlies a Breuil--Kisin frame for $\hat{R}_{\mathfrak{m}} = \hat{R}_{0,\mathfrak{m}_0}\otimes_{W(\kappa)}\Reg{K}$, where $\mathfrak{m}_0 = \mathfrak{m}\cap R_0$. Therefore, our equivalent hypotheses tell us that $\hat{\mathcal{L}}_{\mathfrak{n}}$ is obtained from the evaluation of a prismatic $F$-crystal over $\hat{R}_{\mathfrak{m}}$, and is thus finite free, as desired.
\end{proof}

\begin{notation}
  Suppose now that $R$ is a base $\Reg{K}$-algebra and that $T$ is the completion of $R$ at a maximal ideal. Let $\underline{\Sig}_T$ be the Breuil--Kisin frame for $T$ obtained as a completion of $\Sig_R$ as in the proof above, and set $\mathcal{L}_T = T\otimes_R \mathcal{L}$.
\end{notation}

\begin{lemma}
  \label{lem:complete_local_existence_of_aperture}
Suppose that $\mu$ acts on $\Lambda$ with weights in the interval $[a,b]$ with $|b-a|=r$. Then, if $er<p-1$, $\mathcal{L}_T$ is finite free over $T$. In particular, if $\mathcal{G}$ is adjoint, then Theorem~\ref{thm:p_comp_smooth_characterization} holds.
\end{lemma}
\begin{proof}
  The first assertion follows from the argument in~\cite[Proposition 4.3]{LiuMoon}. The second now follows from the fact that we can choose $\Lambda$ to be the adjoint representation, which has $\mu$-weights $\{-1,0,1\}$.
\end{proof}

\begin{remark}
  \label{rem:complete_local_trivial_criterion}
Suppose that $\kappa$ is algebraically closed. Then one sees that a $\mathcal{G}$-torsor over $U \defn \Spec \Sig_T\backslash V(p,E(u))$ extends to a $\mathcal{G}$-torsor over $\Sig_T$ if and only if it is actually trivial. This is because every $\mathcal{G}$-torsor over $\Sig_T$ is trivial, since $\Sig_T$ is now a strictly Henselian local ring. Moreover, the restriction map for $\mathcal{G}$-torsors from $\Spec \Sig_R$ to $U$ is fully faithful (e.g., see \stacks{0EBJ}).
\end{remark}

\begin{proof}
  [Proof of Theorem~\ref{thm:p_comp_smooth_characterization}]
Without loss of generality, we can assume that $\kappa$ is algebraically closed and that $R$ is a base $\Reg{K}$-algebra with Breuil--Kisin frame $\underline{\Sig}_R$. We only have to show that for every completion $T$ of $R$, the $\mathcal{G}$-torsor $\mathcal{Q}^{\mathrm{an}}_T$ over $U \defn \Spec\Sig_T\backslash V(p,E(u))$ obtained from the analytic prismatic $F$-crystal realization of $\mathsf{Q}$ extends to a $\mathcal{G}$-torsor over $\Sig_T$ or, equivalently, that it is trivial.

Consider the map $\mathcal{G}\to \overline{\mathcal{G}}\defn \mathcal{G}^{\ad}\times \mathcal{G}^{\ab}$: This is a finite flat cover of reductive group schemes, and Lemmas~\ref{lem:case_of_torus} and~\ref{lem:complete_local_existence_of_aperture} together tell us that the theorem is valid with $\mathcal{G}$ replaced with $\overline{\mathcal{G}}$ and $\mathsf{Q}$ replaced with the associated $\overline{\mathcal{G}}(\Int_p)$-torsor over $R[\nicefrac{1}{p}]$. Therefore, $\mathcal{Q}^{\mathrm{an}}_T$ is trivial after change of structure group to $\overline{\mathcal{G}}$, and so admits a reduction of structure group to a $\mathcal{Z}$-torsor over $U$, where $\mathcal{Z} = \ker(\mathcal{G}\to \overline{\mathcal{G}})$. However, by purity of branch locus for torsors under finite flat group schemes~\cite[Theorem 3.1]{Marrama2016-gc}, all $\mathcal{Z}$-torsors over $U$ extend uniquely to $\mathcal{Z}$-torsors over $\Sig_T$. This implies that $\mathcal{Q}^{\an}_T$ also extends to a $\mathcal{G}$-torsor over $\Sig_T$, as desired.
\end{proof}

\begin{remark}
[The Hodge type case]
\label{rem:hodge_type_improvement}
  If $(\mathcal{G},\mu)$ is of \emph{Hodge type}, i.e., there exists a map $(\mathcal{G},\mu)\to (\GL_h,\mu_d)$ (see Example~\ref{ex:gl_n_truncated_bts}) for some $d$ such that the underlying map $\mathcal{G}\to \GL_h$ is a closed immersion, then we can improve the bound to $p<e-1$: Indeed, we can replace the use of the adjoint representation with the faithful representation given by such a map, and this will have $\mu$-weights in $\{0,1\}$. When $\mathcal{G} = \GL_h$ this amounts to the assertion that any crystalline local system over $R[\nicefrac{1}{p}]$ with Hodge--Tate weights in $\{0,1\}$ arises from a $p$-divisible group. This is a result of Liu--Moon~\cite[Theorem 1.2]{LiuMoon}.
\end{remark}

\subsection{Versality}
\label{sub:versality}

When we obtain an aperture using Theorem~\ref{thm:p_comp_smooth_characterization}, versality is not automatic. In applications, it can be understood via a Kodaira--Spencer map, which we now discuss.

\begin{remark}
 [The formal de Rham realization and a versality condition]
\label{rem:the_formal_de_rham}
Supposethat we have a map of stacks  $f\colon \mathcal{X}\to \BT[\mathcal{G},\mu,\mathrm{alg}]{\infty}$ over $\mathcal{O}$. After completion, we obtain a filtered de Rham realization
\[
  \widehat{\mathcal{X}}\to \BT{\infty}\xrightarrow{T_{\dR}} \mathrm{Gr}_{\mu}/\mathcal{G}_{\Reg{E,(v)}}.
\]
By the discussion in Remark~\ref{rem:versality_bt_n}, this gives a map of $p$-complete complexes
\[
\mathbb{L}^{\wedge}_{(\mathrm{Gr}_{\mu}/\mathcal{G}_{\Reg{E,(v)}})/B \mathcal{G}_{\mathcal{O}}}\vert_{\widehat{\mathcal{X}}}\to \mathbb{L}^{\wedge}_{\mathcal{X}/\Reg{E,(v)}},
\]
which is an isomorphism if and only if the map $\widehat{\mathcal{X}}\to \BT{\infty}$ is formally \'etale. 
\end{remark}

\begin{remark}
  [The Kodaira--Spencer map]
\label{rem:Kodaira--Spencer}
Suppose that $\mathcal{X}$ is a smooth scheme over $\Reg{E,(v)}$. Dualizing the map from the previous remark, we obtain a map
\begin{align}\label{eqn:tangent_space_map}
  \mathbb{T}_{\mathcal{X}/\Reg{E,(v)}}^{\wedge}\to \mathbb{T}^{\wedge}_{(\mathrm{Gr}_{\mu}/\mathcal{G}_{\Reg{E,(v)}})/B \mathcal{G}_{\mathcal{O}}}\vert_{\widehat{\mathcal{X}}}
\end{align}
of vector bundles over $\widehat{\mathcal{X}}$. The left-hand side here is the $p$-completed tangent bundle of $\mathcal{X}$ while the right-hand side can be described explicitly: The aperture $\mathfrak{Q}$ classified by $f$ admits a Hodge-filtered de Rham realization $\Fil^\bullet_{\mathrm{Hdg}}T_{\dR}(\mathfrak{Q})$ equipped with a topologically nilpotent integrable connection. Let $\mathfrak{g}$ be the Lie algebra of $\mathcal{G}$, equipped with the adjoint representation. Twisting this by $\Fil^\bullet_{\mathrm{Hdg}}T_{\dR}(\mathfrak{Q})$, we obtain a filtered vector bundle $\Fil^\bullet \mathfrak{g}_{\mathrm{twist}}$ over $\widehat{\mathcal{X}}$. There is now a canonical isomorphism 
\[
  \mathbb{T}^{\wedge}_{(\mathrm{Gr}_{\mu}/\mathcal{G}_{\Reg{E,(v)}})/B \mathcal{G}_{\mathcal{O}}}\vert_{\widehat{\mathcal{X}}}\simeq \mathfrak{g}_{\mathrm{twist}}/\Fil^0 \mathfrak{g}_{\mathrm{twist}}.
\]
Now,~\eqref{eqn:tangent_space_map} now gives a map $\mathbb{T}_{\mathcal{X}/\Reg{E,(v)}}^{\wedge}\to \mathfrak{g}_{\mathrm{twist}}/\Fil^0 \mathfrak{g}_{\mathrm{twist}}$. This is just the Kodaira--Spencer map associated with the integrable connection on $\Fil^\bullet_{\mathrm{Hdg}}T_{\dR}(\mathfrak{Q})$, and so versality comes down to knowing that this is an isomorphism.
\end{remark}

\subsection{Apertures via reduction of structure group}
\label{sub:apertures_via_reduction_of_structure_group}

\begin{setup}
  Let $\kappa$ be an algebraically closed characteristic $p$ field over $\mathcal{O}$, and let $R_1\to R_2$ be a map of complete local normal rings that are flat and formally of finite type over $W(\kappa)$. We will assume that both $R_1$ and $R_2$ have residue field $\kappa$. Suppose that we have a closed immersion $(\mathcal{G}_2,\mu_2)\to (\mathcal{G}_1,\mu_1)$ of two pairs as in Setup~\ref{setup:g_mu}, with $\mu_1$ and $\mu_2$ both defined over $\mathcal{O}$, and with both $\mathcal{G}_1$ and $\mathcal{G}_2$ \emph{reductive}. Suppose also that we have $\mathsf{Q}_i\in \mathrm{Loc}_{\mathcal{G}_i(\Int_p)}(R_i[\nicefrac{1}{p}])$ such that the restriction of $\mathsf{Q}_1$ over $\Spec R_2[\nicefrac{1}{p}]$ is isomorphic to the change of structure group of $\mathsf{Q}_2$ along the map $\mathcal{G}_2(\Int_p)\to \mathcal{G}_1(\Int_p)$.  We will assume that $\mathsf{Q}_1$ lifts to $\mathfrak{Q}_1\in \BT[\mathcal{G}_1,\mu_1]{\infty}(R_1)$. Write $\mathfrak{Q}_{1,0} = \mathfrak{Q}_1\vert_{\kappa^{\mathrm{syn}}}$ for the resulting aperture over $\kappa$.
\end{setup}

\begin{remark}
\label{rem:type_via_lambda}
Let us put ourselves in the situation of Setup~\ref{setup:faithful_repn_tensors}: We will choose a faithful representation $\Lambda$ of $\mathcal{G}_1$, and take $\{s_{\alpha}\}\subset \Lambda^\otimes$ to be a finite collection of tensors whose pointwise stabilizer is $\mathcal{G}_2$. As in Remark~\ref{rem:expl_tensor_packages}, the $F$-gauge $(\Lambda)_{\mathfrak{Q}_1}\rvert_{\kappa^{\mathrm{syn}}}$ corresponds to a $p$-adically filtered $W(\kappa)$-module $\Fil^\bullet L_0$ equipped with an isomorphism 
\[
\sum_{i\in \Int}p^{-i}\varphi^*\Fil^iL_0\isomto L_0.
\]
For $x\colon R_2\to \Reg{K}$, the existence of the reduction of structure $\mathsf{Q}_2$ for $\mathsf{Q}_1\vert_{\Spec R_2[\nicefrac{1}{p}]}$ implies we have global sections 
\[
\{s_{\alpha,x}\}\subset H^0\big(\Spec K,(\Lambda)_{\mathsf{Q}_2|_x}^\otimes\big).
\]
Via the specialization functor $\mathbb{D}_{\mathrm{crys}}$ (see Remark \ref{rem:specialization}), these map to $\varphi$-invariant tensors 
\[
\bigl\{\Dieu_{\mathrm{crys}}(s_{\alpha,x})\bigr\}\in \Dieu_{\mathrm{crys}}\big((\Lambda)_{\mathsf{Q}_{2}|_x}\big)^\otimes \simeq L_0^\otimes.
\]
The isomorphism $ \Dieu_{\mathrm{crys}}\big((\Lambda)_{\mathsf{Q}_{2}|_x}\big)\simeq L_0$ used here is the natural one.
\end{remark}

\begin{remark}
  [Families of $F$-isocrystals]
\label{rem:horizontal_tensors_limpidity}
Suppose that $R_1$ is formally smooth over $W(\kappa)$, equipped with a Frobenius lift $\varphi\colon R_1\to R_1$. Then the rigid analytic space $\widehat{U}_{1,\eta}$ associated with $\widehat{U}_1 = \Spf R_1$ is an open polydisk. Moreover, the $F$-gauge $\mathfrak{L}$ lifting $(\Lambda)_{\mathsf{Q}_1}$ yields a filtered $F$-isocrystal over $R_1/pR_1$, which realizes to a filtered vector bundle $\Fil^\bullet \mathcal{L}_{\Rat}$ over $\widehat{U}_{1,\eta}$, equipped with an integrable connection, and specializing to $\Fil^\bullet L_x[\nicefrac{1}{p}]$ at every classical point $x\in \widehat{U}_{1,\eta}(K)$. Set
\[
\mathcal{L}_{\Rat}^{\nabla} \defn H^0\big(\widehat{U}_{1,\eta},\mathcal{L}_{\Rat}\big)^{\nabla = 0}.  
\]
This is a finite dimensional vector space over $W(\kappa)[\nicefrac{1}{p}]$, and inherits the structure of an $F$-isocrystal from that on $\mathcal{L}_{\Rat}$. Moreover, we have a canonical isomorphism of $F$-isocrystals
\[
  \delta\colon\Reg{\widehat{U}_{1,\eta}}\otimes_{W(\kappa)[\nicefrac{1}{p}]}\mathcal{L}_{\Rat}^{\nabla}\isomto \mathcal{L}_{\Rat}
\]
over $\widehat{U}_{1,\eta}$. This is a consequence of Dwork's trick; see~\cite[Proposition 3.1]{zbMATH03409516}.
\end{remark}

\begin{remark}
  [Application of the de Rham functor of Liu--Zhu]
\label{rem:p-adic_comparison_and_limpidity}
By Remark~\ref{rem:filtered_f-isocrystals}, $\Fil^\bullet \mathcal{L}_{\Rat}$ is also obtained from the filtered $F$-isocrystal associated with the $\Int_p$-local system $(\Lambda)_{\mathsf{Q}_1}$. If $\widehat{U}_{2,\eta}$ is a \emph{smooth} rigid analytic space over $W(\kappa)[\nicefrac{1}{p}]$, then the restriction of $\Fil^\bullet \mathcal{L}_{\Rat}$ over $\widehat{U}_{2,\eta}$ is once again a filtered vector bundle with integrable connection, and is \emph{associated} with $(\Lambda)_{\mathsf{Q}_2}$ via the $D_{\dR}$ functor of Liu--Zhu~\cite[Theorem 3.9]{Liu2017-yz}. In particular, this functor carries the global sections
\[
  \{s_{\alpha,\mathsf{Q}_2}\}\subset H^0\big(\Spec R_2[\nicefrac{1}{p}],(\Lambda)_{\mathsf{Q}_2}^\otimes\big)
\]
to a collection
\[
  \{s_{\alpha,\mathcal{L}_{\Rat}}\}\subset H^0\big(\widehat{U}_{2,\eta},\Fil^0\mathcal{L}_{\Rat}^\otimes\big)^{\nabla = 0}.
\]
\end{remark}

\begin{remark}
  [Type of reduction at every point]
 \label{rem:type_of_reduction}
Suppose that $\mathsf{Q}_1$ lifts to $\mathfrak{Q}_1\in \BT[\mathcal{G}_1,\mu_1]{\infty}(R_1)$, and that we have a finite extension $K/W(\kappa)[\nicefrac{1}{p}]$, and a map $x\colon R_2\to \Reg{K}$ of complete local rings. The $\mathcal{G}_1(\Int_p)$-local system $\mathsf{Q}_{1}|_x$ has a reduction of structure to a $\mathcal{G}_2(\Int_p)$-local system $\mathsf{Q}_{2}|_x$, which is in $\mathrm{Loc}^\mathrm{crys,\mu'_2}_{\mathcal{G}_2(\Int_p)}(K)$ for some cocharacter $\mu'_2$ mapping to the $\mathcal{G}_1$-conjugacy class of $\mu_1$. We will refer to the conjugacy class of $\mu'_2$ as the \defnword{type of $x$}.
\end{remark}

\begin{remark}
  [Compatibility between specializations of points]
 \label{rem:compatibility_specializations}
Suppose that we have two points $x\colon R_2\to \Reg{K}$ and $y\colon R_2\to \Reg{L}$ as in Remark~\ref{rem:type_of_reduction} for finite extensions $K$ and $L$ of $W(\kappa)$. Then, via the specialization functor from Remark~\ref{rem:specialization}, we obtain $\mathcal{G}_2$-bundles $\Dieu_{\mr{crys}}(\mathsf{Q}_{2}|_x)$ and $\Dieu_{\mr{crys}}(\mathsf{Q}_{2}|_y)$ over $\kappa$, which are both canonically reductions of structure group for the $\mathcal{G}_1$-torsor $\mf{Q}_2$. In particular, it makes sense to ask for them to be \emph{equal}; in this case, it follows that $x$ and $y$ have the same type. Explicitly, this equality amounts to asking that, for every $\alpha$, we have
\[
\Dieu_{\mathrm{crys}}(s_{\alpha,x})=\Dieu_{\mathrm{crys}}(s_{\alpha,y})\in L_0^{\otimes}.
\]
\end{remark}

\begin{proposition}
  \label{prop:reduction_of_structure_group}
In the situation of \emph{Remark~\ref{rem:horizontal_tensors_limpidity}}, suppose further that the rigid analytic space $\widehat{U}_{2,\eta}$ is also smooth over $W(\kappa)[\nicefrac{1}{p}]$, and that \emph{some} $x$ as above has type $\mu_2$. Then:
\begin{enumerate}
  \item There exists a lift of $\mathsf{Q}_2$ to $\mathfrak{Q}_2\in \BT[\mathcal{G}_2,\mu_2]{\infty}(R_2)$. 
  \item Suppose that the following conditions hold: 
    \begin{enumerate}
      \item The classifying map $\Spf R_1\to \BT[\mathcal{G}_1,\mu_1]{\infty}$ for $\mathfrak{Q}_1$ is formally unramified;
      \item If $I = \ker(R_1\to R_2)$, then $R_1/I\to R_2$ is finite and inducing an isomorphism on fraction fields;
      \item $\dim R_2 = 1+\dim \mathcal{G}_2 - \dim \mathcal{P}^-_{\mu_2}$.
    \end{enumerate}
  Then $R_1\to R_2$ is in fact surjective and the classifying map $\Spf R_2\to \BT[\mathcal{G}_2,\mu_2]{\infty}$ for $\mathfrak{Q}_2$ is the universal deformation ring over $W(\kappa)$ for $\BT[\mathcal{G}_2,\mu_2]{\infty}$ at $\mf{Q}_{2,0}$.
\end{enumerate}
\end{proposition}
\begin{proof}
  To begin, note that there is a canonical isomorphism of $F$-isocrystals $L_0[\nicefrac{1}{p}]\isomto\mathcal{L}_{\Rat}^{\nabla}$. Suppose that we have $x\colon R_2\to \Reg{K}$ as in our hypothesis. The tensors $\{s_{\alpha,\mathcal{L}_x}\}$ map to the $\varphi$-invariant tensors $
  \bigl\{\Dieu_{\mathrm{crys}}(s_{\alpha,x})\bigr\} \in L_0^\otimes$, and these in turn yield a collection of parallel $\varphi$-invariant tensors 
  \[
  \{\tilde{s}_{\alpha,\mathcal{L}_{\Rat}}\}\defn \delta\big(\left\{1\otimes \Dieu_{\mathrm{crys}}(s_{\alpha,x})\right\}\big)\subset H^0\big(\widehat{U}_{1,\eta},\mathcal{L}_{\Rat}^\otimes\big)^{\nabla=0}.
  \]
  The restriction of this collection over $\widehat{U}_{2,\eta}$ agrees with $\{s_{\alpha,\mathcal{L}_{\Rat}}\}$ at $x$ and so, by horizontality, must agree everywhere.\footnote{We are using the connectedness of $\widehat{U}_{2,\eta}$ here; see~\cite[Lemma 7.3.5]{dejong:formal_rigid}.} Therefore, the tensors $\bigl\{\Dieu_{\mathrm{crys}}(s_{\alpha,x})\bigr\}$ are independent of $x$, and so we see using Remark~\ref{rem:compatibility_specializations} that there is a canonical reduction of structure $\mathfrak{Q}_{2,0}\subset \mathfrak{Q}_{1,0}$ to a $(\mathcal{G}_2,\mu_2)$-aperture over $\kappa$ such that, 
\[
\Dieu_{\mr{crys}}(\mathsf{Q}_{2}|_{x'}) = \mathfrak{Q}_{2,0} \subset \mf{Q}_{1,0}.
\]
for all finite extensions $K'/W(\kappa)[\nicefrac{1}{p}]$ and all maps $x'\colon R_2\to \Reg{K'}$.

  Let $\tilde{R}_i$ be the universal deformation ring for $\mathfrak{Q}_{i,0}$ for $i=1,2$ (see Remark~\ref{rem:bt_deformation_rings}). The map $\tilde{R}_1\to \tilde{R}_2$ is a surjective map of formally smooth complete local Noetherian $W(\kappa)$-algebras: This amounts to the observation (which follows from the Grothendieck--Messing theory explained in Theorem~\ref{thm:gmm}) that $\BT[\mathcal{G}_2,\mu_2]{\infty}\to \BT[\mathcal{G}_1,\mu_1]{\infty}$ is a formally unramified map of formally smooth formal stacks.
  
  We have a map $\tilde{R}_1\to R_1$ classifying $\mathfrak{Q}_1$ as a deformation of $\mathfrak{Q}_{1,0}$. Now, by what we have seen above, for every $x'\colon R_2\to \Reg{K'}$, the composition $\tilde{R}_1\to R_1\to R_2\to \Reg{K'}$ factors through $\tilde{R}_2$. Since $\tilde{R}_1\to \tilde{R}_2$ is surjective, this implies that in fact the composition $\tilde{R}_1\to R_1\to R_2$ factors through $\tilde{R}_2$: In other words, we have a deformation $\mathfrak{Q}_2$ over $R_2$ of $\mathfrak{Q}_{2,0}$ specializing to the various $T_{\et}^{-1}(\mathsf{Q}_2|_{x'})$. To finish, we need to know that $T_{\et}(\mathfrak{Q}_2)=\mathsf{Q}_2$ as reductions of structure to $\mathcal{G}_2(\Int_p)$-torsors within the $\mathcal{G}_1(\Int_p)$-local system $\mathsf{Q}_1\vert_{\Spec R_2[\nicefrac{1}{p}]}$. This equality can be checked over $x$, where it holds by construction.

  Let us now verify (2): Under the hypotheses here, the map $\tilde{R}_1\to R_1$ is surjective, and the map $\tilde{R}_2\to R_2$ is finite and its image has the same field of fractions as $R_2$. Since we have the equality 
  \begin{equation*}
  \dim \tilde{R}_2 = 1+\dim \mathcal{G}_2 - \dim \mathcal{P}^-_{\mu_2} = \dim R_2,
  \end{equation*}
  by hypothesis, and since $\tilde{R}_2$ is normal, we now find that $\tilde{R}_2{\isomto}R_2$. This of course shows that $R_1\to R_2$ is surjective, and identifies $R_2$ with the deformation ring of $\BT[\mathcal{G}_2,\mu_2]{\infty}$ at $\mathfrak{Q}_{2,0}$.
\end{proof}

\begin{setup}
\label{setup:reduction_of_structure_group}
  $\mathcal{X}$ will be a normal algebraic space flat of finite type over $\mathcal{O}$ with generic fiber $X$, equipped with $\mathsf{Q}\in \mathrm{Loc}_{\mathcal{G}(\Int_p)}(X)$. We will assume $\mathcal{X}$ is equidimensional with relative dimension over $\mathcal{O}$ equal to $\dim \mathcal{G} - \dim \mathcal{P}^-_\mu$ . Also, we will require that there exist a closed immersion $(\mathcal{G},\mu)\hookrightarrow(\mathcal{G}^\sharp,\mu^\sharp)$ of pairs as in Setup~\ref{setup:g_mu} with both $\mathcal{G}$ and $\mathcal{G}^\sharp$ reductive, and that there exists a commuting diagram
\begin{equation*}
         \begin{tikzcd}
  X & {\mr{Loc}_{\mc{G}(\bb{Z}_p)}=\mr{BT}^{\mc{G},\mu,\mr{alg}}_\infty\otimes\bb{Q}_p} \\
  {\mc{X}} & {\mr{BT}^{\mc{G},\mu}_\infty} \\
  {\mc{X}^\sharp} & {\mr{BT}^{\mc{G}^\sharp,\mu^\sharp,\mr{alg}}_\infty}
  \arrow[from=1-1, to=1-2]
  \arrow[from=1-1, to=2-1]
  \arrow[from=1-2, to=2-2]
  \arrow[from=2-1, to=3-1]
  \arrow[from=2-2, to=3-2]
  \arrow[from=3-1, to=3-2]
\end{tikzcd}
    \end{equation*}
        where the top horizontal arrow classifies $\mathsf{Q}$.
\end{setup}

\begin{proposition}
\label{prop:Gmu_aperture_pointwise_condition_normal}
In the situation of \emph{Setup~\ref{setup:reduction_of_structure_group}}, suppose that the following condition holds: For all closed points $x_0$ of  $\mathcal{X}$, the tuple
    \[
      \big(R_1 = \hat{\Rg}_{\mathcal{X}^\sharp,x_0},\,\, R_2 = \hat{\Rg}_{\mathcal{X},x_0},\,\, \mathfrak{Q}_1=\mathfrak{Q}^\sharp\vert_{R_1^{\mathrm{syn}}},\,\, \mathsf{Q}_2 = \mathsf{Q}_2\vert_{\Spec R_2[\nicefrac{1}{p}]}\big)
    \]
    satisfies the hypotheses of \emph{Proposition~\ref{prop:reduction_of_structure_group}}. Then $\mathsf{Q}$ lifts to $\mathfrak{Q}\in \BT[\mathcal{G},\mu,\mathrm{alg}]{\infty}(\mathcal{X})$, and the classifying map $\widehat{\mathcal{X}}\to \BT{\infty}$ is formally \'etale.
\end{proposition}

\begin{proof}
We can assume that $\mathcal{X}$ is affine. Proposition~\ref{prop:reduction_of_structure_group} tells us that the complete local ring $\hat{R}_x$ of $\mathcal{X}$ at any geometric closed point $x$ is formally smooth over $\mathcal{O}$ and that in fact $\mathsf{Q}$ lifts to a reduction of structure $\mathfrak{Q}_x\in \BT{\infty}(\hat{R}_x)$ such that the classifying map $\Spf \hat{R}_x\to \BT{\infty}$ is formally \'etale. In particular, $\widehat{\mathcal{X}}$ is smooth over $\hat{\mathcal{O}}$, and the conditions of Lemma~\ref{lem:reduction_to_complete_local} have been verified. This shows the existence of $\mathfrak{Q}$. That the resulting classifying map on $\widehat{\mathcal{X}}$ is formally \'etale now follows from Proposition~\ref{prop:reduction_of_structure_group} once again, and Lemma~\ref{lem:formal_etaleness} below.
\end{proof}

\begin{lemma}
\label{lem:formal_etaleness}
Suppose that we have a $p$-complete ring $R$ and:
\begin{itemize}
  \item A smooth $p$-adic formal algebraic space $\mathfrak{X}$ over $\Spf R$;
  \item A formally cohesive\footnote{See~\cite[Definition 6.2.1]{lurie_thesis}.} and integrable $p$-adic formal prestack $\mathfrak{Y}$ over $\Spf R$ admitting a $p$-completed cotangent complex $\mathbb{L}^{\wedge}_{\mathfrak{Y}/R}$ that is in fact a vector bundle of finite rank over $\mathfrak{Y}$; and
  \item A map $\varpi\colon \mathfrak{X}\to \mathfrak{Y}$;
\end{itemize}
with the following property: For every closed point $x\colon\Spec \kappa\to \mathfrak{X}$ with $\kappa$ algebraically closed, the map of complete local rings $\hat{\Rg}_{\mathfrak{Y},\varpi(x)}\to \hat{\Rg}_{\mathfrak{X},x}$ is an isomorphism. Then $\varpi$ is formally \'etale: that is, the formal cotangent complex $\mathbb{L}^{\wedge}_{\mathfrak{X}/\mathfrak{Y}}$ is nullhomotopic.
\end{lemma} 
\begin{proof}
Note that our hypotheses imply that $\mathfrak{Y}$ admits universal deformation rings, so the given property is sensible; see~\cite[Corollary 6.2.14]{lurie_thesis}, which is just a classical criterion of Schlessinger. Moreover, by our smoothness hypothesis, the cotangent complex $\mathbb{L}^{\wedge}_{\mathfrak{X}/R}$ is a vector bundle of finite rank over $\mathfrak{X}$.

We can of course assume that $\mathfrak{X} = \Spf S$ for some $p$-complete $R$-algebra $S$ with $\mathbb{L}_{(S/pS)/(R/pR)}$ finite locally free over $S/pS$. Via the fundamental sequence
\[
\varpi^*\mathbb{L}^{\wedge}_{\mathfrak{Y}/R}\to \mathbb{L}^{\wedge}_{\mathfrak{X}/R}\to \mathbb{L}^{\wedge}_{\mathfrak{X}/\mathfrak{Y}}
\]
of $p$-completed cotangent complexes and our hypotheses, we find that $\mathbb{L}^{\wedge}_{\mathfrak{X}/\mathfrak{Y}}$ is a perfect complex over $S$ with Tor amplitude in $[-1,0]$. Therefore, to prove the lemma, it is enough to know that, for any map $x\colon S\to \kappa$ with $\kappa$ algebraically closed and with kernel a maximal ideal, we have $x^*\mathbb{L}^{\wedge}_{\mathfrak{X}/\mathfrak{Y}}\simeq 0$ as perfect complexes over $\kappa$. 

Let $\mathcal{C}_\kappa$ be the $\infty$-category of animated Artin $W(\kappa)$-algebras with residue field $\kappa$\footnote{See~\cite[\S 6.2]{lurie_thesis}.} Then, by definition, for any  $(A,\mf{m})$ in $\mathcal{C}_\kappa$ equipped with a local map from $\widehat{\Rg}_{\mf{Y},\varpi(x)}$, we have
\[
\mathcal{M}_x(A)\defn \mathrm{fib}_x\big(\Map_{/\mf{Y}}(\Spec A,\mf{X})\to \Map_{/\mf{Y}}(\Spec \kappa,\mf{X})\big) \simeq \Map_{\widehat{\Rg}_{\mf{Y},\varpi(x)}/}(\widehat{\Rg}_{\mf{X},x},A),
\]
where the mapping space on the right is taken in the $\infty$-category of animated complete local $W(\kappa)$-algebras with residue field $\kappa$. Our hypothesis now says that the right-hand side is contractible.

On the other hand, by the definition of the cotangent complex, if we take $A = \kappa\oplus M$ to be an animated square-zero extension of $\kappa$ by a bounded above complex $M$ of $\kappa$-vector spaces, we have
\[
\mathcal{M}_x(A) \simeq \tau^{\leqslant 0}\mathrm{RHom}_\kappa(x^*\mathbb{L}^{\wedge}_{\mf{X}/\mf{Y}},M).
\]
Therefore, in total, we see that $x^*\mathbb{L}^{\wedge}_{\mf{X}/\mf{Y}}$ is nullhomotopic, as desired.
\end{proof}

\subsection{Behavior under central isogenies}
\label{sub:central_isogenies_BT}

\begin{setup}
 Suppose that we have a central extension  $\mathcal{G}\to \overline{\mathcal{G}}$ of reductive group schemes over $\Int_p$, i.e., a faithfully flat map $\mathcal{G}\to \overline{\mathcal{G}}$ of reductive $\Int_p$-group schemes with kernel $\mathcal{Z}$ central in $\mathcal{G}$. The minuscule cocharacter $\mu$ induces a minuscule cocharacter $\overline{\mu}$ of $\overline{\mathcal{G}}$.
\end{setup}

\begin{definition}
    [Apertures for $\mathcal{Z}$]
\label{def:apertures_for_z}
Even though $\mathcal{Z}$ is not necessarily smooth over $\Int_p$, for $n\geqslant 1$, one can still consider the $p$-adic formal stack $\BT[\mathcal{Z},0]{n}$ over $\Spf \Int_p$ assigning to each $p$-complete ring $R$ the $\infty$-groupoid of $\mathcal{Z}$-torsors over $R^{\mathrm{syn}}\otimes\Int/p^n\Int$ whose restriction over $B\Gm\times \Spec \kappa$ is a trivial $\mathcal{Z}$-torsor for every geometric point $\kappa$ of $\Spf R$. 
\end{definition}

\begin{remark}
    [Representability]
\label{rem:representability_z-apertures}
The formal stack $\BT[\mathcal{Z},0]{n}$ is represented by an \'etale formal algebraic stack over $\Spf \Int_p$. This is a special case of \cite[Theorems 8.7.2 and 8.8.4]{gmm}. These results need to be applied to the \emph{$1$-bounded stack} $\mathcal{X}=(\mc{X}^\diamond,X^0)$ over $\Int_p^{\mathrm{syn}}$ obtained as follows: The underlying stack $\mathcal{X}^\diamond$ is just the classifying stack $B\mathcal{Z}\times \Int_p^{\mathrm{syn}}$. Its \emph{fixed point locus} is the formal stack over $\Int_p$ parameterizing, for each $p$-complete $R$, $\mathcal{Z}$-torsors over $B\Gm\times \Spf R$---that is, of graded $\mathcal{Z}$-torsors. Within it we have the open and closed substack $X^0$ of \emph{trivially} graded $\mathcal{Z}$-torsors, which is of course isomorphic to $B\mathcal{Z}$. 

The pair $\mathcal{X} = (\mathcal{X}^\diamond,X^0)$ is the $1$-bounded stack we are concerned with. In this situation, the \emph{attracting locus} $X^-$ and the stack $X$ defined in \S8.4.3 of \emph{op.\@ cit.\@} are both also $B\mathcal{Z}$. Unwinding definitions shows that the formal stack $\Gamma_{\mathrm{syn}}(\mathcal{X}\otimes\Int/p^n\Int)$ defined in \emph{op.\@ cit.\@} is the same as $\BT[\mathcal{Z},0]{n}$. Now, Theorem 8.8.4 of \emph{op.\@ cit.\@} there tells us that $\BT[\mathcal{Z},0]{n}$ is represented by a finitely presented derived formal Artin $1$-stack over $\Spf\Int_p$, and Theorem 8.7.2 tells us that its deformation theory is governed by that of $X^- = B\mathcal{Z}$ over itself, which in turn shows that it is \'etale.
\end{remark}

\begin{remark}
    [Over perfectoid rings]
\label{rem:perfectoid_rings_z-apertures}
If $R$ is a perfectoid ring, then an object in $\BT[\mathcal{Z},0]{n}$ is the same as a $\mathcal{Z}$-torsor $\mc{D}$ (for the fppf topology) over $\Prism_R/p^n\Prism_R$ equipped with an isomorphism $\varphi^*\mc{D}\isomto\mc{D}$. 
\end{remark}

\begin{lemma}
    \label{lem:z_aperture_loc_const}
For any flat affine group scheme $H$ over $\Int/p^n\Int$, the stack $\mc{Y}$ assigning to each $\Field_p$-algebra $A$ the groupoid of pairs $(\mathcal{D},\alpha)$, where $\mathcal{D}$ is an fppf $H$-torsor over $W_n(A)$, and $\varphi^*\mathcal{D}\isomto\mathcal{D}$ is an isomorphism of $H$-torsors, is canonically equivalent to  $\underline{(BH)(\Int/p^n\Int)}$.
\end{lemma}
\begin{proof}
    This is classical: The stack $\mc{Y}$ is \'etale over $\Field_p$, and we have a natural map $\underline{(BH)(\Int/p^n\Int)}\to \mathcal{Y}$ which is an equivalence on $\overline{\Field}_p$-points by descent.
\end{proof}

\begin{proposition}
    \label{prop:apertures_for_z}
There is a canonical isomorphism from $\BT[\mathcal{Z},0]{n}$ to the $p$-adic formal stack obtained by restricting $\underline{(B\mathcal{Z})(\Int/p^n\Int)}$ to $p$-nilpotent rings.
\end{proposition}
\begin{proof}
 Since both formal stacks are \'etale over $\Spf \Int_p$, it suffices to establish an isomorphism between their restrictions to perfectoid rings. Here, this follows from Remark~\ref{rem:perfectoid_rings_z-apertures} and Lemma~\ref{lem:z_aperture_loc_const}.
\end{proof}

\begin{lemma}
    \label{lem:Z_fppf_cohomology}
There is a canonical equivalence of $\Int$-stacks
\[
\mathrm{fib}\left(\mathrm{Loc}_{\mathcal{G}(\Int/p^n\Int)}\to \mathrm{Loc}_{\overline{\mathcal{G}}(\Int/p^n\Int)}\right)\isomto \underline{(B\mathcal{Z})(\Int/p^n\Int)},
\]
where the left hand side is the stack fiber over the trivial $\overline{\mc{G}}(\Int/p^n\Int)$-torsor.
\end{lemma}
\begin{proof}
    The left hand side is the \'etale sheafification of the constant groupoid $[\overline{\mathcal{G}}(\Int/p^n\Int)/\mathcal{G}(\Int/p^n\Int)]$. Therefore, it is enough to know that there is a canonical equivalence of groupoids
    \[
   [\overline{\mathcal{G}}(\Int/p^n\Int)/\mathcal{G}(\Int/p^n\Int)]\isomto (B\mathcal{Z})(\Int/p^n\Int).
    \]
    This is obtained from assigning to every $\overline{g}\in \ov{\mc{G}}(\Int/p^n\Int)$ the $\mathcal{Z}$-torsor of lifts to $\mc{G}(\Int/p^n\Int)$, and noting that, by Lang's theorem and the smoothness of $\mathcal{G}$, every $\mathcal{Z}$-torsor over $\Int/p^n\Int$ arises in this fashion.
\end{proof}

\begin{proposition}
\label{prop:central_covers}
The natural map 
     \[
  \beta\colon \BT[\mathcal{G},\mu,\mathrm{alg}]{n}\to \BT[\overline{\mathcal{G}},\overline{\mu},\mathrm{alg}]{n}
  \]
  is \'etale and surjective. In fact, it is a torsor for $\underline{(B\mathcal{Z})(\Int/p^n\Int)}$, and so factors as a gerbe banded by $\mathcal{Z}(\Int/p^n\Int)$ followed by a torsor for the abelian group $H^1_{\mathrm{fppf}}(\Spec \Int/p^n\Int,\mathcal{Z})$.
\end{proposition}
\begin{proof}
    The deformation theory from Lemma~\ref{lem:algebraic_deformation_theory} tells us that this map is \'etale. One sees next that $\underline{(B\mathcal{Z})(\Int/p^n\Int)}$ acts on $ \BT[\mathcal{G},\mu,\mathrm{alg}]{n}$. In the generic fiber, this follows from Lemma~\ref{lem:Z_fppf_cohomology}, which actually shows that the action is simply transitive on the fibers of $\beta$. Over the $p$-completion, this can be deduced from Proposition~\ref{prop:apertures_for_z}, and the fact that $\BT[\mathcal{Z},0]{n}$ acts naturally on $\BT{n}$---in fact, simply transitively on the non-empty fibers of $\beta$. 
    
    The only thing that remains to be verified now is that $\beta$ is surjective modulo $p$. Via the quotient description in Remark~\ref{rem:quotient_description_perfectoid_truncated}, for any algebraically closed field $\kappa$ over $k$, we have the map of groupoids
    \begin{align*}
    [\mathcal{G}(W_n(k))/H^{(n)}_\mu(\kappa)]\to [\overline{\mathcal{G}}(W_n(\kappa))/\overline{H}^{(n)}_\mu(\kappa)].
    \end{align*}
   We want to show that this is surjective. This translates to knowing that, given $\overline{g}\in \ov{\mc{G}}(W(\kappa))$, we can find $\overline{h}\in\overline{H}^{(n)}_\mu(\kappa)$ such that $\tau(\overline{h})^{-1}\overline{g}\sigma(\overline{h})$ lifts to $\mathcal{G}(W_n(\kappa))$. 

    The obstruction to lifting $\overline{g}$ is witnessed by the  $\mathcal{Z}$-torsor $\mathcal{B}_{\overline{g}}$ over $W_n(\kappa)$ parameterizing lifts of $\overline{g}$ to $\mathcal{G}(W_n(\kappa))$. Now, the maps 
    \[
     \tau\colon H^{(n)}_\mu(\kappa)\to \mc{G}(W_n(\kappa))\;;\;\tau\colon\overline{H}^{(n)}_\mu(\kappa)\to \ov{\mc{G}}(W_n(\kappa))
    \]
    are obtained as projections using the identifications
    \[
    H^{(n)}_\mu(\kappa) = \mc{G}(W_n(\kappa))\times_{\mc{G}(\kappa/{}^{\mathbb{L}}p^n)}\mc{P}^-_\mu(\kappa/{}^{\mathbb{L}}p^n)\;;\;  \ov{H}^{(n)}_\mu(\kappa) = \ov{\mc{G}}(W_n(\kappa))\times_{\mc{G}(\kappa/{}^{\mathbb{L}}p^n)}\ov{\mc{P}}^-_\mu(\kappa/{}^{\mathbb{L}}p^n).
    \]
    This implies that the $\tau$ maps induce isomorphisms of groupoids
    \[
    [\overline{H}_\mu^{(n)}(\kappa)/H^{(n)}_\mu(\kappa)]\isomto[\overline{\mathcal{G}}(W_n(\kappa))/\ov{\mathcal{G}}(W_n(\kappa))]\isomto(B\mathcal{Z})(W_n(\kappa))
    \]
    Here, the second isomorphism is obtained from the map assigning to $\overline{g}\in \overline{\mathcal{G}}(W_n(\kappa))$ the torsor $\mathcal{B}_{\overline{g}}$. 
    
     By Lemma~\ref{lem:lang_surjective_stacks} below, we can write 
    \begin{align}
        \label{eqn:need_to_show_surjectivity}
     \mathcal{B}_{\overline{g}}\simeq (-1)^*\varphi^*\mc{B}\otimes\mc{B}
    \end{align}
    for some other $\mathcal{Z}$-torsor $\mc{B}$ over $W_n(\kappa)$. From the last paragraph, the $\mathcal{Z}$-torsor $\mathcal{B}$ is isomorphic to $\mathcal{B}_{\overline{h}}$ for some $\overline{h}\in \overline{H}^{(n)}_\mu(\kappa)$. The isomorphism~\eqref{eqn:need_to_show_surjectivity} shows that $\mc{B}_{\tau(\overline{h})^{-1}\overline{g}\sigma(\overline{h})}$ is thus the trivial torsor, and thus $\tau(\overline{h})^{-1}\overline{g}\sigma(\overline{h})$ lifts to $\mathcal{G}(W_n(\kappa))$, as desired.
\end{proof}

\begin{lemma} 
\label{lem:lang_surjective_stacks}
Let $Z$ be a flat commutative affine group scheme over $\mathbb{Z}_p$, and let $\kappa$ be an algebraically closed field of characteristic $p$. Then, for all $n\geqslant 1$, the Lang morphism of groupoids
\begin{equation*}
    \varphi-\mr{id}\colon BZ(W_n(\kappa))\to BZ(W_n(\kappa))
\end{equation*}
is surjective.
\end{lemma}
\begin{proof}
We proceed by induction on $n$. When $n=1$, this is clear, as $BZ(\kappa)$ is a connected groupoid. If the claim holds true for $n-1\geqslant 1$, then to see it holds for $n$ it suffices to show that the induced map
\begin{equation*}
    \varphi-\mr{id}\colon \mr{fib}\big(BZ(W_n(\kappa))\to BZ(W_{n-1}(\kappa))\big)\to \mr{fib}\big(BZ(W_n(\kappa))\to BZ(W_{n-1}(\kappa))\big)
\end{equation*}
is surjective. We now have an identification
\begin{equation*}
    \mr{fib}\big(BZ(W_n(\kappa))\to BZ(W_{n-1}(\kappa))\big)\simeq \mr{Map}_{/\kappa}(\ell_Z[-1]\otimes \kappa,\kappa),
\end{equation*}
where $\ell_Z= e^\ast \bb{L}_{Z/\bb{Z}_p}$. Here, $e$ is the identity section of $Z$. Now, $\ell_Z$ is isomorphic to $M_0\oplus M_1[1]$ for $\bb{F}_p$-spaces $M_i$. We have:
\begin{equation*}
    \mr{Map}_{/\kappa}(\ell_Z[-1]\otimes \kappa,\kappa)=\mr{Map}_{/\kappa}((M_0[-1]\otimes \kappa)\oplus (M_1\otimes \kappa),\kappa)\simeq \mr{Map}_{/\kappa}(M_1\otimes \kappa,\kappa)\simeq M_1^\vee\otimes \kappa.
\end{equation*}
The surjectivity of the Lang map now reduces to that of $\varphi-\mathrm{id}:\kappa\to \kappa$, which is clear.
\end{proof}

\begin{proposition}
[Descending along central covers]
\label{prop:full_faithfulness_2}
Suppose that $\overline{\mathcal{X}}$ is a normal flat Noetherian algebraic space over $\mathcal{O}$ with $\mathsf{Q}\in \mathrm{Loc}_{\overline{\mathcal{G}}(\Int_p)}(\overline{\mathcal{X}}_\eta)$. Suppose that there exist:
\begin{enumerate}
  \item An \'etale cover $\mathcal{X}\to \overline{\mathcal{X}}$ with a lift $\mathsf{Q}\in \mathrm{Loc}_{\mathcal{G}}(\mathcal{X}_\eta)$ of the restriction of $\overline{\mathsf{Q}}$;
  \item A lift $\mathfrak{Q}\in \BT[\mathcal{G},\mu,\mathrm{alg}]{\infty}(\mathcal{X})$ of $\mathsf{Q}$.
\end{enumerate}
Then $\overline{\mathsf{Q}}$ lifts to $\overline{\mathfrak{Q}}\in \BT[\overline{\mathcal{G}},\overline{\mu},\mathrm{alg}]{\infty}(\overline{\mathcal{X}})$, and the $p$-completed classifying map $\widehat{\overline{\mathcal{X}}}\to \BT[\overline{\mathcal{G}},\overline{\mu}]{\infty}$ for $\overline{\mathfrak{Q}}$ is formally \'etale if and only if the corresponding one $\widehat{\mathcal{X}}\to \BT{\infty}$ is so. 
\end{proposition}
\begin{proof}
      The first part of the conclusion is immediate from Corollary~\ref{cor:full_faithfulness}, while the second is a consequence of the \'etaleness assertion from Proposition~\ref{prop:central_covers}.
\end{proof}

\begin{setup}
  \label{setup:lifting_finite_central_cover}
Suppose that we have the following data:
\begin{enumerate}
  \item a flat normal algebraic space $\mathcal{X}$ over $\mathcal{O}$ with generic fiber $X$;
  \item an aperture $\overline{\mathfrak{Q}}\in\BT[\overline{\mathcal{G}},\overline{\mu},\mathrm{alg}]{\infty}(\mathcal{X})$ lifting $\overline{\mathsf{Q}}\in \mathrm{Loc}_{\overline{\mathcal{G}}(\Int_p)}(X)$;
  \item a finite Galois cover $Y\to X$ with Galois group $\Delta$ such that $\overline{\mathsf{Q}}|_Y$ lifts to $\mathsf{Q}\in  \mathrm{Loc}_{\mathcal{G}(\Int_p)}(Y)$.
\end{enumerate}
Let $\mathcal{Y}\to \mathcal{X}$ be the normalization of $\mathcal{X}$ in $Y$. 
\end{setup}

\begin{notation}\label{nota:schematic-ordinary-locus}
    Define $\mathcal{X}^{\ord}$ to be the open subspace of $\mc{X}$ with the same generic fiber but with $p$-adic completion given by the fiber product $\hat{\mathcal{X}}\times_{\BT[{\ov{\mathcal{G}}},\ov{\mu}]{\infty}}\BT[{\ov{\mathcal{G}}},\ov{\mu},\ord]{\infty}$.
\end{notation}

\begin{remark}
\label{rem:canonical_lift_X}
    Suppose that $\mathcal{X}^{\ord}\to \BT[{\ov{\mathcal{G}}},\ov{\mu},\mathrm{alg}]{\infty}$ is formally \'etale after $p$-adic completion. Then, for every perfect field $\kappa$ over $k$ and $x\in \mathcal{X}^{\ord}(\kappa)$, we have a \emph{canonical lift} $x^{\mathrm{can}}\in \mathcal{X}(W(\kappa))$ corresponding to the canonical lift $\ov{\mathfrak{Q}}_x^{\mathrm{can}}$ of the $\ov{\mu}$-ordinary aperture $\ov{\mathfrak{Q}}_x$ (Definition~\ref{defn:canonical_lift_local}).
\end{remark}

\begin{proposition}
    [Ascending along central covers]
\label{prop:ascending_along_central_covers}
 Fix an algebraic closure $\kappa$ of $k$, and suppose that the following is true:
 \begin{enumerate}
      \item $p>2$;
      \item $\Delta$ is a finite abelian group;
     \item $\mathcal{X}$ is smooth over $\mathcal{O}$;
     \item The classifying map $\mathcal{X}\to \BT[\overline{\mathcal{G}},\overline{\mu},\mathrm{alg}]{\infty}$ is formally \'etale after $p$-adic completion;
     \item Given $x\in \mathcal{X}^{\ord}(\kappa)$ with canonical lift $x^{\mathrm{can}}\in \mathcal{X}(W(\kappa))$, the space $Y_{x^{\mathrm{can}}} \defn Y\times_{\mathcal{X},x^{\mathrm{can}}}\Spec W(\kappa)$ is a split \'etale cover of $\Spec W(\kappa)[\nicefrac{1}{p}]$;
     \item For every $\tilde{y}\in Y_{x^{\mathrm{can}}}(W(\kappa)[\nicefrac{1}{p}])$, the representation $\mathsf{Q}\vert_{\tilde{y}}$ is crystalline of type $\mu$.
 \end{enumerate}
 Then:
 \begin{enumerate}
  \item $\mathcal{Y}\to \mathcal{X}$ is once again a Galois cover;
  \item $\mathsf{Q}$ lifts to $\mathfrak{Q}\in \BT[\mathcal{G},\mu,\mathrm{alg}]{\infty}(\mathcal{Y})$.
\end{enumerate}
\end{proposition}

The proof will require a bit of preparation.

\begin{definition}
   Suppose that $\mathcal{Z}\to \mathcal{W}$ is a finite map of $\mathcal{O}$-schemes with \'etale generic fiber. Given an extension of complete discrete valuation fields $\hat{\mathcal{O}}[\nicefrac{1}{p}]\subset K$ and a point $w\in \mathcal{W}(\Reg{K})$, we will say that $f$ is \defnword{strongly unramified at $w$}, or, equivalently, that $w$ is \defnword{strongly $f$-unramified}, if the finite \'etale $K$-scheme $\mathcal{Z}_\eta\times_{\mathcal{W},w}\Spec \Reg{K}$ is the spectrum of a finite product of unramified extensions of $K$.
\end{definition}

\begin{lemma}
    \label{lem:good_reduction_torsors}
Suppose that $p>2$ and that $\mathcal{W}$ is a smooth algebraic space over $\mathcal{O}$, and suppose that $f\colon V\to \mathcal{W}_\eta$ is a finite Galois cover with abelian Galois group $\Delta$ satisfying the following property: 
\begin{itemize}[leftmargin=.8cm]
    \item There is a dense open subspace $U$ of the special fiber $\mathcal{W}_k$ such that, for every $w\in U(\kappa)$, there exists a strongly $f$-unramified lift $\tilde{w}\in \mathcal{W}(W(\kappa))$ of $w$.
\end{itemize}
Then $V$ extends to a Galois cover $\mathcal{V}\to \mathcal{W}$.
\end{lemma}
\begin{proof}
    Let $F\colon \mathcal{V}\to \mathcal{W}$ be the normalization of $\mathcal{W}$ in $V$. We want to show that $F$ is finite \'etale. Since $\mathcal{O}\to \hat{\mathcal{O}}$ is faithfully flat, $F$ is finite \'etale if and only if its base change over $\hat{\mathcal{O}}$ is so. Moreover,  $\mathcal{V}\otimes_{\mathcal{O}} \hat{\mathcal{O}}$ is the normalization of $\mathcal{W}\otimes_{\mathcal{O}}\hat{\mathcal{O}}$ in $V\otimes_{\mathcal{O}[\nicefrac{1}{p}]}\hat{\mathcal{O}}[\nicefrac{1}{p}]$ as $\hat{\mathcal{O}}$ is a filtered colimit of smooth $\mathcal{O}$-algebras \stacks{07GB}. Therefore, without loss of generality, we can assume that $\mathcal{O} = \hat{\mathcal{O}} = W(\kappa)$ is $p$-adically complete and that $k = \kappa$ is algebraically closed.

    By d\'evissage, we can reduce to the case where $\Delta = \Int/\ell \Int$ for some prime $\ell$: This reduction only needs the observation that, if we have a factorization $\mathcal{V}\to \mathcal{W}'\to \mathcal{W}$ with the second map finite \'etale, then the preimage of $U$ in $\mathcal{W}'$ satisfies the same properties as $U$ but now with respect to $\mathcal{V}\to \mathcal{W}'$.
    
    Given our first hypothesis, it follows from Lemma~\ref{lem:generic_reducedness} below that, for every generic point $\eta$ of $\mathcal{V}_\kappa$, the local ring $\Reg{\mathcal{V},\eta}$ is a discrete valuation ring with maximal ideal generated by $p$. 
    
    By purity of the branch locus~\cite[Corollaire 3.3]{sga1}, it suffices to know that, for any such $\eta$, lying above a generic point $\eta'$ of $\mathcal{W}_\kappa$, the map $\Reg{\mathcal{W},\eta'}\to \Reg{\mathcal{V},\eta}$ of discrete valuation rings, both admitting $p$ as a uniformizer, is \'etale. If the finite extension of residue fields $k(\eta)/k(\eta')$ is separable, and in particular, if $\ell\neq p$, then this is immediate. 
    
    Therefore, we can assume that $\ell = p$ and that the extension of residue fields is purely inseparable of degree $p$. In this case, any generator of $\Delta$ will act as the identity on the residue field $k(\eta)$. Since $p>2$, this is impossible: See for instance the proof of~\cite[Proposition 3.22]{moonen:models}.
    \end{proof}

\begin{lemma}
  \label{lem:generic_reducedness}
Let $R$ be a complete discrete valuation ring with uniformizer $\pi$ and algebraically closed residue field $\kappa$. Suppose that $\mathcal{V}$ is an integral scheme, flat and of finite type over $R$ such that, for every generic point $\eta$ of the special fiber $\mathcal{V}_\kappa$, the local ring $\Reg{\mathcal{V},\eta}$ is a discrete valuation ring and
\[
 \mathrm{im}(\mathcal{V}(R)\to \mathcal{V}(\kappa)) 
\]
is set-theoretically Zariski dense in $\mathcal{V}_\kappa$. Then, $\mathcal{V}_\kappa$ is generically reduced, i.e., for any generic point $\eta$ for $\mathcal{V}_\kappa$, one has that $\pi$ is a uniformizer in $\Reg{\mathcal{V},\eta}$.
\end{lemma}
\begin{proof}
  This is essentially~\cite[Lemma 6.6]{guo2024pointwisec}, but, since the hypotheses there are a bit stronger, we recall the proof here. Let $\varpi\in \Reg{\mathcal{V},\eta}$ be a uniformizer, and suppose that we have $\pi = u\varpi^m$ for some $u\in \Reg{\mathcal{V},\eta}^\times$. By localizing on $\mathcal{V}$, we can assume that $\mathcal{V} = \Spec A$ is affine and that we have $\varpi\in A$ and $u\in A^\times$. Furthermore, by our density hypothesis, we can also assume that we have a section $x\colon A\to R$ of the $R$-algebra $A$. This shows $\pi = x(u)x(\varpi)^m\in R$. Clearly, this is possible only if $m = 1$.
\end{proof}

\begin{lemma}
  \label{lem:gerbe_extending}
Suppose that $\mathcal{W}$ is a smooth algebraic space over $\mathcal{O}$ and that $\tilde{\mathcal{W}}\to \mathcal{W}$ is a gerbe banded by a finite abelian group $\Delta$. Suppose there is a section $s\colon \mathcal{W}_\eta\to \tilde{\mathcal{W}}$ with the following property:
\begin{itemize}[leftmargin=.8cm]
    \item There is a dense open subspace $U$ of the special fiber $\mathcal{W}_k$ such that, for every $w\in U(\kappa)$, there exists a lift $\tilde{w}\in \mathcal{W}(W(\kappa))$ such that the section
    \[
    \Spec W(\kappa)[\nicefrac{1}{p}]\xrightarrow{\tilde{w}_\eta}\mathcal{W}_\eta\to \tilde{\mathcal{W}}
    \]
    extends over $\Spec W(\kappa)$.
\end{itemize}
Then $s$ extends to a section $\mathcal{W}\to \tilde{\mathcal{W}}$.
\end{lemma}
\begin{proof}
  To begin, if we have two extensions of $s$ over $\mathcal{W}$, then the space of isomorphisms between them is a $\Delta$-torsor over $\mathcal{W}$, admitting a canonical trivialization over $\mc{W}_\eta$, and so is itself canonically trivial (see for instance \stacks{0BQG}). In particular, to prove the existence of such an extension, we can work \'etale locally and assume that the gerbe is trivial. The section $s$ now corresponds to a $\Delta$-torsor $V\to \mathcal{W}_\eta$, and the extension property amounts to the concrete assertion that the normalization of $\mathcal{W}$ in $V$ is a finite Galois cover $\mathcal{V}\to \mathcal{W}$. Given our hypothesis, this now follows from Lemma~\ref{lem:good_reduction_torsors}.
\end{proof}

\begin{proof}
[Proof of Proposition~\ref{prop:ascending_along_central_covers}]
    Hypothesis (4) tells us that $\mathcal{X}^{\mathrm{ord}}\subset \mathcal{X}$ is open and fiberwise dense, and combining this with hypotheses (2), (5), and Lemma~\ref{lem:good_reduction_torsors} gives us assertion (1): the normalization $\mathcal{Y}\to \mathcal{X}$ is finite Galois. 
    
    Now, for assertion (2). For every $n\geqslant 1$, Proposition~\ref{prop:central_covers} tells us that lifting from $\BT[\overline{\mathcal{G}},\overline{\mu},\mathrm{alg}]{n}$ to $\BT[\mathcal{G},\mu,\mathrm{alg}]{n}$ amounts to trivializing a torsor for the finite group $H^1_{\mathrm{fppf}}(\Spec \Int/p^n\Int,\mathcal{Z})$ followed by trivializing a gerbe banded by $\mathcal{Z}(\Int/p^n\Int)$.  Now, we have such a lift over $Y$, corresponding to such trivializations over it, compatibly for all $n$. The trivialization of the torsor part automatically extends over $\mathcal{Y}$; see \stacks{0BQG}. By hypothesis (6) and Proposition~\ref{prop:gmu_rings_of_integers}, combined with hypothesis (6) and Lemma~\ref{lem:gerbe_extending}, we find that the trivialization of the gerbe part also extends over $\mathcal{Y}$. Taking the limit over $n$ now yields the lift in assertion (2).
\end{proof}

\section{Integral canonical models: Definitions and first properties}\label{sec:canonical}

In this section we apply the material established above to define and study the basic properties of integral canonical models. Notably, we give a proof of Theorem \ref{introthm:mapping_property}.

We will be following the conventions for Shimura varieties as laid out in~\cite{deligne:corvallis}.

\begin{setup}
  Given a Shimura datum $(G,X)$, for $x\in X$, the conjugacy class $\{\mu_x\}$ of the associated Shimura cocharacter $\mu_x$\footnote{For any representation $W$ of $G_{\Real}$, $x$ yields a Hodge structure $W_x$ on $W$ and $\mu_x$ acts on $W^{p,q}_x$ by $z\mapsto z^{-p}$.} is defined over the reflex field $E\subset \Comp$ and is independent of the choice of $x$. We will denote this $E$-rational class by $\{\mu\}$. For a compact open subgroup $K\subset G(\Adele_f)$, let $\Sh_K$ be the canonical model over $E$ for the complex Shimura variety with complex points given by the orbifold
\[
  \Sh_K(\Comp) = G(\Rat)\backslash X\times G(\Adele_f)/K.
\]
Note that $\Sh_K$ is a Deligne--Mumford stack over $E$ that is a quasi-projective scheme when $K$ is \emph{neat} in the sense of~\cite[\S 0.6]{pink:thesis}, and that in general $\Sh_K$ is the stack quotient of $\Sh_{K'}$ by the finite group $K/K'$, for any choice of neat compact open subgroup $K'\subset K$.   
\end{setup}

\subsection{\'Etale realizations on Shimura varieties}
\label{subsec:etale_realization_global}

\begin{definition}
[{\cite[Definition 1.5.4 and Lemma 1.5.5]{kisin2021stf}}]
A torus $T$ over $\Rat$ is \defnword{cuspidal} if it is isogenous to a product of a $\Rat$-split torus and a $\Rat$-torus that is anisotropic over $\Real$. For a general $\Rat$-torus $T$, the \defnword{anti-cuspidal part} $T_{ac}\subset T$ is the minimal $\Rat$-subtorus such that $T/T_{ac}$ is cuspidal.
\end{definition}  

\begin{remark}
    \label{rem:cuspidal_preserved_under_subs}
If $T$ is a cuspidal torus over $\Rat$, then any subtorus is also cuspidal. This follows for instance from characterization (v) in~\cite[Lemma 1.5.5]{kisin2021stf}.
\end{remark}

\begin{remark}
\label{rem:anti-cuspidal}
The anti-cuspidal part is also characterized as the smallest subtorus of the maximal anisotropic subtorus $T_a\subset T$ that contains the maximal $\Real$-split subtorus of $T_{a,\Real}$: in other words, it is the $\Rat$-Zariski closure in $T$ of the maximal $\Real$-split subtorus of $T_{a,\Real}$.
\end{remark}

\begin{definition}
[{\cite[\S 1.5.6]{kisin2021stf}}]
We define $G^c = G/Z(G)_{ac}$, where $Z(G)\subset G$ is the center, and where $Z(G)_{ac}$ is the anti-cuspidal part of the connected component $Z(G)^\circ\subset Z(G)$: we refer to it as the \defnword{cuspidal quotient} of $G$. 
\end{definition}

\begin{remark}
As noted in the footnote on p.\@ 34 of \emph{op.\@ cit.\@}, this definition differs from the one used for instance by Lovering in~\cite{Lovering2017-me}, since we are not assuming Milne's axiom requiring that $Z(G)^{\circ}$ split over a CM field. \emph{If} we made this assumption, then $Z(G)_{ac}$ is the maximal anisotropic subtorus of $Z(G)$ that is $\Real$-split.
\end{remark}

\begin{lemma}
\label{lem:GQK}
For a compact open subgroup $K\subseteq G(\bb{A}_f)$, set $Z(G)_K \defn Z(G)(\Rat)\cap K$. Then:
\begin{enumerate}
  \item The map of complex analytic spaces $X\times G(\Adele_f)\to \Sh_K(\Comp)$ is Galois with Galois group $(G(\Rat)\times K)/Z(G)_K$ where $Z(G)_K$ is embedded diagonally.
  \item For any normal subgroup $K'\subset K$, $\Sh_{K'}\to \Sh_K$ is a finite Galois cover with Galois group $K/K'Z(G)_K$.
  \item For $K$ \emph{neat}, we have $Z(G)_K\subset Z(G)_{ac}(\Rat)$; so, the map $G(\Rat)\to G^c(\Rat)$ factors through $Z(G)_K\backslash G(\Rat)$, and the map $K\to G^c(\Adele_f)$ factors through $K/Z(G)^-_K$, where $Z(G)^-_{K}\subset K$ is the closure of $Z(G)_{K}$ in $K$.
\end{enumerate}
\end{lemma}
\begin{proof}
Only the last assertion needs proof. This follows from~\cite[Lemma 1.5.7]{kisin2021stf}.
\end{proof}

\begin{construction}
[\'Etale realization]
  \label{const:Kc_etale_realization}
By assertion (2) of Lemma~\ref{lem:GQK}, we have a cofiltered system 
\[
\{\Sh_{K'}\to \Sh_K\}_{K'\triangleleft K}
\]
of Galois covers, where each cover has Galois group $K/K'Z(G)_K$. The inverse limit of this system is a pro-Galois cover of $\Sh_K$ with Galois group $K^c \defn K/Z(G)_{K}^-$, which by assertion (3) of the same lemma admits an injective map to $G^c(\Adele_f)$. In fact, since the kernel of $G\to G^c$ is connected by definition, $K^c$ maps isomorphically onto a compact open subgroup of $G^c(\Adele_f)$, which we denote by the same symbol. We will denote the canonical pro-\'etale $K^c$-torsor over $\Sh_K$ by $\mathbf{Et}_{K}$. For every prime $\ell$, if $K^c_{\ell}$ is the projection of $K^c$ into $G^c(\Rat_\ell)$, then we will write $\mathbf{Et}_{K,\ell}$ for the $K^c_\ell$-local system obtained from $\mathbf{Et}_{K}$.
\end{construction}

\begin{lemma}
\label{lem:Gc_functoriality}
If $f\colon (G_1,X_1)\to (G_2,X_2)$ is a map of Shimura data, then $f(Z(G_1)_{ac})\subset Z(G_2)_{ac}$, and thus $f$ induces a map $f^c\colon G_1^c\to G_2^c$. If $f$ is a closed immersion, then so is $f^c$.
\end{lemma} 
\begin{proof}
For the first assertion, see~\cite[Lemma 2.5]{imai2023prismatic}. For the second, using $f$ to identify $G_1$ with a subgroup of $G_2$, we need to verify that $Z(G_1)_{ac} = G_1\cap Z(G_2)_{ac}$. This is equivalent to knowing that $Z(G_1)^{\circ}/(G_1\cap Z(G_2)_{ac})$ is cuspidal. But this torus is a subtorus of the cuspidal torus $Z(G_2)^{\circ}/Z(G_2)_{ac}$, and so we conclude by Remark~\ref{rem:cuspidal_preserved_under_subs}.
\end{proof}

\begin{remark}
    [Functoriality of \'etale realizations]
\label{rem:functoriality_etale_realizations}
In the situation of the lemma above, suppose that we have compact open subgroups $K_i \subset G_i(\Adele_f)$ and $K_i^c\subset G_i^c(\Adele_f)$ for $i=1,2$ with $f(K_1)\subset K_2$, $f^c(K_1^c)\subset f^c(K_2^c)$, and with $K_i$ mapping into $K_i^c$ under $G_i(\Adele_f)\to G_i^c(\Adele_f)$. Then by the theory of canonical models, we obtain a map
\[
\Sh_{K_1}\to \Sh_{K_2}\otimes_{E_2}E_1
\]
of stacks over $E_1$. Here, $E_i$ is the reflex field for $(G_i,X_i)$. Then the construction in Construction~\ref{const:Kc_etale_realization}, combined with Lemma~\ref{lem:Gc_functoriality}, shows that there is a a canonical $K^c_{1}$-equivariant map
\[
\mathbf{Et}_{K_1}\to \mathbf{Et}_{K_2}\vert_{\Sh_{K_1}}.
\]
\end{remark}

\subsection{The de Rham realization on Shimura varieties}
\label{sub:the_de_rham_realization_on_shimura_varieties}

We recall here some properties of the de Rham realization on Shimura varieties. We will maintain the notation and hypotheses stemming from Definition~\ref{defn:unramified_tuples}.

\begin{construction}
  By definition of the reflex field, the geometric conjugacy class $\{\mu\}$ given by
  \[
   \mu_x\colon\Gmh{\Comp}\xrightarrow{z\mapsto (z,1)}\Gmh{\Comp}\times \Gmh{\Comp} \xleftarrow[\sim]{(az,a\overline{z})\mapsfrom a\otimes z}\mathbb{S}_{\Comp}\xrightarrow{h_x}G_{\Comp}
  \]
  is defined over $E\subset \Comp$. This implies the following: There exists a canonical projective homogeneous variety $\mathrm{Gr}_{\{\mbox{-}\mu\}}$ over $E$ parameterizing filtrations on representations of $G$\footnote{More precisely, filtrations on the canonical fiber functor $\Rep_E(G)\to \mathrm{Vect}_E$.} that are \'etale locally split by a cocharacter of $G$ in the geometric conjugacy class $\{\mbox{-}\mu\}$.
 \end{construction}

\begin{remark}
    [The de Rham realization]
\label{rem:de_rham_realization_generic_fiber}
As explained in~\cite[\S 5.2]{Diao2022-yn}, there is a canonical functor from $\Rep_{\Rat}(G^c)$ to the category of variations of $\Rat$-Hodge structures over the complex analytic space $\Sh_{K,\Comp}^{\mathrm{an}}$: Given $V\in \Rep_{\Rat}(G^c)$, the fiber of the associated variation of $\Rat$-Hodge structure at any lift $(x,g)\in X\times G(\Adele_f)$ is canonically isomorphic to the Hodge structure on $V$ arising from the Deligne cocharacter $h_x\colon\mathbb{S}\to G_{\Real}$ at $x$.\footnote{We follow Deligne's convention from~\cite{deligne:corvallis} that $h_x(z)$ acts on the $(p,q)$-part of the Hodge decomposition via $z^{-p}\overline{z}^{-q}$.} 

Moreover, the underlying filtered vector bundle with integrable connection for each object in the image of the functor is algebraic. That is, with every $V\in \Rep_{\Rat}(G^c)$, we can associate a filtered vector bundle $\Fil^\bullet_{\mathrm{Hdg}}\mathbf{dR}_{K,\Comp}(V)$ over $\Sh_{K,\Comp}$, equipped with an integrable connection satisfying Griffiths transversality. This functorial association yields in particular a map $\Sh_{K,\Comp}\to \mathrm{Gr}_{\mbox{-}\mu,\Comp}/G_{\Comp}$ classifying a filtered $G^c$-torsor $\Fil^\bullet_{\mathrm{Hdg}}\mathbf{dR}_{K,\Comp}$ over $\Sh_{K,\Comp}$, and this filtered $G^c$-torsor is equipped with an integrable connection satisfying Griffiths transversality; cf.\@ Remark~\ref{rem:versality_bt_n}. 

By~\cite[Theorems 4.3 and 5.1]{milne:canonical}, this descends canonically to a filtered $G^c$-torsor $\Fil^\bullet_{\mathrm{Hdg}}\mathbf{dR}_{K,\Rat}$ over $\Sh_K$, and the integrable connection also descends. Therefore, for any $V\in \Rep_{\Rat}(G^c)$, we obtain a canonical filtered vector bundle $\Fil^\bullet_{\mathrm{Hdg}}\mathbf{dR}_K(V)$ over $\Sh_K$ with an integrable connection satisfying Griffiths transversality.
\end{remark}

\begin{remark}
    [The Kodaira--Spencer map]
\label{rem:ks_map_rational}
Let $\mathbb{T}_{\Sh_K/E}$ be the tangent bundle for $\Sh_K$. By Griffiths transversality, for any $V\in \Rep_{\Rat}(G^c)$, we obtain a map of vector bundles
\[
\displaystyle \mathbb{T}_{\Sh_K/E}\to \bigoplus_i \underline{\Hom}\left(\gr^i_{\mathrm{Hdg}}\mathbf{dR}_K(V),\gr^{i-1}_{\mathrm{Hdg}}\mathbf{dR}_K(V)\right).
\]
This arises from a single map $\mathbb{T}_{\Sh_K/E}\to \mathbf{dR}_K(\mathfrak{g}_{\Rat})/\Fil^0\mathbf{dR}_K(\mathfrak{g}_{\Rat})$, where $\mathfrak{g}$ is the Lie algebra of $\mathcal{G}_{(p)}$ equipped with the adjoint action, and we refer to this latter map as the \defnword{Kodaira--Spencer map} for $\Sh_K$. By the construction of the Shimura variety, this map is an isomorphism. See~\cite[Proposition 1.1.14]{deligne:corvallis} for a discussion of this over $\Comp$.
\end{remark}

\begin{remark}
  [de Rham property of the $p$-adic \'etale realization]
\label{rem:lz}
Fix a place $v\vert p$ of $E$, and let $\Sh_{K,E_v}^{\mathrm{an}}$ be the rigid analytification of $\Sh_K$ over $E_v$. Any $V\in \Rep_{\Rat}(G^c)$, viewed as a $\Rat_p$-representation of $G$, can be twisted by $\mathbf{Et}_{K,p}$ to obtain a $\Rat_p$-local system $\mathbf{Et}_{K,p}(V)$ over $\Sh_K$. In~\cite{Liu2017-yz}, Liu and Zhu showed that the restriction over $\Sh_{K,E_v}^{\mathrm{an}}$ of $\mathbf{Et}_{K,p}(V)$ is \emph{de Rham}. More precisely, they showed that there is a tensor-functorially associated filtered vector bundle $\Fil^\bullet_{\mathrm{Hdg}}\mathbf{dR}^{\mathrm{LZ},\mathrm{an}}_{K,v}(V)$ over $\Sh_{K,E_v}^{\mathrm{an}}$ equipped with an integrable connection satisfying Griffiths transversality, such that $\mathbf{Et}_{K,p}(V)$ is \emph{associated} with it in the sense of ~\cite[Definition 8.3]{Scholze2013-nm}. In particular, for every classical point $x$ of $\Sh_{K,E_v}^{\mathrm{an}}$, $\mathbf{Et}_{K,p,x}(V)$ corresponds to a de Rham Galois representation, and the fiber of $\Fil^\bullet_{\mathrm{Hdg}}\mathbf{dR}^{\mathrm{LZ},\mathrm{an}}_{K,v}(V)$ at $x$ is canonically isomorphic to the image of $\mathbf{Et}_{K,p,x}(V)$ under Fontaine's $D_{\dR}$ functor.
\end{remark}

\begin{remark}
    [Work of Diao-Lan-Liu--Zhu]
\label{rem:dllz}
In~\cite{Diao2022-yn}, we find the following improvements to \cite{Liu2017-yz}:
\begin{itemize}
    \item Using toroidal compactifications, it is shown that $\Fil^\bullet_{\mathrm{Hdg}}\mathbf{dR}^{\mathrm{LZ},\mathrm{an}}_{K,v}(V)$ admits a canonical algebraization $\Fil^\bullet_{\mathrm{Hdg}}\mathbf{dR}^{\mathrm{LZ}}_{K,v}(V)$ over $\Sh_{K,E_v}$ compatible with tensor structures; see pp.\@ 536--537 of \emph{op.\@ cit.}
    \item In Theorem 5.3.1 of \emph{op.\@ cit.\@} that the tensor functors
    \begin{equation*}
        V \mapsto \Fil^\bullet_{\mathrm{Hdg}}\mathbf{dR}^{\mathrm{LZ}}_{K,v}(V),\qquad 
        V \mapsto \Fil^\bullet_{\mathrm{Hdg}}\mathbf{dR}_{K}(V)\vert_{\Sh_{K,E_v}},
    \end{equation*}
    are canonically isomorphic. Here, we view both as functors from $\Rep_{\Rat}(G^c)$ to filtered vector bundles over $\Sh_{K,E_v}$ with integrable connection.
\end{itemize}
\end{remark}

\subsection{Unramified tuples}
\label{sub:unramified_tuples}

From here on, we will fix a prime $p$.

\begin{definition}
  [Unramified tuples]
\label{defn:unramified_tuples}
A \defnword{$p$-unramified Shimura tuple} or simply \defnword{unramified tuple} (since $p$ will remain fixed), is a tuple $(G,\mathcal{G},X,K)$ where $(G,X)$ is a Shimura datum, $\mathcal{G}$ is a reductive model for $G$ over $\Int_p$, and $K \subset G(\Adele_f)$ is a compact open subgroup of the form $K^pK_p$ where $K_p = \mathcal{G}(\Int_p)$. Such a tuple is \defnword{neat} if $K$ is.
\end{definition}

\begin{setup}
  Fix an unramified tuple $(G,\mathcal{G},X,K)$. This yields a reductive model $\mathcal{G}_{(p)}$ for $G$ over $\Int_{(p)}$.  We will write $\mathcal{G}^c_{(p)}$ for the induced reductive model for $G^c$ and $\mathcal{G}^c$ for its base-change over $\Int_p$. Concretely, $\mathcal{G}^c_{(p)}$ is obtained by taking the quotient of $\mathcal{G}_{(p)}$ by the Zariski closure of $Z(G)_{ac}$.
\end{setup}

\begin{remark}
   If $(G,\mathcal{G},X,K)$ is an unramified tuple, the compact open $K^c\subset G^c(\Adele_f)$ as in Construction~\ref{const:Kc_etale_realization} will always be of the form $K^c = K^c_pK^{c,p}$ where $K^c_p\defn \mathcal{G}^c(\Int_p)\subset G^c(\Rat_p)$.
\end{remark}

\begin{definition}
  [Maps between unramified tuples]
\label{defn:maps_between_unramified_tuples}
A map $(G_1,\mathcal{G}_1,X_1,K_1)\to (G_2,\mathcal{G}_2,X_2,K_2)$ between unramified tuples is a map of $\Rat$-groups $f\colon G_1\to G_2$ and an element $g\in G_2(\Adele_f^p)$ with the following properties:
\begin{itemize}
  \item $f_{\Rat_p}$ extends to a map of $\Int_p$-group schemes $\mathcal{G}_1\to \mathcal{G}_2$;
  \item $f$ induces a map of Shimura data $(G_1,X_1)\to (G_2,X_2)$;
  \item we have $gf(K^p_1)g^{-1}\subset K_2^p$.
\end{itemize}
\end{definition}

\begin{remark}
\label{rem:maps_arising_from_tuples}
   Consider a map $(G_1,\mathcal{G}_1,X_1,K_1)\to (G_2,\mathcal{G}_2,X_2,K_2)$ of unramified tuples, and write $E_2\subset E_1\subset \Comp$ for the reflex fields of $(G_2,X_2)$ and $(G_1,X_1)$, respectively. There is a unique map of Shimura varieties 
\[
\Sh_{K_1}\to\Sh_{K_2}\otimes_{E_2}E_1
\]
whose evaluation on $\Comp$-points is given by
\[
   G_1(\Rat)\backslash X_1\times G_1(\Adele_f)/K_1 \xrightarrow{[(x,h)]\mapsto [(f(x),f(h)g^{-1})]}G_2(\Rat)\backslash X_2\times G_2(\Adele_f)/K_2.
\]
It now follows from Remark~\ref{rem:functoriality_etale_realizations} that we have a canonical $K^c_{1}$-equivariant map
\[
\mathbf{Et}_{K_1}\to \mathbf{Et}_{K_2}\vert_{\Sh_{K_1}}.
\]
This requires the additional observation that the map of Shimura varieties factors as
\[
\Sh_{K_1}\to\Sh_{g^{-1}K_2g}\otimes_{E_2}E_1\xrightarrow[\sim]{\eta_g}\Sh_{K_2}\otimes_{E_2}E_1
\]
where the second map is the isomorphism of Shimura varieties whose evaluation on $\Comp$-points is 
\[
   G_2(\Rat)\backslash X_1\times G_1(\Adele_f)/g^{-1}K_2g \xrightarrow{[(x,h)]\mapsto [(x,hg^{-1})]}G_2(\Rat)\backslash X_2\times G_2(\Adele_f)/K_2.
\]
Moreover, we have a canonical isomorphism of $K^c_{2}\simeq g^{-1}K^c_{2}g$-local systems
\[
\eta_g^*\mathbf{Et}_{K_2}\isomto \mathbf{Et}_{g^{-1}K_2g}
\]
over $\Sh_{g^{-1}K_2g}$.
\end{remark}

\begin{remark}
    \label{rem:unramified_uniformization}
Given an unramified tuple $(G,\mathcal{G},X,K)$ as above and a connected component $X^+\subset X$, let us write $\mathcal{G}_{(p)}(\Int_{(p)})_+\subset\mathcal{G}_{(p)}(\Int_p)$ for the stabilizer of $X^+$. Then we obtain an isomorphism
\[
\mathcal{G}_{(p)}(\Int_{(p)})_+\backslash X^+\times G(\Adele_f^p)/K^p \isomto\Sh_K(\Comp).
\]
This comes down to knowing the following:
\begin{itemize}
    \item $G(\Rat)$ is dense in $G(\Real)$, which is a consequence of real approximation; see~\cite[Theorem 7.7]{Platonov1994-ib};
    \item $G(\Rat)_+K_p = G(\Rat_p)$, which is a consequence of weak approximation; see~\cite[Lemma (2.2.6)]{kisin:abelian}.
\end{itemize}
\end{remark}

\begin{remark}
[Prime-to-$p$ Hecke correspondences and connected components]
\label{rem:prime-to-p_hecke}
There is an action of $G(\Adele_f^p)$ on the underlying cofiltered system of $E$-schemes $\{\Sh_{K^{',p}K_p}\}$: Given $g\in G(\Adele_f^p)$, we obtain a canonical isomorphism $\Sh_{K^{',p}K_p}\isomto\Sh_{gK^{',p}g^{-1}K_p}$ induced over $\Comp$ by the endomorphism  of $X\times G(\Adele^p_f)$ given by $(x,h)\mapsto (x,hg^{-1})$. This gives an action of $G(\Adele_f^p)$ on the inverse limit $\Sh_{K_p}$ of this system.  By~\cite[{}2.1.3]{deligne:corvallis} and~\cite[Lemma (2.2.5)]{kisin:abelian}, the induced action on the connected components of $\Sh_{K_p,\overline{\Rat}}$ factors through a simply transitive action of the quotient $G(\Adele_f^p)/(\bigcap_{K^{',p}}\mathcal{G}_{(p)}(\Int_{(p)})_+K^{',p})$, which is in fact an abelian group. 
\end{remark}

\subsection{Consequences of having a universal aperture}
\label{sub:consequences_of_having_a_universal_aperture}

\begin{setup}
Fix an unramified tuple $(G,\mathcal{G},X,K)$, and a place $v\vert p$ of the reflex field $E$. Set $\hat{\mathcal{O}} = \Reg{E_v}$, and let $k$ be the residue field. Viewing $\{\mu\}$ as an $E_v$-rational conjugacy class, then we in fact have a representative $\mu_v\colon\Gmh{\mathcal{O}}\to \mathcal{G}_{\hat{\mathcal{O}}}$ for $\{\mu\}$: The existence of a representative over $E_v$ follows from the quasi-splitness of $G_{E_v}$~\cite[Lemma 1.1.3]{kottwitz:twisted}, while that over $\hat{\mathcal{O}}$ follows from~\cite[Proposition 1.1.4]{kisin:abelian}. We will view $\mu_v$ as a cocharacter of $\mathcal{G}^c_{\hat{\mathcal{O}}}$. Recall from \S\ref{subsec:tate} the pro-algebraic stack $\BT[\mathcal{G}^c,\mbox{-}\mu_v,\mathrm{alg}]{\infty}$ over $\Reg{E,(v)}$ with generic fiber $\mathrm{Loc}_{\mathcal{G}^c(\Int_p)}$ and $v$-adic completion $\BT[\mathcal{G}^c,\mbox{-}\mu_v]{\infty}$.
\end{setup}

\begin{definition}
\label{defn:ss_K_normal_model}
   Suppose that we have a map of schemes $U\to \Sh_K$, and suppose that $\mathcal{U}$ is a flat integral model for $U$ over $\Reg{E,(v)}$: By this, we mean that $\mathcal{U}$ is a flat separated algebraic space over $\Reg{E,(v)}$ equipped with an identification $\mathcal{U}\otimes_{\Reg{E,(v)}}E\isomto U$. We will say that $\mathcal{U}$ is \defnword{apertile} if it is equipped with an extension $\mathcal{U}\to \BT[\mathcal{G}^c,\mbox{-}\mu_v,\mathrm{alg}]{\infty}$ of the classifying map $U\to \mathrm{Loc}_{\mathcal{G}^c(\Int_p)}$ for $\mathbf{Et}_{K,p}\vert_U$.
\end{definition}

\begin{remark}
  \label{rem:apertile}
 By Theorem~\ref{thm:full_faithfulness}, such an extension is unique up to unique isomorphism if $\mathcal{U}$ is $\eta$-normal. In particular, in this case, being apertile is a \emph{property} of the integral model $\mathcal{U}$.
\end{remark}

\begin{remark}
[Type of $\mathbf{Et}_{K,p}$]
\label{rem:pointwise_lifts}
The existence of the aperture implies in particular that the restriction of $\mathbf{Et}_{K,p}$ over $\hat{\mathcal{U}}_{\eta}$ is a crystalline $\mathcal{G}^c(\Int_p)$-local system of type $\mbox{-}\mu_v$ at every classical point. Conversely, if we knew \emph{a priori} that $\mathbf{Et}_{K,p}$ is crystalline at every classical point of $\hat{\mathcal{U}}_{\eta}$, then, by Remarks~\ref{rem:lz} and~\ref{rem:dllz}, we see that its type at any such point has to be $\mbox{-}\mu_v$. 
\end{remark}

\begin{construction}
  [The integral de Rham realization]
\label{const:integral_de_rham}
Consider the composition
\begin{equation}\label{eqn:formal_integral_de_rham_map}
  \widehat{\mathcal{U}}\to \BT[\mathcal{G}^c,\mbox{-}\mu_v]{\infty}\to \mathrm{Gr}_{\mbox{-}\mu_v}/\mathcal{G}^c_{\Reg{E_v}}
\end{equation}
where the second map is as in Remark~\ref{rem:versality_bt_n}. This map classifies a filtered $\mathcal{G}^c$-bundle over $\widehat{\mathcal{U}}$, whose restriction over $\widehat{\mathcal{U}}_{\eta}$ is canonically isomorphic to the restriction of $\Fil^\bullet_{\mathrm{Hdg}}\mathbf{dR}_{K,\Rat}$: One checks this using Remarks~\ref{rem:dllz} and~\ref{rem:filtered_de_rham_comparison}. This tells us that there exists a map
\[
\mathcal{U}\to \mathrm{Gr}_{\mbox{-}\mu_v}/\mathcal{G}^c_{\Reg{E,(v)}}
\]
whose formal completion is ~\eqref{eqn:formal_integral_de_rham_map} and whose generic fiber classifies $\Fil^\bullet_{\mathrm{Hdg}}\mathbf{dR}_{K,\Rat}\vert_U$. Denote the associated filtered $\mathcal{G}^c$-bundle over $\mathcal{U}$ by $\Fil^\bullet_{\mathrm{Hdg}}\mathbf{dR}_K$.
\end{construction}

\begin{remark}
[The integral Kodaira--Spencer map]
  \label{rem:Kodaira--Spencer_algebraization}
Suppose that $\mathcal{U}$ is \emph{smooth} over $\Reg{E,(v)}$. Then the filtered $\mathcal{G}$-bundle from the previous remark gives us a Kodaira--Spencer map  
\begin{align}
\label{eqn:integral_Kodaira--Spencer_map}
\mathbb{T}_{\mathcal{U}/\Reg{E,(v)}}\to \mathbf{dR}_K(\mf{g})/\Fil^0\mathbf{dR}_K(\mf{g})
\end{align}
whose generic fiber and $v$-adic completion is described in Remark~\ref{rem:ks_map_rational} and Remark~\ref{rem:Kodaira--Spencer}, respectively.
\end{remark}

\begin{lemma}
[A versality condition]
\label{lem:versality_condition} 
The following are equivalent:
\begin{enumerate}
  \item The map $\widehat{\mathcal{U}}\to \BT[\mathcal{G}^c,\mbox{-}\mu_v]{\infty}$ is formally \'etale
  \item $\mathcal{U}$ is smooth over $\Reg{E,(v)}$ and the map~\eqref{eqn:integral_Kodaira--Spencer_map} is an isomorphism.
\end{enumerate}
\end{lemma}
\begin{proof}
    Immediate from Remark~\ref{rem:Kodaira--Spencer}.
\end{proof}

\begin{remark}
    [Integral models for automorphic sheaves]
\label{rem:automorphic_sheaves}
The map $\mathcal{U}\to \mathrm{Gr}_{\mbox{-}\mu_v}/\mathcal{G}^c_{\Reg{E,(v)}}$ is precisely the one yielding integral models over $\mathcal{U}$ for automorphic sheaves. Namely, any $\bb{Z}_p$-representation $V$ of $\mathcal{P}^{-,c}_{\mu_v}$, where  $\mathcal{P}^{-,c}_{\mu_v}$ is the parabolic subgroup of $\mathcal{G}^c$ obtained as the image of $\mathcal{P}^-_{\mu_v}$,  yields a canonical $\mathcal{G}^c$-equivariant vector bundle over the homogeneous space $\mathrm{Gr}_{\mbox{-}\mu_v} \simeq \mathcal{P}^{-,c}_{\mu_v}\backslash \mathcal{G}^c_{\Reg{E,(v)}}$, and this pulls back to a vector bundle on $\mathcal{U}$.
\end{remark}

\begin{remark}
    [Filtered $F$-crystals]
\label{rem:filtered_f-crystals_apertile}
Using Construction~\ref{const:filtered_F-crystal}, we see that every apertile integral model $\mathcal{U}$ is equipped with a canonical exact and monoidal functor from $\Rep_{\Int_p}(\mathcal{G}^c)$ to strongly divisible filtered $F$-crystals over $\mathcal{U}$. For the integral canonical models of abelian type that we will encounter below, this recovers a construction of Lovering as in \cite{Lovering2017-fy}.
\end{remark}

\subsection{Integral canonical models}
\label{sub:integral_canonical_models}

\begin{assumption}{For the rest of this section and the next, we will assume that the level $K^p$ is} neat. We will return to the general situation in \S\ref{sub:icms_shimura_stacks} below.  
\end{assumption}
 
\begin{definition}
  [Integral canonical models]\label{defn:ICM}
Suppose that $E'/E$ is an extension in which $p$ is unramified, and that $v'\vert v$ is a place of $E'$. An integral model $\Ss_{K,(v')}$ for $\Sh_K\otimes_EE'$ over $\Reg{E',(v')}$ is an \defnword{integral canonical model} (\defnword{ICM} for short) if it is apertile, satisfies the condition(s) of Lemma~\ref{lem:versality_condition} (the Serre--Tate property), and if, for any mixed characteristic $(0,p)$ complete discrete valuation field $F$ over $\Reg{E'_{v'}}$ with perfect residue field, the subset 
\begin{equation*} 
\Ss_{K,(v')}(\Reg{F})\subset (\Sh_K\otimes_EE'_{v'})(F)
\end{equation*}
is characterized by the conditions in Definition~\ref{introdef:crys_icms}. If $E=E'$, then we will simply write $\Ss_K$ for the ICM.

The morphism $\varpi\colon \Ss_{K,(v')}\to \mr{BT}^{\mc{G}^c,\mbox{-}\mu_v,\mr{alg}}_\infty$ the \defnword{syntomic realization map}, and the pullback along it of the universal object over $\mr{BT}^{\mc{G}^c,\mbox{-}\mu_v,\mr{alg}}_\infty$ is the \defnword{syntomic realization} over $\Ss_{K,(v')}$.
\end{definition}

\begin{remark}
    [Local model diagram for ICMs]
\label{rem:local_model_diagram}
The Serre--Tate property combined with the discussion in Remark~\ref{rem:versality_bt_n} shows that, if we set
\[
\widetilde{\Ss}_K \defn \Ss_K\times_{\mathrm{Gr}_{\mbox{-}\mu_v}/\mathcal{G}^c_{\Reg{E,(v)}}}\mathrm{Gr}_{\mbox{-}\mu_v}
\]
then we obtain a diagram
\[\begin{tikzcd}[cramped,column sep=2.25em]
	& {\widetilde{\Ss}_K} \\
	{\Ss_K} && {\mathrm{Gr}_{\mbox{-}\mu_v}}
	\arrow[from=1-2, to=2-1]
	\arrow[from=1-2, to=2-3]
\end{tikzcd}\]
where the left arrow is a $\mathcal{G}^c$-torsor and the right arrow is formally smooth and $\mathcal{G}^c$-equivariant with relative cotangent complex a vector bundle of rank $\dim \mathcal{G}^c$.
\end{remark}

Let us restate Theorem~\ref{introthm:mapping_property} in the following form, which is now immediate from Proposition~\ref{prop:uniqueness_maps}.
\begin{theorem}
[Mapping property of ICMs]
\label{thm:mapping_property}
Let $\mathcal{Y}$ be an excellent and $\eta$-normal algebraic space over $\Reg{E',(v')}$ with generic fiber $Y$. If $\Ss_{K,(v')}$ is an ICM for $\Sh_K\otimes_EE'$ over $\Reg{E',(v')}$, then giving a map $\tilde{f}\colon\mathcal{Y}\to \Ss_{K,(v')}$ is equivalent to giving maps $\eta\colon \mathcal{Y}\to \BT[\mathcal{G}^c,\mbox{-}\mu_v,\mathrm{alg}]{\infty}$ and $f\colon Y\to \Sh_K\otimes_EE'$ so that the map $Y\to \mathrm{Loc}_{\mathcal{G}^c(\Int_p)}$ induced by $\eta$ lifts $f$ along the map $\Sh_K\to \mathrm{Loc}_{\mathcal{G}^c(\Int_p)}$ classifying $\mathbf{Et}_{K,p}$. Moreover, we have an identification $\varpi\circ \tilde{f}\simeq \eta$.
\end{theorem}

\begin{corollary}[Descent for coefficients]
\label{cor:descent_coefficients_gen_ICMs}
Suppose that $E'/E$ is an extension in which $p$ is unramified, and $v'\vert v$ is a place of $E'$ such that $\Sh_K\otimes_EE'$ admits an ICM $\Ss_{K,(v')}$ over $\Reg{E',(v')}$. Then $\Sh_K$ admits an ICM over $\Reg{E,(v)}$. 
\end{corollary}
\begin{proof}
We can assume without loss of generality that $E'/E$ is finite Galois. It suffices to show that $\widehat{\Ss}_{K,(v')}$, along with the classifying map $\widehat{\Ss}_{K,(v')}\to \BT[\mathcal{G}^c,\mbox{-}\mu_v]{\infty}$, descends to a formal algebraic space $\widehat{\Ss}_K$ over $\Reg{E_{v}}$. Indeed, this would show (e.g., by the gluing results of~\cite{achinger_youcis}) the existence of an apertile integral model $\Ss_K$ over $\Reg{E,(v)}$ satisfying the equivalent conditions of Lemma~\ref{lem:versality_condition}, and the pointwise conditions would follow from those for $\Ss_{K,(v')}$. 

To show that $\widehat{\Ss}_{K,(v')}$ descends, it suffices to establish that the action of $\Gal(E'_{v'}/E_{v})\leqslant \Gal(E'/E)$ on $\Sh_K\otimes_EE'$ extends to an action on $\Ss_{K,(v')}$, which follows from Theorem~\ref{thm:mapping_property} above.
\end{proof}

\begin{remark}{In what follows, we will exclusively focus on the case $E=E'$. By Corollary \ref{cor:descent_coefficients_gen_ICMs}, there is no loss of generality in doing so.}
\end{remark}

\begin{corollary}
  [Uniqueness]
\label{cor:uniqueness}
An ICM $\Ss_K$ for $\Sh_K$ is unique up to unique isomorphism. 
\end{corollary}

\begin{corollary}
    [Functoriality]
\label{cor:functoriality}
Suppose that we have a map of tuples as in \emph{Definition~\ref{defn:maps_between_unramified_tuples}}, and for $i=1,2$ ICMs $\Ss_{K_i}$ for $\Sh_{K_i}$ over $\Reg{E_i,(v_i)}$. Then the map $\Sh_{K_1}\to \Sh_{K_2}\otimes_{E_2}E_1$ extends to a map $\Ss_{K_1}\to\Ss_{K_2}\otimes_{\Reg{E_2,(v_2)}}\Reg{E_1,(v_1)}$.
\end{corollary}
\begin{proof}
    The only thing to observe here is that, under the map of Shimura varieties, for every prime $\ell$, the restriction of $\mathbf{Et}_{K_2,\ell}$ to $\Sh_{K_1}$ is canonically obtained from $\mathbf{Et}_{K_1,\ell}$ via pushforward along the map $K^c_{1,\ell}\to K^c_{2,\ell}$ induced by
    \[
    G^c_1(\Rat_{\ell})\xrightarrow{f}G^c_2(\Rat_{\ell})\xrightarrow{h\mapsto g_\ell hg_\ell^{-1}}G^c_2(\Rat_{\ell}).
    \]
    This was observed in Remark~\ref{rem:maps_arising_from_tuples}.
\end{proof}

\begin{remark}
[Central covers yield finite \'etale maps]
    \label{rem:finite_maps_etale}
In the context of Corollary~\ref{cor:functoriality}, suppose that $G_1\to G_2$ has central kernel. Then the map $\Ss_{K_1}\to \Ss_{K_2}$ is finite \'etale. That it is \'etale is immediate from the Serre--Tate property and Proposition~\ref{prop:central_covers}. Now, our hypothesis on the kernel implies that the map $\Sh_{K_1}\to \Sh_{K_2}$ is finite. Take $\Ss'_{K_1}$ to be the normalization of $\Ss_{K_2}$ in $\Sh_{K_1}$. We have an open immersion $\Ss_{K_1}\to \Ss'_{K_1}$, which one checks is an isomorphism using the pointwise conditions for ICMs.
\end{remark}

\begin{definition}
[Limpidity]
An ICM $\Ss_{K,(v')}$ for $\Sh_K$ is \defnword{limpid} if it admits an extension of the $K^c_\ell$-local system given by $\mathbf{Et}_{K,\ell}\vert_{\Sh_K\otimes_EE'}$ for all primes $\ell\neq p$.
\end{definition}

\begin{corollary}
    [Prime-to-$p$ Hecke correspondences]
\label{cor:prime-to-p_part}  
Suppose that $K^{',p}\subset K^p$ is a compact open subgroup and write $K'$ for the compact open subgroup $K_pK^{',p}\subset K$.
\begin{enumerate}
  \item If $\Sh_K$ admits a limpid ICM $\Ss_K$ over $\Reg{E,(v)}$, then the normalization of $\Ss_K$ in $\Sh_{K'}$ is also a limpid ICM $\Ss_{K'}$ for $\Sh_{K'}$;
  \item If (1) holds, and $g\in G(\Adele_f^p)$ is such that $gK^{',p}g^{-1}\subset K^p$, then the map $i_g\colon\Sh_{K'}\to \Sh_K$ determined on complex points by
  \[
    G(\Rat)\backslash X\times G(\Adele_f)/K' \xrightarrow{[(x,h)]\mapsto [(x,hg^{-1})]}G(\Rat)\backslash X\times G(\Adele_f)/K,
  \]
  extends to a finite \'etale map of ICMs $\Ss_{K'}\to \Ss_K$.
\end{enumerate}
\end{corollary}
\begin{proof}
    The first assertion follows from the fact that $\mathbf{Et}_{K.\ell}$ extends to a local system over $\Ss_K$, for all $\ell\neq p$, and the second is immediate from Corollary~\ref{cor:functoriality} and Remark~\ref{rem:finite_maps_etale}.
.\end{proof}

\begin{remark}
    [Comparison with the definition of Milne--Moonen]
\label{rem:milne-moonen_comparison}
In \cite{milne:canonical}, ~\cite{moonen:models}, and~\cite{kisin:abelian}, one finds a definition of integral canonical models that involves the inverse limit $\Sh_{K_p}$. An integral model $\Ss_{K_p}$ for $\Sh_{K_p}$ has the \emph{extension property} if, for any regular and formally smooth $\Reg{E,(v)}$-scheme $S$, any map $S_\eta\to \Sh_{K_p}$ of $E$-schemes extends to a map $S\to \Ss_{K_p}$. If $\Ss_{K_p}$ is itself regular and formally smooth, this property characterizes it as an integral canonical model in the sense of Milne. 

Suppose that $(G,\mathcal{G},X)$ admits limpid ICMs over $\Reg{E,(v)}$, and that the following \emph{additional} condition holds: 
\begin{itemize}
    \item Given a discrete valuation field $F$ over $E_v$ and $x\in \Sh_K(F)$, we have $x\in \Ss_K(\Reg{F})$ precisely when $\mathbf{Et}_{K,\ell}\vert_x$ is \emph{unramified} for all $\ell\neq p$.\footnote{Note that this condition holds in essentially all cases where ICMs are known to exist. One should be able to use a theory of ICMs for toroidal compactifications to verify that it is a general fact about ICMs.}
\end{itemize}

We can now take the inverse limit over the resulting inverse system $\{\Ss_K\}$ of finite \'etale maps to get a model $\Ss_{K_p}$ for $\Sh_{K_p}$. We claim that $\Ss_{K_p}$ satisfies Milne's extension property when $p>3$. 

First, we can assume that $S = \Spec R$ is affine. Now, for any mixed characteristic complete discrete valuation field $F$ over $E_v$ with perfect residue field, and an $x\colon \Spec(F)\to \Sh_{K_p}$, the projection of $x_K$ to $\Sh_K$ lifts to an $\Reg{F}$-point of $\Ss_K$: Indeed, the existence of the lift $x$ shows that, for all $\ell\ne p$, the local system $\mathbf{Et}_{K,\ell,x_K}$ is unramified (in fact, trivial). Therefore, the desired assertion follows from our additional assumption above.

Next, by Popescu desingularization \stacks{07GC}, we can present $R$ as a filtered colimit of smooth $\Reg{E,(v)}$-algebras implying $S$ is an inverse limit of smooth affine $\Reg{E,(v)}$-schemes $S_i$. Given $S_\eta\to \Sh_{K_p}$, the projection onto $\Sh_K$ for any finite level $K$ factors through $S_{i,\eta}$ for some $i$. The previous paragraph shows that the restriction of $\mathbf{Et}_{K,p}$ over $S_{i,\eta}$ is pointwise crystalline, and so Proposition~\ref{prop:pointwise_smooth_canonical} below now tells us that we have an extension $S_i\to \Ss_K$, and this implies our claim.
\end{remark}

\begin{remark}
    [Comparison with the definition of Bakker--Shankar--Tsimerman]
\label{rem:bst_comparison}
Combining Remark~\ref{rem:milne-moonen_comparison} with the discussion in~\cite[\S 1.3]{bst} tells us that when $p>3$, ICMs satisfying the additional assumption made above also satisfy the canonicity property of Bakker--Shankar--Tsimerman. 
\end{remark}

\section{Integral canonical models: Constructions and further properties}
\label{sec:icm_constructions}

We will maintain the notation from above. In this section, we will prove our main theorems on the existence of ICMs, as well as prove a few deeper properties of such models. All levels $K\subset G(\Adele_f)$ will be assume to be neat.

\subsection{CM points and canonical models for tori}
\label{sub:cm_points}

\begin{construction}
[CM Shimura varieties]
\label{const:cm_shimura_varieties}
    Let $T$ be a $\bb{Q}$-torus with a map $h_T\colon\mathbb{S}\to T_{\Real}$ and let $\mu_T\in X_*(T)$ be the cocharacter associated with $h_T$. Let $E_T\subset \overline{\Rat}$ be the field of definition of $\mu_T$. Then, for any neat compact open subgroup $K_T\subset T(\Adele_f)$, the tuple $(T,\{h\},K_T)$ yields a zero-dimensional Shimura variety $\Sh_{K_T}$ over $E_T$ with 
    \[
    \Sh_{K_T}(\Comp)\simeq T(\Rat)\backslash T(\Adele_f)/K.
    \]
    Fix a place $w\vert p$ of $E_T$, and let $\Ss_{K_T}$ be the normalization of $\Spec \Reg{E_T,(w)}$ in $\Sh_{K_T}$. We will assume that $K_T$ is of the form $K_{T,p}K_T^p$, where $K_{T,p}\subset T(\Rat_p)$ and $K^p_T\subset T(\Adele_f^p)$ are compact open subgroups.
\end{construction}

\begin{remark}
    [Galois structure of CM Shimura variety]
\label{rem:galois_structure_cm}Consider the reflex norm
\[
  r_{\mu_T}\colon\Res_{E_T/\Rat}\Gm \xrightarrow{\Res_{E_T/\Rat}\mu_T}\Res_{E_T/\Rat}T\xrightarrow{\mathrm{Nm}_{E_T/\Rat}}T.
\]
This induces a map
\begin{align*}
  E_T^\times\backslash\Adele_{E_T}^\times \xrightarrow{r_{\mu_T}(\Adele)} T(\Rat)\backslash T(\Adele)\to T(\Rat)\backslash T(\Adele_f)/K_T
\end{align*}
factoring via the global reciprocity map (with the geometric normalization) through a homomorphism
\[
  \tau_K(\mu_T)\colon\Gal(E_T^{\ab}/E_T)\to T(\Rat)\backslash T(\Adele_f)/K_T
\]
giving the target the structure of a $\Gal(E_T^{\ab}/E_T)$-equivariant finite set. The Shimura variety $\Sh_{K_T}$ is now canonically isomorphic to the \'etale scheme over $E_T$ associated with this datum.
\end{remark}

\begin{remark}
    [Unramifiedness criterion]
\label{rem:unramifiedness_criterion_cm_sv}
Suppose that $w\vert p$ is a $p$-adic place of $E$. The description above, combined with local reciprocity, shows that the normalization $\Ss_{K_T,w}$ of $\Reg{E_T,w}$ in $\Sh_{K_T}\otimes_{E_T}E_{T,w}$ is finite \'etale over $\Reg{E_T,w}$ precisely when we have
\[
r_{\mu_T}(\Reg{E_T,w}^\times)\subset T(\Rat)K_T\cap T(\Rat_p)\subset T(\Adele_f).
\]
Therefore, if $K_T\cap T(\Rat_p)$ is the maximal compact open subgroup of $T(\Rat_p)$, then this is always the case. In this situation, if $\varpi\in E_{T,w}$ a uniformizer, we find that multiplication by $r_{\mu_T}(\varpi)^{-1}$ on $T(\Rat)\backslash T(\Adele_f)/K_T$ gives an endomorphism of $\Ss_{K_T,w}$ that is a lift of the $q$-Frobenius on the mod-$\varpi$ fiber, where $q = |k(w)|$.
\end{remark}

\begin{remark}
  [\'Etale realizations for tori]
\label{rem:etale_tori}
Suppose that we have a finite extension $F/E_{T,w}$ within $\overline{\Rat}_p$ with maximal abelian extension $F^{\ab}\subset \overline{\Rat}_p$, and a point $x\in \Sh_{K_T}(F)$. Then, for each prime $\ell$, we can describe the local systems $\mathbf{Et}_{K_T,\ell}\vert_x$ over $\Spec F$ explicitly following the discussion in~\cite[\S4.1]{daniels2023canonical} (see also \cite[\S4.3.13]{kisin2021stf}). Taking the inverse limit of the maps $\tau_K(\mu_T)$ over $K$ yields a map
\[
\Gal(E_T^{\ab}/E_T) \to T(\Rat)^-\backslash T(\Adele_f).
\]
The restriction of this map to $\Gal(\overline{\Rat}_p/F)$ factors through $T(\Rat)^-\backslash T(\Rat)^-K_T$, and so gives a map
\begin{align}\label{eqn:reciprocity_map}
\Gal(F^{\ab}/F)\to T(\Rat)^-\backslash T(\Rat)^-K_T\to T^c(\Rat)\backslash T^c(\Rat)K^c_T \simeq K^c_T,
\end{align}
where the last isomorphism uses the neatness of $K_T$ (and hence of $K^c_T$). Projecting onto the $\ell$-adic component now gives us the $K^c_{T,\ell}$-valued Galois representation corresponding to $\mathbf{Et}_{K_T,\ell}\vert_x$.
\end{remark}

\begin{remark}
  [Crystallinity at $p$]
\label{rem:tori_crystallinity}
The description in Remark~\ref{rem:etale_tori}, combined with~\cite[Proposition 4.3.14]{kisin2021stf} and \cite[Lemma 4.3]{daniels2023canonical}, shows that $\mathbf{Et}_{K_T,p}$ is crystalline of type $\mbox{-}\mu_T$.
\end{remark}

\begin{remark}
  [Unramifiedness at $\ell\neq p$]
\label{rem:tori_unramified}
For $\ell\neq p$, the description from Remark~\ref{rem:etale_tori} shows that the local system $\mathbf{Et}_{K_T,\ell}$ is unramified. Indeed, it is enough to know that the image of the composition
\[
\Reg{F}^\times\to F^\times \to \Gal(F^{\ab}/F)\xrightarrow{\eqref{eqn:reciprocity_map}}K^c_T
\]
has trivial projection onto $K^{c,p}_T$ (the middle arrow is the local reciprocity map for $F$). This is now clear, as the description in terms of the reflex norm shows that this composition lifts to a map $\Reg{F}^\times\to K^c_{T,p}$.
\end{remark}

Integral canonical models always exist for CM Shimura varieties:
\begin{proposition}
  [ICMs for tori]
\label{prop:integral_canonical_models_for_tori}
Suppose that $T_{\Rat_p}$ extends to a torus $\mathcal{T}$ over $\Int_p$ such that $K_{T,p} = \mathcal{T}(\Int_p)$. Then $\Ss_{K_T}$ is a limpid ICM for $\Sh_{K_T}$.
\end{proposition}
\begin{proof}
  In this case, $E_T$ is unramified at $p$ by Remark \ref{rem:unramifiedness_criterion_cm_sv}, and one finds using Remark~\ref{rem:unramifiedness_criterion_cm_sv} that $\Ss_{K_T}$ is a finite \'etale scheme over $\Reg{E_T,(w)}$. Moreover, the stacks $\BT[\mathcal{T},\mbox{-}\mu_T]{n}$ are \'etale over $\Reg{E_{T},w}$ and are in fact non-canonically isomorphic to the $\Reg{E_{T},w}$-stack $\mathrm{Loc}_{\underline{\mathcal{T}(\Int/p^n\Int)}}$ of $\mathcal{T}(\Int/p^n\Int)$-local systems; see~\cite[Proposition 10.4.1]{gmm}. Therefore, \emph{any} map $\widehat{\Ss}_{K_T,w}\to \BT[\mathcal{T},\mbox{-}\mu_T]{\infty}$ will be formally \'etale.  Proposition~\ref{prop:gmu_rings_of_integers} now reduces us to checking that, for every finite extension $L/E_{T,w}$ and every $x\in \Ss_K(\Reg{L})$, the following statements hold:

\begin{enumerate}
  \item  $\mathbf{Et}_{K_T,p}\vert_x$ belongs to $\mathrm{Loc}^{\mathrm{crys},{\mbox{-}\mu_T}}_{\mathcal{T}(\Int_p)}(L)$;
  \item For $\ell\neq p$, $\mathbf{Et}_{K_T,\ell}\vert_x$ is unramified.
\end{enumerate}
  
  These follow from Remarks~\ref{rem:tori_crystallinity} and~\ref{rem:tori_unramified}.
\end{proof}

\begin{construction}
\label{const:cm_points}
    Suppose that we have an unramified tuple $(G,\mathcal{G},X,K)$ and an ICM $\Ss_K$ for $\Sh_K$ over $\Reg{E,(v)}$ for a place $v\vert p$. As a special case of Construction~\ref{const:cm_shimura_varieties}, suppose that we have a torus $T$ over $\Rat$ and a closed immersion $i\colon T\to G$ over $\Rat$ such that, for some $x\in X$, the map $h_x\colon\mathbb{S}\to G_{\Real}$ factors through $i_{\Real}$ via a map $h_T\colon\mathbb{S}\to T_{\Real}$. Given $g\in G(\Adele^p_f)$, set $K_{T,g} = i^{-1}(gKg^{-1})\subset T(\Adele_f)$. Then, for any place $w\vert p$ of $E_T$ above $v$, we obtain an integral model $\Ss_{K_{T,g},(w)}$ for $\Sh_{K_{T,g}}$ over $\Reg{E_T,(w)}$ as the normalization of $\Spec \Reg{E,(v)}$ in $\Sh_{K_{T,g}}$.
\end{construction}

\begin{lemma}
    [Extension of CM points]
\label{lem:extension_cm_points}
The map $\Sh_{K_{T,g}}\to \Sh_K\otimes_E E_T$ given on $\Comp$-points by 
\[
T(\Rat)\backslash T(\Adele_f)/K_{T,g}\xrightarrow{[t]\mapsto [(x_T,i(t)g]}G(\Rat)\backslash X \times G(\Adele_f)/K
\]
extends to a map of integral models $\Ss_{K_{T,g},(w)}\to \Ss_K \otimes_{\Reg{E,(v)}}\Reg{E_T,(w)}$.
\end{lemma}
\begin{proof}
To alleviate notation, set $K_T = K_{T,g}$. Given Theorem~\ref{thm:mapping_property}, it is enough to know that the restriction of $\mathbf{Et}_{K,p}$ over $\Sh_{K_T}$ lifts to a map
\[
  \Ss_{K_T,(w)}\to \BT[\mathcal{G}^c,\mbox{-}\mu_v,\mathrm{alg}]{\infty}.
\]
As $\Ss_{K,T,(w)}$ is a product of spectra of localizations of rings of integers in number fields, we reduce to knowing by Lemma~\ref{lem:reduction_to_complete_local} that, for every complete local ring $\hat{\Rg}$ of $\Ss_{K,T,(w)}$, the local system $\mathbf{Et}_{K,p}|_{\hat{\Rg}[\nicefrac{1}{p}]}$ lifts to $\BT[\mathcal{G}^c,\mbox{-}\mu_v]{\infty}(\hat{\Rg})$.

By Proposition~\ref{prop:gmu_rings_of_integers}, this is equivalent to checking that the restriction of $\mathbf{Et}_{K,p}$ over $\hat{\Rg}[\nicefrac{1}{p}]$ is crystalline of type $\mbox{-}\mu_v$, which is immediate from Remark~\ref{rem:tori_crystallinity} and Remark~\ref{rem:maps_arising_from_tuples}. 
\end{proof}

\subsection{Reduction of structure group}
\label{sub:reduction_of_structure_group}

We can now prove Theorem~\ref{introthm:reduction_of_structure}, whose statement we now recall:
\begin{theorem}
  [Reduction of structure group]
\label{thm:reduction_of_structure_group}
Let $(G,\mathcal{G},X,K)\to (G^\sharp,\mathcal{G}^\sharp,X^\sharp,K^\sharp)$ be a map of neat unramified tuples such that $\mathcal{G}\to \mathcal{G}^\sharp$ is a closed immersion. Let $E^\sharp$ be the reflex field for $(G^\sharp,X^\sharp)$, and let $v\vert p$ be a place for $E$ lying above a place $v^\sharp$ for $E^\sharp$. Suppose that $\Sh_{K^\sharp}$ admits a limpid ICM $\Ss_{K^\sharp}$ over $\Reg{E^\sharp,(v^\sharp)}$. Then the normalization of $\Reg{E,(v)}\otimes_{\Reg{E^\sharp,(v^\sharp)}}\Ss_{K^\sharp}$ in $\Sh_K$ gives a limpid ICM for $\Sh_K$ over $\Reg{E,(v)}$.
\end{theorem}
\begin{proof}
By Lemma~\ref{lem:Gc_functoriality}, the map $G^c\to G^{\sharp,c}$ is also  a closed immersion, and therefore so is the map $\mathcal{G}^c\to \mathcal{G}^{\sharp,c}$.

Suppose first that $\Sh_K\to \Sh_{K^\sharp}$ is a closed immersion. We then have the following diagram
\[\begin{tikzcd}[sep=2.25em]
	{\Sh_K} & {\mathrm{Loc}_{\mathcal{G}(\Int_p)} = \BT[\mathcal{G}^c,\mbox{-}\mu_v,\mathrm{alg}]{\infty}[\nicefrac{1}{p}]} \\
	{\Ss_K} & {\BT[\mathcal{G}^c,\mbox{-}\mu_v,\mathrm{alg}]{\infty}} \\
	{\Ss_{K^\sharp}} & {\BT[\mathcal{G}^\sharp,\mbox{-}\mu_v,\mathrm{alg}]{\infty}}
	\arrow[from=1-1, to=1-2]
	\arrow[from=1-1, to=2-1]
	\arrow[from=1-2, to=2-2]
	\arrow["{?\exists}", dotted, from=2-1, to=2-2]
	\arrow[from=2-1, to=3-1]
	\arrow[from=2-2, to=3-2]
	\arrow[from=3-1, to=3-2]
\end{tikzcd}\]
 The bottom left vertical arrow is the normalization of a closed immersion by construction, and the bottom horizontal arrow is formally \'etale. Therefore, to fill in the middle dotted arrow and also know that it is formally \'etale, we can appeal to Proposition~\ref{prop:Gmu_aperture_pointwise_condition_normal}, whose hypotheses are verified by Remark~\ref{rem:pointwise_lifts}. For $\ell\neq p$, the $K^{\sharp,c}_\ell$-local system $\mathbf{Et}_{K^\sharp,\ell}$ extends over $\Ss_{K^\sharp}$ by hypothesis, and this immediately implies that the $K^c_\ell$-local system $\mathbf{Et}_{K,\ell}$ extends over $\Ss_K$.

 To finish, we need to know that $\Ss_K$ satisfies the pointwise condition for ICMs: This just amounts to the observation that by construction the equality
 \[
  \Ss_K(\Reg{F}) = \Sh_K(F)\cap \Ss_{K^\sharp}(\Reg{F})\subset \Sh_{K^\sharp}(F)
 \]
holds for any discrete valuation field $F$ over $E$.

In general: By~\cite[Lemma (2.1.2)]{kisin:abelian}, we can find a compact open subgroup $K^{\sharp,p}_1\subset K^{\sharp,p}$ such that, with $K^\sharp_1 = K^{\sharp,p}_1K^\sharp_p$, we have $K = K_1^\sharp\cap G(\Adele_f^p)$ and the map $\Sh_{K}\to \Sh_{K^\sharp_1}$ is a closed immersion. By the above discussion, it is enough to know that $\Sh_{K^\sharp_1}$ admits an ICM whenever $\Sh_{K^\sharp}$ does so, and this is (1) of Corollary~\ref{cor:prime-to-p_part}.
\end{proof}

\subsection{Spaces of quasi-isogenies}
\label{sub:spaces_of_isogenies}

 In this subsection, we will study how Hecke correspondences can be used to define spaces of quasi-isogenies between points on a limpid ICM $\Ss_K$. Throughout, we assume that $G = G^c$.

\begin{remark}
 \label{rem:hecke_action}
As noted in Construction~\ref{const:Kc_etale_realization}, we have a canonical pro-Galois cover $\Sh\to \Sh_K$ with Galois group $K\subset G(\Adele_f)$,\footnote{This is where we use the assumption $G = G^c$.} obtained as the inverse limit of the Shimura varieties $\Sh_{K'}$ with $K'\unlhd K$ of finite index. Now, $G(\Adele_f)$ acts on this inverse system in the usual way (see Remark~\ref{rem:prime-to-p_hecke} for the description of the action of $G(\Adele_f^p)$ on the prime-to-$p$ Hecke tower), and this lifts to an action of $G(\Adele_f)$ on $\Sh$ as an $E$-scheme that extends the action of $K$ on it as an $\Sh_K$-scheme. 

Similarly, if $\Ss_{K_p}\to \Ss_K$ is the inverse limit of the prime-to-$p$ Hecke tower of ICMs $\Ss_{K'}$ with $K'\unlhd K$ of the form $K^{',p}K_p$ (see Corollary~\ref{cor:prime-to-p_part}), then we obtain an action of $G(\Adele_f^p)$ on $\Ss_{K_p}$ as an $\Reg{E,(v)}$-algebraic space, extending the action of $K^p$ on it as an $\Ss_K$-algebraic space.
\end{remark}

\begin{definition}
    [Quasi-isogenies in characteristic $0$]
\label{def:isogenies}
For an $E$-algebra $R$ and $x,y\in \Sh(R)$, a \defnword{quasi-isogeny from $x$ to $y$} is an element $g\in G(\Adele_f)$ with $x\cdot g = y$. We set $\mathrm{QIsog}(x,y)$ to be the set of such quasi-isogenies, and shorten this to $\mathrm{Aut}^{\circ}(x)$ when $x=y$, in which case it is a subgroup of $G(\Adele_f)$.
\end{definition}

\begin{definition}
 [Prime-to-$p$ quasi-isogenies in general]
If $R$ is an $\Reg{E,(v)}$-algebra and $x,y\in \Ss_{K_p}(R)$, then a \defnword{prime-to-$p$ quasi-isogeny from $x$ to $y$} is an element $g\in G(\Adele_f^p)$ with $x\cdot g = y$. We set $\mathrm{QIsog}^p(x,y)$ to be the set of such quasi-isogenies, and shorten this to $\mathrm{Aut}^{(p)}(x)$ when $x=y$, in which case it is a subgroup of $G(\Adele_f^p)$.
\end{definition}

\begin{remark}
    [Descent to finite level]
\label{rem:isogenies_descent}
Suppose that $K\subset G(\Adele_f)$ is a compact open of the form $K_pK^p$. Viewing $\mathrm{QIsog}^p$ as an algebraic space\footnote{For $g\in G(\Adele_f^p)$, the open and closed subspace of $\mathrm{QIsog}^p$ consisting of quasi-isogenies lying in $K^pgK^p$ is the pullback over $\Ss_{K_p}\times \Ss_{K_p}$ of the Hecke correspondence $\Ss_{K\cap gKg^{-1}}\to \Ss_K\times \Ss_K$.} over $\Ss_{K_p}\times \Ss_{K_p}$, we see that it is canonically $K^p$-equivariant, and so descends to an algebraic space $\mathrm{QIsog}^p_K$ over $\Ss_K\times \Ss_K$. Similarly, $\mathrm{QIsog}$ also descends to a scheme $\mathrm{QIsog}_K$ over $\Sh_K\times \Sh_K$. For $x\in \Ss_K(R)$ (resp.\@ $x\in \Sh_K(R)$), we will write $\Aut^p_K(x)$ (resp.\@ $\Aut^\circ_K(x)$) for $\mathrm{QIsog}^p_K(x,x)$ (resp.\@ $\mathrm{QIsog}_K(x,x)$). 
\end{remark}

\begin{remark}
    [Algebraic groups obtained from \'etale realizations]
\label{rem:reductive_groups_realizations}
Suppose that $\kappa$ is a field over $\Reg{E,(v)}$, and that we have $x\in \Ss_K(\kappa)$. We can associate with this an algebraic group $I_{\ell,x}$ over $\Rat_\ell$ when $\ell$ is a prime invertible in $\kappa$. For this, we use the $K_{\ell}$-local system $\mathbf{Et}_{K,\ell}\vert_x$ to obtain a $\Rat_\ell$-linear exact tensor functor $\omega_{\ell,x}\colon\Rep_{\Rat_\ell}(G)\to \mathrm{Loc}_{\Rat_\ell}(\kappa)$, where the target is the  category of $\Rat_\ell$-local systems over $\Spec \kappa$. We now take $I_{\ell,x}$ to be the affine algebraic group of automorphisms of $\omega_{\ell,x}$.

Concretely, this can be realized as follows: Choose a faithful representation $V$ for $G$ and tensors $\{s_{\alpha}\}\subset V^\otimes$ whose pointwise stabilizer is $G$ and let $\mathbf{Et}_{K,\ell}(V)\vert_x$ be the associated $\Rat_\ell$-local system over $\kappa$ equipped with global sections $\{s_{\alpha,\ell,x}\}$. Fixing a separable closure $\overline{\kappa}$ of $\kappa$ with Galois group $\Gamma_\kappa \defn \Gal(\overline{\kappa}/\kappa)$, this local system corresponds to a continuous $\Gamma_\kappa$-representation on a $\Rat_\ell$-vector space $V_{\ell,x}$ equipped with invariant tensors $\{s_{\alpha,\ell,x}\}$. The pointwise stabilizer $G_{\ell,x}\subset \GL(V_{\ell,x})$ of these tensors is a reductive group, equipped with a $G(\Rat_\ell)$-conjugacy class of isomorphisms with $G_{\Rat_\ell}$. We now find that $I_{\ell,x}$ can be identified with the commutant of $\rho(\Gamma_\kappa)$ in $G_{\ell,x}$. 

An obvious variant of the construction also gives an algebraic group $I_{\Adele_f^p,x}$ over $\Adele_f^p$, whose base-change over $\Rat_\ell$ recovers $I_{\ell,x}$ for each prime $\ell\neq p$. 
\end{remark}

\begin{remark}
    [Algebraic groups associated with $F$-isocrystals with $G$-structure]
\label{rem:algebraic_groups_f-isocrystal_with_G-structure}
Suppose that $\kappa$ is a perfect field over $k(v)$ and that we have $x\in \Ss_K(\kappa)$. Let $\mf{Q}_x$ be the pullback of the syntomic realization to $x$. By Construction~\ref{const:newton_map}, we can associate an $F$-isocrystal with $G$-structure over $\kappa$, which can be viewed as an exact $\bb{Q}_p$-linear tensor functor
\[
\omega_{p,x}\colon \Rep_{\Rat_p}(G)\to \mathrm{Isoc}_\kappa,
\]
of $\Rat_p$-linear Tannakian categories. Here, the target is the category of $F$-isocrystals over $\kappa$: Pairs $(M,F)$, where $M$ is a finite-dimensional $W(\kappa)[\nicefrac{1}{p}]$-vector space and $F\colon \sigma^\ast M\isomto M$ is an isomorphism where $\sigma$ is the absolute Frobenius on $\kappa$. We define $I_{p,x}$ as the group of automorphisms of $\omega_{p,x}$. 

This can be described more concretely when the associated object in $\mathrm{Isoc}_G(\kappa) = [LG/_{\mathrm{Ad}_\varphi}LG](\kappa)$ admits a lift to $\delta_x \in G(W(\kappa)[\nicefrac{1}{p}])$. In this case, for any $\Rat_p$-algebra $R$, we have a canonical identification
 \[
 I_{p,x}(R) = \left\{g\in G(W(\kappa)\otimes_{\Int_p}R):\; g^{-1}\delta_x \varphi(g) = \delta_x\right\}.
 \]
\end{remark}

\begin{remark}
    [$\ell$-adic $q$-Frobenius elements]
\label{rem:frobenius_elements_ell-adic}
Suppose that $\kappa$ is a finite extension of $k(v)$ of order $q$, and fix an algebraic closure $\overline{\kappa}$ for it. Then for all $\ell\neq p$, the action of the geometric Frobenius element in $\Gal(\overline{\kappa}/\kappa)$ on $\mathbf{Et}_{K,\ell}\vert_{\overline{x}}$ yields an element $\gamma_{\ell,x}\in I_{\ell,x}(\Rat_\ell)$. If we choose an algebraic closure $\overline{\kappa}$ of $\kappa$ and a lift $\tilde{x}\in \Ss_{K_p}(\overline{\kappa})$ of $x$, then we obtain an inclusion $I_{\ell,x}\subset I_{\ell,\tilde{x}} \simeq G_{\Rat_\ell}$, which identifies $I_{\ell,x}$ with the commutant in $G_{\Rat_\ell}$ of $\gamma_{\ell,x}$. 
\end{remark}

\begin{remark}
    [The $p$-adic $q$-Frobenius element]
\label{rem:frobenius_element_p-adic}
In the situation of Remark~\ref{rem:frobenius_elements_ell-adic}, we also have an element $\gamma_{p,x}\in I_{p,x}(\bb{Q}_p)$.  By Lang's theorem, the $F$-isocrystal with $G$-structure obtained from $\mf{Q}_x$ is represented by some $\delta_x\in \mathcal{G}(W(\kappa))\varphi(\mu(p))\mathcal{G}(W(\kappa))$ well-defined up to the $\mathrm{Ad}_\varphi$-action of $\mathcal{G}(W(\kappa))$. If $q = p^r$, then we have the element $\gamma_{p,x} = \delta_x\sigma(\delta_x)\cdots \sigma^{r-1}(\delta_x)$, which is seen to live in $I_{p,x}(\Rat_p)\subset G(W(\kappa)[\nicefrac{1}{p}])$, and is in fact central there. Moreover, the natural map $I_{p,x}\otimes_{\Rat_p}W(\kappa)[\nicefrac{1}{p}]\to G\otimes_{\Rat_p}W(\kappa)[\nicefrac{1}{p}]$ is an isomorphism onto the commutant of $\gamma_{p,x}$ in the target.
\end{remark}

\begin{remark}
    [Group schemes over $\Int_p$]
\label{rem:group_schemes_Zp}
Suppose that $x\in \Ss_K(\kappa)$ is valued in a perfect field $\kappa$ over $\Reg{E,(v)}$. Then the group $I_{p,x}$ admits a canonical $\Int_p$-model $I_{\Int_p,x}$:

\begin{enumerate}
\item If $\kappa$ has characteristic $0$, this is obtained via the same method used for $I_{p,x}$ itself: One instead uses the exact tensor functor $\Rep_{\Int_p}(\mathcal{G})\to \mathrm{Loc}_{\Int_p}(\kappa)$ associated with $\mathbf{Et}_{K,p}\vert_x$.

\item If $\kappa$ has characteristic $p$, then we employ the exact tensor functor $\Rep_{\Int_p}(\mathcal{G})\to \mathrm{Vect}(\kappa^{\mathrm{syn}})$ associated with the syntomic realization $\mf{Q}_x$ to the same effect.
\end{enumerate}
\end{remark}

\begin{lemma}
    [The specialization map for group schemes]
\label{lem:specialization_group_schemes}
Let $F$ be a complete discrete valuation field over $E_v$ with perfect residue field $\kappa$ and fix $x\in \Ss_K(\Reg{F})$ with associated points $x_0\in \Ss_K(\kappa)$ and $x_\eta\in \Sh_K(F)$. Then:
\begin{enumerate}
    \item For $\ell\neq p$, there is a canonical specialization isomorphism $I_{\ell,x_\eta}\isomto I_{\ell,x_0}$.
    \item There is a canonical specialization homomorphism of group schemes $I_{\Int_p,x_\eta}\to I_{\Int_p,x_0}$.
\end{enumerate}
\end{lemma}
\begin{proof}
    The $\ell\neq p$ case is immediate from the definitions and specialization for the \'etale fundamental group. For the $\Int_p$-group schemes, one uses the crystalline specialization map from Remark~\ref{rem:specialization}.
\end{proof}

\begin{remark}
    [Realizations of quasi-isogenies]
\label{rem:realizations_quasi-isogenies}
Suppose that we have $x\in \Ss_K(\kappa)$ with $\kappa$ a field. If $\kappa$ is in characteristic $p$, there are canonical realization maps
\[
\Aut^p_K(x) \to I_{\Aff^p_f,x}(\Aff^p_f) = \Aut(\mathbf{Et}_{K,\Adele_f^p}\vert_x)\;;\; \Aut^p_K(x) \to I_{\Int_p,x}(\Int_p) = \Aut(\mathfrak{Q}_x).
\]
The first of these is obtained via descent from the obvious inclusion $\Aut^p(\tilde{x})\subset G(\Aff^p_f)$, for any $\tilde{x}\in \Ss_{K_p}(\ov{\kappa})$ lifting $x$. The second is obtained from the fact that the syntomic realization is invariant for the prime-to-$p$ Hecke action. If $\kappa$ is a field extension of $E$, then we similarly have an \'etale realization map
\[
\Aut^{\circ}_K(x)\to  I_{\Aff^p_f,x}(\Aff^p_f)\times I_{p,x}(\Rat_p).
\]
\end{remark}

\begin{lemma}
    [Algebraic groups obtained from quasi-isogenies: Over $\Comp$]
\label{lem:reductive_groups_from_isogenies_C}
Let $x\in \Sh(\Comp)$ and let $x^{\ad}$ be its image in $\Sh^{\ad}(\Comp)$. Then there is a canonical reductive algebraic group $I_x$ over $\Rat$ with the following properties:
\begin{enumerate}
    \item There is a natural inclusion $Z(G)\hookrightarrow I_x$ and $(I_{x}/Z(G))\otimes\Real$ is compact.
    \item There is a map $I_x\to G^{\ab}$ restricting to the natural map $Z(G)\to G^{\ab}$.
    \item We have $I_x(\Rat) = \Aut^\circ(x)$.
    \item For every prime $\ell$, there is a canonical realization map of algebraic groups $I_x\otimes\Rat_\ell\to G_{\Rat_\ell}$ that is a closed immersion, and is compatible with the realization map $\Aut^{\circ}(x)\to I_{\ell,x}(\Rat_\ell)\simeq G(\Rat_\ell)$.
    \item If $\sigma\in \Aut(\Comp)$, then there is a canonical isomorphism $I_x\isomto I_{\sigma(x)}$ extending the identity on $Z(G)$, which is compatible with the natural identification $\Aut^{\circ}(x)\isomto \Aut^{\circ}(\sigma(x))$, and also with the canonical identification $I_{\ell,x}\isomto I_{\ell,\sigma(x)}$ obtained from the isomorphism $\mathbf{Et}_{K,\ell}\vert_x \isomto \mathbf{Et}_{K,\ell}\vert_{\sigma(x)}$.
\end{enumerate}
\end{lemma}
\begin{proof}
    There exists a canonical exact tensor functor $\omega_{\mathrm{Hdg},x}:\Rep_{\Rat}(G)\to \mathrm{Hdg}_{\Rat}$---where the target is the category of $\Rat$-Hodge structures---lifting the various $\ell$-adic functors $\omega_{\ell,x}$ along the forgetful functors $\mathrm{Hdg}_{\Rat}\to \mathrm{Vect}_{\Rat_{\ell}}$. If $x$ lifts to a point $(h,g)\in X\times G(\Adele_f)$, this associates with each representation $V$ of $G$ the $\Rat$-Hodge structure split by $h$. Let $I'_x = \underline{\Aut}(\omega_{\mathrm{Hdg},x})$ denote the automorphism group of this functor: this is an algebraic group over $\Rat$. Concretely, $I'_x\subset G$ is the commutant of the $\Rat$-Zariski closure of the image of $h$. Let $I_x\subset I'_x$ be the open and closed subgroup consisting of the connected components that meet $\Aut^\circ(x) = I'_x(\Rat)$.
    
    We have an inclusion $Z(G)\hookrightarrow I_x$ arising from the fact that $h$ commutes with $Z(G)$. Moreover, let $\Sh^{\ab}$ be the inverse limit of CM Shimura varieties associated with $G^{\ab}$ and the composition $\mathbb{S}\xrightarrow{h}G_{\Real}\to G^{\ab}_{\Real}$ for some (hence any) $h\in X$. Let $x^{\ab}\in \Sh^{\ab}(\Comp)$ be the image of $x$. Then we can repeat the above construction with $x$ replaced by $x^{\ab}$, and obtain in this way an algebraic group $I_{x^{\ab}}$ that is canonically isomorphic to $G^{\ab}$. In fact, $\omega_{\mathrm{Hdg},x^{\ab}}$ is simply the restriction of $\omega_{\mathrm{Hdg},x}$ to $\Rep_{\Rat}G^{\ab}$, and so we obtain a canonical map $I_x\to I_{x^{\ab}}\simeq G^{\ab}$ compatible with the map $Z(G)\to G^{\ab}$.

    Moreover, $(I_x/Z(G))\otimes\Real$ is compact: This is a manifestation of the fact that $\omega_{\mathrm{Hdg},x}$ is a \emph{polarized} $\Rat$-Hodge structure with $G$-structure. This also implies that $I_x$ is reductive. 
    
    For the last property, note that we a canonical isomorphism
    \[
    \alpha\colon I_{x}(\Rat)\isomto \Aut^\circ(x) \isomto \Aut^\circ(\sigma(x))\isomto I_{\sigma(x)}(\Rat). 
    \]
    This is compatible with the isomorphisms $I_{\ell,x}\isomto I_{\ell,\sigma(x)}$ induced by the identifications between the $\ell$-adic realizations. Since the rational points of any connected reductive group are Zariski dense, and since every connected component of $I_x$ (resp. $I_{\sigma(x)}$) meets $\Aut^\circ(x)$ (resp. $\Aut^\circ(\sigma(x))$) by definition, this implies that $\alpha$ algebraizes to an isomorphism of connected reductive groups $I_x \isomto I_{\sigma(x)}$ compatible with all $\ell$-adic realizations.
\end{proof}

\begin{lemma}
    [Algebraic groups obtained from quasi-isogenies]
\label{lem:reductive_groups_from_isogenies}
Suppose that $\kappa$ is a characteristic $0$ field and that we have $x\in \Sh_K(\kappa)$. Then there is a canonical reductive algebraic group $I_x$ over $\Rat$ with the following properties:
\begin{enumerate}
    \item There is a natural inclusion $Z(G)\hookrightarrow I_x$ and $(I_{x}/Z(G))\otimes\Real$ is compact.
    \item We have $I_x(\Rat) = \Aut^{\circ}_K(x)$.
    \item For every prime $\ell$, there is a canonical realization map of algebraic groups $I_x\otimes\Rat_\ell\to I_{\ell,x}$ that is a closed immersion and is compatible with the realization map $\Aut^{\circ}_K(x)\to I_{\ell,x}(\Rat_\ell)$.
\end{enumerate}
Moreover, if $\kappa'/\kappa$ is a field extension, and $x'\in \Ss_K(\kappa')$ is the image of $x$, then, for any choice of $\ell$, one can characterize $I_x$ as the largest algebraic subgroup of $I_{x'}$ mapping to $I_{\ell,x}\subset I_{\ell,x'}$.
\end{lemma}
\begin{proof}
 Suppose first that $\kappa$ is algebraically closed and admits an embedding $\iota\colon\kappa\hookrightarrow \Comp$, and that we have a lift $\tilde{x}\in \Sh(\kappa)$ of $x$. Then Lemma~\ref{lem:reductive_groups_from_isogenies_C} gives us a reductive group $I_{\iota(\tilde{x})}$ with $I_{\iota(\tilde{x})}(\Rat) = \Aut^{\circ}(\tilde{x})$, and assertion (5) there tells us that this group is in fact independent of the choice of embedding $\iota$. We can therefore denote it by $I_{\tilde{x}}$. In fact, this group is independent of the choice of lift $\tilde{x}$ as well, and gives us the group $I_x$ with all the desired properties when $\kappa$ is algebraically closed. 

 In general, choose an algebraic closure $\ov{\kappa}$ of $\kappa$ and write $\overline{x}\in \Sh_K(\ov{\kappa})$. Assertion (5) of Lemma~\ref{lem:reductive_groups_from_isogenies_C} tells us that there is a canonical $\Gal(\overline{\kappa}/\kappa)$-action on $I_{\overline{x}}$ that agrees with the natural action on $I_{\ell,\overline{x}}$, and we now see that the commutant $I_x\subset I_{\overline{x}}$ of this action now gives us the group for arbitrary $\kappa$, and that it satisfies all given properties.
\end{proof}

We now arrive at some natural conjectures. 
\begin{conjecture}
    [Tate conjecture in characteristic $0$]
\label{conj:tate}
If $F$ is a finitely generated field over $E$, then for any $x\in \Sh_K(F)$, and for all primes $\ell$, the realization map $I_x\otimes\Rat_\ell\to I_{\ell,x}$ is an isomorphism.
\end{conjecture}

\begin{conjecture}
    [Tate conjecture in characteristic $p$]
\label{conj:tate_char_p}
Suppose that $\kappa$ is a perfect field over $k(v)$ and fix a point $x\in \Ss_K(\kappa)$. Then:
\begin{enumerate}
    \item There exists a reductive group $I_x$ over $\Rat$ with realization maps $\alpha_\ell\colon I_x\otimes_{\Rat}\Rat_\ell\to I_{\ell,x}$ for all $\ell$ such that 
    \[
    I_x(\Int_{(p)})\defn I_x(\Rat)\cap I_{\Int_p,x}(\Int_p) =\Aut^p(x),
    \]
    and such that, for all $\ell$, the map $I_x(\Int_{(p)})\to I_{\ell,x}(\Rat_\ell)$ is the realization map from \emph{Remark~\ref{rem:realizations_quasi-isogenies}}.
    \item If $\kappa$ is finite, then, for all $\ell$, the realization map $I_x\otimes\Rat_\ell\to I_{\ell,x}$ is an isomorphism.
\end{enumerate}
\end{conjecture}

\begin{remark}
    When $(G,X)$ is of pre-abelian type, Conjecture~\ref{conj:tate} can be deduced from and Faltings's isogeny theorem: Indeed, in the Siegel-type case the notion of quasi-isogeny reduces to the classical notion for abelian varieties. Conjecture~\ref{conj:tate_char_p}, on the other hand, can be deduced from Kisin's generalization of Tate's theorem~\cite[Corollary (2.3.2)]{Kisin2017-qa}. We will see in Theorem~\ref{thm:tate_conjecture} that the latter also holds for all limpid ICMs when $x$ is a $\mu$-ordinary point. This will be extended to a full proof (for finite $\kappa$) of the conjecture in~\cite{lee_madapusi}.
\end{remark}

For later use, we record the following lemma.
\begin{lemma}
    [Criterion for being a CM point]
\label{lem:cm_point_criterion}
Let $F$ be a field extension of $E$. Then the following are equivalent for a point $x\in \Sh_K(F)$:
\begin{enumerate}
    \item $x$ is a CM point;
    \item after replacing $F$ with a finite extension, the groups $I_{x}$ and $G$ have the same rank.
\end{enumerate}
In fact, in this case, for every maximal torus $T\subset I_x$, there exists a tuple $(T,h_T,i,g)$ as in \emph{Construction~\ref{const:cm_points}} such that $x$ is in the image of $\Sh_{K_{T,g}}\to \Sh_K$.
\end{lemma}
\begin{proof}
    We only prove the more non-trivial implication, (2) implies (1), as this will be used below. In the process, we will actually prove the more precise assertion following the equivalence.
  
    Choose an embedding $F\hookrightarrow \Comp$ and let $x_{\Comp}\in \Sh_K(\Comp)$ be the associated $\Comp$-point. By lifting $x_{\Comp}$ further to $X\times G(\Adele^p_f)$ (see Remark~\ref{rem:unramified_uniformization}), we can identify $I_{x_{\Comp}}$ with the commutant in $G$ of the associated Hodge cocharacter $h\colon \mathbb{S}\to G_{\Real}$. 
    
    By replacing $F$ with a finite extension within $\Comp$, we can ensure that the map $I_x\to I_{x_{\Comp}}$ is an isomorphism. Therefore, if $T\subset I_x$ is a maximal torus, this allows us to view $T_{\Real}$ as a maximal torus in $G_{\Real}$, and in particular identifies it with its own commutant in $G_{\Real}$. This implies that $h$ factors through $T_{\Real}$, and hence shows that $x_{\Comp}$ is a CM point of $\Sh_K$, as desired.
\end{proof}

\subsection{The global \texorpdfstring{$\mu$}{mu}-ordinary locus and canonical lifts}
\label{sub:global_mu-ord}

Here, we will find some results about the $\mu$-ordinary loci in limpid ICMs. We will fix a neat unramified tuple $(G,\mathcal{G},X,K)$ with a limpid ICM $\Ss_K$ over $\Reg{E,(v)}$. As in the previous subsection, we assume $G=G^c$.

\begin{remark}
Some of the results here are already known for Shimura data of abelian type by work of Moonen~\cite{Moonen2004-cc} and Shankar--Zhou~\cite{Shankar2021-rp} (see also~\cite[\S 2.4]{Kisin2025-mz}), and in the exceptional case by work of Bakker--Shankar--Tsimerman~\cites{bst,bakker2025finitenessfunctionfieldvaluedpoints}.
\end{remark}

\begin{remark}
     As in \S\ref{sub:the_ordinary_locus}, we will assume that $\mu_v$ factors through a  maximal torus $\mathcal{T}\subset \mathcal{G}$ contained in a Borel subgroup. If $d = [E_v:\Rat_p] = [k(v):\Field_p]$, then we will set $\nu = -\sum_{i=0}^{d-1}\varphi^i(\mu_v)$: this is a cocharacter of $\mathcal{T}$ defined over $\Int_p$. Set $q_0 = p^d$. We will write $\mathcal{P}^+_\nu\subset \mathcal{G}$ for the parabolic subgroup whose Lie algebra is the direct sum of the $i$-weight spaces for the adjoint action of $\nu$ on $\Lie \mathcal{G}$, and $\mathcal{M}_\nu$ will be its Levi quotient, i.e., the centralizer of $\nu$ in $\mathcal{G}$.
\end{remark}

\begin{definition}
    [The global $\mu$-ordinary locus]
Let $\Ss_K$ be an ICM for $\Sh_K$ over $\Reg{E,(v)}$. Then the $\mu$\defnword{-ordinary locus} of $\Ss_K$ is the open subspace $\Ss^{\ord}_K\subset \Ss_K$, defined as in Notation~\ref{nota:schematic-ordinary-locus}.
\end{definition}

\begin{remark}
    \label{rem:ordinary_etale_realization}
By construction and Proposition \ref{prop:ordinary_locus_description}, when restricted to the completion $\widehat{\Ss}^{\ord}_K$, the syntomic realization map factors through the open substack $\BT[\mathcal{P}^+_\nu,\mu]{\infty} = \BT[\mathcal{G},\mbox{-}\mu_v,\mathrm{ord}]{\infty}\subset \BT[\mathcal{G},\mbox{-}\mu_v]{\infty}$.
\end{remark}

\begin{theorem}
    [Properties of the $\mu$-ordinary locus]
\label{thm:properties_mu_ord_locus} The following statements hold:
\begin{enumerate}
    \item The $\mu$-ordinary locus is open and fiberwise dense in $\Ss_{K,k(v)}$.
    \item For any perfect field $\kappa$ in characteristic $p$ and $x\in \Ss^{\mathrm{ord}}_{K,k(v)}(\kappa)$, the deformation space $\mathrm{Def}_x$ for $\Ss_K$ at $x$ admits a canonical structure of a tower
    \[
    \mathrm{Def}_x = \widehat{U}_{x,n}\to \cdots \to \widehat{U}_{x,0} = \Spf W(\kappa)
    \]
    where each map in the tower is relatively represented by a formal $p$-divisible group;
    \item In particular, $x$ admits a \defnword{canonical lift} $x^{\mathrm{can}}\in \Ss_K(W(\kappa))$ corresponding to the successive composition of identity sections of the tower.
\end{enumerate}
\end{theorem}
\begin{proof}
    The first assertion is clear from the definition and the smoothness of the map $\widehat{\Ss}_K\to \BT[\mathcal{G}^c,\mbox{-}\mu_v]{1}$. The next two assertions are immediate from the definition of an ICM and the results of \S\ref{sub:the_ordinary_locus}. 
\end{proof}

\begin{lemma}
    \label{lem:p-adic_comparison_canonical_lift}
For $x\in \Ss^{\ord}_K(\kappa)$, the specialization map $I_{\Int_p,x^{\mathrm{can}}_\eta}\to I_{\Int_p,x}$ induces an isomorphism $I_{p,x^{\mathrm{can}}_\eta}\isomto I_{p,x}$.
\end{lemma}
\begin{proof}
Suppose first that $\kappa$ is algebraically closed. In this case, we claim that both groups are isomorphic to the generic fiber $M_\nu$ of $\mathcal{M}_\nu$. It is easy to see using $p$-adic comparison that the specialization map $I_{p,x^{\mathrm{can}}_\eta}\to I_{p,x}$ is an embedding of algebraic groups, and so---assuming the claim---it must necessarily be an isomorphism. The result for arbitrary $\kappa$ now follows by working over $\overline{\kappa}$, and using Galois descent.

Now on to the claim: For $I_{p,x}$, we have the explicit description from Remark~\ref{rem:algebraic_groups_f-isocrystal_with_G-structure}, where, since $\kappa$ is algebraically closed, we can replace $\delta_x$ with a $\varphi$-conjugate $g_0^{-1}\delta_x\varphi(g_0)$ for some $g_0\in \mathcal{G}(W(\kappa))$, and assume that in fact $\delta_x = \varphi(\mu_v(p))^{-1}$. This gives an isomorphism of $I_{p,x}$  with the $\Rat_p$-group given by
\[
J:R\mapsto \left\{g\in G(W(\kappa)\otimes_{\Int_p}R):\;g^{-1}\varphi(\mu_v(p))^{-1}\varphi(g) = \varphi(\mu_v(p))^{-1}\right\}.
\]
Since $\varphi(\mu_v(p))\in M_\nu(\Rat_p)$ is central, there is an obvious inclusion of groups $M_\nu\hookrightarrow J$, and dimension considerations show that this inclusion is an isomorphism.

To find $I_{p,x^{\mathrm{can}}_\eta}$, we fix an algebraic closure $L$ of $W(\kappa)[\nicefrac{1}{p}]$, and write $\overline{x}\in \Sh_K(L)$ for the point above $x^{\mathrm{can}}_\eta$. We can fix a trivialization of $\mathbf{Et}_{K,p}\vert_{\overline{x}}$ to view it as a representation $\Gal(L/W(\kappa)[\nicefrac{1}{p}))\to \mathcal{G}(\Int_p)$. Unwinding the definition of the canonical lift shows that the image of this representation can be identified with $\nu(\Int_p^\times)\subset \mathcal{G}(\Int_p)$. The group $I_{p,x^{\mathrm{can}}_\eta}$ is therefore isomorphic to the commutant of $\nu(\Int_p^\times)$ in $G_{\Rat_p}$, which is precisely $M_\nu$. \end{proof}

\begin{remark}
    [The canonical lift and prime-to-$p$ Hecke correspondences]
\label{rem:hecke_and_canonical_lift}
Canonical lifts are preserved by prime-to-$p$ Hecke correspondences. More precisely, if $g\in G(\Adele_f^p)$ and $x\in \Ss^{\mathrm{ord}}_{K_p}(\kappa)$, then we have $x\cdot g\in \Ss^{\mathrm{ord}}_{K_p}(\kappa)$, and $x^{\mathrm{can}}\cdot g = (x\cdot g)^{\mathrm{can}}$. To see this, note simply that the syntomic realization on $\Ss_{K_p}$ is $G(\Adele_f^p)$-invariant. In particular, this tells us that, for $x,y\in \Ss^{\mathrm{ord}}_{K}(\kappa)$, the canonical maps 
\[
\mathrm{QIsog}^p_K(x^{\mathrm{can}}_\eta,y^{\mathrm{can}}_\eta)\isomfrom\mathrm{QIsog}^p_K(x^{\mathrm{can}},y^{\mathrm{can}})\isomto \mathrm{QIsog}^p_K(x,y)
\]
are all bijections.
\end{remark}

\begin{remark}
    [The analytic fiber of the canonical Frobenius lift]
\label{rem:analytic_fiber_Phi}
Recall from Proposition~\ref{prop:frobenius_liftings_BTP} that we have a canonical $q_0$-Frobenius lift $\widetilde{\Phi}$ on $\BT[\mathcal{P}_\nu^+,\mbox{-}\mu_v]{n} \simeq \BT[\mathcal{G},\mbox{-}\mu_v,\ord]{n}$. For any $p$-complete $\Reg{E_v}$-algebra $R$, we have the following commuting diagram:
\[
\begin{tikzcd}
\BT[\mathcal{P}_\nu^+,\mbox{-}\mu_v]{n}(R) \arrow[r, "T_{\et}"] \arrow[d, "\tilde{\Phi}"'] & \mathrm{Loc}_{\mathcal{P}_\nu^+(\Int/p^n\Int)}(R[\nicefrac{1}{p}]) \arrow[d, "\mathrm{int}(\nu(p)^{-1})"] \\
\BT[\mathcal{P}_\nu^+,\mbox{-}\mu_v]{n}(R) \arrow[r, "T_{\et}"']               & \mathrm{Loc}_{\mathcal{P}_\nu^+(\Int/p^n\Int)}(R[\nicefrac{1}{p}])
\end{tikzcd}
\]
Here, the right vertical arrow is induced by the endomorphism of $\mathcal{P}_\nu^+(\Int/p^n\Int)$ induced by the adjoint action of $\nu(p)^{-1}$. Indeed, this is essentially immediate from the construction of $\tilde{\Phi}$ in \emph{loc.\@ cit.}
\end{remark}

\begin{remark}
    [Global $p$-Hecke correspondences]
\label{rem:p-hecke_correspondences}
Fix $h\in G(\Rat_p)$: Associated with this and $K_h = K\cap hKh^{-1}$, we have the $p$-Hecke correspondence $\Sh_{K_h}\xrightarrow{(s_h,t_h)} \Sh_K\times \Sh_K$ where $t_h$ is induced from the inclusion $K_h\subset K$ and $s_h$ from conjugation by $h^{-1}$. This latter map can also be described as follows: Over $\mathrm{Loc}_{\mathcal{G}(\Int_p)}$, we have the finite \'etale cover given by the stack quotient
\[
\alpha_h\colon [\underline{\mathcal{G}(\Int_p)h\mathcal{G}(\Int_p)}/\underline{\mathcal{G}(\Int_p)}]\to [*/\underline{\mathcal{G}(\Int_p)}] = \mathrm{Loc}_{\mathcal{G}(\Int_p)}.
\]
The pullback of $\alpha_h$ over $\Sh_K$ along the $p$-adic \'etale realization map $\Sh_K\to \mathrm{Loc}_{\mathcal{G}(\Int_p)}$ is now canonically isomorphic to $t_h\colon \Sh_{K_h}\to \Sh_K$.
\end{remark}

\begin{lemma}
[Canonical pseudo-Frobenius lift on the $\mu$-ordinary locus]
    \label{lem:canonical_frobenius_lift_SsK_ord}
There exists a canonical morphism $\widehat{\Phi}^{\ord}_K\colon \widehat{\Ss}^{\ord}_K\to \widehat{\Ss}^{\ord}_K$ such that the following diagram commutes:
\[
\begin{tikzcd}
\widehat{\Ss}^{\ord}_K \arrow[r, "\varpi"] \arrow[d, "\widehat{\Phi}^{\ord}_K"'] & \BT[\mathcal{G},\mbox{-}\mu_v,\mathrm{ord}]{\infty} \arrow[d, "\widetilde{\Phi}"] \\
\widehat{\Ss}^{\ord}_K \arrow[r, "\varpi"']               & \BT[\mathcal{G},\mbox{-}\mu_v,\mathrm{ord}]{\infty}.
\end{tikzcd}
\]
Here, the right vertical arrow is the inverse limit over $n$ of the maps constructed in \emph{Proposition~\ref{prop:frobenius_liftings_BTP}}. The reduction-mod-$p$ of $\widehat{\Phi}^{\ord}_K$ is a purely inseparable finite flat map.
\end{lemma}
\begin{proof}
  In this proof, for any stack $X$ over $E$, we will write $X^{\an}$ for the analytification of its base-change over $E_v$. 

   Set $K_\nu = K\cap \nu(p)K\nu(p)^{-1}$, and consider $t_\nu \defn t_{\nu(p)}\colon \Sh_{K_\nu}\to \Sh_K$. Over $\widehat{\Ss}^{\ord}_{K,\eta}$, the tube around the mod-$p$ ordinary locus, the $p$-adic \'etale realization map factors through $\mathrm{Loc}^{\an}_{\mathcal{P}^{+}_\nu(\Int_p)}$. The description in Remark~\ref{rem:p-hecke_correspondences} therefore shows that $t_\nu\vert_{\widehat{\Ss}^{\ord}_{K,\eta}}$ is canonically isomorphic to the base-change of
   \[  [\underline{\mathcal{G}(\Int_p)\nu(p)\mathcal{G}(\Int_p)}/\underline{\mathcal{G}(\Int_p)}]\times_{\mathrm{Loc}_{\mathcal{G}(\Int_p)}}\mathrm{Loc}^{\an}_{\mathcal{P}^{+}_\nu(\Int_p)} \simeq [\underline{\mathcal{G}(\Int_p)\nu(p)\mathcal{G}(\Int_p)}/\underline{\mathcal{P}^+_\nu(\Int_p)}]^{\an}
   \]
   along $\widehat{\Ss}^{\ord}_{K,\eta}\to \mathrm{Loc}^{\an}_{\mathcal{P}^{+}_\nu(\Int_p)}$. Now, since 
   \[
   \mathcal{P}^+_\nu(\Int_p)\nu(p)\mathcal{P}^+_\nu(\Int_p) = \nu(p)\mathcal{P}^+_\nu(\Int_p),
   \]
   we have a canonical section
   \[
   \mathrm{Loc}^{\an}_{\mathcal{P}^{+}_\nu(\Int_p)}\isomto[\underline{\nu(p)\mathcal{P}^+_\nu(\Int_p)}/\underline{\mathcal{P}^+_\nu(\Int_p)}]^{\an}\to  [\underline{\mathcal{G}(\Int_p)\nu(p)\mathcal{G}(\Int_p)}/\underline{\mathcal{P}^+_\nu(\Int_p)}]^{\an}.
   \]
This shows that we have a canonical section
   \begin{align}\label{eqn:ordinary_p-hecke_section}
   \widehat{\Ss}^{\ord}_{K,\eta}\to \Sh^{\an}_{K_\nu}\times_{t_\nu^{\an},\Sh_K^{\an}}\widehat{\Ss}^{\ord}_{K,\eta}.
   \end{align}
   Composing this with the map $s_\nu^{\an}\colon \Sh^{\an}_{K_\nu}\to \Sh^{\an}_K$ gives us a map $\Phi':\widehat{\Ss}^{\ord}_{K,\eta}\to \Sh^{\an}_K$ such that the diagram
   \[
     \begin{tikzcd}
\widehat{\Ss}^{\ord}_{K,\eta} \arrow[r] \arrow[d, "\Phi'"] & \mathrm{Loc}_{\mathcal{P}^+_\nu(\Int_p)}^{\an} \arrow[d, "\mathrm{int}(\nu(p)^{-1})"] \\
\Sh^{\an}_K \arrow[r]               & \mathrm{Loc}_{\mathcal{P}^+_\nu(\Int_p)}^{\an} 
\end{tikzcd}
   \]
   commutes. By Remark~\ref{rem:analytic_fiber_Phi}, the composition of the top horizontal with the right vertical arrow factors as
   \[
   \widehat{\Ss}^{\ord}_{K,\eta}\xrightarrow{\varpi_\eta}\BT[\mathcal{G},\mbox{-}\mu_v,\mathrm{ord}]{\infty,\eta}\xrightarrow{\tilde{\Phi}_\eta}\BT[\mathcal{G},\mbox{-}\mu_v,\mathrm{ord}]{\infty,\eta}\to \mathrm{Loc}_{\mathcal{P}^+_\nu(\Int_p)}^{\an}.
   \]
  The argument from Proposition~\ref{prop:uniqueness_maps} now shows that $\Phi'$ factors through $\widehat{\Ss}^{\ord}_{K,\eta}$ and is the adic fiber of an endomorphism $\widehat{\Phi}^{\ord}_K$ of $ \widehat{\Ss}^{\ord}_{K}$ such that we have a commuting diagram as claimed.

  That $\widehat{\Phi}^{\ord}_K$ is finite flat and purely inseparable mod-$p$ is immediate from the formal \'etaleness of the syntomic realization and the fact that $\tilde{\Phi}$ induces the $q_0$-Frobenius relative Frobenius map on complete local rings mod-$p$.
\end{proof}

\begin{notation}
    If $K^{',p}\subset K^p$ is a finite index subgroup and $x\in \widehat{\Ss}_K^{\ord}(R)$ lifts to $x'\in \widehat{\Ss}_{K'}^{\ord}(R)$, then $\widehat{\Phi}^{\ord}_{K'}(x')$ is a lift of $\widehat{\Phi}^{\ord}_{K}(x)$. For this reason, we will simply set $\Phi(x) \defn \widehat{\Phi}^{\ord}_K(x)$.
\end{notation}

\begin{remark}
[Fixed points of $\Phi$]
\label{rem:fixed_points_Phi}
 At this point, we do not know yet that $\widehat{\Phi}^{\ord}_K$ is a lift of $q_0$-Frobenius. We will verify that below in Theorem~\ref{thm:ordinary_locus_frobenius_lift}. However, after fixing an algebraic closure $\overline{k(v)}$ for $k(v)$, for any $m\geqslant 1$ with $q_m = q_0^m$, we can still make sense of the fixed point set
 \[
 {}^\Phi\Ss^{\ord}_{K}(\Field_{q_m})\defn \left\{x\in \Ss^{\ord}_K(\overline{k(v)}):\;\Phi^m(x) = x\right\}.
 \]
 Since $\widehat{\Phi}^{\ord}_K$ is a purely inseparable homeomorphism mod-$p$, this fixed point set is \emph{finite}.
\end{remark}

\begin{remark}
[Interaction with canonical lifts]
\label{rem:interaction_Phi_canonical_lift}
It is not difficult to see that $\widetilde{\Phi}$ preserves canonical lifts, and therefore so does $\widehat{\Phi}^{\ord}_K$. More precisely, for $x\in \Ss_K(\kappa)$, one has that $\Phi(x^{\mathrm{can}})$ is the canonical lift of $\Phi(x)$.
\end{remark}

\begin{proposition}
    [The pseudo-Frobenius quasi-isogeny]
\label{prop:frobenius_q-isogeny}
Suppose that $\kappa$ is field over $k(v)$ and that $x\in \Ss^{\ord}_K(\kappa)$.  Then there exists a canonical element
\[
\varphi_x\in \mathrm{QIsog}_K(\Phi(x^{\mathrm{can}})_\eta,x^{\mathrm{can}}_\eta)(W(\kappa)[\nicefrac{1}{p}])
\]
with the following property: If $K^{',p}\subset K^p$ is a finite index subgroup, and $y\in \Ss^{\ord}_{K'}(\kappa')$ is a lift of $x$ over an extension $\kappa'/\kappa$, then $\varphi_x$ is carried to $\varphi_y$ under the natural pullback map
\[
\mathrm{QIsog}_K(\Phi(x^{\mathrm{can}})_\eta,x^{\mathrm{can}}_\eta)(W(\kappa)[\nicefrac{1}{p}])\to \mathrm{QIsog}_{K'}(\Phi(y^{\mathrm{can}})_\eta,y^{\mathrm{can}}_\eta)(W(\kappa)[\nicefrac{1}{p}]).
\]
\end{proposition}
\begin{proof}
    This follows from the proof of Lemma~\ref{lem:canonical_frobenius_lift_SsK_ord}, which shows that $\varphi_x$ corresponds to the image of $x^{\mathrm{can}}_\eta$ under the section~\eqref{eqn:ordinary_p-hecke_section}.
\end{proof}

\begin{notation}
[The pseudo-Frobenius element]
    Suppose that $x\in {}^\Phi\Ss_K^{\ord}(\Field_{q_m})$ and that $\Phi^m(x) = x$. Set
    \[
    {}^\Phi\gamma_x = \varphi_x\circ \varphi_{\Phi(x)}\circ\cdots \circ \varphi_{\Phi^{m-1}(x)}\in \Aut^\circ_K(x^{\mathrm{can}}_\eta) = I_{x^{\mathrm{can}}_\eta}(\Rat).
    \]
    Fix a lift $\tilde{x}\in \Ss_{K_p}(\overline{k(v)})$ of $x$. This fixes isomorphisms
    \[
I_{\tilde{x}^{\mathrm{can}}_\eta,\Aff^p_f}\isomto I_{\tilde{x},\Aff^p_f}\isomto G_{\Adele_f^p}.
    \]
    In particular, we can view ${}^\Phi\gamma_x$ as an element of $G(\Adele_f^p)$. Let
    \[
     {}^\Phi I_{x^{\mathrm{can}}_\eta}\subset I_{x^{\mathrm{can}}_\eta}\;;\; {}^\Phi I_{\Adele_f^p,x}\subset G_{\Adele_f^p}\;;\; {}^\Phi I_{\ell,x}\subset G_{\Rat_\ell}
    \]
    be the commutants of (the image of) ${}^\Phi\gamma_x$. For every $\ell\neq p$, we have a canonical realization map ${}^\Phi I_{x^{\mathrm{can}}_\eta}\otimes\Rat_\ell\to {}^\Phi I_{\ell,x}$.
\end{notation}

\begin{lemma}
    \label{lem:kisin_tate_argument}
For every $\ell\neq p$, ${}^\Phi I_{\ell,x}$ is a reductive group over $\Rat_\ell$ of the same (geometric) rank as $G$, and we have $\mathrm{rank}({}^\Phi I_{x^{\mathrm{can}}_\eta}) = \mathrm{rank}(I_{x^{\mathrm{can}}_\eta}) = \mathrm{rank}(G)$.
\end{lemma}
\begin{proof}
    Consider the maps
   \begin{equation*}
    G(\Adele_f^p) \xrightarrow{g\mapsto \tilde{x}^{\mathrm{can}}_\eta\cdot g}\Sh_{K_p}(W(\ov{\kappa})[\nicefrac{1}{p}])\to \Sh_K(W(\ov{\kappa})[\nicefrac{1}{p}])
\end{equation*}
and
\begin{equation*}
    G(\Adele_f^p) \xrightarrow{g\mapsto \tilde{x}\cdot g}\Ss_{K_p}(\ov{\kappa})\to \Ss_K(\ov{\kappa}).
\end{equation*}
   By Remark~\ref{rem:hecke_and_canonical_lift}, the first map is obtained from the second by taking canonical lifts and then going to the generic fiber. By construction it is invariant under the left action of $I_x(\Int_{(p)})$ on $G(\Adele_f^p)$, and on the right by $K^p$. If we restrict to the subgroup ${}^\Phi I_{\Adele_f^p,x}(\bb{A}_f^p)$ of $G(\Adele_f^p)$, then the image of the second map actually lands in ${}^\Phi\Ss^{\ord}_K(\Field_{q^m})$. In particular, this implies that the image of ${}^\Phi I_{\ell,x^{\mathrm{can}}_\eta}(\Rat_\ell)$ in $\Sh_K(W(\kappa)[\nicefrac{1}{p}])$ is finite. In turn, this shows that, for any $\ell\neq p$, and any compact open $U_\ell \subset {}^\Phi I_{\ell,x^{\mathrm{can}}_\eta}(\bb{Q}_\ell)$, the set
   \begin{equation}\label{eqn:finite_double_coset}
   {}^\Phi I_{x^{\mathrm{can}}_\eta}(\Rat)\backslash {}^\Phi I_{\ell,x^{\mathrm{can}}_\eta}(\Rat_\ell)/U_\ell
   \end{equation}
   is finite. 

  Since $I_{x^{\mathrm{can}}_\eta}/Z(G)$ is compact over $\Real$, the element ${}^\Phi\gamma_x$ has to be semi-simple. Then, as ${}^\Phi I_{\ell,x^{\mathrm{can}}_\eta}$ is the centralizer in $G_{\Rat_\ell}$ of ${}^\Phi \gamma_{\ell,x}$, we see that ${}^\Phi I_{\ell,x}$ is a  reductive group over $\Rat_\ell$. 
  
  Let $\ell$ be a prime such that the characteristic polynomial of ${}^\Phi\gamma_x$ (with respect to some faithful representation of $I_{x^{\mathrm{can}}_\eta}$) and $G_{\Rat_\ell}$ are both split. Then in fact ${}^\Phi I_{\ell,x}$ is a \emph{split} reductive group over $\Rat_\ell$. Therefore, as in~\cite[Corollary (2.1.7)]{kisin:abelian}, we find from the finiteness of~\eqref{eqn:finite_double_coset} that the map ${}^\Phi I_{x^{\mathrm{can}}_\eta}\otimes\Rat_\ell\to {}^\Phi I_{\ell,x}$ contains the connected component ${}^\Phi I^\circ_{\ell,x}$ in its image for this choice of $\ell$. Therefore, ${}^\Phi I_{x^{\mathrm{can}}_\eta}$---and \emph{a fortiori} $I_{x^{\mathrm{can}}}$---has the same rank as $G$.
\end{proof}

\begin{theorem}
    [Canonical lifts are CM]
\label{thm:canonical_lifts_are_CM}
The canonical lift $x^{\mathrm{can}}_\eta$ of $\Sh_K(W(\Field_{q_m})[\nicefrac{1}{p}])$ is a CM point.
\end{theorem}
\begin{proof}
    This is immediate from Lemmas~\ref{lem:cm_point_criterion} and~\ref{lem:kisin_tate_argument}. 
\end{proof}

\begin{remark}
[Torsion points]
    There is a natural notion of a \emph{torsion point} of the deformation space $\mathrm{Def}_x$, and the argument in Theorem~\ref{thm:canonical_lifts_are_CM} can be used to show that each of these torsion points is also a CM lift of $x$.
\end{remark}

\begin{lemma}
    [Unramifiedness of CM liftings]
\label{lem:canonical_lift_unramified}
Suppose that we have a map of Shimura data $(T,\{h_T\}) \to (G,X)$ with $T\subset G$ a maximal torus such that $x^{\mathrm{can}}_\eta$ is in the image of $\Sh_{K_T}\otimes_EE_{T,w}$ for some place $w\vert v$ of the reflex field $E_T$. Then $w$ is a split place over $E$, and the reflex norm $r_{\mu_T}$ satisfies $r_{\mu_T}(\Reg{E_T,w}^\times)\subset K_{T,p} = T(\Rat_p)\cap K_p$. In particular, the integral model $\Ss_{K_T,(w)}$ is \'etale over $\Reg{E,(v)}$.
\end{lemma}
\begin{proof}
     After replacing $m$ with a sufficiently large multiple, we can assume that $x$ is in the image of $\Ss_{K_T,(w)}(\Field_{q_m})$ for some place $w\vert p$ of the reflex field $E_T$, and we can also assume that $I_{x^{\mr{can}}_\eta}$ does not change if we replace $W(\Field_{q_m})[\nicefrac{1}{p}]$ with an algebraically closed extension. We now have an embedding $T\hookrightarrow I_{x^{\mr{can}}_\eta}$, yielding an embedding $T_{\Rat_p}\hookrightarrow I_{p,x}$. As seen in the proof of Lemma~\ref{lem:p-adic_comparison_canonical_lift}, we can use conjugation by an element $g_0\in \mathcal{G}(W(\overline{k(v)}))$ to identify $I_{p,x}$ with a subgroup of $M_\nu$. Moreover, under this identification, the cocharacter $\mu_v$ of $M_\nu$ is carried to a central cocharacter of $I_{p,x}$, and hence to a cocharacter of $T_{\Rat_p}$ defined over $E_v$, that splits the Hodge filtration for the canonical lift. As such, it has to be equal to the cocharacter $\mu_{T,w}$ of $T_{E_{T,w}}$ arising from the Deligne cocharacter $h_T$. This tells us that $\mu_{T,w}$ is already defined over $E_v$, and hence that $E_v = E_{T,w}$. It also tells us that we in fact have $r_{\mu_T}(\Reg{E_{T,w}}^\times)\subset \mathcal{G}(\Int_p)$. The last assertion now follows from Remark~\ref{rem:unramifiedness_criterion_cm_sv}.
\end{proof}

\begin{notation}
    If  $X$ is an algebraic stack over $\Field_p$, $R$ is an $\Field_p$-algebra, $x\in X(R)$, we will write $x^{(p)}\in X(R)$ for the Frobenius conjugate $\varphi^*x$ of $x$. As usual, we will use the notation $x^{(p^r)}$ for the $r$-th iterate of this operation.
\end{notation}

\begin{theorem}
    [Canonical Frobenius lift on ordinary locus]
\label{thm:ordinary_locus_frobenius_lift}
The endomorphism $\widehat{\Phi}^{\ord}_K$ of $\widehat{\Ss}^{\ord}_K$ is a lift of the $q_0$-Frobenius endomorphism of $\Ss^{\ord}_{K,k(v)}$. In particular, we have ${}^\Phi\Ss_K^{\ord}(\Field_{q_m}) = \Ss_K^{\ord}(\Field_{q_m})$, and for any $x\in \Ss_K^{\ord}(\Field_{q_m})$, there is a canonical element $\gamma_x\in I_{x^{\mathrm{can}}_\eta}(\Rat)$ carried, for every $\ell$, to $\gamma_{\ell,x}$ under the maps
\[
I_{x^{\mathrm{can}}_\eta}(\Rat)\to I_{\ell,x^{\mathrm{can}}_\eta}(\Rat_\ell)\to I_{\ell,x}(\Rat_\ell).
\]
\end{theorem}
\begin{proof}
     For the first assertion, it is enough to know that the action of $\widehat{\Phi}^{\ord}_K$ on the underlying topological space of $\Ss_{K,k(v)}$ agrees with that of $q_0$-Frobenius. This follows from Lemma~\ref{lem:canonical_lift_unramified} and the last sentence of Remark~\ref{rem:unramifiedness_criterion_cm_sv}.

     Set $\varphi_x$ be the quasi-isogeny from Proposition~\ref{prop:frobenius_q-isogeny}. We now see that it is in fact a quasi-isogeny from $x^{(q_0),\mathrm{can}}_\eta$ to $x^{\mathrm{can}}_\eta$. We will take $\gamma_x \defn {}^\Phi\gamma_x\in I_{x^{\mathrm{can}}_\eta}(\Rat)$. By the property of $\varphi_x$ explained in \emph{loc.\@ cit.}, for any finite index subgroup $K^{',p}\subset K^p$ with $K' = K^{',p}K_p$, and any lift $y\in \Ss_{K'}(\ov{k(v)})$ of $x$, $\gamma_x$ lifts to the quasi-isogeny
    \[
    \varphi_y\circ \cdots\circ \varphi_{y^{(q_0^{r-1})}}\in \mathrm{QIsog}(y^{(q_m),\mathrm{can}},y^{\mathrm{can}})(W(\ov{k(v)})).
    \]
    This implies that the action of $\gamma^{-1}_x$ on the fiber of the map $\Ss_{K'}(\ov{k(v)})\to \Ss_K(\ov{k(v)})$ carries $y$ to $y^{(q_m)}$, implying precisely that, for any $\ell\neq p$, the $\ell$-adic realization of $\gamma_x$ is the \emph{geometric} $q_m$-Frobenius element of $\Gal(\ov{k(v)}/\Field_{q_m})$.

    To finish, we need to observe that $\gamma_x$ is carried to $\gamma_{p,x}$ under the $p$-adic realization, but this is immediate from its construction.
\end{proof}

\begin{theorem}
    [$\mu$-ordinary Tate conjecture]
\label{thm:tate_conjecture}
For any $x\in \Ss^{\mathrm{ord}}_K(\Field_{q_m})$, and all primes $\ell$, the realization map
\begin{equation}\label{eqn:mu_ord_tate_conjecture_map}
     I_{x^{\mathrm{can}}_\eta}\otimes_{\Rat}\Rat_\ell \to I_{\ell,x^{\mathrm{can}}_\eta}\isomto I_{\ell,x}
    \end{equation}
is an isomorphism.
\end{theorem}
\begin{proof}
     Let $\gamma_x\in I_{x^{\mathrm{can}}_\eta}(\Rat)$ be the $q$-Frobenius quasi-isogeny from Theorem~\ref{thm:ordinary_locus_frobenius_lift}. The proof of Theorem~\ref{thm:canonical_lifts_are_CM} unwinds to show that, for any maximal torus $T\subset I_{x^{\mathrm{can}}_\eta}$, $T$ can be identified (up to $G(\Rat)$-conjugacy) with a maximal torus of $G$. Using this, the element $\gamma_x\in T(\Rat)$ can be identified with an element $\gamma_0\in G(\Rat)$. On the other hand, for $\ell\neq p$, we can use a lift $\tilde{x}\in \Ss_{K_p}(\overline{k(v)})$ of $x$ to identify $I_{\ell,x}$ with a subgroup of $G_{\Rat_\ell}$, namely the commutant of the image $\gamma_\ell\in G(\Rat_\ell)$ of $\gamma_{\ell,x}$. Under these identifications, the image of $\gamma_0\in G(\Rat)$ in $G(\Rat_\ell)$ is $G(\Rat_\ell)$-conjugate to $\gamma_\ell$. This shows that the groups $I_{\ell,x}$ are all reductive, and that their dimension is independent of $\ell$. For $\ell=p$, the same argument also applies, except that we use the knowledge that $I_{p,x}\otimes W(\kappa)$ can be identified with the centralizer of the image $\gamma_p\in G(W(\kappa)[\nicefrac{1}{p}])$ of $\gamma_{p,x}$. 
     
     The proof of Lemma~\ref{lem:kisin_tate_argument} showed that the map~\eqref{eqn:mu_ord_tate_conjecture_map} contains the connected component $I^{\circ}_{\ell,x} = {}^\Phi I^{\circ}_{\ell,x}$ for a particular choice of $\ell$, and we now see that it must contain this component for \emph{every} choice of $\ell$.

    If we replace $m$ with a multiple $rm$, and denote by $x'\in \Ss_K(\Field_{q_{rm}})$ the $\Field_{q_{rm}}$-point underlying $x$, then we have $\gamma_{x'} = \gamma^r_x$. We can arrange for $r$ to be chosen such that the Zariski closure of the subgroup $\langle \gamma^{r}_x\rangle\subset T(\Rat)$ generated by $\gamma^r_x$ is connected. This implies that the groups $I_{\ell,x'}$ are all connected, and hence the maps~\eqref{eqn:mu_ord_tate_conjecture_map} are isomorphisms when $x$ is replaced by $x'$. We now deduce the result for $x$ by noting that $I_{x^{\mathrm{can}}_\eta}\subset I_{x^{',\mathrm{can}}_\eta}$ (resp.\@ $I_{\ell,x}\subset I_{\ell,x'})$ is the commutant of $\gamma_x$ (resp.\@ $\gamma_{\ell,x}$).
\end{proof}

\begin{remark}
[$\ell$-independence of $q$-Frobenius]
\label{rem:ell_indepdencce_mu_ordinary}
    As in~\cite[{}5.1.3]{Kisin2025-mz}, write $\mathrm{Conj}_G$ for the GIT quotient of $G$ acting on itself by conjugation: We have a canonical map $G\xrightarrow{g\mapsto [g]} \mathrm{Conj}_G$, and, for any algebraically closed field $F$, this map identifies $\mathrm{Conj}_G(F)$ with the set of semisimple conjugacy classes in $G(\kappa)$.  The proof above shows that the elements $\gamma_{\ell,x}\in I_{\ell,x}(\Rat_\ell)$ are semisimple for all $\ell$, yielding conjugacy classes $[\gamma_\ell]\in \mathrm{Conj}_G(\Rat_\ell)$ (for $\ell\neq p$) and $[\gamma_p]\in \mathrm{Conj}_G(W(\kappa)[\nicefrac{1}{p}])$, and that there is an elliptic element $\gamma_0\in G(\Rat)$ such that $[\gamma_0]\in \mathrm{Conj}_G(\Rat)$ is carried to $[\gamma_\ell]$ for all primes $\ell$; see~\cite[Corollary (2.3.1)]{Kisin2017-qa}.
\end{remark}

\subsection{Ascent for ICMs}
\label{sub:ascent_for_icms} 

Here, we will combine the results of \S\ref{sub:central_isogenies_BT} with the theory of the canonical CM lifting from \S\ref{sub:global_mu-ord} to prove Theorem~\ref{introthm:ascent}.

\begin{setup}
\label{setup:central_covers_shimura}
   Suppose that $(G,\mathcal{G},X,K)\to (\overline{G},\overline{\mathcal{G}},\overline{X},\overline{K})$ is a map of unramified tuples such that $ {G}\to \overline{{G}}$ is a surjective map of reductive group schemes with central kernel and where $\ov{G}=\ov{G}^c$. Let $\overline{E}$ be the reflex field for $(\overline{G},\overline{X})$, and let $v\vert p$ be a place for $E$ lying above a place $\overline{v}$ for $\overline{E}$.
\end{setup}

\begin{remark}
    [CM points and central covers]
\label{rem:cm_points_central_covers}
Suppose that we have $\overline{g}\in \overline{G}(\Adele_f^p)$ and $(T,i,h_T)$ as in Construction~\ref{const:cm_points} with $\overline{T}\subset \overline{G}$ the image of $T$. Then we can consider the fiber product
\[
\Sh_K\times_{\Sh_{\overline{K}}}\Sh_{\overline{K}_{\overline{T},\overline{g}}}\to \Sh_{\overline{K}_{\overline{T},\overline{g}}}.
\]
Examining $\Comp$-points, and using Remark~\ref{rem:unramified_uniformization}, one finds that, if non-empty, this fiber product is a disjoint union of CM Shimura varieties of the form $\Sh_{K_{T',g'}}$, where $i'\colon (T',h_{T'})\hookrightarrow (G,X)$ is a $\overline{\mathcal{G}}_{(p)}(\Int_{(p)})_+$-conjugate of $i$. 
\end{remark}

\begin{remark}
    [Unramifiedness of canonical lifts along central covers]
\label{rem:cm_points_central_covers_unramified}
Suppose that $\Ss_{\overline{K}}$ is a limpid ICM for $\Sh_{\overline{K}}$ over $\Reg{\overline{E},\overline{v}}$, and we have $x\in \Ss^{\ord}_{\overline{K}}(\kappa)$ with $\kappa$ an algebraic closure of $k(\overline{v})$. Let $x^{\mathrm{can}}\in \Ss_{\overline{K}}(W(\kappa))$ be the canonical lift. By Theorem~\ref{thm:canonical_lifts_are_CM}, we can find $(\overline{T},h_{\overline{T}},g)$ with $\overline{T}\subset \overline{G}$ a maximal torus such that the associated map $\Sh_{\overline{K}_{\overline{T}}}\to \Sh_{\overline{K}}$ contains the generic fiber $x^{\mathrm{can}}_\eta$ in its image. 

Suppose that $x^{\mathrm{can}}_\eta$ is in the image of $\Sh_K$. Let $T\subset G$ be the pre-image of $\overline{T}$; then there is a unique lift $h_T\colon \mathbb{S}\to T_{\Real}$ of $h_{\overline{T}}$ that lies in $X$, and Remark~\ref{rem:cm_points_central_covers} tells us that all pre-images of $x^{\mathrm{can}}_\eta$ in $\Sh_K$ are in the image of $\Sh_{K_{T',g'}}$ for some $\overline{\mathcal{G}}_{(p)}(\Int_{(p)})_+$-conjugate $(T',h_{T'})$ of $(T,h_T)$. By Lemma~\ref{lem:canonical_lift_unramified}, the reflex field $E_T = E_{\overline{T}}$ is unramified over $\overline{v}$. Moreover, for any place $w\vert v$ of $E_T$, one can extend $\Sh_{K_T}\otimes_{E_T}E_{T,w}$ to a finite \'etale scheme over $\Reg{E_T,w}$: It suffices to show $r_{\mu_T}(\Reg{E_T,w}^\times)\subset K_p$, which follows because, by the same lemma, we already know that the image in $\overline{G}(\Rat_p)$ is contained in $\overline{K}_p$, and this implies that the image lies in $K_p Z(G)(\mathbb{Q}_p)$, which has $K_p$ as its unique maximal compact open subgroup.
\end{remark}

We can now prove Theorem~\ref{introthm:ascent} whose statement we now recall:
\begin{theorem}
    [Ascent for ICMs]
\label{thm:ascent}
Suppose that $p>2$ and that $\Sh_{\overline{K}}$ admits a limpid ICM $\Ss_{\overline{K}}$ over $\Reg{\overline{E},\overline{v}}$. Then $\Sh_K$ admits a limpid ICM $\Ss_K$ over $\Reg{E,(v)}$ and the natural map $\Ss_K\to \Ss_{\overline{K}}$ is finite \'etale.
\end{theorem}

\begin{proof}
 Let $\Ss_K$ the normalization of $\Ss_{\overline{K}}\otimes_{\Reg{\overline{E},(\overline{v})}}\Reg{E,(v)}$ in $\Sh_K$. We claim that this is the desired ICM for $\Sh_K$.
  
  We begin by checking that $\Ss_K\to \Ss_{\overline{K}}$ is finite \'etale and that it is an apertile integral model for $\Sh_K$. To do this, we apply the criteria from Proposition~\ref{prop:ascending_along_central_covers}: the non-trivial ones to check are (5) and (6). We need to check that, for all  $x\in \Ss_{\overline{K}}(\kappa)$ with canonical lift $x^{\mr{can}}$,
  \[
   \Ss_K\times_{\Ss_{\overline{K}},x^{\mr{can}}}\Spec W(\kappa)\to \Spec W(\kappa)
  \]
  is a totally split finite \'etale cover, all of whose points are crystalline.  For this, note that Remark~\ref{rem:cm_points_central_covers_unramified} tells us that the pre-image of the canonical lift $x^{\mathrm{can}}_\eta$ in $\Sh_K$ consists of CM points defined over an extension of $E$ in which $p$ is unramified. Taking Remark~\ref{rem:tori_crystallinity} into account, this simultaneously verifies both (5) and (6) from Proposition~\ref{prop:ascending_along_central_covers}. 

  What we have seen so far already tells us that $\Ss_K$ is an apertile integral model satisfying the Serre--Tate property. Moreover, if $\Ss_{\overline{K}}$ is a limpid ICM, for any complete mixed characteristic discrete valuation field $F$ over $E_v$ with perfect residue field, $\Ss_K(\Reg{F})\subset \Sh_K(F)$ contains all the potentially crystalline points: this follows from the corresponding assertion for $\Ss_{\overline{K}}$. 

  It still remains to be checked that the local systems $\mathbf{Et}_{K,\ell}$ extend over $\Ss_K$. But this follows by simply replacing $K$ with finite index subgroups $K'\subset K$ with $K'_p = K_p$ and noting that our proof has actually constructed ICMs $\Ss_{K'}$ for all such $K'$.
\end{proof}

\subsection{Abstract twisting and descent for ICMs}
\label{sub:abstract_twisting}

Here, we give a different approach to the twisting construction of Kisin~\cite{kisin:abelian} that works around the lack of a moduli interpretation for general ICMs. We once again fix an unramified tuple $(G,\mathcal{G},X,K)$ with a limpid ICM $\Ss_K$ over $\Reg{E,(v)}$, and also maintain our assumption that $G = G^c$.

\begin{remark}
    [Adjoint action]
\label{rem:adjoint_action}
Let $G^{\ad}(\Rat)_1\subset G^{\ad}(\Rat)$ be the subgroup of elements preserving the $G(\Real)$-conjugacy class $X$ under conjugation. Given $\gamma\in G^{\ad}(\Rat)_+$, for any $K\subset G(\Adele_f)$, with $K^\gamma = \mathrm{int}(\gamma)(K)$, we have an isomorphism $i_\gamma\colon \Sh_K \isomto  \Sh_{K^{\gamma}}$ given on $\Comp$-points by
\[
[(x,g)]\mapsto [(\gamma\cdot x,\mathrm{int}(\gamma)(g))].
\]
When these Shimura varieties admit ICMs over $\Reg{E,(v)}$, then this map also extends. Our goal here is to give a more concrete description of this extension under certain conditions. For convenience, we will actually consider the isomorphism $i_\gamma\colon \Sh\isomto \Sh$ induced by taking the inverse limit over all levels $K$. 
\end{remark}

\begin{remark}
[Adjoint action on ICMs]
\label{rem:adjoint_action_icms}
If $\gamma\in \mathcal{G}^{\ad}_{(p)}(\Int_{(p)})_1$, where
\[
 \mathcal{G}^{\ad}_{(p)}(\Int_{(p)})_1 =  \mathcal{G}^{\ad}_{(p)}(\Int_{(p)})\cap  G^{\ad}(\Rat)_1\subset G^{\ad}(\Rat),
\]
then the action of $\gamma$ on $\Sh$ descends to an isomorphism $i_\gamma\colon \Sh_{K_p}\isomto \Sh_{K_p}$ where $\Sh_{K_p}$ is the inverse limit of the prime-to-$p$ Hecke tower. Moreover, this action extends to one on the inverse limit of ICMs $\Ss_{K_p}$: To see this, it suffices, by Theorem~\ref{thm:mapping_property}, to observe that, for all $K^p\subset G(\Adele_f^p)$, we have a commuting diagram
\[
\begin{diagram}
    \Sh_{K_pK^p}&\rTo&\mathrm{Loc}_{\mathcal{G}^c(\Int_p)}\\
    \dTo^{i_\gamma}&&\dTo_{\iota_\gamma}\\
    \Sh_{K_p\mathrm{int}(\gamma)(K^p)}&\rTo&\mathrm{Loc}_{\mathcal{G}^c(\Int_p)}
\end{diagram}
\]
where
\[
\iota_\gamma\colon \mathrm{Loc}_{\mathcal{G}^c(\Int_p)}\isomto \mathrm{Loc}_{\mathcal{G}^c(\Int_p)}
\]
is the automorphism induced by the adjoint action of $\gamma$.
\end{remark}

\begin{remark}
    \label{rem:central_action}
For any $x\in \Ss_{K_p}(R)$, every element $z\in Z(\mathcal{G}_{(p)})(\Int_{(p)})$ satisfies $x\cdot z = z$: It is enough to check this in the generic fiber, and here this follows by an explicit check on $\Comp$-points. Therefore, for every neat compact open subgroup $K^p\subset G(\Adele_f^p)$, we obtain an action of the finite group
\[
Z(G)(\Adele_f^p)/\big(Z(\mathcal{G}_{(p)})(\Int_{(p)})\cdot (K^p\cap Z(G)(\Adele_f^p))\big)
\]
on $\Ss_{K_pK^p}$. We will see below that the adjoint action from Remark~\ref{rem:adjoint_action_icms} can be described in terms of this central action after replacing the Shimura variety with one from a Weil restricted Shimura datum.
\end{remark}

\begin{notation}
    In this subsection, $F$ will always denote a totally real field in which $p$ is unramified. Write $\mathrm{TR}_{p\text{-unr}}$ for the set of such fields.
\end{notation}

\begin{definition}
    [Weil restriction] 
\label{defn:weil_restriction}
For an unramified Shimura datum $(G,\mc{G},X)$, we define the \defnword{Weil restriction} $\Res_{F/\Rat}(G,\mathcal{G},X)$ to be the unramified Shimura datum with underlying reductive group $\Res_{F/\Rat}G$, reductive model $\Res_{\Reg{F}\otimes\Int_p/\Int_p}\mathcal{G}$ and Hermitian symmetric domain 
\[
\Res_{F/\Rat}X = \prod_{\tau:\colon F\to \Real}X
\]
viewed as a homogeneous space under
\[
(\Res_{F/\Rat}G)(\Real) = \prod_{\tau:\colon F\to \Real}G(\Real).
\]
\end{definition}

\begin{remark}
The reflex field of the Weil restricted Shimura datum $\Res_{F/\Rat}(G,X)$ is still $E$, and we have a canonical closed immersion of unramified Shimura data
\begin{align}\label{eqn:wr_embedding_data}
(G,\mathcal{G},X)\to \Res_{F/\Rat}(G,\mathcal{G},X)
\end{align}
\end{remark}

\begin{assumption}
[The WR condition]
    \label{assump:weil_restriction}
We will say that $(G,\mc{G},X)$ satisfies the \textbf{WR condition} if for every $F$ in $\mr{TR}_{p\text{-unr}}$, every neat unramified tuple with underlying unramified Shimura datum $\Res_{F/\Rat}(G,\mathcal{G},X)$ admits an ICM over $\Reg{E,(v)}$. Moreover, we require that for any sufficiently small compact open subgroup $K_F\subset G(\Adele_f\otimes_{\Rat}F)$ that is part of such a tuple, the central action of
\[
Z({G^{\mathrm{der}}})(\Adele_f^p\otimes F)/\big(Z({G^{\mathrm{der}}})(\Reg{F,(p)})\cdot (K^p_F\cap Z({G^{\mathrm{der}}})(\Adele_f^p\otimes F))\big)
\]
on the ICM $\Ss_{K_F}$ is free.
\end{assumption}

\begin{remark}
    The freeness of the central action in the generic fiber can be seen from the complex uniformization in Remark~\ref{rem:unramified_uniformization} applied to $\Sh_{K_F}$ as $K_F\subset G(\Adele_f^p\otimes F)$ varies.
\end{remark}

\begin{remark}
    \label{rem:WR_implies_limpid}
Note in particular that, if $(G,\mc{G},X)$ satisfies the WR condition, then every unramified tuple $(G,\mc{G},X,K)$ admits a \emph{limpid} ICM $\Ss_K$ over $\Reg{E,(v)}$.
\end{remark}

\begin{remark}
 \label{rem:freeness_after_inverse_limit}
By Lemma~\ref{lem:freeness_in_inverse_systems} below, the freeness condition for the WR condition is equivalent to the following: If $K_{F,p} = \mathcal{G}(\Reg{F}\otimes\Int_p)$, then the action of $Z({G^{\mathrm{der}}})(\Adele_f^p\otimes F)/Z({G^{\mathrm{der}}})(\Reg{F,(p)})$ on $\Ss_{K_{F,p}}$ is free.
\end{remark}

\begin{lemma}
  \label{lem:freeness_in_inverse_systems}
Let $J$ be a cofiltered small category and $\{X_j\}_{j\in J}$ a cofiltered inverse system of quasicompact, separated schemes with finite \'etale transition maps. Suppose that $\Gamma$ is a profinite group acting on $\{X_j\}_{j\in J}$ with the action on each $X_j$ is via a finite quotient $\Gamma_j$. Then the following are equivalent:
\begin{enumerate}
  \item $\Gamma$ acts freely on the inverse limit $X = \varprojlim_j X_j$;
  \item there exists $j_0\in J$ such that, for all $j\to j_0$, $\Gamma_j$ acts freely on $X_j$.
\end{enumerate}
\end{lemma}
\begin{proof}
  Clearly (2) implies (1). For the converse, for each $j\in J$, set
  \[
    F_j = \bigcup_{1\neq \sigma\in \Gamma_j}\mathrm{Fix}(\sigma).
  \]
  This is a closed subscheme of $X_j$. We claim that for $j_1\to j_2$ the map $X_{j_1}\to X_{j_2}$ restricts to a map $F_{j_1} \to F_{j_2}$. The key thing to check is that, if $\sigma\in \Gamma_{j_1}$ has a geometric fixed point on $F_{j_1}$, then it maps non-trivially to $\Gamma_{j_2}$: Indeed, if $\sigma\in \ker(\Gamma_{j_1}\to \Gamma_{j_2})$, then it is a non-trivial automorphism of the finite \'etale cover $X_{j_1}\to X_{j_2}$, and hence cannot have any geometric fixed points.

  Now, we have an inverse system $\{F_j\}_{j\in J}$ of quasicompact schemes with finite transition maps, and (1) tells us that their inverse limit is empty; see \stacks{01Z2}. The only way for this to be possible is for the $F_j$ to be eventually empty, giving us the desired conclusion.
\end{proof}

\begin{remark}
[Preservation of WR under closed immersions]
    \label{rem:modified_weil_restriction_closed_immersion}
Suppose that we have a closed immersion $(G,\mathcal{G},X)\hookrightarrow (G^\sharp,\mathcal{G}^\sharp,X^\sharp)$ as in Theorem~\ref{thm:reduction_of_structure_group}. Then the induced map
\[
\Res_{F/\Rat}(G,\mathcal{G},X)\to \Res_{F/\Rat}(G^\sharp,\mathcal{G}^\sharp,X^\sharp)
\]
is once again a closed immersion. In particular, by \emph{loc.\@ cit.\@}, if $(G^\sharp,\mathcal{G}^\sharp,X^\sharp)$ satisfies Assumption~\ref{assump:weil_restriction} and if $G\to G^\sharp$ is an isomorphism on derived subgroups, then $(G,\mathcal{G},X)$ satisfies Assumption~\ref{assump:weil_restriction}.
\end{remark}

\begin{remark}
 [Preservation of WR under products with CM data]
\label{rem:mwr_products}
Suppose that $(T,\mathcal{T},\{h_T\})$ is a CM unramified tuple. If $(G,\mathcal{G},X)$ satisfies Assumption~\ref{assump:weil_restriction}, then so does the product tuple
\[
(G\times T,\mathcal{G}\times\mathcal{T},X\times \{h_T\}).
\]
The only non-trivial portion of this claim follows from Proposition~\ref{prop:integral_canonical_models_for_tori}.
\end{remark}

\begin{definition}
[Modified Weil restriction]
    \label{defn:modified_weil_restriction}
The following notion will is used in \S\ref{sub:shimura_varieties_of_pre-abelian_type}. The \defnword{modified Weil restriction} $\Res'_{F/\Rat}(G,\mathcal{G},X)$ is the unramified Shimura datum with underlying reductive group\footnote{If $G = T$ is a torus, then the modified Weil restriction is just the original Shimura datum.}
\[
\Res'_{F/\Rat}G \defn \ker\bigg(\Res_{F/\Rat}G\to \Res_{F/\Rat}G^{\ab}/G^{\ab}\bigg),
\]
with reductive model 
\[
\Res'_{(\Reg{F}\otimes\Int_p/\Int_p)}\mathcal{G} = \ker\bigg(\Res_{(\Reg{F}\otimes\Int_p)/\Int_p}\mathcal{G}\to \Res_{(\Reg{F}\otimes\Int_p)/\Int_p}\mathcal{G}^{\ab}/\mathcal{G}^{\ab}\bigg),
\]
and Hermitian symmetric domain $\Res'_{F/\Rat}X$ given by the orbit under $(\Res'_{F/\Rat}G)(\Real)$ of the diagonal embedding $X\hookrightarrow \prod_{\tau\colon F\to \Real}X$. The closed immersion~\eqref{eqn:wr_embedding_data} factors through a map
\[
(G,\mathcal{G},X)\to \Res'_{F/\Rat}(G,\mathcal{G},X).
\]
\end{definition}

\begin{proposition}
    \label{prop:mwr_to_wr}
The following are equivalent:
\begin{enumerate}
    \item The unramified Shimura datum $(G,\mathcal{G},X)$ satisfies \emph{Assumption~\ref{assump:weil_restriction}}.
    \item For all $F\in \mathrm{TR}_{p\emph{-unr}}$, all neat unramified tuples with underlying unramified Shimura datum $\Res'_{F/\Rat}(G,\mathcal{G},X)$ also admit ICMs over $\Reg{E,(v)}$, and, if  for $K'_{F,p}\subset (\Res'_{F/\Rat}\mathcal{G})(\Int_p)$ compact open, the central action of the group $Z({G^{\mathrm{der}}})(\Adele_f^p\otimes F)/Z({G^{\mathrm{der}}})(\Reg{F,(p)})$ on $\Ss_{K'_{F,p}}$ is free.
\end{enumerate}
\end{proposition}
\begin{proof}
    That (2)$\Rightarrow$(1) follows from Theorem~\ref{thm:reduction_of_structure_group}. For (1)$\Rightarrow$(2), note that the map  
    \[
    \Res_{F/\Rat}G \to (\Res_{F/\Rat}G)^c\times \Res_{F/\Rat}G^{\ab}
    \]
    is a closed immersion: Indeed, this comes down to the observation that $\Res_{F/\Rat}G^{\mathrm{der}}$ maps isomorphically onto the derived subgroup of $(\Res_{F/\Rat}G)^c$, a consequence of our standing assumption that $G = G^c$. Therefore, appealing to Proposition~\ref{prop:integral_canonical_models_for_tori} and Theorem~\ref{thm:reduction_of_structure_group}, we see that it is enough to prove that all neat unramified tuples with underlying group $(\Res_{F/\Rat}G)^c$ admit ICMs with the central action coming from the derived subgroup free for sufficiently small level. This now follows from Lemma~\ref{lem:descent_same_derived_subgroup} below.
\end{proof}

\begin{lemma}
    \label{lem:descent_same_derived_subgroup}
Let $(G_1,\mathcal{G}_1,X_1)\to (G_2,\mathcal{G}_2,X_2)$ be a map of unramified Shimura data such that $G_1^{\mathrm{der}}\to {G}_2^{\mathrm{der}}$, $G_1\to G_1^c$, and $G_2\to G_2^c$ are isomorphisms, and such that both Shimura data admit the same reflex field $E$. Suppose that all neat unramified tuples of the form $(G_1,\mathcal{G}_1,X_1,K_1)$ admit ICMs over $\Reg{E,(v)}$. Then all neat unramified tuples of the form $(G_2,\mathcal{G}_2,X_2,K_2)$ also admit ICMs over $\Reg{E,(v)}$. Moreover, $Z({G_1^{\mathrm{der}}})(\Adele_f^p)/Z({G_1^{\mathrm{der}}})(\Int_{(p)})$ acts freely on $\Ss_{K_{1,p}}$ if and only if it acts freely on $\Ss_{K_{2,p}}$.
\end{lemma}
\begin{proof}
    To begin, let us make the following observation: If $G$ is a reductive group with $G = G^c$ and $K\subset G(\Adele_f)$ is a neat compact open subgroup, then $K\cap G(\Rat)$ is contained in $G^{\mathrm{der}}(\Rat)$. Indeed, if $\nu\colon G\to G^{\ab}$ is the natural map, the image of $K\cap G(\Rat)$ in $G^{\ab}(\Rat)$ is contained in $\nu(K)\cap G^{\ab}(\Rat)$, which is trivial, by the neatness of $\nu(K)$ and the fact that $G^{\ab}$ is a cuspidal torus. Therefore, if $G$ is part of a Shimura datum $(G,X)$, and $X^+\subset X$ is a connected component, then the image of $X^+\times \{g\}$ in $\Sh_K(\Comp)$ is isomorphic to $\Gamma_g\backslash X^+$, where $\Gamma = gKg^{-1}\cap G^{\mathrm{der}}(\Rat)$.

    Let $f\colon G_1\to G_2$ be the underlying map of groups. Suppose that we have an extension of the given map to a map of neat unramified tuples 
    \[
   (G_1,\mathcal{G}_1,X_1,K_1)\to (G_2,\mathcal{G}_2,X_2,K_2)
    \]
    given by $g\in G_2(\Adele^p_f)$ satisfying $gf(K_1)g^{-1} \subset K_2$. If $K_1$ is such that $K_1\cap G_1^{\mathrm{der}}(\Rat)$ maps isomorphically to $K_2\cap G_2^{\mathrm{der}}(\Rat)$, then the observation from the first paragraph shows that the image of $X_1^+\times \{1\}$ in $\Sh_{K_1}(\Comp)$ maps isomorphically onto the image of $X_1^+ \times \{g\}$ in $\Sh_{K_2}$. Therefore, given any connected component $\Sh^+_{K_2,\Comp}\subset \Sh_{K_2,\Comp}$, we can find a pair $(K_1,g)$, and a connected component $\Sh^+_{K_1,\Comp}\subset \Sh_{K_1,\Comp}$ mapping isomorphically onto $\Sh^+_{K_2,\Comp}$.

    Now, there is a finite $p$-unramified extension $E'/E$ such that $\Sh^+_{K_1,\Comp}$ is defined over $E'$, and, for any place $v'\vert v$ of $E'$, we obtain an integral model for it over $\Reg{E',(v')}$ as a connected component for $\Ss_{K_1}\otimes_{\Reg{E,(v)}}\Reg{E',(v')}$. This in turn gives us an integral model for the corresponding connected component of $\Sh_{K_2}\otimes_EE'$. Repeating this for all geometrically connected components of $\Sh_{K_2}$ and expanding $E'$ if necessary, we obtain an integral model $\Ss_{K_2,(v')}$ for $\Sh_{K_2}$ over $\Reg{E',(v')}$, which is an ICM: The existence of the syntomic realization and the Serre--Tate property follow from Proposition~\ref{prop:full_faithfulness_2}, while the pointwise conditions are easily checked using their validity for $\Ss_{K_1}$. The existence of the ICM $\Ss_{K_2}$ now follows from Corollary~\ref{cor:descent_coefficients_gen_ICMs}.

    The construction also shows the central action of the derived subgroup is free on $\Ss_{K_1}$ if and only if it is so on $\Ss_{K_2}$.
 \end{proof}

\begin{convention}
For the rest of this subsection, we will fix an unramified Shimura datum $(G,\mathcal{G},X)$ for which Assumption~\ref{assump:weil_restriction} holds.
\end{convention}

\begin{remark}
    \label{rem:from_mwr_to_maps}
For any $F\in \mathrm{TR}_{p\text{-unr}}$, write $\Sh_F$ for the inverse limit of Shimura varieties constructed in Remark~\ref{rem:hecke_action} from the datum $\Res_{F/\Rat}(G,X)$. We also obtain an inverse limit $\Ss_{K_{F,p}}$ of the prime-to-$p$ Hecke tower of ICMs for $\Res_{F/\Rat}(G,\mathcal{G},X)$ over $\Reg{E,(v)}$. We now have canonical closed immersions
\[
\Sh \hookrightarrow \Sh_F\;;\; \Ss_{K_p}\hookrightarrow \Ss_{K_{F,p}}
\]
of algebraic spaces over $\Reg{F,(v)}$. Suggestively, given an $\Reg{E,(v)}$-algebra $R$ and $x\in \Sh(R)$ (resp.\@ $x\in \Ss_{K_p}(R)$), we will write $x\otimes\Reg{F}$ for the associated point in $\Sh_F(R)$ (resp.\@ $\Ss_{K_{F,p}}(R)$). 
\end{remark}

\begin{remark}
    \label{rem:desc_of_adjoint_action}
Suppose that $\gamma\in \mathcal{G}^{\ad}_{(p)}(\Int_{(p)})_1$ lifts to  $\tilde{\gamma}\in \mathcal{G}_{(p)}(\Reg{F,(p)})$ for some $F\in \mathrm{TR}_{p\text{-unr}}$. Then we obtain an isomorphism $j_{\tilde{\gamma}}\colon \Ss_{K_{F,p}}\xrightarrow[\sim]{x\mapsto x\tilde{\gamma}^{-1}}\Ss_{K_{F,p}}$. Looking at $\Comp$-points, we find that the diagram
    \[
      \begin{diagram}
          \Ss_{K_p}&\rTo&\Ss_{K_{F,p}}\\
          \dTo^{i_\gamma}&&\dTo_{j_{\tilde{\gamma}}}\\
          \Ss_{K_p}&\rTo&\Ss_{K_{F,p}}
      \end{diagram}
    \]
    commutes. Therefore, $y = i_\gamma(x)$ if and only if $y\otimes\Reg{F} = (x\otimes\Reg{F})\cdot \tilde{\gamma}^{-1}$.
\end{remark}

\begin{lemma}
    [Freeness of the central action]
\label{lem:central_action_free}
The action of 
\[
Z(G)(\Adele_f^p\otimes F)/\bigl(Z(\mathcal{G}_{(p)})(\Reg{F,(p)})\cdot (K^p_F\cap Z(G)(\Adele_f^p\otimes F))\bigr)
\]
on $\Ss_{K_{F,p}K^p_F}$ is free for $K^p_F\subset G(\Adele_f^p\otimes F)$ sufficiently small.
\end{lemma}
\begin{proof}
    Set $K_F= K_{F,p}K^p_F$. We have a map of unramified tuples 
    \[
    (\Res_{F/\Rat}(G,\mathcal{G},X),K_F) \to (\Res_{F/\Rat}(G^{\ab},\mathcal{G}^{\ab},X^{\ab}),K^{\ab}_F),
    \]
    where we write $X^{\ab} = \{h_x^{\ab}\}$ for the abelianization of some (hence any) Deligne cocharacter $h_x$ associated with $x\in X$, and $K^{\ab}_F\subset G^{\ab}(\Adele_f)$ is the image of $K_F$. Note that $K^{\ab}_F = K^{\ab}_{F,p}K^{\ab,p}_F$, where we write $K_{F,p}^{\ab} = \mathcal{G}^{\ab}(\Reg{F,(p)})$. We obtain a map of ICMs $\Ss_{K_F}\to \Ss_{K^{\ab}_F}$ (the target exists because of Proposition~\ref{prop:integral_canonical_models_for_tori}). The codomain is a torsor for $G^{\ab}(\Adele^p_f\otimes F)/\big(\mathcal{G}_{(p)}^{\ab}(\Reg{F,(p)})\cdot K^{\ab,p}_F\big)$. 
    Indeed, this follows from Remark~\ref{rem:galois_structure_cm} and the density of $G^{\ab}(F)$ in $G^{\ab}(F\otimes\Rat_p)$: Since $G^{\ab}$ is a torus splitting over an unramified extension of $\Rat_p$, this density is a consequence of~\cite[Proposition 7.8]{Platonov1994-ib}. 

   It is now enough to show that
   \[
   \ker\bigg(Z(G)(\Adele_f^p\otimes F)/\big(Z(\mathcal{G}_{(p)})(\Reg{F,(p)})\cdot(K^p_F\cap Z(G)(\Adele_f^p\otimes F))\big)\to G^{\ab}(\Adele^p_f\otimes F)/\big(\mathcal{G}_{(p)}^{\ab}(\Reg{F,(p)})\cdot K^{\ab,p}_F\big)\bigg)
   \]
   acts freely on $\Ss_K$ for $K^p$ sufficiently small, which follows from our assumption on the central action arising from the derived subgroup.
\end{proof}

\begin{notation}
    As in~\cite[(3.3.2)]{kisin:abelian}, set
    \[
    \mathcal{A}(\mathcal{G}_{(p)}) = G(\Adele_f^p)/Z(\mathcal{G}_{(p)})(\Int_{(p)})^- \ast_{\mathcal{G}_{(p)}(\Int_{(p)})/Z(\mathcal{G}_{(p)})(\Int_{(p)})} \mathcal{G}^{\ad}_{(p)}(\Int_{(p)})^+.
    \]
    Here, the notation is Deligne's from~\cite{deligne:corvallis}, and denotes the quotient
    \[
    \left[G(\Adele_f^p)/Z(\mathcal{G}_{(p)})(\Int_{(p)})^- \rtimes \mathcal{G}^\mr{ad}_{(p)}(\Int_{(p)})^+\right]/(\mathcal{G}_{(p)}(\Int_{(p)})/Z(\mathcal{G}_{(p)})(\Int_{(p)}))
    \]
    where we view $\mathcal{G}_{(p)}(\Int_{(p)})/Z(\mathcal{G}_{(p)})(\Int_{(p)})$ as a subgroup of the semi-direct product via the diagonal embedding. Within this, we have the subgroup
    \[
    \mathcal{A}(\mathcal{G}_{(p)})^\circ = \mathcal{G}_{(p)}(\Int_{(p)})^-_+/Z(\mathcal{G}_{(p)})(\Int_{(p)})^- \ast_{\mathcal{G}_{(p)}(\Int_{(p)})/Z(\mathcal{G}_{(p)})(\Int_{(p)})} \mathcal{G}^{\ad}_{(p)}(\Int_{(p)})^+,
    \]
    which can be identified with the completion of $\mathcal{G}_{(p)}^{\ad}(\Int_{(p)})_+$ with respect to the topology generated by subgroups of the form $K^p\cap \mathcal{G}_{(p)}^{\mathrm{der}}(\Int_{(p)})$ for $K^p\subset G(\Adele_f^p)$ sufficiently small; see \cite[{}2.1.15]{deligne:corvallis}.
\end{notation}

\begin{remark}
    The prime-to-$p$ Hecke action of $G(\Adele_f^p)/Z(\mathcal{G}_{(p)})(\Int_{(p)})^-$ combined with the adjoint action of $ \mathcal{G}_{(p)}(\Int_{(p)})_+$ explained in Remark~\ref{rem:adjoint_action} gives an action of $ \mathcal{A}(\mathcal{G}_{(p)})$ on $\Ss_{K_p}$. Explicitly, this is the descent of the action given for a pair $(h,\gamma^{-1})\in G(\Adele_f^p)\rtimes  \mathcal{G}^{\ad}_{(p)}(\Int_{(p)})^+$ by $x\mapsto i_{\gamma}(x\cdot h)$.
\end{remark}

\begin{lemma}
    \label{lem:lifting_over_totally_real}
Set
\[
\Delta \defn \ker\left(\mathcal{A}(\mathcal{G}_{(p)})\to \mathcal{A}(\mathcal{G}^{\ad}_{(p)})\right).
\]
For $[(h,\gamma^{-1})]\in \Delta$, $\gamma$ can be lifted to $\mathcal{G}(\Reg{F,(p)})$ for some $F\in \mathrm{TR}_{p\emph{-unr}}$. 
\end{lemma}
\begin{proof}
  The obstruction to lifting $\gamma$ to $\mathcal{G}$ yields a class $\alpha_\gamma\in H^1(\Int_{(p)},Z({\mathcal{G}_{(p)}}))$. By assumption, $\gamma$ lifts to $\mathcal{G}(\Rat_\ell)$ for all $\ell\neq p$. This shows that $\alpha_\gamma$ has trivial restriction to $H^1(\Rat_\ell,Z(G))$ for all $\ell\ne p$. Now, our assumptions imply that $Z(G) = Z(G)^c$ splits over a CM field $\tilde{N}$, which we can take to be Galois and $p$-unramified. Let $N\subset \tilde{N}$ be the maximal totally real subfield. Then, by the cyclicity of $\tilde{N}/N$ and Chebotarev density, we find that the restriction of $\alpha_\gamma$ to $H^1(\Reg{N,(p)},Z(\mathcal{G}_{(p)}))$ is in the image of a class $\alpha'_\gamma\in H^1(\Reg{N,(p)},Z({\mathcal{G}_{(p)}})^\circ)$. 

  We claim that any such lift $\alpha'_\gamma$ has trivial restriction to $H^1(\Real,Z(G)^\circ)$ for any embedding $N\to \Real$. For this, note that the assumption $\gamma\in \mathcal{G}(\Int_{(p)})_1$ tells us that, for any $x\in X$, with stabilizer $G_x\subset G_{\Real}$, $\alpha_\gamma$ has trivial restriction to $H^1(\Real,G_x)$. The claim now follows from~\cite[Lemma (4.4.5)]{Kisin2017-qa}.

  Let $\mathcal{X}\to \Spec \Reg{N,(p)}$ be a $Z({\mathcal{G}_{(p)}})^\circ$-torsor with cohomology class $\alpha'_\gamma$. Let $\tilde{F}$ be the maximal $p$-unramified totally real extension of $N$. It is now enough to know that $\mathcal{X}(\Reg{\tilde{F},(p)})$ is non-empty. We have shown that $\mathcal{X}\otimes_{\Reg{N,(p)},\tau}\Real$ is trivial for all $\tau:N\to \Real$, and, by Lang's theorem, the base-change $\mathcal{X}\otimes_{\Reg{N,(p)}}\Reg{N,v}$ is also trivial for any place $v\vert p$ of $N$. Therefore, the desired statement follows from \cite[Th\'eor\`eme 1.3]{Moret-BaillySkolem} applied with $R=\Reg{N,(p)}$ and $\Sigma$ the set of archimedean places of $N$.
\end{proof}

\begin{lemma}
    \label{lem:free_action_Delta_G_Gad}
The group $\Delta$ acts freely on $\Ss_{K_p}$. In fact, for all sufficiently small $K^p\subset G(\Adele_f^p)$, the (finite) quotient of $\Delta$ acting faithfully on $\Ss_{K_pK^p}$ acts freely.
\end{lemma}
\begin{proof}
    Suppose that we have $[(h,\gamma^{-1})]\in \Delta$ and $x\in \Ss_{K_p}(R)$ such that $x\cdot (h,\gamma) = x$. Let $\tilde{\gamma}\in \mathcal{G}_{(p)}(\Reg{F,(p)})$ be a lift of $\gamma$ with $F\in \mathrm{TR}_{p\text{-unr}}$: this exists by Lemma~\ref{lem:lifting_over_totally_real}.
    
    By Remark~\ref{rem:desc_of_adjoint_action}, we find that we have
    \[
    (x\otimes\Reg{F})\cdot k = (x\otimes\Reg{F})\cdot z,
    \]
    where $z \defn h\tilde{\gamma}^{-1}\in Z(G)(\Adele_f^p\otimes_{\Rat}F)$ and $k\in K^p$. By Lemma~\ref{lem:central_action_free}, we have $z\in Z(G)(\Reg{F,(p)})^-$. This means that $ z\tilde{\gamma} = h$ is an element of 
    \[
    Z(G)(\Reg{F,(p)})^-\mathcal{G}_{(p)}(\Reg{F,(p)}) \cap G(\Adele_f^p)
    \]
    lifting $\gamma$. Since $G = G^c$, this intersection is contained in
    \[
   (\Res_{\Reg{F,(p)}/\Int_{(p)}}\mathcal{G}_{(p)})^c(\Int_{(p)})\cap G(\Adele_f^p) = G(\Int_{(p)})\subset (\Res_{F/\Rat}G)^c(\Adele_f^p).
    \]
    So, by Lemma~\ref{lem:freeness_in_inverse_systems}, it suffices to observe that $\Delta$ acts on $\Ss_{K_pK^p}$ via a finite quotient by \cite[Lemma E.6]{Kisin2017-qa}.
\end{proof}

We now prove Theorem~\ref{introthm:descent} in the following formulation:

\begin{theorem}
    [Descent for ICMs]
\label{thm:descent_for_icms}
Let $(G_1,\mathcal{G}_1,X_1)\to (G_2,\mathcal{G}_2,X_2)$ be a a map of unramified Shimura data with $G_1 = G_1^c$, and suppose the WR condition (i.e., \emph{Assumption~\ref{assump:weil_restriction}}) holds for $(G_1,\mathcal{G}_1,X_1)$. Then, for ${K}_2^p$ sufficiently small, $\Sh_{{K}_2}$ admits an ICM $\Ss_{{K}_2}$ over $\Reg{{E}_2,({v}_2)}$ and the map $\Ss_{K_1}\to \Ss_{{K}_2}\otimes_{\Reg{{E}_2},({v}_2)}\Reg{E_1,(v_1)}$ is finite \'etale.
\end{theorem}

\begin{proof}
    The proof will be along the lines of Lemma~\ref{lem:descent_same_derived_subgroup}, again making use of Proposition~\ref{prop:full_faithfulness_2} and Corollary~\ref{cor:descent_coefficients_gen_ICMs}. After choosing a geometrically connected component $\Sh^+_{K_1,E'}\subset \Sh_{K_1}\otimes_EE'$, and a place $v'\vert p$ of $E'$ above $v$, it suffices to check that the action of $({K}_2\cap {G}_2^{\mathrm{der}}(\Rat))/(K_1\cap G_1^{\mathrm{der}}(\Rat))$ on $\Sh^+_{K_1,E'}$ extends to a free action on the corresponding connected component $\Ss^+_{K_1,(v')}$ of $\Ss_{K_1}\otimes_{\Reg{E_1,(v_1)}}\Reg{E',(v')}$. By Lemma~\ref{lem:freeness_in_inverse_systems}, we are reduced to knowing that the inverse limit
    \[
    \varprojlim_{f(K_1^p)\subset{K}_2^p\subset {G}_2(\Adele_f^p)}({K}_2\cap {G}_2^{\mathrm{der}}(\Rat)/(K_1\cap G_1^{\mathrm{der}}(\Rat)) = \varprojlim_{f(K_1^p)\subset{K}_2^p\subset {G}_2(\Adele_f^p)}({K}_2^p\cap {\mathcal{G}}^{\mathrm{der}}_{2,(p)}(\Int_{(p)}))/(K_1^p\cap \mathcal{G}^{\mathrm{der}}_{1,(p)}(\Int_{(p}))
    \]
    acts freely on $\Ss_{K_{1,p}}$. This can be identified with the kernel of the map
    \[
    \mathcal{A}(\mathcal{G}_{1,(p)})^\circ\to \mathcal{A}(\overline{\mathcal{G}}_{1,(p)})^\circ,
    \]
    and so the theorem now follows from Lemma~\ref{lem:free_action_Delta_G_Gad}.
\end{proof}

\begin{remark}
    \label{rem:only_finitely_many_wr_needed}
A closer analysis of the proof of Theorem~\ref{thm:descent_for_icms} shows that we only need the condition in Assumption~\ref{assump:weil_restriction} to hold for \emph{finitely many} $F\in \mathrm{TR}_{p\text{-unr}}$. This follows from the finiteness of the group~\cite[Th\'eor\`eme 7.1]{zbMATH03230408}
\[
\ker\bigg(H^1(\Int_{(p)},Z({\mathcal{G}_{(p)}})) \to \prod_{\ell\neq p}H^1(\Rat_\ell,Z(G))\bigg).
\]
\end{remark}

\subsection{Shimura varieties of pre-abelian type}
\label{sub:shimura_varieties_of_pre-abelian_type}

\begin{definition}
  A Shimura datum $(G,X)$ is of \defnword{Hodge type} if $G$ admits a faithful symplectic representation $H$ on which some (hence any) element of $X$ induces a Hodge structure with weights $(-1,0),(0,-1)$. It is of \defnword{abelian type} if there exists a Shimura datum $(G_1,X_1)$ of Hodge type and a map $G_1^{\mathrm{der}}\to G_2^{\mathrm{der}}$ of derived groups inducing an isomorphism on the corresponding adjoint Shimura data. It is of \defnword{pre-abelian type} if the adjoint Shimura datum $(G^{\mathrm{ad}},X^{\mathrm{ad}})$ is of abelian type.
\end{definition}

\begin{remark}
    [Classification of pre-abelian Shimura data]
\label{rem:pre-abelian}
By the classification of Shimura data by Deligne~\cite[\S 1.3]{deligne:corvallis}, a Shimura datum $(G,X)$ is of pre-abelian type if and only if each simple factor of the adjoint Shimura datum $(G^{\ad},X^{\ad})$ is classical of type $A,B,C,D^{\Real},D^{\mathbb{H}}$. 

Note that a simple adjoint Shimura datum has underlying group of the form $\Res_{F/\Rat}H$ where $F$ is a totally real field and $H$ is an absolutely simple adjoint group over $F$. The meaning of type $D^{\Real}$ is that, for each embedding $\tau:F\to \Real$, the group $H_{\tau}\defn H\otimes_{F,\tau}\Real$ is the adjoint quotient of an orthogonal group associated with a quadratic form of signature either $(2n,2)$, $(2n,n)$ or $(2n+2,0)$. The meaning of type $D^{\mathbb{H}}$ is instead that $H_{\tau}$ is the adjoint quotient of $\mathrm{O}^*(2n)$, the group of isometries of a skew-Hermitian space over the ring of quaternions $\mathbb{H}$. The Shimura datum is of \emph{abelian} type if the derived subgroup of every simple factor of $G_{\Real}$ of type $D^{\mathbb{H}}$ is a quotient of $\SO^*(2n)$.
\end{remark}

\begin{example}
    [A non-abelian but pre-abelian Shimura datum]
\label{ex:non-abelian_pre-abelian}
The basic example of a pre-abelian-type (but not abelian-type) Shimura datum is obtained as follows. Take a totally real field $F_0$ and a quaternion algebra $B$ over $F_0$ such that for some distinguished archimedean place $v_0$ of $F_0$ we have $B_{v_0}\simeq \bb{H}$, and $B_v\simeq \mr{Mat}_2(\bb{R})$ for archimedean places $v\ne v_0$. We then take a left $B$-module $V$ of rank $n\geqslant 4$ with a Hermitian pairing which is definitive at $v\ne v_0$ and signature $(n,0)$ at $v_0$. Set $H^{\ad}$ to be $\mr{PU}(V,h)$ and $G^\mr{ad}=\mr{Res}_{F_0/\bb{Q}}\, H^{\ad}$. Define $h\colon\bb{S}\to G^\mr{ad}_\bb{R}$ to have trivial projection to each archimedean place $v\ne v_0$,  be such that if $W=V\otimes_{F_0,v_0}\bb{R}$ then the projection to $v_0$ defines a Hodge structure of weight $\{(-1,0),(0,-1)\}$ where $W_\bb{C}^{(-1,0)}$ is the maximal isotropic subspace. This defines a Shimura datum $(G^\mr{ad},X^\mr{ad})$. Let $H$ be the central cover of $H^{\ad}$ determined by the property that $H_{\Comp}$ is isomorphic to the GSpin cover of $H^{\ad}_{\Comp}\simeq \mr{PSO}(2n)$. Then $X^\mr{ad}$ admits a lift $X$ to $G=\mr{Res}_{F_0/\mathbb{Q}}\,H$, and $(G,X)$ is of pre-abelian but not of abelian type.
\end{example}

\begin{proposition}
    [Hodge type data satisfy WR]
\label{prop:hodge_type_mwr}
Suppose that $(G,\mc{G},X)$ is of Hodge type. Then it satisfies \emph{Assumption~\ref{assump:weil_restriction}}.
\end{proposition}
\begin{proof}
Suppose first that $(G,X)$ is of Siegel type associated with a symplectic space $V$ such that $V_{\Rat_p}$ admits a self-dual lattice $V_{\Int_p}$ stabilized by $\mathcal{G}$. We begin by verifying that $(G,\mathcal{G},X)$ admits limpid ICMs over $\Int_{(p)}$ for any neat level $K^p\subset G(\Adele_f^p)$. The point is that $\Sh_K$ has an integral model $\Ss_K$ over $\Int_{(p)}$, given by a moduli space of polarized abelian varieties with additional level structure. By Theorem~\ref{thm:dieudonne}, we have an $F$-gauge over $\Ss_K$ associated with the $p$-divisible group of the universal abelian scheme $\mathcal{A}\to \Ss_{K}$. We can actually view this as a $(\GL_{2g},\mu_g)$-aperture over $\Ss_{K}$, where $g$ is the relative dimension of $\mathcal{A}$. Now, if $\rank V_{\Int_{p}} = 2g$, Proposition~\ref{prop:Gmu_aperture_pointwise_condition_normal} applies with
  \[
  \mathcal{X} = \mathcal{X}^\sharp = \Ss_K;\; (\mathcal{G},\mu) = (\GSp_{2g},\mu_g)\;;\; (\mathcal{G}^\sharp,\mu^\sharp) = (\GL_{2g},\mu_g).
  \]
  That $\widehat{\Ss}_{K,v}\to \BT[\GL_{2g},\mu_g]{\infty}$ is formally unramified follows from classical Serre--Tate theory and Theorem~\ref{thm:dieudonne}. 

This tells us that $\Ss_K$ satisfies the Serre--Tate property. The extension of the canonical $\ell$-adic local systems over $\Ss_K$ for $\ell\neq p$ is obtained from the $\ell$-adic Tate modules of $\mathcal{A}$.

It remains to check the pointwise condition. This comes down to the fact that, by a result of Coleman--Iovita~\cite[Theorem 1]{coleman_iovita}, for an abelian variety $A$ over a local field,  crystallinity of the $p$-adic Tate module is equivalent to saying that $A$ has good reduction.

Now, let us proceed to the general Hodge type case, where we have a closed immersion from $(G,X)$ into a Siegel Shimura datum $(G^\sharp,X^\sharp)$. This immediately implies that we have $G=G^c$. Moreover, by Zarhin's trick, we can assume that we actually have a closed immersion of unramified Shimura data
\[
(G,\mathcal{G},X) \to (G^\sharp,\mathcal{G}^\sharp,X^\sharp). 
\]
Therefore, the already considered case of Siegel type, combined with Theorem~\ref{thm:reduction_of_structure_group} tells us that, for any neat $K^p\subset G(\Adele_f^p)$ and any place $v\vert p$ of $E$, $\Sh_K$ admits a limpid ICM over $\Reg{E,(v)}$. 

A point $x\in \Ss_{K_p}(\kappa)$ corresponds to a polarized abelian variety $\mathcal{A}_x$ over $\kappa$, equipped with an isomorphism 
\[
\bigg(\varprojlim_{(p,n) =1}\mathcal{A}_x[n](\kappa)\bigg)\otimes\Rat\isomto \Adele_f^p\otimes_{\Rat}V,
\]
and gives rise to an embedding 
\[
Z({G^{\mathrm{der}}})(\Int_{(p)})\hookrightarrow \Aut^p(\mathcal{A}_x)
\]
where $\Aut^p(\mathcal{A}_x)$ is the group of prime-to-$p$ quasi-isogenies from $\mathcal{A}_x$ to itself; see~\cite[Lemma 4.5.2]{Kisin2018-nm}. Note that the \'etale realization of quasi-isogenies gives us an inclusion
\[
\Aut^p(\mathcal{A}_x)\subset G^\sharp(\Adele_f^p)
\]
Moreover, the intersection with $ \Aut^p(\mathcal{A}_x)$ of $Z({G^{\mathrm{der}}})(\Adele_f^p)$ inside $G^\sharp(\Adele_f^p)$ is exactly $Z({G^{\mathrm{der}}})(\Int_{(p)})$. This tells us that the action of $Z({G^{\mathrm{der}}})(\Adele_f^p)/Z(G^{\mathrm{der}})(\Int_{(p)})$ on $\Ss_{K_p}$ is free.

We are now in a position to verify Assumption~\ref{assump:weil_restriction} for $(G,\mathcal{G},X)$. Indeed, using Proposition~\ref{prop:mwr_to_wr}, we can work with the modified Weil restrictions instead. Now, note that $\Res'_{F/\Rat}(G,\mathcal{G},X)$ is once again of Hodge type: It embeds into the Siegel Shimura datum associated with $[F:\Rat]$-copies of $V$. Therefore, the proposition follows from the preceding discussion.
\end{proof}

\begin{remark}
   The alert reader will note that polarized \emph{$p$-divisible} groups do not make an appearance in the above proof (and they do not in Kisin's original argument in~\cite{kisin:abelian} either). One could shorten the proof slightly by noting that the moduli stack of prime-to-$p$ degree polarized abelian schemes of dimension $g$ carries a canonical principally polarized $p$-divisible group of height $2g$, which can then be reinterpreted as a $(\GSp_{2g},\mu_g)$-aperture; see~\cite[\S 11.6]{gmm}. However, this reinterpretation has not been carried out fully satisfactorily (see  footnote 43 of \emph{loc.\@ cit.\@}), and so we have preferred to avoid using it in our argument above. 
\end{remark}

\begin{theorem}
  [Canonical models of pre-abelian type]
\label{thm:abelian}
Suppose that one of the following holds: 
\begin{enumerate}
    \item $(G,\mc{G},X)$ is of abelian type;
    \item $p>2$ and $(G,\mc{G},X)$ is of pre-abelian type.
\end{enumerate}
Then, for $K^p$ sufficiently small, $(G,\mc{G},X,K)$ admits a limpid ICM over $\Reg{E,(v)}$.
\end{theorem}
\begin{proof}
  Theorem~\ref{thm:descent_for_icms} combined with Proposition~\ref{prop:hodge_type_mwr} and Remark \ref{rem:WR_implies_limpid} shows that $(G,\mathcal{G},X)$ admits limpid ICMs for sufficiently small $K^p$ whenever we can find a surjective map $(G',\mathcal{G}',X')\to (G,\mathcal{G},X)$ where the source is of Hodge type. In particular, we find that the theorem holds whenever $(G,\mathcal{G},X)$  is of \emph{adjoint} abelian type. If $p>2$, we can conclude the proof using Theorem~\ref{thm:ascent}.
  
  If $p=2$ and $(G,X)$ is of abelian type, then, by a construction of Lovering~\cite[\S 4.6]{Lovering2017-me}, we can find a map of Shimura data $(\tilde{G},\tilde{\mathcal{G}},\tilde{X})\to (G,\mc{G},X)$ where $\tilde{G}\to G$ is a central cover, and a closed immersion
  \[
  (\tilde{G},\tilde{\mc{G}},\tilde{X}) \to (G',\mathcal{G}',X'),
  \]
  where the target is a product of an unramified Shimura datum of Hodge type with one of CM type, and is such that $\tilde{G}^{\mathrm{der}} = (G')^{\mathrm{der}}$. Now, Remarks~\ref{rem:modified_weil_restriction_closed_immersion} and~\ref{rem:mwr_products} tell us that $(\tilde{G},\tilde{\mc{G}},\tilde{X})$ satisfies Assumption~\ref{assump:weil_restriction}, and so we conclude using Theorem~\ref{thm:descent_for_icms} again.
\end{proof}

\begin{remark}
    [Schematic nature of pre-abelian type ICMs]
The ICMs $\Ss_K$ for pre-abelian type Shimura data are actually quasi-projective schemes over $\Reg{E,(v)}$. When $(G,X)$ is Siegel type, this is~\cite[{}VII.\@ Th\'eor\`eme 4.2]{moret-bailly:pinceaux}, and this implies the statement successively for ICMs of Hodge type, of adjoint abelian type (by descent of quasi-projectivity along quotients by finite groups), and finally for all ICMs of pre-abelian type (by ascent of quasi-projectivity along finite maps).
\end{remark}

\subsection{Canonicity of the models of Bakker--Shankar--Tsimerman}
\label{sub:bst}

Here, we will find a proof of Theorem~\ref{introthm:exceptional}. We can assume that the rank of $G$ is at least $2$, since otherwise the Shimura datum is of pre-abelian type and so is already addressed by Theorem~\ref{thm:abelian}. By Remark~\ref{introrem:adjoint_is_enough}, we can also assume that $G$ is \emph{adjoint} (and thus $G=G^c$).

\begin{proof}
  [Proof of Theorem~\ref{introthm:exceptional}]
We will take a prime $p>3$ large enough such that the following holds: 
\begin{enumerate}
  \item There is a smooth model $\Ss_K$ of $\mr{Sh}_K$ over $\Reg{E,(v)}$ with log-smooth compactification.
  \item The filtered $G$-bundle with integrable connection $\Fil^\bullet_{\mathrm{Hdg}}\mathbf{dR}_K$ extends to a filtered $\mathcal{G}_{(p)}$-bundle with integrable connection on $\Ss_K$, which we denote by the same symbol. We can also assume that this satisfies Griffiths' transversality and that the associated Kodaira--Spencer map
  \[
    \mathbb{T}_{\Ss_K/\Reg{E,(v)}}\to \mathbf{dR}_K(\mathfrak{g})/\Fil^0_{\mathrm{Hdg}}\mathbf{dR}_K(\mathfrak{g})
  \]
  is an isomorphism.
  \item For some choice of faithful representation $\Lambda$ of $\mathcal{G}$, the filtered de Rham realization $\Fil^\bullet_{\mathrm{Hdg}}\mathbf{dR}_K(\Lambda)$ lifts to a Fontaine--Laffaille module (up to suitable twist) over $\Ss_K$. This implies in particular that $\mathbf{Et}_{K,p}$ is crystalline at every classical point of $\widehat{\Ss}_{K,\eta}$. That this is possible follows from \cite[Theorem 7.1]{pila2024canonicalheights}.
\end{enumerate}
At this point, we are almost there. Theorem~\ref{thm:p_comp_smooth_characterization}, combined with Remark~\ref{rem:pointwise_lifts} and Lemma~\ref{lem:versality_condition}, tells us that we have an extension $\Ss_K\to \BT[\mathcal{G}^c,\mbox{-}\mu_v,\mathrm{alg}]{\infty}$ which is $p$-adically formally \'etale. 

The only thing that remains to be checked is that $\Ss_K$ satisfies the pointwise condition for being an ICM. We thank Ananth Shankar for conveying the following argument. We must know that, if $x\in (\Sh_K\otimes_EE_v)(F)$ is such that $\mathbf{Et}_{K,p}\vert_x$ is potentially crystalline, then we in fact have $x\in \Ss_K(\Reg{F})$. This is a consequence of the fact that---by the proof of~\cite[Theorem 7.1]{pila2024canonicalheights}---the Fontaine--Laffaille module lifting $\Fil^\bullet_{\mathrm{Hdg}}\mathbf{dR}_K(\Lambda)$ extends to a logarithmic Fontaine--Laffaille module over a suitable toroidal compactification $\overline{\Ss}_K$ of $\Ss_K$ with underlying filtered logarithmic connection given over the generic fiber by the so-called canonical extension of $\Fil^\bullet_{\mathrm{Hdg}}\mathbf{dR}_K(\Lambda)$. 

Concretely, this means that for every $x\in (\Sh_K\otimes_EE_v)(F)$ specializing to the boundary of $\overline{\Ss}_K$, the Galois representation $\mathbf{Et}_{K,p}\vert_x$ is \emph{semistable} and its associated filtered $(\varphi,N)$-module over $F$ has \emph{non-trivial} monodromy operator $N$ that can be read off from the characteristic $0$ boundary stratification. In particular, $x$ is potentially crystalline if and only if $N=0$. Therefore, $x$ is potentially crystalline if and only if $\mathrm{sp}(x)\in \ms{S}_K\otimes k\subset \ov{\ms{S}}_K\otimes k$ or, equivalently, $x\in \wh{\ms{S}}_K(F)=\ms{S}_K(\mc{O}_F)$, as desired.
\end{proof}

\begin{remark}
    Since we only require the above conditions for \emph{adjoint} Shimura data, and also do not require any ampleness conditions or the existence of the prime-to-$p$ \'etale realizations, the lower bound on $p$ we arrive at here is logically smaller than that in~\cite{bst}.
\end{remark}

\begin{remark}
    As noted in Remark~\ref{introrem:exceptional_icms_limpid}, after further extending the bound on the prime $p$, work of Klevdal--Patrikis implies that the ICM $\Ss_K$ is actually limpid. That is, for all $\ell\neq p$, the local system $\mathbf{Et}_{K,\ell}$ extends to a local system on $\Ss_K$; see~\cite[Theorem 1.3]{Klevdal2025-ka}.
\end{remark}

\begin{remark}
    Finally, we observe that Theorem~\ref{thm:ascent}, combined with the previous remark, gives a somewhat different route to Patrikis's extension of the results of~\cite{Klevdal2025-ka} to the non-adjoint case in~\cite{patrikis2025compatibility}. Of course, both routes traverse through a study of the $\mu$-ordinary locus.
\end{remark}

\section{Applications and some complements}
\label{sec:applications_and_further_results}

In this section, we collect some immediate applications and extensions of our theory of ICMs. In addition to discussing cohomological applications and integral models at (some) parahoric levels, we prove Theorems \ref{introthm:surjective}, \ref{introthm:newton_strata}, and \ref{introthm:Ekedahl--Oort} from the introduction. We freely use notation and terminology from \S\ref{sec:canonical}--\ref{sec:icm_constructions}.

\subsection{Surjectivity}
\label{sub:surjectivity}

The following result of Vasiu~\cite[Main Theorem A]{Vasiu2002-rj}, will be used in the proof of Theorem~\ref{introthm:surjective}, which is the aim of this subsection.
\begin{theorem}
[Vasiu]
\label{thm:vasiu_boundedness}
Let the notation be as in \emph{Setup~\ref{setup:g_mu}}. Then there exists $m(\mathcal{G},\mu)\geqslant 1$ such that, for any algebraically closed field $\kappa$ over $\mathcal{O}$, and any $n\geqslant m(\mathcal{G},\mu)$, the map of groupoids
\[
\BT{\infty}(\kappa)\to \BT{n}(\kappa)
\] 
induces a bijection on isomorphism classes.
\end{theorem}
\begin{proof}
This is just a reformulation of the cited result of Vasiu, using the quotient descriptions of $\BT{n}(\kappa)$ and $\BT{\infty}(\kappa)$ given to us by Remarks~\ref{rem:quotient_description_perfectoid_truncated} and~\ref{rem:quotient_description_perfectoid}. 
\end{proof}

\begin{remark} 
An equicharacteristic version of Theorem \ref{thm:vasiu_boundedness} is given by \cite[Theorem 6.31]{viehmann2025modulitruncatedshtukasdisplays}. The argument in \emph{loc.\@ cit.\@} can be combined with work of Zhu on the Witt vector affine Grassmannian~\cite{Zhu2017-kn} to give an alternate proof of Theorem \ref{thm:vasiu_boundedness}.
\end{remark}

\begin{definition}
\label{def:de_jong_property}
  We will say that an unramified Shimura tuple $(G,\mathcal{G},X,K)$ admitting an ICM $\ms{S}_K$ over $\ms{O}_{E,(v)}$ satisfies the \defnword{de Jong extension property (at $v$)} if the following condition holds. Suppose that $C$ is a smooth curve over an algebraically closed field $\kappa$ over $k(v)$ with smooth compactification $\overline{C}$. Then a map $C\to \Ss_K$ extends to $\overline{C}\to \Ss_K$ if and only if the composition with $\Ss_K\to \BT[\mathcal{G}^c,\mbox{-}\mu_v]{\infty}$ extends to a map $h\colon \ov{C}\to \BT[\mathcal{G}^c,\mbox{-}\mu_v]{\infty}$. 
\end{definition}

\begin{remark}
[Functoriality properties of the de Jong extension property]
    \label{rem:de_jong_finite_maps}
Suppose that we have a map of unramified tuples $(G_1,\mathcal{G}_1,X_1,K_1)\to (G_2,\mathcal{G}_2,X_2,K_2)$ such that both admit ICMs $\Ss_{K_1}$ and $\Ss_{K_2}$ over $\Reg{E_1,(v_1)}$ and $\Reg{E_2,(v_2)}$, respectively, with $v_1\vert v_2$ both lying over $p$. If $G_1\to G_2$ is a central cover or if $G_1\to G_2$ is a closed immersion, then the map  $\Ss_{K_1}\to \Ss_{K_2}$ is finite (see Remark~\ref{rem:finite_maps_etale} for the former and Theorem~\ref{thm:reduction_of_structure_group} for the latter). From this, one checks that, if $\Ss_{K_2}$ admits the de Jong extension property, then so does $\Ss_{K_1}$. 
\end{remark}

\begin{remark}
[Lifting the extension property along central covers]
\label{rem:lifting_dejong_central_covers}
When $G_1\to G_2$ is a central cover, $\Ss_{K_1}\to \Ss_{K_2}$ is finite \'etale onto its image, and one sees that the de Jong extension property for $\Ss_{K_1}$ implies the extension property for any curve mapping into the image of $\Ss_{K_1}$ in $\Ss_{K_2}$. By varying the element $g\in G_2(\Adele_f^p)$ giving the map of tuples, one finds using Remark \ref{rem:prime-to-p_hecke} that, if $\Ss_{K_1}$ satisfies the de Jong extension property, then so does $\Ss_{K_2}$.
\end{remark}

\begin{remark}
    \label{rem:de_jong_proper_automatic}
The extension property is automatic if $\Sh_K$ admits a proper ICM $\Ss_K$ over $\Reg{E,(v)}$.
\end{remark}

\begin{remark}
[Extension property for pre-abelian type ICMs]
    \label{rem:de_jong_siegel}
The extension property holds for the Siegel Shimura datum by results of de Jong~\cite[Corollaries 1.2 and 2.5]{de-Jong1998-ki}. Using Remarks~\ref{rem:de_jong_finite_maps} and~\ref{rem:lifting_dejong_central_covers}, one sees that it also holds for all ICMs for unramified Shimura tuples of pre-abelian type.
\end{remark}

\begin{remark}
    [Log de Rham extensions]
\label{rem:log_de_rham_extensions}
Suppose that $\Ss_K$ admits a log-smooth compactification $\overline{\Ss}_K$ such that, for some faithful representation $\Lambda$, $\mathbf{dR}_K(\Lambda)$ extends to a module with logarithmic connection $\overline{\mathbf{dR}}_K(\Lambda)$ on $\overline{\Ss}_K$ with the following property: 
\begin{itemize}
    \item For any smooth curve $C$ over a finite field and any map $C\to \overline{\Ss}_K$ whose image is not entirely inside the boundary (i.e., $\overline{\Ss}_K\backslash \Ss_K$), the residue of the logarithmic connection of $\overline{\mathbf{dR}}_K(\Lambda)$ is \emph{non-zero} at every point of $C$ landing inside the boundary. 
\end{itemize}
Then $\Ss_K$ has the de Jong extension property: Any map $C\to \Ss_K$ extends to a map $\overline{C}\to \overline{\Ss}_K$, and the fact that the syntomic realization over $C$ extends over $\overline{C}$ implies that $\overline{\mathbf{dR}}_K(\Lambda)\vert_{\overline{C}}$ has trivial residue at every point.
\end{remark}

\begin{remark}
    [de Jong extension property for exceptional ICMs]
\label{rem:de_jong_exceptional}
The argument for the verification of the pointwise criterion in the proof of Theorem~\ref{introthm:exceptional} in fact shows that the models $\Ss_K$ there carry log de Rham extensions with the property described in Remark~\ref{rem:log_de_rham_extensions}. In particular, these models satisty the de Jong extension property.
\end{remark}

Theorem~\ref{introthm:surjective} is now implied by combining Remarks  \ref{rem:de_jong_proper_automatic},~\ref{rem:de_jong_siegel} and~\ref{rem:de_jong_exceptional} with the following result:

\begin{theorem}
    \label{thm:surjectivity}
Suppose that $(G,\mathcal{G},X,K)$ admits an ICM over $\Reg{E,(v)}$ satisfying the de Jong extension property. Then, for all $1\leqslant n\leqslant \infty$, the map of formal Artin stacks
  \[
    \widehat{\Ss}_{K}\to \BT[\mathcal{G}^c,\mbox{-}\mu_v]{n}
  \]
  is surjective, and smooth for finite $n$.
\end{theorem}

\begin{proof}
Note that, for every finite $n$, the classifying map $\widehat{\Ss}_{K}\to \BT[\mathcal{G}^c,\mbox{-}\mu_v]{n}$ is a formally smooth map of smooth formal Artin stacks over $\mathcal{O}$, and is therefore smooth. 

By smoothness, the map $|\widehat{\Ss}_{K}|\to |\BT[\mathcal{G}^c,\mbox{-}\mu_v]{n}|$ of the underlying topological spaces is open, and its surjectivity amounts to saying that its image intersects every connected component and that its image is closed under specialization. The first is easy, since $\BT[\mathcal{G}^c,\mbox{-}\mu_v]{n}$ is connected (see Theorem~\ref{thm:gmm}). 

For the surjectivity, observe the following consequence of Theorem~\ref{thm:vasiu_boundedness}: the map $|\BT[\mathcal{G}^c,\mbox{-}\mu_v]{\infty}|\to |\BT[\mathcal{G}^c,\mbox{-}\mu_v]{n}|$ is a homeomorphism for $n$ sufficiently large. Therefore, to complete the proof, it is sufficient to know that the map $|\widehat{\Ss}_{K}|\to |\BT[\mathcal{G}^c,\mbox{-}\mu_v]{\infty}|$ has image closed under specialization, and this follows from the de Jong extension property. 
\end{proof}

\begin{remark}
    [Non-emptiness of central leaves]
\label{rem:central_leaves}
The surjectivity in Theorem~\ref{thm:surjectivity} is equivalent to saying that all central leaves (which are essentially the fibers of the formal syntomic realization map) are non-empty.
\end{remark}

\subsection{Non-emptiness of strata}
\label{sub:non-emptiness}

Throughout the following, $\Ss_K$ will be an ICM for an unramified Shimura tuple $(G,\mc{G},X,K)$ over $\Reg{E,(v)}$. For visual simplicity, in what follows we will abuse notation and use the symbol $G^c$ instead of $G^c_{\mathbb{Q}_p}$ when discussing $F$-isocrystals with $G^c_{\mathbb{Q}_p}$-structure.

\begin{construction}
  [Newton strata]
By Construction~\ref{const:newton_map}, for any perfect $k(v)$-algebra $R$, we obtain maps
\begin{equation}\label{eq:newton-map}
 \Ss_K(R)\to \mr{BT}^{\mc{G}^c,\mbox{-}\mu_v}_\infty(R)\to \mathrm{Isoc}_{G^c}(R).
\end{equation}
For any class $[b]\in B(G^c)$, this gives a topological subspace $\Ss_{K}[b]\subset \Ss_{K}\otimes k(v)$ defined by the property that, for any perfect field $\kappa$, the $\kappa$-points of $\Ss_K$ contained in this subset are the ones whose image under the composition of \eqref{eq:newton-map} with the natural map $\mr{Isoc}_{G^c}(R)\to B(G^c)$ is $[b]$.  We call $\ms{S}_K[b]$ the \defnword{Newton stratum} associated with $[b]$. Newton strata do form a stratification by \cite[Theorem 3.6]{rapoport_richartz}.
\end{construction}

The next result proves Theorem~\ref{introthm:newton_strata}.

\begin{theorem}
\label{thm:newton_non_emptiness}
The Newton stratum $\Ss_{K}[b]$ is non-empty if and only if $[b]\in B(G^c,\mbox{-}\mu_v)$.
\end{theorem}
\begin{proof} By Lemma~\ref{lem:extension_cm_points}, every special point $x\colon\Spec F\to \Sh_K$ extends to a map $\Spec \Reg{F,(w)}\to \Ss_K$ for some place $w\vert v$ of $F$. The argument in~\cite[Proposition 1.3.10]{Kisin2022-eo} then gives the desired conclusion.
\end{proof}

\begin{remark} 
If $(G,\mc{G},X,K)$ satisfies the de Jong extension property, then, by Theorem~\ref{introthm:surjective}, it suffices to verify that there exists $\mathcal{Q}\in \BT[\mathcal{G}^c,\mbox{-}\mu_v]{\infty}(\overline{\Field}_p)$ with $[b_{\mathcal{Q}}] = [b]$. This follows from Lemma~\ref{lem:kottwitz_map}.
\end{remark}

\begin{definition}
    \label{defn:EO_strata}
Consider the composition of maps 
\[
\Ss_K\otimes k(v) \to \BT[\mathcal{G}^c,\mbox{-}\mu_v]{1}\otimes k(v) \to \mathcal{G}^c\text{-}\mathrm{zip}_{\mbox{-}\mu_v}
\]
of $k(v)$-stacks. The target admits a stratification by smooth locally closed substacks as in Remark~\ref{rem:stratification_G-zips}. An \defnword{Ekedahl--Oort stratum} of $\Ss_{K}\otimes k(v)$ is the pre-image of any such stratum under the above map.
\end{definition}

\begin{proposition}
\label{prop:eo_non-emptiness}
Suppose that $\Ss_K$ satisfies the de Jong extension property. Then every Ekedahl--Oort stratum of $\Ss_{K}\otimes k(v)$ is non-empty.
\end{proposition}
\begin{proof}
Given the surjectivity of the syntomic realization map (Theorem~\ref{introthm:surjective}), this comes down to the surjectivity assertion in Proposition~\ref{prop:apertures_to_gzips}.
\end{proof}

\subsection{Mapping properties for smooth inputs}
\label{sub:neronian_properties}

In this subsection, we will give more concrete mapping properties for ICMs in terms of pointwise criteria.

\begin{definition}
    [Pointwise crystallinity]
\label{defn:pointwise_conditions}
Suppose that we have an algebraic space $\mathcal{X}$ over $\Reg{E,(v)}$ equipped with a map $f:\mathcal{X}_\eta\to \Sh_K$. Let  $F/E_v$ be a finite extension and let $x\in \mathcal{X}(\Reg{F})$. Then, we say that $f$ is \defnword{crystalline at $x$} if $\mathbf{Et}_{K,p}|_x$ is crystalline. We say that $f$ is \defnword{pointwise crystalline} if it is crystalline at all such $x$. 

\end{definition}

\begin{proposition}
  [A pointwise mapping property for smooth inputs]
\label{prop:pointwise_smooth_canonical}
Suppose that one of the following holds:
\begin{enumerate}[label=(\alph*)]
    \item $p>3$;
    \item $p>2$ and $(G,X)$ is of pre-abelian type.
\end{enumerate}
Suppose that $\Ss_K$ is an ICM for $\Sh_K$ over $\Reg{E,(v)}$, and that $\mathcal{X}$ is a smooth $\Reg{E,(v)}$-scheme with generic fiber $X$, equipped with a map $f\colon X \to \Sh_K$. Then $f$ extends to a map $\mathcal{X}\to \Ss_K$ if and only if $f$ is pointwise crystalline.
\end{proposition}
\begin{proof}
 Suppose first that $p>3$. Given Theorem~\ref{thm:mapping_property}, it is enough to know that $\mathbf{Et}_{K,p}\vert_X$ lifts to a $(\mathcal{G}^c,\mbox{-}\mu_v)$-aperture over $\mathcal{X}$. But this is immediate from Theorem~\ref{thm:p_comp_smooth_characterization} and Remark~\ref{rem:pointwise_lifts}. 

 Next, suppose that $(G,X)$ is of Hodge type and that $p>2$. Then the same argument applies after taking Remark~\ref{rem:hodge_type_improvement} into account. The pre-abelian case can be deduced from this via standard arguments.
\end{proof}

We have the following immediate corollary:
\begin{corollary}
  [N\'eronian property for proper models]
\label{cor:proper_neronian}
Suppose that $\Ss_K$ is \emph{proper} over $\Reg{E,(v)}$, and with the conditions from \emph{Proposition~\ref{prop:pointwise_smooth_canonical}}, $\Ss_K$ is a N\`eron model for $\Sh_K$: In the notation of that proposition, every map $X\to \Sh_K$ extends to a map $\mathcal{X}\to \Ss_K$.
\end{corollary}

\begin{remark}
[Proper models of pre-abelian type]
  A necessary condition for properness of $\Ss_K$ is that $G^{\ad}$ be anisotropic, since this is known to be equivalent to the generic fiber $\Sh_K$ being proper. With this condition in hand, it follows from~\cite[Theorem 1]{mp:toroidal} that, when $(G,X)$ is of pre-abelian type, and $p>2$, the ICMs from Theorem~\ref{thm:abelian} are proper and so are N\`eron models of their generic fibers. 
\end{remark}

\begin{remark}
[Sufficient to check crystallinity at one point]
\label{rem:Diao-Yiao-all-or-nothing}
  Suppose one knew that $\mathbf{Et}_{K,p}$ is a \emph{pointwise semistable} local system on $\Sh_K$: In the abelian type case, this follows from the structure of the boundary of the integral model of a toroidal compactification; see \cite[Theorem 4.39]{Wu2025} and earlier results in \cite{mp:toroidal}. This is also true for the models in Theorem~\ref{introthm:exceptional} as seen during its proof. We expect that this is true in general as well. 
  
  Under such a pointwise semistability condition, results of Guo--Yang~\cite[Theorem 1.1]{guo2024pointwisec} and Diao--Yao~\cite[Theorem 1.1]{diao2025monodromyrigiditycrystallinelocal} allow us to considerably weaken the hypotheses of Proposition~\ref{prop:pointwise_smooth_canonical}. Namely, in this case, statement (1) from this proposition can be upgraded to only require $f$ to be crystalline at $x$ for \emph{some} $x\in\mc{X}(\ms{O}_F)$ in every connected component of $\mc{X}$. 
\end{remark}

\begin{remark}
    \label{rem:pre-abelian-diao-yao}
Suppose that $p>2$ and $(G,\mc{G},X,K)$ is an unramified Shimura tuple of \emph{pre-abelian type}. We do not know yet if $\mathbf{Et}_{K,p}$ is semistable over $\Sh_K$. However, since it suffices to know that the composition $X\to \Sh_K\to \Sh_{K^{\ad}}$ extends to a map $\mathcal{X}\to \Ss_{K^{\ad}}$, the strengthened mapping property from Remark~\ref{rem:Diao-Yiao-all-or-nothing} still holds for $\Ss_K$ itself.
\end{remark}

\subsection{Applications to cohomology}
\label{sub:applications_to_cohomology}

In this subsection, we state the obvious cohomological results about integral canonical models which follow from the results of \S\ref{ss:F-gauges-and-cohomology}. See \emph{loc.\@ cit.\@} for the terminology and notation used below.

\begin{notation}
Let $\mathscr{S}_K$ be an ICM for $(G,\mc{G},X,K)$. For a $\bb{Z}_p$-representation $\Lambda$ of $\mathcal{G}^c$ we write $\mathbf{Et}_{K}(\Lambda)$ for the $\bb{Z}_p$-local system on $\mr{Sh}_K$ given by $(\Lambda)_{\mathbf{Et}_{K,p}}$. We uniquely extend this to an association $V\mapsto \mathbf{Et}_{K}(V)$ for $V\in\mr{Rep}_{\bb{Q}_p}(G^c)$ so that $\mathbf{Et}_K(\Lambda)[\nicefrac{1}{p}]$ when $V=\Lambda[\nicefrac{1}{p}]$. 

Similarly, denote by $\mb{Syn}_K(\Lambda)$ the $F$-gauge on $\widehat{\mathscr{S}}_K$ obtained by twisting $\Lambda$ by the syntomic realization. Finally, we shorten the notation $T^+_\mr{crys}(\mb{Syn}_K(\Lambda))$ to $\mathbf{Crys}^+_K(\Lambda)$. We again uniquely extend this to associations $V\mapsto \mathbf{Syn}_K(V)$ and $V\mapsto \mb{Crys}^+_K(V)$ in the obvious way.
\end{notation}

\begin{proposition}
\label{prop:cohomology-crystalline-automorphic-local-system} 
Suppose that $\ms{S}_K$ is a proper ICM over $\mathscr{O}_{E,(v)}$, and let $C$ denote a completed algebraic closure of $E_v$. Then, for all $i\geqslant 0$ and $\bb{Q}_p$-representations $V$ of $G^c$, the $\mr{Gal}(\overline{E}_v/E_v)$-representation $H^i_\mathrm{et}\big(\Sh_K\otimes C,
\mathbf{Et}_K(V)\big)$ is crystalline and there is a natural identification of filtered $F$-isocrystals:
\begin{equation*}
    D_\mr{crys}\bigg(H^i_\mathrm{et}\big(\mr{Sh}_K\otimes {C},\mathbf{Et}_K(V)\big)\bigg)\simeq H^i_\mr{crys}\big((\ms{S}_{k(v)}/W(k(v))_\mr{crys},\mb{Crys}_K^+(V)\big).
\end{equation*}
Moreover, for all $1\leqslant n\leqslant \infty$ we have a natural identification
\begin{equation*}
    T_{\et}\big(Rf_\ast(\mb{Et}_K(\Lambda)/p^n)\big)\simeq R(f_\eta)_\ast T_{\et}(\mb{Syn}_K(\Lambda)/p^n).
\end{equation*}
\end{proposition}

\begin{proposition} Fix notation as in \emph{Proposition \ref{prop:cohomology-crystalline-automorphic-local-system}}. For all $i\geqslant 0$ we have a natural identification
\begin{equation*}
    H^1_f\bigg(H^i_\mathrm{et}\big(\mr{Sh}_K\otimes C,\mb{Et}_K(V)\big)\bigg)\simeq H^1\big(\mathscr{O}_{E_v}^\mr{syn},R^if_\ast\mb{Syn}_K(V)\big).
\end{equation*}
\end{proposition}

\subsection{Integral canonical models for Shimura stacks}
\label{sub:icms_shimura_stacks}

In this subsection, we will drop the assumption that $K$ be neat, but still assume that $(G,\mathcal{G},X,K)$ is an unramified tuple. In this generality, $\Sh_K$ is a Deligne--Mumford stack and not necessarily a scheme over $E$. However, we still have the map $\Sh_K\to \mathrm{Loc}_{\mathcal{G}^c(\Int_p)}$ classifying the $\mc{G}^c(\bb{Z}_p)$-local system $\mathbf{Et}_{K,p}$.

\begin{definition}
    [Stacky integral canonical models]
\label{defn:stacky_icms}
A \defnword{stacky} \defnword{ICM} for $\Sh_K$ over $\Reg{E,(v)}$ is a separated Deligne--Mumford stack $\Ss_K$ of finite type over $\Reg{E,(v)}$ with generic fiber $\Sh_K$ such that, for some neat compact open subgroup $K'\subset K$ with $K'_p=K_p$, normalization of $\Ss_K$ in $\Sh_{K'}$ is a limpid ICM for $\Sh_{K'}$. 
\end{definition}

We can now collect the properties of stacky ICMs.
\begin{remark}
    [Serre--Tate property for stacky ICMs]
\label{rem:stacky_Serre--Tate}
If $\Ss_K$ is a stacky ICM for $\Sh_K$, by Corollary~\ref{cor:full_faithfulness}, we see that there is an extension $\Ss_K\to \BT[\mathcal{G}^c,\mbox{-}\mu_v,\mathrm{alg}]{\infty}$ of the classifying map for $\mathbf{Et}_{K,p}$, which is $p$-adically formally \'etale.
\end{remark}

\begin{remark}
    [Mapping property for stacky ICMs]
\label{rem:mapping_property_stacky}
Using \'etale descent, the mapping property from Theorem~\ref{thm:mapping_property} holds verbatim for stacky ICMs, and one can even take $\mathcal{Y}$ to be an excellent separated $\eta$-normal Deligne--Mumford stack over $\Reg{E,(v)}$, with the following caveat: One has to impose the condition that, for each $\ell\neq p$, the restriction of $\mathbf{Et}_{K,\ell}$ over $\mathcal{Y}_\eta$ extends to a $K^c_\ell$-local system over $\mathcal{Y}$ represented by pro-algebraic spaces over $\Reg{E,(v)}$.
\end{remark}

As a consequence, we get the following generalization of Corollaries~\ref{cor:uniqueness} and~\ref{cor:functoriality}:
\begin{proposition}
   [Uniqueness and functoriality for stacky ICMs]
\label{prop:uniqueness_funct_stacky} 
Stacky ICMs are unique up to unique isomorphism and are functorial for maps of unramified tuples.
\end{proposition}
    
\begin{proposition}
    [Existence of stacky ICMs]
\label{prop:existence_stacky}
Suppose that for some neat open subgroup $K'\subset K$ with $K' = K^{',p}K_p$, $\Sh_{K'}$ admits a limpid ICM $\Ss_{K'}$ over $\Reg{E,(v)}$ in the sense of \emph{Definition~\ref{defn:ICM}}. Then $\Sh_K$ admits a stacky ICM over $\Reg{E,(v)}$. In particular, if one of the following conditions holds:
\begin{enumerate}
    \item  $p$ is large enough;
    \item $(G,X)$ is of abelian type; 
    \item $p>2$ and $(G,X)$ is of pre-abelian type.
\end{enumerate}
Then $\Sh_K$ admits a stacky ICM over $\Reg{E,(v)}$.
\end{proposition}
\begin{proof}
 The map $\Sh_{K'}\to \Sh_{K}$ is a finite Galois cover of Deligne--Mumford stacks over $E$. Let $\Delta$ be its Galois group: the action of $\Delta$ on $\Sh_{K'}$ extends to an action on $\Ss_{K'}$ by Theorem~\ref{thm:mapping_property} and the fact that $\mathbf{Et}_{K',p}$ is $\Delta$-equivariant. Therefore, we can set $\Ss_K = [\Ss_{K'}/\Delta]$: This is a Deligne--Mumford stack over $\Reg{E,(v)}$, and is clearly a stacky ICM for $\Sh_K$ by construction.
\end{proof}

\subsection{Integral canonical models over global rings of integers}
\label{sub:global_rings_of_integers}

\begin{setup}
    We drop our fixed prime $p$. Instead, take an open subgroup $K\subset G(\Adele_f)$ and let $\Xi_K$ be the set of primes $\ell$ where $K\cap G(\Rat_\ell)$ is \emph{not} hyperspecial. For $S\supseteq \Xi_K$ a set of primes, there exists a reductive group scheme $\mathcal{G}^S$ over $\Int[S^{-1}]$ with $\mathcal{G}^S(\Int_p) = K_p$ for $p\notin S$. Write $\mathcal{G}^{S,c}$ for the quotient of $\mc{G}^S$ by the Zariski closure of $Z(G)_{ac}$.
\end{setup}

\begin{remark}
    \label{rem:everywhere_unramified}
If $G$ admits a reductive model $\mathcal{G}_{\Int}$ over $\Int$, and we take $K = \mathcal{G}_{\Int}(\widehat{\Int})$, then $\Xi_K$ is empty. For instance if $G = \GSp_{2g,\Rat}$, then we can take $\mathcal{G}_{\Int} = \GSp_{2g,\Int}$. In this case, the associated Shimura stack $\Sh_K$ is the moduli stack $\mathsf{A}_{g,\Rat}$ over $\Rat$ of principally polarized abelian varieties of dimension $g$.
\end{remark}

\begin{definition}
    [Global integral canonical models]
Let $S\supseteq \Xi_K$ be a set of primes. An integral model $\Ss_K$ for $\Sh_K$ over $\Reg{E}[S^{-1}]$ is a \defnword{(stacky) limpid} \defnword{integral canonical model} for $\Sh_K$ if, for every $p\notin S$ and every place $v\vert p$ of $E$, $\Ss_K\otimes_{\Reg{E}}\Reg{E,(v)}$ is a (stacky) limpid ICM for $\Sh_K$ over $\Reg{E,(v)}$. 
\end{definition}

Let us list some results concerning these global ICMs that follow easily from the local versions. The version of these results for global stacky ICMs are formulated and proven similarly.
\begin{proposition}
    [Mapping property for global ICMs]
\label{prop:mapping_property_global_ICMs}
Let $\mathcal{Y}$ be an excellent separated normal Deligne--Mumford stack flat over $\Reg{E}[S^{-1}]$ with generic fiber $Y$ over $E$. Suppose that we are given a map $f\colon Y\to \Sh_K$ with the following properties:
\begin{enumerate}
    \item For every prime $p\notin S$ and every $v\vert p$, the map $Y\to \mathrm{Loc}_{\mathcal{G}^{S,c}(\Int_p)}$ classifying $\mathbf{Et}_{K,p}$ extends to a map $\mathcal{Y}\otimes_{\Reg{E}}\Reg{E,(v)}\to \BT[\mathcal{G}^{S,c}_{\Int_p},\mbox{-}\mu_v,\mathrm{alg}]{\infty}$.
    \item For all primes $\ell$, the $K^c_\ell$-local system $\mathbf{Et}_{K,\ell}\vert_Y$ extends to a $K^c_\ell$-local system on $\mathcal{Y}[\ell^{-1}]$ represented by pro-algebraic spaces over $\Reg{E,(v)}$. 
\end{enumerate}
Then $f$ extends to a map $\mathcal{Y}\to \Ss_K$.
\end{proposition}

\begin{proposition}
[Uniqueness and functoriality for global ICMs]
    \label{prop:global_icms_uniqueness_functoriality}
An ICM $\Ss_K$ over $\Reg{E}[S^{-1}]$ is determined uniquely up to unique isomorphism. In fact, for a map $f\colon (G_1,X_1)\to (G_2,X_2)$ of Shimura data with $f(K_1)\subset K_2$, and ICMs $\Ss_{K_i}$ over $\Reg{E_i}[S^{-1}]$, the canonical map $\Sh_{K_1}\to \Sh_{K_2}\otimes_{E_2}E_1$ extends to a map
\[
\Ss_{K_1}\to \Ss_{K_2}\otimes_{\Reg{E_2}}\Reg{E_1}.
\]
\end{proposition}

\begin{theorem}
[Existence of global ICMs]
    \label{thm:global_icms_existence} The following statements hold.
\begin{enumerate}
    \item If $(G,X)$ is of abelian type (resp.\@ pre-abelian type), then $\Sh_K$ admits a limpid ICM $\Ss_K$ over $\Reg{E}[\Xi_K^{-1}]$ (resp.\@ over $\Reg{E}[2^{-1}\Xi_K^{-1}]$;
    \item In general, $\Sh_K$ admits a limpid ICM over $\Reg{E}[S^{-1}]$ for some large enough set of primes $S$ containing $\Xi_K$.
\end{enumerate}
\end{theorem}

\begin{remark}
    [Global ICMs of Hodge type]
\label{rem:hodge_type_global_icms}
Suppose that $(G,X)$ is of Hodge type and that we have an embedding $(G,X)\hookrightarrow(\GSp(V),S^{\pm}(V))$ into a Siegel Shimura datum associated with a symplectic space $V$ over $\Rat$. Then the proof of Theorem~\ref{introthm:abelian} shows that the global (stacky) ICM $\Ss_K$ for $\Sh_K$ is obtained by taking the normalization of a moduli stack of polarized abelian varieties over $\Reg{E}[S^{-1}]$ in $\Sh_K$. In particular, $\mathsf{A}_g$, the stack of principally polarized abelian varieties, is a global stacky ICM for its generic fiber.
\end{remark}

\begin{remark}
  [Global ICMs of abelian type]
    For Shimura varieties of abelian type, a version of statement (1) in the theorem above is due to Lovering~\cite{Lovering2017-me}. His proof amounts to a careful finite-level analysis of a combination of Deligne's quotient construction from~\cite{deligne:corvallis} with Kisin's twisting construction from~\cite{kisin:abelian}, whereas the result here is an immediate consequence of our local characterization of ICMs at finite level.
\end{remark}

\begin{remark}
    [Automorphic bundles over global ICMs]
\label{rem:global_automorphic_bundles}
Let $\mathrm{Gr}_X$ be the projective homogeneous $G$-equivariant variety over $E$ parameterizing parabolic subgroups of $G$ of type $\{\mbox{-}\mu\}$. This has a natural integral model $\mathrm{Gr}^S_X$ over $\Reg{E}[S^{-1}]$ defined in the same fashion using parabolic subgroups of $\mathcal{G}^S$. The Hodge-filtered de Rham realizations from Construction~\ref{const:integral_de_rham} for varying places $v\vert p$ with $p\notin S$ can be glued together to get a canonical map
\[
\Ss_K \to \mathrm{Gr}^S_X/\mathcal{G}^S.
\]
As in Remark~\ref{rem:automorphic_sheaves}, for any finite extension $F/E$ and any parabolic subgroup $\mathcal{P}\subset \mathcal{G}^S_{\Reg{F}}$ of type $\{\mbox{-}\mu\}$, a functor from representations of $\mathcal{P}$ to vector bundles over $\Ss_K\otimes_{\Reg{E}}\Reg{F}$. In the abelian-type case, this recovers Lovering~\cite{Lovering2017-me}.
\end{remark}

\subsection{Integral canonical models at parahoric level}
\label{sub:parahoric_levels}
Here we explain that any unramified Shimura datum $(G,\mathcal{G},X)$ which admits an \'etale integral canonical model automatically admits a Pappas--Rapoport integral canonical model as studied in \cite{pappas2023padic}. We then apply results of Takaya from \cite{Takaya} to obtain Pappas--Rapoport integral canonical models for subhyperspecial (e.g., Iwahori) level. Finally, we give an analogue of the mapping property from Theorem \ref{thm:mapping_property} in this parahoric situation.

\begin{notation}
    In the following we denote by $\mathbf{Perf}_{A}$, for a $p$-adically complete ring $A$, the category of characteristic $p$ perfectoid spaces $S$ equipped with a map $S\to \Spd(A)$. For any other undefined piece of notation or terminology, we direct the reader to the comprehensive discussion in \cite[\S2-3]{pappas2023padic}. 
\end{notation}

\begin{definition}[{Tame local Shimura datum}] A \textbf{tame local Shimura datum} is a triple $(\mc{G},b,\{\mu\})$ where $\mc{G}$ is a parahoric group $\mathbb{Z}_p$-scheme with generic fiber $G$, $\{\mu\}$ is a conjugacy class of minuscule cocharacters of $G_{\overline{\mathbb{Q}}_p}$,  and $b$ is an element of $G(\breve{\mathbb{Q}}_p)$ inducing an element of $B(G,\mbox{-}\mu)$.
\end{definition}

\begin{definition}[{Local Shimura variety}]\label{defn:local-shimura-variety} Given a tame local Shimura datum $(\mc{G},b,\{\mu\})$ with reflex field $F$, we obtain a presheaf
\begin{equation*}
    \mc{M}_{\mathcal{G},b,\{\mu\}}^\mr{int}\colon \mathbf{Perf}_{\Reg{\breve{F}}}\to \mathbf{Set},
\end{equation*}
assigning to $S\to \Spd(\Reg{\breve{F}})$ the set of isomorphism classes of tuples $(S^\sharp, \mathscr{P}, \phi_\mathscr{P}, i_r)$, where:
\begin{itemize}
    \item $S^\sharp$ is the untilt of $S$ over $\Reg{\breve{E}}$ associated with $S \to \Spd(\Reg{\breve{F}})$,
    \item $(\mathscr{P},\phi_\mathscr{P})$ is a $\mc{G}$-shtuka on $S$ with one leg along $S^\sharp$ bounded by $\{\mu\}$ (see \cite[Definition 2.4.3]{pappas2023padic}),
    \item and $i_r$ is a framing (see loc.\@ cit.\@).
\end{itemize}
By \cite[\textsection 25.1]{Scholze2020-bx}, the presheaf $\mc{M}_{\mc{G},b,\{\mu\}}^\mr{int}$ is a small $v$-sheaf (in the sense of \cite[Definition 12.1]{Scholze2017-di}), which is called the \textbf{integral local Shimura variety} associated with the tame local Shimura datum $(\mc{G},b,\{\mu\})$.
\end{definition}

\begin{definition}[{Tame Shimura datum/tuple}] A \textbf{tame Shimura datum} is a triple $(G,\mathcal{G},X)$ where $(G,X)$ is a Shimura datum and $\mathcal{G}$ is a parahoric model of $G_{\mathbb{Q}_p}$. A \textbf{tame Shimura tuple} is a quadruple $(G,\mathcal{G},X,{K})$ where $(G,\mathcal{G},X)$ is a tame Shimura datum and $K\subset G(\mathbb{A}_f)$ is a compact open subgroup with $K=K_pK^p$ and where $K_p=\mathcal{G}(\mathbb{Z}_p)$. We say $(G,\mathcal{G},X,K)$ is \textbf{neat} if $K$ is.
\end{definition}

\begin{remark}
   Fix a neat tame Shimura tuple $(G,\mathcal{G},X,K)$ with reflex field $E$ and a place $v$ of $E$ of lying over $p$. In the following we shall abusively write $\Sh_K$ for $\Sh_K(G,X)_{E_v}$. We will write $\{\mu_v\}$ for the conjugacy class of $G_{\overline{\bb{Q}}_p}$ determined by the Hodge cocharacter of $(G,X)$.

As in Construction~\ref{const:Kc_etale_realization} for one may build pro\'etale $\mathcal{G}^c(\mathbb{Z}_p)$-torsor $\mathbf{Et}_{K,p}$ on $\Sh_K$. Here $\mathcal{G}^c$ is the parahoric model of $G^c$ induced by the parahoric model $\mathcal{G}$ of $G$ (e.g., see \cite[\S4.1]{DanielsYoucis}). As in the reductive case, we treat $\{\mu_v\}$ also as a conjugacy class of cocharacters of $G^c_{\overline{\bb{Q}}_p}$. 
\end{remark}

\begin{definition}[{Generic shtuka associated with a tame Shimura datum}]\label{defn:generic-shtuka} We denote by $\mathscr{P}_{K,E_v}$ the $\mathcal{G}^c$-shtuka bounded by $\{\mbox{-}\mu_v\}$ on $\Sh_K^\lozenge\to\Spd(E)$ defined by $U_\mathrm{sht}(\mathbf{Et}_{K,p})$ with notation as in \cite[\S3.2.2]{imai2024tannakianframeworkprismaticfcrystals}.
\end{definition}

\begin{construction}[{Formal neighborhoods of local Shimura varieties associated with $\overline{k}(v)$-points of an integral model of $\Sh_K$}] Let $\Ss_K$ be flat normal $\Reg{E_v}$-model of $\Sh_K$ equipped with an extension $\mathscr{P}_K$ of $\ms{P}_{K,E_v}$ to $\Ss_K^{\lozenge /}$ (see \cite[Definition 2.1.9]{pappas2023padic}).\footnote{Note that $\mathscr{P}_K$ is necessarily bounded by $\{\mu\}$ by \cite[Lemma 2.1]{daniels2023canonical}.} For any point $x$ of $\Ss_K(\overline{k}(v))$, the pullback $x^\ast \mathscr{P}_K$ defines a $\mc{G}^c$-shtuka over $\Spd(\bar{k}(v))$ which by \cite[Example 2.4.9]{pappas2023padic} naturally determines an object $(\mc{P}_x, \phi_x)$ of $\mr{BT}^{\mc{G}^c,\mbox{-}\mu_v}_\infty(\overline{k}(v))$ and by Construction \ref{const:newton_map} an element $b_x$ in $G(\breve{\mathbb{Q}}_p)$ inducing an element of $B(G^c,\mbox{-}\mu_v)$. 

The triple $(\mc{G}^c, b_x, \{\mu_v\})$ defines a tame local Shimura datum, and so we have the integral local Shimura variety $\mc{M}^\mr{int}_{\mc{G}^c,b_x,\{\mu_v\}}$. Moreover, the pair $(\mc{P}_x,\id)$ determines a point $x_0$ of the affine Deligne--Lusztig set $X_{\mc{G}^c}(b_x,\mu_v^{-1})(\bar{k}(v))$ (e.g., as in \cite[Definition 3.3.1]{pappas2023padic}), and so a point of $\mc{M}^\mr{int}_{\mc{G}^c,b_x,\{\mu_v\}}(\overline{k}(v))$ via the identification in \cite[Proposition 2.61.(1)]{GleasonConnectedComponents}. One may then consider the formal neighborhood $(\mc{M}^{\mr{int}}_{\mc{G}^c,b_x,\{\mu_v\}})^\wedge_{/x_0}$; see \cite[Definition 4.18]{GleasonSpecialization}.
\end{construction}

\begin{definition}[{Pappas--Rapoport integral models}] Fix a tame Shimura datum $(G,\mc{G},X)$. Then, a system $\{\Ss_K\}$ of normal flat $\Reg{E_v}$-models of $\Sh_K$, as one ranges over all sufficiently small neat tame Shimura tuples $(G,\mc{G},X,K)$, is a \textbf{Pappas--Rapoport integral canonical model} for $(G,\mc{G},X)$ if it satisfies conditions (i), (ii), (iii) and (iv) of~\cite[Definition 4.3]{DanielsYoucis}.

A Pappas--Rapoport integral canonical model is \textbf{crystalline} if, for any $K$, and for any mixed characteristic complete discrete valuation field $F$ over $E_v$ with perfect residue field, the following are equivalent for $x\in \Sh_K(F)$:
   \begin{enumerate}
     \item $x\in \Ss_K(\Reg{F})$;
     \item $\mathbf{Et}_{K,p}\vert_x$ is crystalline;
     \item $\mathbf{Et}_{K,p}\vert_x$ is potentially crystalline.
   \end{enumerate}
\end{definition}

\begin{proposition}\label{prop:ICMs-are-PR ICMS} Suppose that $(G,\mathcal{G},X)$ is an unramified Shimura datum admitting ICMs $\Ss_K$ for all sufficiently small neat $K^p\subset G(\Adele_f^p)$. Suppose also that these integral models satisfy the additional pointwise condition from Remark~\ref{rem:milne-moonen_comparison}. Then the collection $\{\Ss_K\}$ is a crystalline Pappas--Rapoport integral canonical model for $(G,\mathcal{G},X)$.
\end{proposition}
\begin{proof} 
The crystallinity condition is immediate from our definition of ICMs, so it is enough to check the conditions from~\cite[Definition 4.3]{DanielsYoucis} for Pappas--Rapoport integral canonical models. Conditions (i) and (ii) both follow from Remark \ref{rem:milne-moonen_comparison}. For condition (iii), observe each $\Ss_K$ carries a prismatic $F$-crystal in the sense of \cite[Definition 3.18]{imai2024tannakianframeworkprismaticfcrystals} with de Rham local system $\mathbf{Et}_{K,p}$. Therefore, applying the shtuka realization functor from \cite[Construction 3.19]{imai2024tannakianframeworkprismaticfcrystals} to this prismatic $F$-crystal yields a shtuka $\mathscr{P}_K$ on $\Ss_K^{\lozenge /}$. That $\mathscr{P}_K$ models $\mathscr{P}_{K,E}$ is clear as $(\mathscr{P}_K)_E$ is (by definition) $U_\mathrm{sht}(\mathbf{Et}_{K,p})=\mathscr{P}_{K,E}$. To check (iv), we may use the formal \'etaleness of the syntomic realization $\varpi$ and \cite[Proposition 3.32]{imai2023prismatic} (see also \cite[\S10.2]{gmm}) to reduce ourselves to constructing an isomorphism
\begin{equation*}
    \Theta_x\colon \left(\mc{M}^{\mr{int}}_{\mc{G}^c,b_x,\{\mu_v\}}\right)^\wedge_{/x_0}\isomto \Spf(R_{\mc{G},\{\mu_v\}}),
\end{equation*}
where the target is as in \emph{loc.\@ cit.\@}, and with the property that pullback of $\ms{P}_K$ on the target is the universal $\mathcal{G}^c$-shtuka on the source. But this is the content of \cite[Theorem 5.3.5]{ito2023deformation}.
\end{proof}

\begin{definition}[{Subhyperspecial parahoric}] Let $F/\mathbb{Q}_p$ be a finite extension and $H$ a reductive group over $F$. A subgroup of $H(F)$ is \textbf{subhyperspecial parahoric} if it is a parahoric subgroup in the sense of \cite[Definition 4.1.4]{KalethaPrasad} contained in $\mc{H}(\Reg{F})$ for a reductive $\Reg{F}$-model $\mc{H}$ of $H$.
\end{definition}

\begin{remark}[{Explicit description of subhyperspecial parahorics}] For a reductive model $\mathcal{H}$ of $H$ over $\Reg{F}$, the subhyperspecial parahorics contained in $\mathcal{H}(\Reg{F})$ are precisely those of the form $\mathrm{red}^{-1}(P(k))$ where $k$ is the residue field of $F$, $P\subset \mathcal{H}_k$ is a parabolic subgroup, and 
\begin{equation*}
    \mathrm{red}\colon \mathcal{H}(\Reg{F})\to \mathcal{H}(k)
\end{equation*}
is the reduction map. Equivalently, the subhyperspecial parahoric subgroups of $H(K)$ are $\mathcal{H}'(\Reg{K})$ for a parahoric group $\Reg{K}$-scheme $\mathcal{H}'$ (see \cite[Definition 2.6]{DanielsYoucis}) obtained as the dilatation (see \cite[\S A.5]{KalethaPrasad}) of a reductive $\Reg{F}$-model $\mathcal{H}$ along a parabolic subgroup $P\subset \mathcal{H}_k$. We call these $\Reg{K}$-models $\mathcal{H}'$ of $H$ \textbf{parabolic dilatations of $\mathcal{H}$}. We will just call $\mc{H}'$ a \defnword{parabolic dilatation} if $\mc{H}$ is implicit.
\end{remark}

\begin{example}[{Iwahoris are subhyperspecial}] Considering $\mathrm{red}^{-1}(B(k))$ for $B\subset \mathcal{H}_k$ a Borel subgroup we obtain the (or more precisely \emph{a}) Iwahori subgroup of $H(F)$ as in \cite[Definition 4.1.3]{KalethaPrasad}.
\end{example}

\begin{proposition}\label{prop:PR ICM-for-minorizing} Let $(G,\mathcal{G},X)$ be an unramified Shimura datum satisfying \emph{\textbf{(SV5)}}\footnote{Recall that this means that $Z(G)(\bb{Q})$ is discrete in $Z(G)(\bb{A}_f)$. This, in particular, implies that $G=G^c$.} and admitting ICMs satisfying the condition from \emph{Remark~\ref{rem:milne-moonen_comparison}}. Then, for any parabolic dilatation $\mathcal{G}'$ of $\mathcal{G}$, the tame Shimura datum $(G,\mathcal{G}',X)$ admits a crystalline Pappas--Rapoport integral canonical model.
\end{proposition}
\begin{proof} This is immediate from Proposition \ref{prop:ICMs-are-PR ICMS} and \cite[Theorem 6.20]{Takaya}.
\end{proof}

\begin{corollary}\label{cor:PR ICMs-for-pre-abelian-and-large-p} Suppose that $(G,\mathcal{G}',X)$ is a tame parahoric Shimura datum with $\mathcal{G}'$ a parabolic dilatation. Suppose further that either:
\begin{enumerate}
    \item $(G,X)$ is of abelian-type;
    \item $(G,X)$ is of pre-abelian type and $p>2$;
    \item $p$ is sufficiently large.
\end{enumerate}
Then, $(G,\mathcal{G}',X)$ admits a crystalline Pappas--Rapoport integral canonical model.
\end{corollary}

\begin{remark}[{Agreement with Kisin--Pappas--Zhou}] When $(G,\mathcal{G}',X)$ is of abelian type and $p>3$ the models constructed by Corollary \ref{cor:PR ICMs-for-pre-abelian-and-large-p} agree with the Kisin--Pappas--Zhou models constructed in \cite{Kisin2018-nm} and \cite{KPZ}. Indeed, this follows from the unicity of Pappas--Rapoport integral canonical models (see \cite[Theorem 4.3.1]{pappas2023padic}) and the results of \cite{DanielsYoucis}, building on previous work from \cite{pappas2023padic}, \cite{daniels2025conjecturepappasrapoport}, and \cite{daniels2023canonical}. 
\end{remark}

\begin{definition}[{Moduli stack of $(\mc{G},\mu)$-shtukas}] Let notation be as in Definition \ref{defn:local-shimura-variety}. Denote by $\mr{Sht}_{\mc{G},\{\mu\}}$ the $v$-stack on $\mathbf{Perf}_{\Reg{F}}$ sending $S\to \Spd(\Reg{F})$ to the groupoid of $\mc{G}$-shtukas over $S$ with leg along $S^\sharp$ bounded by $\{\mu\}$.
\end{definition}

The proof of the following result is along the same lines as that of Theorem~\ref{thm:mapping_property}. One has to replace the use of Theorem~\ref{thm:full_faithfulness}
 with~\cite[Theorem 2.7.7]{pappas2023padic}, and also appeal to \cite[Proposition 18.4.1]{Scholze2020-bx}.
\begin{theorem}\label{thm:PR ICM-mapping-property} 
Let $(G,\mathcal{G},X)$ be a tame Shimura datum admitting a crystalline Pappas--Rapoport integral canonical model $\{\Ss_K\}$ over $\ms{O}_{E_v}$. Suppose that $\mathcal{X}$ is an $\eta$-normal algebraic space over $\Reg{E_v}$. Then a map $\mathcal{X}_\eta\to \Sh_K$ extends to a map $\mathcal{X}\to \Ss_K$ if and only if the composition $\mathcal{X}_\eta^\lozenge\to \Sh^\lozenge_K\to \mr{Sht}_{\mc{G},\{\mbox{-}\mu_v\}}$ extends to a map $\mathcal{X}^{\lozenge/}\to \mr{Sht}_{\mc{G},\{\mbox{-}\mu_v\}}$.
\end{theorem}

\printbibliography
\end{document}